\let\oldvec\vec% Store \vec in \oldvec
\documentclass[invmat,numbook]{svjour} 

\let\newvec\vec% Store svjour vec in \newvec

\usepackage[adobe-utopia]{mathdesign}
\usepackage[full]{textcomp}

\DeclareSymbolFont{cmcal}{OMS}{cmsy}{m}{n}
\SetSymbolFont{cmcal}{bold}{OMS}{cmsy}{b}{n}
\DeclareSymbolFontAlphabet{\mathcal}{cmcal}

\let\vec\oldvec% Restore \vec from \oldvec
\usepackage{amsmath}
\let\vec\newvec% Restore \vec from \newvec

\usepackage[all,cmtip]{xy}
\usepackage{xcolor}
\usepackage{enumitem}

\usepackage{bm}
\usepackage{mathtools}
\usepackage[colorlinks=true]{hyperref} 

\usepackage[super]{nth}

\usepackage{etoolbox}
\tracingpatches
\makeatletter
\patchcmd{\@citeo}{\hskip0.1em}{\kern0.1em}{}{}
\patchcmd{\@citex}{\hskip0.1em}{\kern0.1em}{}{}
\makeatother

\usepackage{pifont}

\setlist[itemize]{label=$\bullet$,labelindent=8pt,leftmargin=*}
\setlist[enumerate]{labelindent=8pt,leftmargin=*}
\newcommand{\mb}[1]{\bm{#1}}

\newcommand{\eq}[1]{Eq.\ (\ref{#1})}
\newcommand{\eqn}[2]{
\begin{equation}\label{#1}
#2
\end{equation}
}

\newcommand{\eqnalign}[2]{
	\begin{equation}
	\begin{aligned}\label{#1}
	#2
	\end{aligned}
	\end{equation}
}

\newcommand{\Fr}[2]{\dfrac{#1}{#2}}
\newcommand{\fr}[2]{{\tfrac{#1}{#2}}}
\renewcommand{\tilde}{\widetilde}
\renewcommand{\Im}{\operatorname{Im}}
\renewcommand{\a}{\alpha}
\renewcommand{\b}{\beta}

\renewcommand{\d}{\delta} 
\renewcommand{\k}{\kappa}
\renewcommand{\l}{\lambda} 
\renewcommand{\r}{\rho}

\renewcommand{\S}{\Sigma}
 
\renewcommand{\t}{\tau}

\newcommand{\trytodefine}[2]{\ifdefined#1 \renewcommand{#1}{#2} \else \newcommand{#1}{#2}\fi}

\trytodefine{\G}{\Gamma}
\trytodefine{\U}{\Upsilon}
\trytodefine{\C}{\mathbb{C}}
\newcommand{\rd}{\partial}
\newcommand{\pr}{\prime}
\newcommand{\ppr}{{\prime\prime}}
\newcommand{\Des}{\operatorname{\mathfrak{Des}}}

\DeclareMathOperator{\Diff}{Diff}
\DeclareMathOperator{\End}{End}
\DeclareMathOperator{\Aut}{Aut}
\DeclareMathOperator{\Ker}{Ker}
\DeclareMathOperator{\tr}{tr}

\newcommand{\ceq}{\vcentcolon\hspace{.5pt}=}

\newcommand{\category}[1]{\textbf{#1}}
\newcommand{\HProb}{PAlg}
\newcommand{\Prob}{P}
\newcommand{\homotopycat}{{\category{ho}}}
\newcommand{\homotopycategory}{\homotopycat}
\newcommand{\UsL}{UsL}

\newcommand{\padj}{homotopy}
\newcommand{\CorAname}{{ungraded correlation algebra}}
\newcommand{\CorAnames}{{ungraded correlation algebras}}
\newcommand{\GCCorAname}{{correlation algebra}}
\newcommand{\GCCorAnames}{{correlation algebras}}
\newcommand{\Algebraic}{Algebraic}
\newcommand{\algebraic}{algebraic}
\newcommand{\overk}{}
\newcommand{\fieldk}{\Bbbk}
\newcommand{\vk}{\varkappa}

\newcommand{\comprandvars}{{\bm{\mathcal{S}}}}

\newcommand{\Comm}{{\color{blue}C}}
\newcommand{\classical}{{\color{blue}cl}}
\newcommand{\nc}{{\color{blue}nc}}
\newcommand{\Lie}{{\color{blue}L_\infty}}

\newcommand{\binarycomm}{{\color{blue}bC}}

\newcommand{\e}{\varepsilon} \newcommand{\ep}{\epsilon}
 
\newcommand{\g}{\gamma}

\newcommand{\La}{\Lambda}

\newcommand{\m}{\mu}
\newcommand{\n}{\nu}

\newcommand{\vr}{\varrho}

\newcommand{\s}{\sigma} 

\newcommand{\w}{\varphi}

\newcommand{\R}{\mathbb{R}}
\newcommand{\Z}{\mathbb{Z}}

\newcommand{\mq}{\mathfrak{q}}

\newcommand{\ma}{\mathfrak{a}}

\newcommand{\mc}{\mathfrak{c}}

\newcommand{\mm}{\mathfrak{m}}

\newcommand{\mg}{\mathfrak{g}}

\newcommand{\mF}{\mathfrak{F}}

\newcommand{\sA}{\mathcal{A}}

\newcommand{\sC}{\mathcal{C}}

\newcommand{\sG}{\mathcal{G}}

\newcommand{\sM}{\mathcal{M}}
\newcommand{\sN}{\mathcal{N}}

\newcommand{\sP}{\mathcal{P}}

\newcommand{\sS}{\mathcal{S}}

\newcommand{\sV}{\mathcal{V}}

\newcommand{\dia}{\diamond}

\title{Homotopy Theory of Probability Spaces I}
\subtitle{Classical independence and homotopy Lie algebras}
\author{
Jae-Suk Park\thanks{
This work was supported by IBS-R003-D1.
}}

\institute{
Center for Geometry and Physics, Institute for Basic Science (IBS), Pohang 37673, Republic of Korea
\\
Department of Mathematics, POSTECH, Pohang 37673, Republic of Korea
}

\dedication{Dennis Sullivan used to ask me ``What is a Quantum Field Theory?''\\
My temporary answer, only until and in honor of his $70$th birthday, is ``A homotopy probability space that is enhanced with a weight filtration generated by the Planck constant $\hbar$.''
}

\date{}

\begin{document}
\maketitle

\begin{abstract}
This is the first installment of a series of papers whose aim is to lay a foundation for homotopy probability theory by establishing its basic principles and practices. 
The notion of a homotopy probability space is an enrichment of the notion of an algebraic probability space with ideas from algebraic homotopy theory. This enrichment uses a characterization of the laws of random variables in a probability space in terms of symmetries of the expectation. The laws of random variables are reinterpreted as invariants of the homotopy types of  infinity morphisms between certain homotopy algebras. The relevant category of homotopy algebras is determined by the appropriate notion of independence for the underlying probability theory.
This theory will be both a natural generalization and an effective computational tool for the study of classical algebraic probability spaces, while keeping the same central limit.
This article is focused on the commutative case, where the laws of random variables are also described in terms of certain affinely flat structures on the formal moduli space of a naturally defined family attached to the given algebraic probability space. Non-commutative probability theories will be the main subject of the sequels.
\end{abstract}

\newpage

\setcounter{tocdepth}{3}
\tableofcontents
\parindent=0mm
\parskip=2mm

\section{Introduction}

Our purpose here is to arrange a marriage between probability and homotopy theory, more specifically an enrichment of the notion of a classical algebraic probability space with ideas from algebraic homotopy theory. The outcome of this will be a notion of \padj{} probability space wherein the laws of random variables are invariants of the homotopy types of certain $\infty$-morphisms. 
Commutative homotopy probability spaces will be both a generalization and an effective computational tool for the study of classical algebraic probability spaces, while keeping the same central limit theorem. 
A similar enrichment of the algebraically defined notion of a non-commutative probability space via algebraic homotopy theory also seems possible and will be the main subject in a sequel to this paper.

A probability theory, in general, refers to the study of probability spaces under a particular notion of independence. For example, a classical algebraic probability space over a field $\fieldk$ of characteristic zero is a unital associative and commutative algebra of random variables $A_{\classical}=\big(A, 1_A,\cdot\big)$ together with a unit-preserving linear map $\iota:A\rightarrow\fieldk$ called expectation. 
The target $\fieldk$ of the expectation morphism $\iota$ is itself a unital algebra but $\iota$ is not expected to preserve the algebra structure. In fact the deviation and higher deviations from $\iota$ being an algebra homomorphism measure the strength and depth of correlations between and among events thought of as singletons, pairs, triples, and so on.
The precise organizing data of this infinite hierarchy is called the {\em classical cumulant morphism} $\underline{\k}=\k_1, \k_2,\k_3,\dotsc$, where $\k_n$, for all $n\geq 1$, is a linear map from $S^n A$ to $\fieldk$, determined by the expectation morphism $\iota$ together with the multiplications in both $A_{\classical}$ and $\fieldk$ via the following combinatorial formula:

\eqn{ytry}{
\iota(x_1\dotsm x_n) 
=\sum_{\pi \in P(n)} \k\big(x_{B_1}\big)\dotsm\k\big( x_{B_{|\pi|}}\big),
}
where the sum is over all ``classical'' partitions $P(n)$ of the set $[n]= \{1,2,\dotsc,n\}$.
The $\k_1$ is the expectation $\iota$, $\k_2$ measures the failure of $\iota$ to be an algebra homomorphism, $\k_2(x,y) = \iota(x\cdot y)-\iota(x)\iota(y)$, and so on. 

On the other hand, the {\em moment morphism} $\underline{\m}=\m_1, \m_2,\m_3,\dotsc$, is a tuple of linear maps from $S^n A$ to $\fieldk$ for $n\geq 1$defined by $\m_n(x_1,\dotsc,x_n)\ceq \iota(x_1\dotsm x_n)$. Again, the moment morphism depends only on the multiplication in $A_{\classical}$ and on $\iota$.
Two random variables $x,y\in A$ are {\em classically independent} 
if $\k_n(x+y,\dotsc, x+y)= \k_n(x,\dotsc, x) +\k_n(y,\dotsc, y)$ for all $n\geq 1$. This implies that there are no correlations between $x$ and $y$ in the sense that their joint moment generating function can be factorized into the product of the moment generating functions for $x$ and for $y$.

In non-commutative probability theory \cite{Voi1,Voi2}, there are notions of cumulant morphisms -- free or boolean, each with a corresponding notion of independence defined by formulas that differ from \eq{ytry} only by replacing the classical partitions in the sum with non-crossing or interval partitions, respectively \cite{Sp,SW}. 

On the other hand, in algebraic homotopy theory, one deals with similar algebraic structures enriched by an underlying cochain complex. More specifically, one expects to study cochain maps that are algebra preserving morphisms {\em up to homotopy}. The eponymous homotopy is part of an infinite hierarchy of homotopies, a sequence of multilinear maps which together are combined into a package often called an $\infty$-morphism.

A cochain homotopy can similarly be considered as the first part of a different infinite hierarchy of homotopies, which upgrades the notion of cochain homotopy to $\infty$-homotopy of $\infty$-morphisms.\footnote{There is another infinite hierarchy which begins with $\infty$-morphisms, then $\infty$-homotopies, then $\infty$-homotopies of $\infty$-homotopies, and so on. This other hierarchy will not be considered in this paper.}
Homotopy associative algebras ($A_\infty$-algebras) and homotopy Lie algebras ($L_\infty$-algebras) together with the corresponding notions of $\infty$-morphism and $\infty$-homotopy are two famous examples of such enrichments, of associative algebras and Lie algebras respectively \cite{Sta63,SS}. 
In this paper and in a sequel, we shall encounter certain simple variants of $L_\infty$ and $A_\infty$-algebras called unital $sL_\infty$-algebra and unital $sA_{\infty}$-algebra, respectively.

These two ways of thinking about maps between algebras that are not algebra preserving can be combined, yielding the notion of a {\em homotopy probability space}, which is a natural cochain enhancement of a classical algebraic probability space. Then the organizing data of successive deviations from being an algebra morphism will be replaced by an $\infty$-morphism of homotopy algebras incorporating both the notions of algebra morphisms up to homotopy and independence. In this framework, the joint distribution (or law) of random variables is defined to be a certain invariant of the corresponding $\infty$-homotopy type. The commutative probability world with {\em classical} independence corroesponds to the {\em Lie} world of $\infty$-homotopy theory. The non-commutative probability world and {\em boolean} independence correspond to the {\em associative} world of $\infty$-homotopy theory. 
For the celebrated free probability theory of Voiculescu, which studies non-commutative probability spaces with {\em free} independence, there also seems to be a corresponding but hereto unknown $\infty$-homotopy theory,
which is under investigation.

Broadly speaking, this work is a spin-off from the author's program~\cite{P1,P} to characterize path integrals of quantum field theory in terms of the symmetries of the quantum expectation which should satisfy a certain coherence with a particular weight filtration generated by the Planck constant $\hbar$.  A similar idea will be adopted here in a simplified form,
without the ``$\hbar$-conditions,''
to study classical algebraic probability spaces after using infinitesimal symmetries of the expectation to enhance it to a homotopical version. In return, many results in this paper will be used as background materials in the forthcoming work on homotopy theory of quantum fields~\cite{P}.

A classical algebraic probability space $\xymatrix{A_{\classical}\ar[r]^\iota &\fieldk}$ generally comes with non-trivial {\em symmetries of the expectation morphism} $\iota$, due to the large size of its kernel. In fact, because the expectation morphism splits, there is always a faithful representation $\vr: \Aut_\fieldk(\ker\iota)\rightarrow \Aut_\fieldk(A)$ from the group of linear automorphisms of the kernel of $\iota$ into the group of linear automorphisms on $A$. This representation satisfies $\iota\circ \Im \vr =\iota$. It is not surprising that the corresponding space of coinvariants $A_{\Aut_\fieldk(\ker\iota)}\ceq A/DA$, where $DA$ is the submodule generated by elements of the form $\vr(g)(x) - x$, $g\in \Aut_\fieldk(\ker\iota)$ and $x \in A$, is isomorphic to $\fieldk$. However, this fact is even true for the coinvariants for very small subgroups of $\Aut_\fieldk(\ker\iota)$. In particular, we will discuss a natural $\Z_2$ subgroup of $\Aut_\fieldk(\ker\iota)$ (assuming that $\iota$ is not an isomorphism) such that $A_{\Z_2}$ is also ismorphic to $\fieldk$. The property $\iota\circ \Im \vr =\iota$ implies that $\iota$ factors into the quotient map into $A_{\Z_2}$ followed by a unique linear map from $A_{\Z_2}$ to $\fieldk$. 
We will, however, not use the fundamental symmetry, which completely determines the expectation morphism, but instead consider certain {\em infinitesimal symmetries}, which shall lead to a natural enrichment of the classical algebraic probability space to a homotopy probability space.

An infinitesimal symmetry of the classical algebraic probability space $\xymatrix{A_{\classical}\ar[r]^\iota &\fieldk}$ is defined to be a Lie algebra representation $\vr:\mg\rightarrow L\!\Diff_\fieldk(A)$, where $L\!\Diff_\fieldk(A)$ denotes the Lie algebra of linear algebraic differential operators on $A$, such that $\Im\vr(\mg) \in \Ker \iota$, i.e., $\iota\big(g.x \big)=0$ for all $g\in \mg$ and all $x\in A$ where $g.x \ceq \vr(g)(x)$. Then the expectation morphism $\iota$ uniquely factors through the coinvariants $A_\mg \ceq A/\mg. A$, which induces a linear map 
$\iota_\mg:A_\mg\rightarrow\fieldk$􏰆.
The map $\iota_\mg$ is unit-preserving in the sense that it takes the image of $1_{A}$ under the quotient map to $1$ in $\fieldk$. Now it is natural to wonder whether the structure of the classical algebraic probability space on $A$ can be used to induce the same kind of structure on $A_\mg$.

However, there is no natural assignment of an algebraic structure to $A_\mg$ from the algebra $A_{cl}$ unless $\mg. A \subset \Ker \iota$ is an ideal of $A$. This is the case if and only if the expectation morphism $\iota$ is an algebra homomorphism---meaning that there is absolutely no correlation. We have no interest in such cases. We shall resolve this difficulty with a little help from $sL_\infty$-homotopy theory together with critical revisions of some of the basic notions in classical algebraic probability theory. This will allow us to determine the law of random variables as a generating function of certain homotopy invariants of our revised structure.

One of our results is that there is the structure of a   \algebraic{} probability space on $A_\mg$, that is, $\xymatrix{\big(A_\mg\big)_{\Comm}\ar[r]^{\quad\iota_\mg} &\fieldk}$. \Algebraic{} probability spaces generalize classical algebraic probability spaces, and have corresponding notions of moments, cumulants and independence. The generalized theory shares the same central limit, {\em the Gaussian distribution}, with classical algebraic probability space. In a \algebraic{} probability space, the unital commutative and associative algebra of a classical algebraic probability space is replaced with a commutative {\em \CorAname{}}. In an \CorAname{}, the binary product is replaced with an (possibly infinite) sequence $\underline{m}=m_2,m_3,\dotsc$ of $n$-ary maps and one recovers the classical case when $m_n=0$ for all $n\geq 3$.

The structure of a \algebraic{} probability space on $A_\mg$ will be particularly useful if the infinitesimal symmetry is ``large enough'' so that $A_\mg$ is finite dimensional. Then we can determine all of the laws of random variables in the original classical algebraic probability space up to finite ambiguities. Denote by $\big(A_\mg\big)_{\Comm}=\big(A_\mg, 1_{A_\mg}, \underline{m}^{A_\mg}\big)$ the structure of an \CorAname{} on $A_\mg$. This structure is equivalent to a formal {\em torsion-free flat affine connection} $\nabla$ on the tangent space $T_o \sM\simeq A_\mg$ of the formal based manifold $\sM$, whose algebra of functions is the topological algebra $\fieldk[\![t_{A_\mg}]\!]\simeq\widehat{S(A_\mg^*)}$. The connection $\nabla$ defines distinguished {\em affine flat coordinates} on $\sM$ at the base point in formal power series $\fieldk[\![t_{A_\mg}]\!]$ such that the moment generating function associated to $A_\mg$ is determined by the affinely flat coordinates up to a finite number of unknowns corresponding to the expectation values of elements in a basis of $A_\mg$. In our context, in fact, there is a bijective correspondence between such formal torsion free flat connections and choices of affine flat coordinates, which we will exploit to simplify our definitions.

I am deeply indebted to Dennis Sullivan for his persisting op. cited question to me as well as for many in-depth discussions and comments during the last decade.
This article, which is dedicated to Sullivan's \nth{70} birthday, is a refined version of my Nov.\ 2, 2011 Einstein Chair Mathematics Seminar at the Graduate Center of CUNY \cite{P0}, where the notion of homotopy probability theory was first introduced and a comprehensive program toward its foundation and possible applications were sketched.
Some parts of the program have been already turned into the paper \cite{PP}, which includes a completely worked out non-trivial application of binary \padj{} probability spaces, with Jeehoon Park, and two short papers \cite{DPT1,DPT2} which have adopted an operadic viewpoint, with Gabriel C. Drummond-Cole and John Terilla. I am grateful to my collaborators for making those accounts available. I am also grateful to Cheolhyun Cho and Calin Lazaroiu for reading the draft 
and making useful comments.
I would also like to thank  Gabriel C. Drummond-Cole for his invaluable service for editing my messy first draft of mine, which has greatly improved and clarified the contents of
this paper.

\section{Summary}

The notion of a commutative homotopy probability space is the natural cochain enrichment of a \algebraic{} probability space. To begin with, we introduce the graded version of an \CorAname{}. 
\begin{definition}
A {\em \GCCorAname{}} is a pair $(V, \underline{M})$, where 
\begin{enumerate}
\item $V$ is a pointed graded vector space and  
\item $\underline{M}=M_2, M_3, \dotsc$ is a family of linear maps $M_n:S^n V\rightarrow V$ of degree $0$ for all $n\geq 0$.
\end{enumerate}
To be precise we should call this a {\em commutative} \GCCorAname{} because there is a corresponding notion of noncommutative \GCCorAname{} appropriate for non-commutative homotopy probability. However, as this paper is written entirely from the commutative point of view, we omit the word commutative throughout.

This data is required to satisfy the condition that $M_{n+1}\big(v_1,\dotsc, v_n, 1_V\big)=M_n\big(v_1,\dotsc, v_n\big)$ for all $n\geq 1$.
\end{definition}
There is a more complicated but equivalent definition in terms of products $m_n:V^{\otimes n}\to V$ which will be hinted at later.

This definition looks quite general, but here are some key examples:

\begin{itemize}
\item[--]A unital graded commutative associative algebra $(V,1_V,\cdot)$ is a \GCCorAname{} with $M_n$ being iterated multiplication using $\cdot$. Using $\underline{m}$ we instead get $m_2 = \cdot$ and $m_n=0$ for $n>2$. This correspondence is bijective. In this special case, if the degree of $V$ is concentrated to zero, we have a unital commutative associative algebra.
\item[--]If $V$ is finite dimensional, the family $\underline{m}$ corresponds to a torsion-free and flat graded affine connection in formal power series on the tangent bundle of a formal based supermanifold and the family $\underline{M}$ corresponds to flat affine coordinates in formal power series, both subject to a unit condition.
\end{itemize}
Then we can define homotopy probability algebras.
\begin{definition}
A {\em commutative homotopy probability algebra} on a graded vector space $\sC$ is a quartet $\sC_{\Comm}=\big(\sC, 1_\sC, \underline{M}, K\big)$, where the trio $\big(\sC, 1_\sC, \underline{M}\big)$ is a \GCCorAname{} and the trio $\big(\sC, 1_\sC, K\big)$ is a pointed cochain complex. 
A morphism of commutative homotopy probability algebras is a morphism of their underlying pointed cochain complexes. 
\end{definition}

Commutative homotopy probability algebras form a category, denoted by $\category{\HProb}_{\Comm}{\overk}$. 
Objects of its homotopy category $\homotopycat\category{\HProb}_{\Comm}{\overk}$ are those of $\category{\HProb}_{\Comm}{\overk}$, while morphisms are homotopy classes of pointed cochain maps. 
Note that the ground field $\fieldk$ is the initial object in both $\category{\HProb}_{\Comm}{\overk}$ and $\homotopycat\category{\HProb}_{\Comm}{\overk}$. 

We generalize the idea of a classical algebraic probability space $\xymatrix{A_{\classical}\ar[r]^\iota &\fieldk}$ to an object $\sC_{\Comm}$ with a morphism $[\mc]$ to the initial object $\fieldk$ in the homotopy category $\homotopycat\category{\HProb}_{\Comm}{\overk}$:
\[
\xymatrix{\sC_{\Comm}\ar[r]^{[\mc]} &\fieldk}.
\]
In this paper, we will only consider the commutative case, so for ease we will refer to commutative homotopy probability algebras and commutative homotopy probability spaces as merely {\em probability algebras} and {\em probability spaces}, reserving the modifier {\em classical} for the ordinary unrefined notions.

In practice, we study a \padj{} probability space in the category $\category{\HProb}_{\binarycomm}{\overk}$ by choosing a representative $\mc$ of the homotopy type $[\mc]$;
\[
\xymatrix{\sC_{\Comm}\ar[r]^\mc &\fieldk},
\] 
so that $\mc$ is a linear map of degree $0$ from $\sC$ to $\fieldk$, satisfying $\mc(1_\sC)=1$ and $\mc\circ K =0$. The morphism $\mc$ shall be called the expectation and is defined up to homotopy, so we should consider structures and quantities that are invariants of homotopy types of expectation morphisms.

A \padj{} probability algebra where $M_2$ is a commutative associative product and $M_n$ for all $n\geq 2$ is repeated multiplication using $M_2$, is called binary.

The (homotopy) category of binary \padj{} probability algebras is a full subcategory of the (homotopy) category of \padj{} probability algebras. A \padj{} probability space with binary \padj{} probability algebra is called a binary \padj{} probability space. For example, from the classical algebraic probability space $\xymatrix{A_{\classical} \ar[r]^\iota &\fieldk}$ with an infinitesimal symmetry $\vr:\mg\rightarrow \hbox{Diff}_\fieldk(A)$, we obtain a binary \padj{} probability spac $\xymatrix{\sA_{\binarycomm}\ar[r]^\mc & \fieldk}$ concentrated in non-positive degree as follows:

\begin{itemize}
\item
The binary \padj{} probability algebra $\sA_{\binarycomm}=(\sA, 1_\sA, \cdot, K)$ is obtained from the Koszul complex $\big(\sA = A\otimes S(\mg[-1]), K\big)$ for Lie algebra {\em homology} of the $\mg$-module $A$ \cite{Kos}, after changing the sign of the degree to make this a cochain complex. The zeroth cohomology $H^0$ is isomorphic to the coinvariants $A_\mg$;

\item
The pointed cochain map $\mc: (\sA, 1_A, K)\rightarrow (\fieldk, 1, 0)$ is defined to be $\mc=\iota$ on $\sA^0=A$ and zero in all other degrees.
\end{itemize}

To facilitate the description of the classical cumulant morphism $\underline{\k}$, we construct a family of functors $\Des_\La$ from the category $\category{\HProb}_{\Comm}{\overk}$ to the category $\category{\UsL}_\infty{\overk}$ of unital $sL_\infty$-algebras such that $\underline{\k}=\Des_\La(\iota)$ and every $\Des_\La$ induces a well-defined functor $\text{ho}\!\Des$ at the level of homotopy categories  [{\bf Theorem \ref{maina}}]. On a particular \padj{} probability algebra, two different functors in the family $\Des_\La$ assign the same unital $sL_\infty$-algebra [{\bf Definition \ref{GDesA}}]. On a particular morphism of \padj{} probability algebras, two different functors in the family $\Des_\La$ assign homotopic unital $sL_\infty$-morphisms [{\bf Definition \ref{GDesM}}] whose homotopy type only depends on the homotopy type of the morphism of \padj{} probability algebras. In particular the ground field $\fieldk$ after forgetting the multiplication is a unital $sL_\infty$-algebra $\fieldk=(\fieldk, 1,\underline{0})$, which is the initial object in both $\category{\UsL}_\infty{\overk}$ and $\homotopycat\category{\UsL}_\infty{\overk}$, and $\Des_\La \big(\fieldk\big)=\fieldk=\text{ho}\!\Des\big(\fieldk\big)$.

In the full sub-category of binary \padj{} probability algebras, for example, neither the differentials in objects nor the morphisms between objects are required to be compatible with the binary product. In particular, the differential may fail to be a derivation of the product and a morphism may fail to be an algebra map even up to homotopy. In fact these failures are essential in capturing {\em non-trivial correlations} among random variables algebraically. We measure these failures and successive failures up to homotopy by adopting the organizing principle for the notion of independence in classical algebraic probability theory to construct the family of functors $\Des_\La: \category{\HProb}_{\binarycomm}{\overk}\Longrightarrow \category{\UsL}_\infty{\overk}$ [{\bf Definition \ref{bdesc}}].

The extension of the notion of random variables to the homotopical setting is as follows. A finite set of random variables in a classical algebraic probability space is replaced with a finite dimensional space $\sV$ of homotopical random variables. The space $\sV$ is defined to be a finite dimensional graded vector space $V$, considered as an $sL_\infty$-algebra $(V, \underline{0})$ with zero $sL_\infty$-structure, together with a representative $sL_\infty$-morphism $\underline{\w}^V$ in a fixed homotopy type $[\sV]$ into the $sL_\infty$-algebra $\big(\sC,\underline{\ell}^K\big)$.

Note the following reformulation. Let $V$ be a graded vector space and $S(V)=V\oplus S^2 V\oplus \cdots$ be its reduced symmetric product. Regard $V$ as an $sL_\infty$-algebra $(V,\underline{0})$ with zero $sL_\infty$-structure and $S(V)$ as a cochain complex with zero differential. Then, for each $sL_\infty$-morphism $\underline{\w}:(V,\underline{0}) \rightarrow (\sC,\underline{\ell}^K)$ we construct a cochain map $\Pi^{\underline{\w}}:(S(V),0)\rightarrow (\sC,K)$ such that the cochain homotopy type of $\Pi^{\underline{\w}}$ depends only on the $sL_\infty$-homotopy type of $\underline{\w}$. Conversely, for each cochain map $\Pi:(S(V),0)\rightarrow (\sC,K)$ we construct an $sL_\infty$-morphism $\underline{\w}^\Pi:(V,\underline{0}) \rightarrow (\sC, \underline{\ell}^K)$ such that the $sL_\infty$-homotopy type of $\underline{\w}^\Pi$ depends only on the cochain homotopy type of $\Pi$. In both constructions we use the \GCCorAname{} structure on $\sC$ and rely on fact that the target $sL_\infty$-algebra $(\sC, \underline{\ell}^K)$ is in the image of the functor $\Des_\La$. In particular, these constructions are not merely the natural correspondence between $sL_\infty$-algebras and morphisms on the one hand and differentials on symmetric coalgebras and differential graded coalgebra morphisms on the other.

Write the image of $\mc$ under the functor $\Des_\La$ as $\underline{\phi}^{\mc,\underline{\La}}$. Then, we can consider the following two compositions in the category of $sL_\infty$-algebras and the category of cochain complexes, respectively:
\[
\xymatrix{\big(V,\underline{0}\big)\ar@{..>}[r]^{\underline{\w}^V}
\ar@{.>}@/_1.5pc/[rr]_{\underline{\k}^\sV\ceq \underline{\phi}^{\mc,\underline{\La}}\bullet \underline{\w}^V}
&\big(\sC, \underline{\ell}^K\big) \ar@{..>}[r]^{\underline{\phi}^{\mc,\underline{\La}}}
&(\fieldk,\underline{0})}
\quad\Big|\quad
\xymatrix{
\big(S(V),0\big)\ar[r]^{\Pi^{\underline{\w}^V}} 
\ar@{.>}@/_2pc/[rr]_{\qquad\underline{\m}^\sV = \mc\circ \Pi^{\underline{\w}^V}}
&\big(\sC, K\big) \ar[r]^\mc
& (\fieldk,0)}.
\]
These compositions, called the {\em cumulant morphism} and the {\em moment morphism} respectively, satisfy the following properties:
\begin{enumerate}
\item Both the cumulant and moment morphism depend only on the homotopy type $[\sV]$ of the $sL_\infty$-morphism $\underline{\w}^V$ and the homotopy type $[\mc]$ of the cochain map $\mc$.
\item
The moment and cumulant morphisms $\underline{\m}^\sV$ and $\underline{\k}^\sV$ are related with each other by the following equations for all $n\geq 1$ and $v_1,\cdots, v_n \in V$:
\[
\m^{\sV}(v_1,\dotsc,v_n) 
\ceq \sum_{\pi \in P(n)} \ep(\pi)\k^{\sV}\big(v_{B_1}\big)\dotsm\k^{\sV}\big( v_{B_{|\pi|}}\big),
\]
where $\ep(\pi)$ is a sign factor. 
\end{enumerate}

We should remark that our approach can be interpreted as a powerful computational tool to determine the law of homotopical random variables. For example, the cumulant morphism $\underline{\k}^{\sV}=\underline{\w}^{\mc,\underline{\La}}\bullet \underline{\w}^V$ can be computed by choosing the simplest choice of representative $\underline{\w}^V$ among $sL_\infty$-morphisms of the homotopy type $[\sV]$. We say a space $\sV$ of homotopical random variables is homotopically completely integrable if there is a representative $sL_\infty$-morphism $\underline{\w}^V$ such that $\w_n^V = \vk_n\cdot 1_\sC$, for all $n\geq 1$, where $\vk_n: S^n V\rightarrow \fieldk$ is a linear map of degree $0$. Then, we show that $\k^\sV_n = \vk_n$, for all $n\geq 1$. We say a probability space is homotopically completely integrable if every space $\sV$ of homotopical random variables is homotopically completely integrable. 
\begin{theorem}
A \padj{} probability space such that every degree zero element with zero expectation belongs to the image of the differential is homotopically completely integrable.
\end{theorem}
For example a binary probability space obtained from a classical algebraic probability space equipped with an infinitesimal symmetry whose coinvariants are one-dimensional is homotopically completely integrable.

We show that the unital $sL_\infty$-algebra $\sC_{\Lie}$ is always formal; in fact it is quasi-isomorphic to the unital $sL_\infty$-algebra $H_{\Lie}=(H, 1_H,\underline{0})$ with zero $sL_\infty$-structure $\underline{0}$ on the cohomology $H$ [{\bf Theorem \ref{lemc}}]. This allows us to define a complete space $\comprandvars$ of homotopical random variables as an object $\sS_{\Lie}=(\sS, 1_\sS, \underline{0})$ with zero $sL_\infty$-structure $\underline{0}$ together with a quasi-isomorphism $\underline{\w}^\sS$ from the fixed homotopy type $[\comprandvars]$ to the unital $sL_\infty$-algebra $\sC_{\Lie}$.
It follows that the underlying graded vector space $\sS$ of $\comprandvars$ is isomorphic to the cohomology $H$ but $\comprandvars$ is equipped with further canonical structure depending only on the homotopy type of the $sL_\infty$-quasi-isomorphism to $\sC_{\Lie}$. 
\begin{theorem}
On a complete space $\comprandvars$ of homotopical random variables, there is a canonical \GCCorAname{} structure $\big(\sS, 1_\sS, \underline{M}^\comprandvars\big)$ and a unital linear map $\iota_\comprandvars: \sS \rightarrow \fieldk$ satisfying
\begin{align*}
\iota_\sS\left(M^\comprandvars_n(s_1,\dotsc, s_n)\right)&=\m^\comprandvars_n(s_1,\dotsc, s_n)=
\sum_{\pi \in P(n)}\ep(\pi)\k^\comprandvars\big(s_{B_1}\big)\cdots \k^\comprandvars\big(s_{B_{|\pi|}}\big).
\end{align*}
\end{theorem}

Now we can complete the circle. On $\comprandvars$, via the structure
$\big(\sS, 1_\sS, \underline{M}^\comprandvars\big)$
of \GCCorAname{}, we have the structure of the \padj{} probability algebra $\sS_{\Comm}=\big(\sS, 1_\sS, \underline{M}^\comprandvars,0\big)$ with zero differential and a morphism $\iota_\comprandvars$ to the initial object $\fieldk$ in $\category{\HProb}_{\Comm}{\overk}$, so that we have a uniquely defined \padj{} probability space with zero-differential: 
\[
\xymatrix{\sS_{\Comm}\ar[r]^{\iota_\comprandvars} & \fieldk}.
\]
We also note that the first component of any unital $sL_\infty$-quasi-isomorphism 
 in the homotopy type $[\comprandvars]$ is a quasi-isomorphism from the \padj{} probability algebra $\sS_{\Comm}$ to the \padj{} probability algebra $\sC_{\Comm}$. Then, 
\begin{theorem}
We have $\sS_{\Lie}=\text{ho}\!\Des\left(\sS_{\Comm}\right)$, $\text{ho}\!\Des([f])= [\sS]$ and $\underline{\k}^\comprandvars= \text{ho}\!\Des\left(\iota^\comprandvars\right)$ as depicted in the following diagram:
\[
\xymatrix{
&&&\fieldk&&
\\
&&&\fieldk\ar@{=>}[u]&&
\\
& &\ar@{=>}[dll] \sS_{\Comm}\ar[ru]^{\iota_{\!\comprandvars}}\ar[rr]_{[f]}&&\sC_{\Comm}\ar@{=>}[drr] \ar[lu]_{[c]}&
 \\
\sS_{\Lie}\ar@{-->}[rrrrrr]_{\text{ho}\!\Des([f])}\ar@{-->}[rrruuu]^{\underline{\k}^\comprandvars}
&&&&&&
\sC_{\Lie}\ar@{-->}[llluuu]_{\text{ho}\!\Des([c])}
}.
\]
Here the inner diagram commutes in the homotopy category $\homotopycat\category{\HProb}_{\Comm}{\overk}$ of \padj{} probability algebras and the outer diagram commutes in the homotopy category $\homotopycat\category{\UsL}_\infty{\overk}$ of unital $sL_\infty$-algebras.
\end{theorem}

Next we consider the case when the cohomology $H$ is finite dimensional. 
We show that the deformation functor attached to the unital $sL_\infty$-algebra $\sC_{\Lie}=\Des_\La\big(\sC_{\Comm}\big)$ is pro-representable by the topological algebra $\widehat{S(H^*)}$ so that we have a formal based super manifold $\sM$ whose algebra of functions is $\widehat{S(H^*)}$ and whose tangent space $T_o\sM$ at the base point $o\in \sM$ is isomorphic to $H$. In fact, the space $\sM$ is the moduli space of complete spaces of homotopical random variables and each complete space $\comprandvars$ gives formal affinely flat coordinates at $o$.

Choose a basis $\{e_\a\}_{\a\in J}$ on $\sS\simeq H$ with the distinguished element $e_0=1_\sS$. 
Let $t_{\sS}=\{t^\a\}_{\a\in J}$ be the dual basis, which gives affine coordinates on $\sM$ around a formal neighbourhood of the base point via any unital $sL_\infty$-quasi-isomorphism 
(of the homotopy type $[\comprandvars]$) from $\sS_{\Lie}$ to $\sC_{\Lie}$. Then $\left\{\rd_\a \ceq \Fr{\rd}{\rd t^\a}\right\}$ is a formal frame field on $T_o\sM$. We use the same symbol $\rd_\a$ for the graded derivation of the formal power-series ring $\fieldk[\![t_\sS]\!]$ such that $\rd_\a t^\b = \d_\a{}^\b$.
We will define a $(2,1)$-tensor $A^\comprandvars_{\a\b}{}^\g$ and a $1$-tensor $T_\comprandvars^\g$ in $\fieldk[\![t_\sS]\!]$ using the structure constants of the \GCCorAname{} structure on $\comprandvars$. Then:

\begin{theorem}
\begin{enumerate} \item
The formal $(2,1)$-tensor $\{A^\comprandvars_{\a\b}{}^\g\}$ is the connection one-form for a graded flat and torsion-free affine connection $\nabla$ on $T\sM$ in a formal neighborhood of $o \in \sM$, i.e., $\nabla_{\rd_\a} \rd_\b =\sum_\g A^\comprandvars_{\a\b}{}^\g\rd_\g$, such that $A_{0\b}{}^\g =\d_{\b}{}^\g$.
\item The formal $1$-tensor $\{T_\comprandvars^\g\}$ gives affine flat coordinates 
for the connection $\nabla$, that is,
$\nabla_{\tilde\rd_\a}\tilde \rd_\b =0$ where $\tilde\rd_\a = \sum_\g \rd_\a T_\comprandvars^\g\rd_\g$. These affine flat coordinates satisfy the following equations:
\begin{align*}
T_\comprandvars^\g\big|_{\underline{t}=\underline{0}} &=0
,\\
\rd_\b T_\comprandvars^\g\big|_{\underline{t}=\underline{0}} &=\d_\b{}^\g
,\\
\rd_0 T_\comprandvars^\g&= T_\comprandvars^\g +\d_0^\g.
\end{align*}
\end{enumerate}
\end{theorem}
Let $Z_\comprandvars$ be the moment generating function of $\comprandvars$:
\[
Z_\comprandvars\ceq 1 + \sum_{n=1}^\infty\Fr{1}{n!} \sum_{\r_1,\dotsc,\r_n \in J} 
t^{\r_n}\cdots t^{\r_1} \m^{\comprandvars}_{n}\big(e_{\r_1},\dotsc, e_{\r_n}\big)
\in \fieldk[\![t_\sS]\!]
.
\]
Then,
\begin{theorem}
The affine flat coordinates $\{T_\comprandvars^\g\}$ determine the moment generating function via the equation
\[
Z_\comprandvars= 1+\sum_{\g\in J} T_\comprandvars^\g\; \iota^\comprandvars(e_\g)
\]
up to the finite unknowns $\{\iota^\comprandvars(e_\g)\}$ in $\fieldk$.
Furthermore, the moment generating function $Z_\comprandvars$ satisfies the following system of formal differential equations that, for all $\a,\b \in J$:
\begin{align*}
\left(\rd_\a\rd_\b- \sum_{\g\in J} A^\comprandvars_{\a\b}{}^\g \rd_\g\right)Z_\comprandvars=0,\\
\left(\rd_0-1\right)Z_\comprandvars=0.\\
\end{align*}
\end{theorem}

It should be noted that a space $\sV$ of homotopical random variables does not need to be complete and incomplete spaces of homotopical random variables are an important ingredient of the theory of \padj{} probability spaces. In general the moment generating function of the space $\sV$ satisfies a system of higher order ($n\geq 2$) formal PDEs (which collapse to ODEs when $\sV$ is one-dimensional) if $H$ is finite dimensional. Even in certain cases where $H$ is of infinite dimension, we may obtain a system of differential equations.

This paper is organized as follows.

Section~\ref{sec:preludes} is devoted to a gentle introduction to some background material: In Sections~\ref{subs:classical algebraic probability spaces}~and~\ref{subs:cats for algprob}, we recall classical algebraic probability spaces and reformulate them category-theoretically. In Sect.~\ref{subs:Linfty}, we review and examine the homotopy category of unital $sL_\infty$-algebras in a combinatorial language based on partitions after some cultural background.

Section~\ref{sec:infsym} is mainly about symmetries of classical algebraic probability spaces and their homotopical realization.
In Sect.~\ref{subs:symmetries}, we discuss the general notion of symmetries of the expectation map and describe the fundamental symmetry. In Sect.~\ref{subs:infsyms}, we specialize to consider infinitesimal symmetries that are defined in terms of certain Lie algebra representations.
We then give four examples---the Gaussian distribution, a probability theoretic realization of Griffiths' period integrals, matrix Airy integrals, and semi-circular distributions, to demonstrate the power of the notion of infinitesimal symmetries in practice. 
In Sect.~\ref{subs:realization}, we construct a homotopical realization of a classical algebraic probability space with an infinitesimal symmetry, which provides us a prototype binary \padj{} probability spaces. The general theory of these binary \padj{} probability spaces is developed in Sect.~\ref{subs:binary prob space}, which also contains a worked out example of a homotopically completely integrable theory.

In Section~\ref{sec:intermission} we reexamine the fundamental assumption of classical algebraic probability theory that random variables form a commutative and associative algebra. 
In Sect.~\ref{subs: pre-alg}, we introduce the notion of pre-algebraic probability space where we relax the assumption that there is underlying binary product in the space of random variables. We define the cumulant morphism and independence and show that pre-algebraic probability spaces share the same central limit as classical algebraic probability spaces. In Sect.~\ref{subs: when assoc}, we consider an obstruction theory to determine whether a pre-algebraic probability space has an underlying commutative and associative algebra. This leads us to define {\CorAnames{}} and a non-differential version of \padj{} probability space in Sections~\ref{subs: comm prob space}~and~\ref{subs:formal geo I}. Then, we examine the formal geometric picture of finite dimensional probability spaces in this context.

Section~\ref{sec:maintheory} is devoted to the underlying theory of \padj{} probability spaces in full generality.
Sections~\ref{subs:corralg}--\ref{subs:completely integrable} contain many definitions in full generality: those of \GCCorAnames{}, \padj{} probability algebras and spaces, the descendant functor to the (homotopy) category of unital $sL_\infty$-algebras, spaces of homotopical random variables, associative correlation functions, and homotopically completely integrable \padj{} probability spaces.
In the final Sections~\ref{subs:formality}--\ref{subs:geometry II} we prove some facts that are true in general in this framework. We show that the unital $sL_\infty$-algebras in the codomain of the descendant functor are formal and vanish in homology. 
We study complete spaces of homotopical random variables and show that each complete space is endowed with a \padj{} probability space structure with zero differential. 
Finally, we consider the case that cohomology is finite dimensional and the case that the cohomology has a finite super-selection sector. Adopting ideas of deformation theory we construct formal graded torsion-free affine flat connection and affine flat coordinates on the associated moduli space.

\section{Conventions}
Fix a ground ring $R$ with unit.

A \emph{graded module} $V$ is the direct sum of modules indexed by \emph{degree}:
\[V=\bigoplus_{j\in \Z}V^j.\] By an element of $V$ we shall mean an element of one particular $V^j$ unless otherwise specified. If $x$ is a non-zero element of $V$ in the summand $V^j$ then $|x|$ denotes the degree $j$ of $x$. We will use this notation for arbitrary elements of $V$; although the degree of zero is undefined, this will not create ambiguity. We always consider ungraded modules as graded modules concentrated in degree zero. A degree $k$ map between graded modules $V$ and $W$ consists of a sequence of homomorphisms $V^j\to W^{j+k}$. The dual of a graded module $V$ is the graded module $V^*$ which has $(V^*)^j\coloneq (V^{-j})^*$.

A {\em pointed (graded) module} is a (graded) module $V$ equipped with a specified element $1_V$ (of degree zero) called the \emph{unit}. A pointed map between two pointed modules preserves the unit. We shall be sloppy about the distinction among $V$ considered as a module, $V$ considered as a pointed module, and the pair $(V,1_V)$ considered as a pointed module with underlying module $V$.

The \emph{reduced free tensor algebra} generated by $V$ is $T(V)= \bigoplus_{n=1}^\infty T^n V$, where $T^n V=V^{\otimes n}$. The \emph{reduced free supersymmetric algebra} generated by $V$ is $S(V)=\bigoplus_{n=1}^\infty S^n V$, which is the reduced free tensor algebra $T(V)$ modulo the ideal generated by $x\otimes y -(-1)^{|x||y|}y\otimes x$.  Both of these inherit a grading induced from that in $V$ by requiring it to be additive over the tensor product. We say an element in $T^n V$ or in $S^n V$ has \emph{word length} $n$.

A \emph{cochain complex} $C$ is a graded module equipped with a homomorphism $d: C^\bullet \rightarrow C^{\bullet +1}$ of degree $1$ called the \emph{differential} which satisfies $d^2=0$. A \emph{cochain map} between cochain complexes $C$ and $C^\pr$ is a homomorphism $f:C^\bullet \rightarrow C^{\prime\bullet}$ of degree $0$ such that $f d = d^\pr f$. A cochain homotopy between $C$ and $C^\pr$ is a homeomorphism $s:C^\bullet \rightarrow C^{\prime\bullet -1}$ of degree $-1$. It is trivial to check that $s d + d^\pr s$ is a cochain map, which is said to be \emph{cochain homotopic to zero} by the cochain homotopy $s$. This is denoted by $s d + d^\pr s\sim 0$.  The two cochain maps $f$ and $\tilde f$ are called \emph{homotopic} or said to \emph{have the same cochain homotopy type}, if $f-\tilde f\sim 0$. This is denoted $f\sim \tilde f$.
Two cochain complexes $C$ and $C^\pr$ are said to be \emph{homotopy equivalent} or \emph{have the same homotopy type} if there are cochain maps $f:C\rightarrow C^\pr$ and $g: C^\pr \rightarrow C$ such that $g\circ f\sim I_C$ and $f\circ g \sim I_{C^\pr}$.

A \emph{pointed cochain complex} is a graded module  which is both a cochain complex and a pointed module, with the compatibility that the unit is in the kernel of the differential, but not the image of the differential. A map of pointed cochain complexes is a cochain map which preserves the unit; a pointed cochain homotopy is required to kill the unit.

Henceforth, unless otherwise specified, assume $R$ is a ground field $\fieldk$ of characteristic zero. For certain results, we will assume that $\fieldk$ contains $\sqrt{N}$ for all natural numbers $N$; this will be specified when it is needed. All vector spaces, chain complexes, and algebras of all sorts are taken over $\fieldk$ unless otherwise specified.

Let $\hbox{Perm}_n$ be the group of permutations of the set $[n]=\{ 1,\ldots, n\}$. For a permutation $\s \in \hbox{Perm}_n$, we define the map $\hat\s :T^n V\rightarrow T^n V$ by requiring, for elements $x_1,\dotsc, x_n \in V$,
\[
\hat\s\left(x_1\otimes x_2\otimes\cdots \otimes x_n\right) =
\ep(\s)x_{\s(1)}\otimes x_{\s(2)}\otimes\cdots\otimes x_{\s(n)}.
\]
Here $\ep(\s)=\pm 1$ is the Koszul sign determined by decomposing $\hat\s$ as composition of transpositions $\hat\t$, where $\hat\t:x\otimes y \mapsto (-1)^{|y||x|} y\otimes x$.

A linear map $L_n$ of degree $|L_n|$ from $T^n V$ to a graded vector space $W$ descends to a linear map from $S^n V$ to $W$ if 
\[L_n\big(x_1\otimes\cdots\otimes x_n\big)
= \ep(\s)L_n\big(x_{\s(1)}\otimes\cdots\otimes x_{\s(n)}\big).
\] 
We shall use the notation $L_n\big(x_1\otimes\cdots\otimes x_n\big)$ and $L_n\big(x_1,\dotsc, x_n\big)$ interchangably.
For a given set of $n$ elements $x_1,\dotsc, x_n \in V$,
the Koszul sign $\ep(\pi)$ for a partition $\pi =B_1\sqcup \cdots \sqcup B_{|\pi|} \in P(n)$, is defined to be the Koszul sign $\ep(\s)$ of the permutation $\s$ defined by
\[
x_{B_1}\otimes x_{B_2}\otimes\cdots\otimes x_{B_{|\pi|}}
=x_{\s(1)}\otimes x_{\s(2)}\otimes \cdots\otimes x_{\s(n)},
\]
where $x_B = x_{j_1}\otimes x_{j_2}\otimes \cdots\otimes x_{j_{|B|}}$ if $B=\left\{x_{j_1}, x_{j_2}, \cdots, x_{j_B}\right\}$.

A \emph{partition} of the set $[n]$ is a decomposition $\pi = B_1\sqcup B_2\sqcup \cdots\sqcup B_{|\pi|}$ of $[n]$ into pairwise disjoint non-empty subsets $B_i$ called blocks. We denote the number of blocks in the partition $\pi$ by $|\pi|$ and the size of a block $B$ by $|B|$. 
We shall use the {\em strictly ordered} representation for a partition. That is, blocks are ordered by the maximum element of each block and each block is ordered via the ordering induced from the natural numbers. For $k,k^\pr$ in $[n]$, we use the notation $k \sim_\pi k^\pr$ if both $k$ and $k^\pr$ belong to the same block in the partition $\pi$ and the notation $k\nsim_\pi k^\pr$ otherwise.
We denote the set of all partitions of $[n]$ by $P(n)$.
If $f_r$ is a sequence of multilinear operators $S^r X\to Y$ or $X^{\otimes r}\to Y$ for $r\ge 1$ and $B=\{j_1,\ldots, j_r\}$ is a block of a partition $\pi$, the notation $f\big(x_{B}\big)$ is taken to mean $f_r(x_{j_1},\dotsc, x_{j_r})$.

We denote the binomial coefficient $\fr{n!}{k!(n-k)!}$ by $\binom{n}{k}$. For an operator $F$ depending on $\t$ we use the notation $\dot{F}$ to mean $\fr{d}{d\t}F$.

\section{Preludes}\label{sec:preludes}

This section is a preparation for the marriage of classical algebraic probability theory with algebraic homotopy theory; its intention is to gently relay an elementary but self-contained story from each side. In particular, we shall specialize to the commutative version of classical algebraic probability theory and relate it with the homotopy theory of Lie algebras. 

\subsection{Classical algebraic probability spaces and classical independence}\label{subs:classical algebraic probability spaces}

The original notion of a classical algebraic probability space, with a measure theoretic foundation, is due to Kolmogorov. 
We do not want to go back that far but instead begin by recalling (commutative) classical algebraic probability spaces. 
Algebraic probability spaces, pioneered by Voiculescu, focus primarily on the algebra of random variables 
and their expectations and correlations.
This makes it so that non-commutative generalizations are natural; see \cite{Voi2,VDN,Voi3} for introductions. 
We also refer the textbook \cite{Tao} for a gentle discussions on the transition from measure theoretic probability spaces 
to classical algebraic probability spaces. 

\begin{definition}
A classical algebraic probability space over $\C$ is a pair $(A_\nc; \iota)$, 
where $A_\nc=(A, 1_A,\cdot)$ is a unital commutative and associative algebra over $\C$ with binary product $\cdot$ 
and unit $1_A$ and $\iota$ is a $\C$-linear functional, called the \emph{expectation}, from the $\C$-vector space $A$ to 
$\C$ such that $\iota(1_A)=1$.
Elements of $A$ are called random variables. For a random variable $x$, the value $\iota(x)$ is called the expectation value 
of $x$ and is sometimes denoted by $\left<x\right>$. 
\end{definition}

\begin{remark}
In the literature, there are often additional requirements placed on a classical algebraic probability space. 
For instance, it is often assumed that $A_{\classical}$ is a $\C^*$-algebra, that $\iota$ is a state, and so on. 
The above definition can be considered ``the bare minimum'' for classical algebraic probability theory. 
Every statement in this paper, unless otherwise specified, is valid if one replaces $\C$ with our arbitrary characteristic zero 
field $\fieldk$.
\hfill$\natural$
\end{remark}

The expectation value $\iota(x)$ of a random variable $x$ is also called the first moment of $x$.
The higher moments of $x$ are $\iota(x^n)=\left<x^n\right>$ for $n\geq 2$. 
One often gathers all the moments of $x$ together into a generating function 
$Z(t)=\iota\left( e^{t x}\right) \in \fieldk[\![t]\!]$, where $e^{tx} =1_A + tx +\Fr{1}{2!}t^2 x^2+\Fr{1}{3!}t^3 x^3+ \cdots$.
More generally, for a set of random variables $\{x_1,\dotsc, x_k\}$, one associates the family of joint moments:
\eqn{aaz}{
\left\{ \m_n(x_{j_1},\dotsc, x_{j_n}\big)
\ceq \iota(x_{j_1}\cdots x_{j_n})\big| n\geq 1, 1\leq j_1,\dotsc, j_n \leq k\right\}
.}
The joint distributions (also called \emph{(the law)}) of the set $\{x_1,\dotsc, x_k\}$ of random variables is the map
\[
\hat\m: \fieldk[t_1,\dotsc, t_k]\rightarrow \fieldk
\]
defined by $\hat\m(P(t_1,\dotsc, t_k)) = \iota\big(P(x_1,\dotsc, x_k)\big)$.
It is clear that $\hat\m$ is determined completely by the family of joint moments, 
which is conveniently described by its generating function 
\eqn{avaa}{
Z(t_1,\dotsc, t_k)\ceq \iota\left( e^{\g}\right)=1+
\sum_{n=1}^\infty\Fr{1}{n!}\m_n(\g,\dotsc,\g)\in \fieldk[\![t_1,\dotsc,t_k]\!],
}
where $\g =\sum_{i=1}^k t_i x_i$. 
We call $e^{\g}$ the \emph{moment generating density} of the set $\{x_1,\dotsc, x_k\}$ of random variables.

There is an equivalent description of the law of random variables in terms of the family of \emph{joint cumulants}. 
The cumulant morphism is the family $\underline{\k}=\k_1, \k_2, \k_3,\dotsc$, where $\k_n$ is a linear map 
from $S^n(A)$ to $\fieldk$ for every $n\geq 1$, determined by the expectation morphism $\iota$ 
and the products in $A$ and $\fieldk$. The construction of the cumulants involves a sum over all partitions 
of the set $[n]= \{1,2,\dotsc,n\}$. 

\begin{definition}
The {\em cumulant morphism} is the family of operations $\underline{\k}=\k_1, \k_2, \k_3,\dotsc$ defined recursively 
by the equation
\eqn{aab}{
\iota(x_1\cdots x_n) 
=\sum_{\pi \in P(n)} \k\big(x_{B_1}\big) \cdots\k\big( x_{B_{|\pi|}}\big).
}
\end{definition}

\begin{remark}
\eq{aab} suffices to define the cumulant morphism 
$\underline{\k}=\k_1, \k_2, \k_3,\cdots$ recursively since $P(n)$ contains the partition $\{1,2,\dotsc, n\}$ 
and every other partition has blocks of length strictly less than $n$.
From the following table of the classical partitions for $n\le 3$:
\begin{align*}
&P(1)=\{1\}
,\\
& P(2)=\{1,2\},\{1\}\sqcup \{2\}
,\\
&P(3)=\{1,2,3\}, 
\;\;\{1,2\}\sqcup \{3\},\{2\}\sqcup \{1,3\},\{1\}\sqcup \{2,3\},
\;\;\{1\}\sqcup\{2\}\sqcup\{3\}
,\\
\end{align*}
we have
\begin{align*}
\k_1(x)=&\iota(x),
\\
\k_2(x_1,x_2)=& \iota(x_1\cdot x_2) - \k_1(x_1)\k_1(x_2)
\\
\k_3(x_1,x_2,x_3)=&
\iota(x_1\cdot x_2\cdot x_3)
-\k_1(x_1)\k_1(x_2)\k_1(x_3) 
\\
&
- \k_2(x_1, x_2)\k_1(x_3) 
-\k_1(x_2)\k_2(x_1, x_3)
-\k_1(x_1)\k_2(x_2, x_3). 
\end{align*}
It is clear that $\k_n$ is a linear map from $S^n(A)$ to $\fieldk$ for every $n\geq 1$. It is not hard to see that $\k_N=0$ 
for all $N\geq n$ if $\k_n=0$, that $\k_1(1_A)=1$, and that $\k_n\big(x_1,\dotsc, x_{n-1},1_A\big)=0$ for all $n\geq 2$ 
and every $x_1,\dotsc, x_{n-1}\in A$.
\hfill$\natural$
\end{remark}

For the set of random variables $\{x_1,\dotsc, x_k\}$, considered earlier, one can associate the family of joint cumulants
\eqn{aaza}{
\left\{ \k_n\big(x_{j_1},\dotsc, x_{j_n}\big)\big| n\geq 1, 1\leq j_1,\dotsc, j_n \leq k\right\}.
}
For example, the joint cumulant $\k_2\big(x_{j_1}, x_{j_2}\big)$ of two random variables $x_{j_1}$ and $x_{j_2}$ 
measures the strength of their correlation 
\[
\k_2\big(x_{j_1}, x_{j_2}\big)=\big<x_{j_1}\cdot x_{j_2}\big> - \big<x_{j_1}\big>\big<x_{j_2}\big>,
\] 
which is known as the covariance, since we also have
\[
\k_2\big(x_{j_1}, x_{j_2}\big)=
\Big<\left(x_{j_1}-\big<x_{j_1}\big>\right)\cdot \left(x_{j_2}-\big<x_{j_2}\big>\right)\Big>.
\]

The family of joint cumulants may be conveniently described by its generating function
\eqn{aac}{
F(t_1,\dotsc,t_k) = \k_1(\g) + \sum_{n=2}^\infty\Fr{1}{n!}\k_n(\g,\dotsc,\g)\in \fieldk[\![t_1,\dotsc,t_k]\!],
}
where $\g =\sum_{i=1}^k t_i x_i$. It is not difficult to check the following relation: 
\eqn{aad}{
Z(t_1,\dotsc, t_k)\ceq \iota\left(e^\g\right)=e^{F(t_1,\dotsc,t_k)}.
}
Hence the probability law for random variables is determined by the joint cumulant generating function. 

Finally, we say two random variables $x$ and $y$ are {\em classically independent} if
\[
\k_n(x+y,\dotsc, x+y) = \k_n(x,\dotsc, x)+\k_n(y,\dotsc, y)
\] 
for all $n\geq 1$.
It follows that $\iota\left(e^{t_1 x + t_2 y}\right) = \iota\left(e^{t_1 x}\right) \iota\left(e^{t_2 y}\right)$ if $x$ and $y$ are classically independent. This relation with the notion of classical independence motivates us to use the adjective \emph{classical} for the cumulant morphism $\underline{\k}$. We shall also refer to $P(n)$ as the set of all \emph{classical} partitions of $[n]$.
 
There are similar combinatorial definitions of free and boolean cumulants, due to~\cite{Sp,SW}, 
which are based on non-crossing and interval partitions of $[n]$, respectively. 
These cumulants are related to the notions of free and boolean independence. 
The notion of classical cumulants was introduced by T.\ N.\ Thiele around 1889. 
The two reviews \cite{NS,NL} may serve a newcomer, like this author, as an excellent starting point for the history, concepts, 
and applications of classical cumulants and their non-commutative analogues in classical algebraic probability theory.
 
 \subsection{Categories and functors for classical algebraic probability theories}\label{subs:cats for algprob}

Consider a classical algebraic probability space $(A_{\classical}; \iota)$.
Note that the ground field $\fieldk$ itself is also a unital algebra and $\iota$ is a linear map between the underlying vector 
spaces $A$ and $\fieldk$ which preserves the unit but is {\em not} required to be an algebra homomorphism. 
The failure of $\iota$ to be an algebra homomorphism is measured by the second cumulant 
$\k_2(x,y)=\iota(x\cdot y) -\iota(x)\iota(y)$; higher cumulants measure successive failures.
 This simple observation motivates us to relax the usual notion of morphisms of unital algebras to introduce
  the notion of probability morphisms. 

\begin{definition}\label{promo}
Let $A_{\classical}=\big(A, 1_A,\cdot\big)$ and $A^\pr_{\classical}=\big(A^\pr,1_{A^\pr},\cdot^\pr\big)$ be unital commutative and associative algebras.
A probability morphism is a pointed linear map $f:A\rightarrow A^\pr$.

Let $f$ be a probability morphism. The \emph{descendant} $\underline{\phi}^f$ of $f$ is the family $ \phi^f_1,\phi^f_2,\phi^f_3,\dotsc$, where $\phi^f_n$ is a linear map from $S^n(A)$ to $A^\pr$ for all $n\geq 1$ defined recursively by the equation
\[
f(x_1\cdots x_n) =\sum_{\pi \in P(n)} \phi^f\big(x_{B_1}\big) \cdots\phi^f\big( x_{B_{|\pi|}}\big).
\]
\end{definition}

\begin{remark} 
Note that $\phi^f_1= f$ so that $\phi^f_1(1_A)=1_{A^\pr}$. It is straightforward to check that $\phi_n(x_1,\dotsc, x_{n-1}, 1_A)=0$ for $n\geq 2$ and $x_1,\dotsc, x_{n-1}\in A$.
The map $\phi^f_2$ measures how far $f$ is from being an algebra map:
\[
\phi^f_2(x_1, x_2)= f(x_1\cdot x_2) - f(x_1)\cdot^\pr f(x_2).
\]
In general, for $n\geq 2$, 
\[
\phi^f_n\big(x_1,\dotsc, x_n\big) =
\phi^f_{n-1}\big(x_1,\dotsc, x_{n-2}, x_{n-1}\cdot x_{n}\big)
-\sum_{\substack{\pi \in P(n)\\|\pi|=2\\ n-1\nsim_\pi n}}\phi^f\big(x_{B_1}\big)\cdot^\pr \phi^f\big(x_{B_2}\big).
\]
\hfill$\natural$
\end{remark}
Let $\category{\Prob}_{\classical}{\overk}$ be the category whose objects are unital commutative and associative algebras and whose morphisms are probability morphisms. 
We use the notation $A_{\classical}=(A,1_A, \cdot)$ for an object in $\category{\Prob}_{\classical}{\overk}$.

We also define another category $\category{\Prob}_{\Lie}{\overk}$. Its objects are pointed vector spaces $A_{\Lie}$. 
Given objects $A_{\Lie}$ and $A^\pr_{\Lie}$, a morphism from $A_{\Lie}$ to $A^\pr_{\Lie}$ is a a family $\underline{\phi}=\phi_1, \phi_2,\phi_3,\dotsc$ of linear maps $\phi_n:S^n(A)\rightarrow A^\pr$ such that $\phi_1$ is a pointed map and $\phi_n(x_1,\dotsc, x_{n-1}, 1_A)=0$ for $n\geq 2$ and $x_1,\dotsc, x_{n-1}\in A$.

Given a pair of composable morphisms $\underline{\phi}$ and $\underline{\phi}^\pr$, their composition $\underline{\phi}^\pr\bullet \underline{\phi}$ is defined by
\[
\left(\underline{\phi}^\pr\bullet \underline{\phi}\right)_n\big(x_1,\dotsc, x_n)
=\sum_{\pi\in P(n)}\phi^\pr_{|\pi|}\left(\phi\big(x_{B_1}\big), \cdots, \phi\big(x_{B_{|\pi|}}\big)\right).
\]
It can be checked that $\underline{\phi}^\pr\bullet \underline{\phi}$ is a morphism in $\category{\Prob}_{\Lie}$ from the source of $\phi$ to the target of $\phi^\pr$ and that composition is associative.

We now define a functor, called the \emph{descendant functor}, 
$\Des: \category{\Prob}_{\classical}{\overk}\Longrightarrow \category{\Prob}_{\Lie}{\overk}$ as follows:
\begin{itemize}
\item
For an object $A_{\classical}=\big(A,1_{A},\cdot\big)$ in $\category{\Prob}_{\classical}{\overk}$ we assign the object $A_{\Lie}=\big(A, 1_{A}\big)$ by forgetting the product.
\item For a morphism $f$ in $\category{\Prob}_{\classical}{\overk}$ we assign its descendant $\underline{\phi}^f$.
\end{itemize}

It is not difficult to check that if $f$ and $g$ are a composable pair of probability morphisms, then $\underline{\phi}^g\bullet \underline{\phi}^f$ is the descendant of the composite probability morphism $g\circ f$, that is:
\[
\underline{\phi}^g\bullet \underline{\phi}^f = \underline{\phi}^{g\circ f},
\]
so that we have functoriality of $\Des$, i.e., $\Des(g\circ f)=\Des(g)\bullet \Des(f)$.

We will overload the notation $\fieldk$ by using it for objects in our categories. It will denote the object of $\category{\Prob}_{\classical}$ which is $\fieldk$ equipped with its product and unit and the object of $\category{\Prob}_{\Lie}$ which is $\fieldk$ equipped with its unit. Then $\Des(\fieldk) = \fieldk$.This abuse of notation will not cause particular confusion. 

Note that $\fieldk$ is the initial object in both $\category{\Prob}_{\classical}$ and $\category{\Prob}_{\Lie}$; for any $A_{\classical}$ there is a unique probability morphism from $\fieldk$, and likewise for $\category{\Prob}_{\Lie}$.

Now we may regard classical algebraic probability theory as the study of the category $\category{\Prob}_{\classical}{\overk}$ and the functor $\Des$ into the category $\category{\Prob}_{\Lie}{\overk}$, where a morphism $\iota$ to the initial object $\fieldk$ in $\category{\Prob}_{\classical}{\overk}$ has a special interpretation. 

\begin{corollary}
A classical algebraic probability space is a diagram $\xymatrix{A_{\classical}\ar[r]^\iota &\fieldk}$ in the category $\category{\Prob}_{\classical}{\overk}$; the descendant $\underline{\phi}^\iota$ of the probability morphism $\iota$ to $\fieldk$ is the cumulant morphism $\underline{\k}$.
\end{corollary}

\subsection{The category and homotopy category of unital \texorpdfstring{$sL_\infty$}{sL-infinity}-algebras}\label{subs:Linfty}

This subsection is a self-contained introduction to the category and homotopy category of (unital) homotopy sLie algebras, which are called $sL_\infty$-algebras in this paper. Homotopy sLie algebras ($sL_\infty$-algebras) are the degree shifted versions of homotopy Lie algebras ($L_\infty$-algebra), which are the nearest kin to the homotopy associative algebras ($A_\infty$-algebras) discovered by Stasheff~\cite{Sta63}. The similarly shifted version of $A_\infty$-algebras, called $sA_\infty$-algebras, shall appear in a sequel.
Homotopy algebras have played crucial roles in algebraic topology, especially in the rational homotopy theory of Quillen, Sullivan and Chen~\cite{Quillen,Su77,Chen}.
Homotopy Lie algebras in particual appeared in Sullivan's theory of minimal models of rational homotopy types (in fact here it was the dual concept of homotopy Lie coalgebras), as well as in the context of deformation theory~\cite{SS}. 

We do not claim originality in this subsection, whose contents is standard (see \cite[Section 4]{Konst} and the references therein). The presentation and organization of the homotopy category of $sL_\infty$-algebras based on classical partitions of $[n]$, which is well-suited to our purposes, may be novel. We also add a bit of cultural background for a probability theorist. 

Recall that topology is the study of the category $\category{Top}$, whose objects are topological spaces and whose morphisms are continuous functions. Two topological spaces $X$ and $Y$ are said to be homeomorphic if there is a pair of continuous functions $f:X\rightarrow Y$ and $g: Y \rightarrow X$ such that $f\circ g=I_X$ and $g\circ f = I_Y$. On the other hand, continuous functions come with a natural notion of equivalence. 
We say that $f$ and $\tilde f$ are \emph{homotopic} or \emph{have the same homotopy type}, and denote this $f\sim \tilde f$, if the two functions can be continuously deformed into each other. 
Two topological spaces $X$ and $Y$ are said to be \emph{homotopy equivalent} or \emph{have the same homotopy type} if there are continuous functions $f:X\rightarrow Y$ and $g: Y \rightarrow X$ such that $f\circ g\sim I_X$ and $g\circ f \sim I_Y$. This leads to the study of the homotopy category $\homotopycategory\category{Top}$, which has the same objects as $\category{Top}$ but whose morphisms are homotopy types of continuous functions. In the homotopy category, a homotopy equivalence becomes an isomorphism.

In general, given a category $\category{C}$ with an equivalence relation $\sim$ called homotopy on the set of morphisms between any pair of objects which satisfies certain coherence conditions, one can build a homotopy category $\homotopycategory\category{C}$. The homotopy category $\homotopycategory\category{C}$ has the same objects as $\category{C}$ but its morphisms are equivalence classes of morphisms in $\category{C}$. We will not express these conditions explicitly but will rather construct the categories we need in an ad hoc manner. A note for the unwary: this is not the only meaning in the wider literature of ``homotopy category.'' However, it is the one we will always take in this paper.

One often studies the homotopy category $\homotopycategory\category{Top}$ by constructing and examining functors from $\category{Top}$ to a category $\category{C}$, itself equipped with a notion of homotopy, which take homotopic maps in $\category{Top}$ to homotopic maps in $\category{C}$. Such a functor induces a functor $\homotopycategory\category{Top}\rightarrow\homotopycategory\category{C}$. We call such functors \emph{homotopy functors}.

Another ubiquitous example of a homotopy category is that of cochain complexes over a commutative unital ring $R$. Let $\category{Kom}(R)$ be the category of cochain complexes, whose objects are cochain complexes over $R$ and whose morphisms are cochain maps. 
Cochain homotopy yields a natural notion of homotopy equivalence. Then one can form the homotopy category $\homotopycategory\category{Kom}(R)$ of cochain complexes. As above, its objects are the same as those of $\category{Kom}(R)$ and its morphisms are homotopy types of cochain maps.

We call a category $\category{Alg}_P$ \emph{algebraic} if each object has an underlying cochain complex and each morphism is a cochain map of the underlying cochain complexes. We call an algebraic category \emph{homotopy algebraic} if it is equipped with a notion of homotopy and if morphisms are homotopic in $\category{Alg}_P$ if and only if they are homotopic in $\category{Kom}(R)$. One origin of homotopy algebra is via the study of algebraic models of $\homotopycategory\category{Top}$. Here, an \emph{algebraic model} means a homotopy algebraic category $\category{Alg}_P$ (``the category of homotopy $P$-algebras'') together with a homotopy functor (or family of homotopy functors), $\mF :\category{Top}\Rightarrow \category{Alg}_P$. As in the discussion above, an algebraic model induces a well-defined functor $\text{ho}\mF :\homotopycategory\category{Top}\Rightarrow \homotopycategory\category{Alg}_P$. Then any invariants of homotopy types in the algebraic category are invariants of homotopy types of topological spaces.

The simplest example of an algebraic model is the functor of cochains. Consider, for example, the singular cochain complex $C(X)$ of the topological space $X$ with coefficients in $R$.
Then a continuous function $f:X\rightarrow Y$ induces a cochain map $C(f):C(Y) \rightarrow C(X)$. Furthermore, if $\tilde f$ is a continuous function homotopic to $f$, then the induced cochain maps $C(f)$ and $C(\tilde f)$ are cochain homotopic. 
This construction is also (contravariant) functorial---$C(g\circ f)=C(f)\circ C(g):C(Z) \rightarrow C(X)$ for any continuous function $g:Y\rightarrow Z$, and $C(\tilde f)\circ C(\tilde g)$ is cochain homotopic to $C(f)\circ C(g)$ whenever $f\sim \tilde f$ and $g\sim \tilde g$. It follows that $C(X)$ and $C(Y)$ have the same cochain homotopy type if $X$ and $Y$ have the same homotopy type. 

The basic invariant of cochain homotopy types is cohomology. Note that a cochain homotopy equivalence induces an isomorphism of cohomology since the differentials vanish on cohomology. Therefore, any cochain map that is homotopic to zero induces the zero map on cohomology. It follows that singular cohomology is an invariant of the homotopy type of a topological space.

Cochain homotopy equivalences provide one particular class of quasi-isomorphisms, but in general not every quasi-isomorphism is a cochain homotopy equivalence. Note that the cohomology $H$ of a cochain complex $(C,d)$ can be viewed as a cochain complex $(H,0)$ with zero differential, whose cohomology is $H$ itself. If the ground ring is a field $\fieldk$, one can construct a cochain quasi-isomorphism $f:(H,0)\rightarrow (C,d)$ as follows. Choose a basis $\{e_\a\}_{\a\in J}$ of $H$ as a graded vector space. Next choose a set of representatives $\{f(e_\a)\}_{\a\in J}$ in $C$, that is, $d f(e_\a)=0$ and the cohomology class $[f(e_\a)]$ of $f(e_\a)$ is $e_\a$ for all $\a \in J$. Then extend $f$ linearly over $H$. The cochain map $f$ is a quasi-isomorphism which induces the identity map on cohomology. Such a quasi-isomorphism is unique up to cochain homotopy.

A more sophisticated example of an algebraic category is the category of differential graded Lie algebras (DGLAs).

\begin{itemize}
\item
A \emph{graded Lie algebra} is a pair $(V, [\hbox{ },\hbox{ }])$, where $V$ is a graded vector space and $[\hbox{ },\hbox{ }]$ is a linear map from $V\otimes V$ to $V$ of degree $0$ such that, for all elements $x,y \in V$
\begin{align*}
[x,y] +(-1)^{|x||y|}[y,x]=0,
\\
[x,[y,z]] - [[x,y],z] - (-1)^{|x||y|}[y, [x,z]]=0,
\end{align*}
where $|x|$ denotes the degree of $x$.
\item
A \emph{Lie algebra} is a graded Lie algebra concentrated in degree zero.
\item
A \emph{differential graded Lie algebra} or DGLA $\big(V, d, [\hbox{ },\hbox{ }]\big)$ is both a cochain complex $\big(V,d\big)$ and a graded Lie algebra $\big(V, [\hbox{ },\hbox{ }]\big)$ such that the differential $d$ is a derivation of the bracket $[\hbox{ },\hbox{ }]$:
\[
d[x,y] = [dx, y]+ (-1)^{|x|}[x, dy].
\]
\end{itemize}
In this paper we shall work with a variant of DGLAs called sDGLAs (the s is for ``shifted'').
\begin{definition}
A graded sLie algebra is a duo $(V, \ell_2)$, where $V$ is a graded vector space and $\ell_2$ is a linear map from $V\otimes V$ to $V$ of degree $1$ such that
\begin{align*}
\ell_2(x,y) -(-1)^{|x||y|}\ell_2(y,x)=0,
\\
\ell_2\big(x,\ell_2(y,z)\big) + (-1)^{|x|}\ell_2\big(\ell_2(x,y),z\big) + (-1)^{(|x|+1)|y|}\ell_2(y, \ell_2(x,z)\big)=0.
\end{align*}
An sDGLA $(V,d,\ell_2)$ is both a cochain complex $(V,d)$ and a graded sLie algebra $(V, \ell_2)$ such that the differential $d$ is compatible with the bracket $\ell_2$:
\eqn{compia}{
d\ell_2(x,y) + \ell_2(dx, y)+ (-1)^{|x|}\ell_2(x, dy)=0.
}
\end{definition}
\begin{remark}
Let $(V, d, [\hbox{ },\hbox{ }])$ be a DGLA. Then $(V[-1], d,\ell_2)$ is a sDGLA with $\ell_2(x,y) =(-1)^{|x|}[x,y]$ for all $x,y \in V$,
where $V[-1]^j \ceq V^{j +1}$ for all $j$. It is clear that everything we state about sDGLAs can be translated into statements about DGLAs. \hfill$\natural$
\end{remark}
The cohomology of an sDGLA is defined as that of its underlying cochain complex. Then the compatibility condition \eq{compia}
implies that the cohomology $H$ of an sDGLA has the canonical structure of a graded sLie algebra.

A morphism between sDGLAs $\big(V, d, \ell_2\big)$ and $\big(V^\pr, d^\pr, \ell_2^\pr\big)$ is a linear map $f:V\rightarrow V^\pr$ of degree zero which is both a cochain map and a sLie algebra homomorphism, i.e., $fd=d^\pr f$ and $f\big(\ell_2(x,y)\big)= \ell_2^\pr\big(f(x),f(y)\big)$. 
It follows that $H(f):H\rightarrow H^\pr$ is a homomorphism between the graded sLie algebras $H$ and $H^\pr$.
A morphism $f$ of sDGLA is called a quasi-isomorphism if it is a quasi-isomorphism of the underlying cochain complexes, that is, if $H(f):V\rightarrow V^\pr$ is an isomorphism. Then $H(f)$ is also an isomorphism between the graded sLie algebras $H$ and $H^\pr$. It is clear that sDGLAs form an algebraic category.

However, naively trying to define a homotopy category of sDGLAs via a notion of homotopy as above runs into problems. A cochain map $\tilde \phi_1=f + \l_1 d + d^\pr \l_1$ that is homotopic to an sDGLA morphism $f$ by the cochain homotopy $\l_1$ is rarely an Lie algebra homomorphism but instead is only a homeomorphism up to homotopy; a tedious computation shows that
\eqn{firstho}{
\tilde \phi_1\big(\ell_2(x,y)\big) -\ell_2^\pr\big(\tilde \phi_1(x),\tilde \phi_1(y)\big)
=d^\pr \tilde\phi_2(x,y) - \tilde\phi_2(dx,y) -(-1)^{|x|}\tilde\phi_2(x,dy),
}
where $\tilde\phi_2$ is a linear map from $S^2V$ to $V$ of degree $0$ defined on elements $x,y \in V$ as
\begin{align*}
\tilde\phi_2(x,y) = 
&
\l_1\big(\ell_2(x,y)\big) +\ell_2^\pr\big(\l_1(x), f(y)\big) +(-1)^{|x|}\ell_2^\pr\big(f(x), \l_1(y)\big)
\\
&
+\Fr{1}{2}\ell_2^\pr\big(\l_1(x), (\tilde\phi_1-f)(y)\big) + \Fr{1}{2}(-1)^{|x|}\ell_2^\pr\bigr((\tilde\phi_1 -f)(x), \l_1(y)\big).
\end{align*}

One can iterate further, beginning by applying the Jacobi-identities to \eq{firstho}, to obtain an infinite sequence $\underline{\tilde\phi}=\tilde\phi_1,\tilde\phi_2,\tilde\phi_3,\dotsc$, where $\tilde\phi_n:S^n V\rightarrow V^\pr$ is a linear map of degree $0$ and the maps $\tilde\phi$ satisfy an infinite sequence of relations. This leads to the notion of $sL_\infty$-morphisms, which have a well-defined notion of homotopy. Once this framework has been built, the sDGLA morphism $f$ itself can be regarded as an $sL_\infty$-morphism $\underline{\phi}=\phi_1,0,0,0,\dotsc$ with $\phi_1=f$. In this context, the $sL_\infty$-morphism $f$ will not be $sL_\infty$-homotopic to $\phi_1$ but rather to the \emph{sequence} $\underline{\tilde\phi}$ by an $sL_\infty$-homotopy $\underline{\l}=\l_1,\l_2,\l_3,\dotsc$

Eventually one also arrives at the notion of $sL_\infty$-algebras as a natural generalization of sDGLAs. Then $sL_\infty$-morphisms and the corresponding homotopies lead naturally to a category and homotopy category of $sL_\infty$-algebras. 

Before we leave this cultural background, we also recall the notions of a differential graded algebra and differential graded commutative algebra.
\begin{definition}
A differential graded algebra (DGA) is a triple $(\sA, d,\cdot)$,
where the pair $(\sA, d)$ is a cochain complex and the pair $(\sA,\cdot)$ is a graded vector space equipped with a degree zero linear map from $\sA\otimes \sA \rightarrow \sA$ defined by $x\otimes y \mapsto x\cdot y$ that is associative, i.e., $x_1\cdot(x_2\cdot x_3)=(x_1\cdot x_2)\cdot x_3$ such that $d(x\cdot y) = dx\cdot y +(-1)^{|x|}x\cdot y$. 
\end{definition}

\begin{remark}
Let $(\sA, d,\cdot)$ be a DGA.
The super-commutator $[\hbox{ },\hbox{ }]$ of the associative multiplication is defined by $[x,y]\ceq x\cdot y -(-1)^{|x||y|}y\cdot x$ for all $x,y \in \sA$. Then the trio $(\sA, d, [\hbox{ },\hbox{ }])$ is a DGLA. \hfill$\natural$
\end{remark}

\begin{definition}
A differential graded commutative algebra (CDGA) is a DGA $(\sA, d,\cdot)$ such that the super-commutator $[\hbox{ },\hbox{ }]$ vanishes identically.
\end{definition}

\subsubsection{The category of unital \texorpdfstring{$sL_\infty$}{sL-infinity}-algebras}

\begin{definition}[Unital $sL_\infty$-Algebras]
\label{l-alg}
Let $V$ be a pointed graded vector space.
The structure of a \emph{unital $sL_\infty$-algebra} on $V$, denoted $(V,\underline{\ell}\big)$, consists of a family $\underline{\ell}=\ell_1, \ell_2, \ell_3,\cdots$ of linear maps $\ell_n:S^n V\rightarrow V$ of degree $1$ such that, for all $n\geq 1$,
\begin{enumerate}
\item for $x_1,\dotsc, x_n \in V$,
\[
\sum_{\substack{\pi \in P(n)\\ |B_i| = n-|\pi|+1}}\ep(\pi,i)
\ell_{|\pi|}\big( x_{B_1}, \cdots, x_{B_{i-1}}, \ell(x_{B_i}), x_{B_{i+1}},\dotsc, x_{B_{|\pi|}}\big)
=0,
\]
where $\ep(\pi,i)= \ep(\pi)(-1)^{|x_{B_1}|+\cdots +|x_{B_{i-1}}|}$, and

\item $\ell_n\big(x_1,\dotsc, x_{n-1}, 1_V\big)=0$ for all $x_1,\dotsc, x_{n-1}\in V$.
\end{enumerate}
Note that the condition $|B_i| = n-|\pi|+1$ in the summation implies that all the sets $B_j$ are singletons for $j\ne i$.
\end{definition}

\begin{remark}
Use the notation $d$ for $\ell_1$.
Here are the first few relations:

- For $n=1$, we have $d^2=d1_V=0$.
Hence $(V,1_V,d)$ is a pointed cochain complex.

- For $n=2$, we have $ d\ell_2\big(x_1, x_2\big)+ \ell_2\big(d x_1, x_2\big)+(-1)^{|x_1|} \ell_2\big(x_1, d x_2\big)=0$ so $d$ is a derivation of $\ell_2$. 
 
- For $n=3$, 
\begin{align*}
d \ell_3(x_1, x_2,x_3)
+\ell_3\big( dx_1, x_2,x_3\big)
+(-1)^{|x_1|} \ell_3\big(x_1, d x_2,x_3\big)
+(-1)^{|x_1|+|x_2|} \ell_3\big(x_1, x_2,d x_3\big)
\\
=
-\ell_2\big(\ell_2(x_1, x_2),x_3\big)
-(-1)^{|x_1|} \ell_2\big(x_1, \ell_2 (x_2, x_3)\big)
-(-1)^{(|x_1|+1)|x_2|} \ell_2\big(x_2, \ell_2 (x_1, x_3)\big).
\\
\end{align*}
Note that the vanishing condition of the right hand side in the final equation above is the graded Jacobi identity for $\ell_2$.
Hence $\ell_2$ may fails to satisfy the graded Jacobi identity---the failure of $\ell_2$ being a graded Lie bracket is measured by the homotopy $\ell_3$.

%\begin{itemize}
%\item For $n=1$, we have $d^2=d1_V=0$.
%Hence $(V,1_V,d)$ is a pointed cochain complex.
%\item For $n=2$, we have $ d\ell_2\big(x_1, x_2\big)+ \ell_2\big(d x_1, x_2\big)+(-1)^{|x_1|} \ell_2\big(x_1, d x_2\big)=0$ so $d$ is a derivation of $\ell_2$. 
%\item For $n=3$, 
%\begin{align*}
%d \ell_3(x_1, x_2,x_3)
%+\ell_3\big( dx_1, x_2,x_3\big)
%+(-1)^{|x_1|} \ell_3\big(x_1, d x_2,x_3\big)
%+(-1)^{|x_1|+|x_2|} \ell_3\big(x_1, x_2,d x_3\big)
%\\
%=
%-\ell_2\big(\ell_2(x_1, x_2),x_3\big)
%-(-1)^{|x_1|} \ell_2\big(x_1, \ell_2 (x_2, x_3)\big)
%-(-1)^{(|x_1|+1)|x_2|} \ell_2\big(x_2, \ell_2 (x_1, x_3)\big).
%\\
%\end{align*}
%Note that the vanishing condition of the right hand side in the final equation above is the graded Jacobi identity for $\ell_2$.
%Hence $\ell_2$ fails to satisfy the graded Jacobi identity---and the failure of $\ell_2$ being a graded Lie bracket is measured by the homotopy $\ell_3$.
%\end{itemize}
\hfill$\natural$
\end{remark}

\begin{definition}[Unital $sL_\infty$-morphisms]
\label{l-mor}
A morphism $\underline{\phi}$ between unital $sL_\infty$-algebras $\big(V,1_V,\underline{\ell}\big)$ 
and $\big(V^\pr,1_{V^\pr},\underline{\ell}^\pr\big)$ is a family $\underline{\phi}=\phi_1,\phi_2,\phi_3,\cdots$ such that, for all $n\geq 1$,
\begin{enumerate}
\item $\phi_n: S^n V\rightarrow V^\pr$ is a linear map of degree $0$;
\item for every tuple of elements $x_1,\dotsc, x_n \in V$,
\begin{align*}
\sum_{\substack{\pi \in P(n)}}\ep(\pi)
\ell^\pr_{|\pi|}&\left(\phi\left( x_{B_1}\right), \cdots, \phi\big( x_{B_{|\pi|}}\big)\right)
\\
&
=
\sum_{\substack{\pi \in P(n)\\ |B_i| = n-|\pi|+1}}\ep(\pi,i)
\phi_{|\pi|}\left( x_{B_1}, \cdots, x_{B_{i-1}}, \ell(x_{B_i}), x_{B_{i+1}},\dotsc, x_{B_{|\pi|}}\big)\right),
\end{align*}
and
\item $\phi_1(1_V)=1_{V^\pr}$, $\phi_{k}\big(x_1,\dotsc, x_{k-1}, 1_V\big)=0$ for all $k\geq 2$ and all $x_1,\dotsc, x_{k-1}\in V$.
\end{enumerate}
\end{definition}
\begin{remark}
Set $\ell_1=d$ and $\ell^\pr_1=d^\pr$.
\begin{itemize}
\item
For $n=1$, we have $\phi_1 d = d^\pr \phi_1$ and $\phi_1(1_V)=1_{V^\pr}$ so that $\phi_1$ is a pointed cochain map from $(V, 1_V, d)$ to $(V^\pr,1_{V^\pr}, d^\pr)$.
\item
 For $n=2$, we have
\begin{align*}
\phi_1\big(\ell_2(x_1,x_2)\big) - \ell^\pr_2 & \big(\phi_1(x_1),\phi_1(x_2)\big)
\\
=&
d^\pr \phi_2(x_1, x_2)\big)
- \phi_2\big(d x_1, x_2\big)-(-1)^{|x_1|} \phi_2\big(x_1, dx_2\big),
\end{align*}
so that $\phi_1$ is an algebra homomorphism up to the homotopy $\phi_2$.
\end{itemize}
\hfill$\natural$
\end{remark}

\begin{definition}[Lemma] \label{lmocom}
Unital $sL_\infty$-algebras form a category with composition of morphisms $\underline{\phi}:(V,\underline{\ell}) \rightarrow (V^\pr,\underline{\ell}^\pr)$ and $\underline{\phi}^\pr: (V^\pr,\underline{\ell}^\pr)\rightarrow (V^\ppr,\underline{\ell}^\ppr)$ given by
\[
\left(\underline{\phi}^\pr\bullet \underline{\phi}\right)_n\big(x_1,\dotsc, x_n\big)
\ceq \sum_{\pi \in P(n)}\ep(\pi)\phi^\pr_{|\pi|}\left(\phi\big(x_{B_1}\big),\dotsc, \phi\big(x_{B_{|\pi|}}\big)\right).
\]
\end{definition}
\begin{proof}
Exercise.
\hfill\qed\end{proof}

\begin{remark}
The explicit formulas for $\left(\underline{\phi}^\pr\bullet \underline{\phi}\right)_n$ for $n=1,2,3$ are
\begin{align*}
\left(\underline{\phi}^\pr\bullet \underline{\phi}\right)_1(x)=
&
\phi^\pr_1\big(\phi_1(x)\big)
,\\
\left(\underline{\phi}^\pr\bullet \underline{\phi}\right)_2(x,y)=
&
\phi^\pr_1\big(\phi_2(x,y)\big) +\phi^\pr_2\big(\phi_1(x),\phi_1(y)\big)
,\\
\left(\underline{\phi}^\pr\bullet \underline{\phi}\right)_3(x,y,z)=
&
\phi^\pr_1\big(\phi_3(x,y,z)\big) 
+\phi^\pr_2\big(\phi_1(x),\phi_2(y,z)\big)
+\phi^\pr_2\big(\phi_2(x,y),\phi_1(z)\big)
\\
&
+(-1)^{|x||y|}\phi^\pr_2\big(\phi_1(y),\phi_2(x,z)\big)
+\phi^\pr_3\big(\phi_1(x),\phi_1(y),\phi_1(z)\big).
\end{align*}
\hfill$\natural$
\end{remark}

We denote by $\category{\UsL}_\infty{\overk}$ the category of unital $sL_\infty$-algebras whose objects are unital $sL_\infty$-algebras and whose morphisms are unital $sL_\infty$-morphisms. Forgetting the unit and unit conditions, we have $\category{sL}_\infty{\overk}$, the category of $sL_\infty$-algebras and $sL_\infty$-morphisms. 

\subsubsection{{The homotopy category of unital \texorpdfstring{$sL_\infty$}{sL-infinity}-algebras.}}
We now turn to homotopies between unital $sL_\infty$-morphisms. 
Let $(V,1_V,\underline{\ell})$ and $(V^\pr,1_{V^\pr},\underline{\ell}^\pr)$ be unital $sL_\infty$-algebras. Recall a pointed cochain homotopy is an arbitrary linear map from $V$ to $V^\pr$ of degree $-1$ such that $s(1_{V})=0$.

%Consider their underlying
%pointed cochain complexes $(V, 1_V, K)$ and $(V^\pr,1_{V^\pr}, K^\pr)$.
%Recall that a cochain homotopy $s$ is an arbitrary linear map from $V$ to $V^\pr$ of degree $-1$
%such that $s(1_{V})=0$. For a given cochain map $f:V\rightarrow V^\pr$, there is a corresponding cochain map $\tilde f$ 
%determined uniquely by $s$. Also the notation $f\sim \tilde{f}$ that $f$ is cochain homotopic to $\tilde f$ is
%an equivalence relation and the composition of chain maps respect the equivalence relation. 
\begin{definition}
\label{flow}
A (time-independent) unital $sL_\infty$-homotopy between two $sL_\infty$-algebras $V$ and $V^\pr$ is an infinite sequence $\underline{\l} =\l_1,\l_2,\cdots$,
where $\l_n$ is a linear map from $S^n V$ to $V^\pr$ of degree $-1$ satisfying
\[
\l_n(x_1,\dotsc, x_{n-1}, 1_V)=0,
\] 
for all $x_1,\dotsc, x_{n-1} \in V$.
\end{definition}

In the category of cochain complexes, a chain map and a chain homotopy uniquely determine another chain map. This leads to an equivalence relation that respects composition of morphisms. The following definition parallels the situation in cochain complexes using $sL_\infty$-morphisms and homotopies, with some modification to the uniqueness statement.
\begin{definition}[Lemma]\label{l-hom}
Let $\underline{\eta}(\t)$ be an one-parameter family of unital $sL_\infty$-homotopies between $V$ and $V^{\pr}$ with polynomial dependence in the parameter $\t\in[0,1]$. Let $\underline{\phi}$ be a sequence of maps $\phi_1,\phi_2,\cdots$, where $\phi_n:S^n V \to V^{\pr}$ is of degree $0$.

Consider the following recursively defined system of \emph{homotopy flow} equations generated by $\underline{\eta}(\t)$. 
\begin{align*}
\dot\Phi_n\big(&x_1,\dotsc, x_n\big)
\\
=
&
d^\pr \eta_n(x_1,\dotsc,x_n)
\\
&
+\sum_{\substack{\pi \in P(n)\\ |B_i| = n-|\pi|+1}}\ep(\pi,i)
\eta_{|\pi|}\left( x_{B_1}, \cdots, x_{B_{i-1}}, \ell(x_{B_i}), x_{B_{i+1}},\dotsc, x_{B_{|\pi|}}\right)
\\
&
+\sum_{\substack{\pi \in P(n)\\|\pi|\neq 1}}\sum_{i=1}^{|\pi|}\ep(\pi,i)
\ell^\pr_{|\pi|}\left({\Phi}\left( x_{B_1}\right), \cdots, {\Phi}\big(x_{B_{i-1}}\big)
 ,\eta\big(x_{B_i}\big), {\Phi}\big( x_{B_{i+1}}\big),\dotsc, {\Phi}\big( x_{B_{|\pi|}}\big)\right)
\end{align*}
Here $\underline{\Phi}(\t)=\Phi_1(\t),\Phi_2(\t),\cdots$, where $\Phi_n(\t)$ is a one-parameter family of linear maps from $S^n V$ to $V^\pr$ of degree $0$, 

Then the homotopy flow equations have a unique solution $\underline{\Phi}$ such that $\underline{\Phi}(0)=\underline{\phi}$. If $\underline{\phi}$ is a unital $sL_\infty$-morphism, then the unique solution $\underline{\Phi}$ is a smooth family of unital $sL_\infty$-morphisms.

We say $\underline{\tilde\phi}\ceq \underline{\Phi}(1)$ is \emph{homotopic} to or has the same \emph{$sL_\infty$-homotopy type} as $\underline{\phi}$. We denote this by $\underline{\phi}\sim_\infty \underline{\tilde\phi}$.

We can extract a time-independent unital $sL_\infty$-homotopy $\underline{\l}=\l_1,\l_2,\l_3,\dots$ from $\underline{\eta}$ by setting $\l_n = \int^1_0 \eta_n \text{d}\t$.
\end{definition}

\begin{proof} 
Consider the homotopy flow equation for $n=1$:
\[
\dot\Phi_1= d^\pr \eta_1 + \eta_1 d,
 \]
whose right hand side is a polynomial in $\t$.
It follows that $\Phi_1$ has the unique polynomial solution:
\[
\Phi(\t)_1 = \phi_1 + d^\pr \xi_1 +\xi_1 d\hbox{ where }\xi_1 =\int^\t_0 \eta(u)_1 \text{d} u.
\]
In general $\Fr{d}{d\t}\Phi_n$ depends on $\underline{\eta}$ and $\Phi_k$ for $1\leq k\leq n-1$.
Fix $n >1$ and assume that $\Phi_k$ has a unique solution, polynomial in $\t$, for all $1\leq k\leq n-1$. Then it is clear that $\Phi_n$ has a unique solution:
\[
\Phi(\t)_n =\phi_n +\int^\t_0 \dot\Phi_n(u)\text{d}u,
\]
which is a polynomial in $\t$. Hence, by induction, $\underline{\Phi}$ is determined uniquely by the given initial conditions $\underline{\phi}$ and $\underline{\eta}$.

Next, consider the sequence $\underline{\Xi}=\Xi_1, \Xi_2,\cdots$ of degree $1$ linear maps $\Xi_n:S^n V\to V'$ defined as
\eqnalign{xin}{
\Xi_n(x_1,\cdots, x_n)\ceq &\sum_{\substack{\pi \in P(n)}}\ep(\pi)
\ell^\pr_{|\pi|}\left(\Phi\left( x_{B_1}\right), \cdots, \Phi\big( x_{B_{|\pi|}}\big)\right)
\\
&
-
\sum_{\substack{\pi \in P(n)\\ |B_i| = n-|\pi|+1}}\ep(\pi,i)
\Phi_{|\pi|}\left( x_{B_1}, \cdots, x_{B_{i-1}}, \ell(x_{B_i}), x_{B_{i+1}},\cdots, x_{B_{|\pi|}}\big)\right).
}
Note that $\underline{\Xi}$ measures the failure of the family $\underline{\Phi}$ to be an $sL_\infty$-morphism.
Hence $\underline{\Xi}$ is a polynomial family parametrized by $\t \in [0,1]$ with the initial condition $\Xi_n\big|_{\t=0}=0$ for all $n\geq 1$ if $\underline{\phi}=\Phi_n\big|_{\t=0}$ is an $sL_\infty$-morphism. We would like to show that in this case $\Xi_n =0$ for all $n\geq 1$.

By applying $\fr{d}{d\t}$ to \eq{xin}, we have, for all $n\geq 1$,
\begin{align*}
\dot{\Xi}_n&\big(x_1,\cdots, x_n\big)
\\
=
&
\sum_{\substack{\pi \in P(n)}}\sum_{i=1}^{|\pi|}\ep(\pi)
\ell^\pr_{|\pi|}\left({\Phi}\left( x_{B_1}\right), \cdots, {\Phi}\big(x_{B_{i-1}}\big)
 ,\dot\Phi\big(x_{B_i}\big), {\Phi}\big( x_{B_{i+1}}\big),\cdots, {\Phi}\big( x_{B_{|\pi|}}\big)\right)
\\
&
-\sum_{\substack{\pi \in P(n)\\ |B_i| = n-|\pi|+1}}\ep(\pi,i)
\dot\Phi_{|\pi|}\left( x_{B_1}, \cdots, x_{B_{i-1}}, \ell(x_{B_i}), x_{B_{i+1}},\cdots, x_{B_{|\pi|}}\big)\right),
\end{align*}
For $n=1$, we have $\dot{\Xi}_1= K^\pr\dot{\Phi}_1 + \dot{\Phi}_1 K$.
From the homotopy flow equation $\dot{\Phi}_1 = K^\pr \eta_1 +\eta_1 K$, we obtain that $\dot{\Xi}_1=0$.
Hence $\Xi_1 =0$ due to the initial condition that $\Xi_1\big|_{\t=0}=0$.
For $n\geq 2$,
after replacing $\dot\Phi\big(x_{B_i}\big)$ using the homotopy flow equation, we have
\begin{align*}
{\dot{\Xi}}_n\big(x_1, &\cdots, x_n\big)
\\
=
&
-\!\!\!\!\!\sum_{\substack{\pi \in \sP(n)}}\sum_{i<j}^{|\pi|}\ep(\pi,i)
\ell^\pr_{|\pi|}\left({\Phi}\left( x_{B_1}\right), \cdots, {\Xi}\big(x_{B_{i}}\big),
\cdots,\eta\big(x_{B_j}\big),\cdots, {\Phi}\big( x_{B_{|\pi|}}\big)\right)
\\
&
+\!\!\!\!\!\sum_{\substack{\pi \in \sP(n)}}\sum_{i<j}^{|\pi|}\ep(\pi,i)
\ell^\pr_{|\pi|}\left({\Phi}\left( x_{B_1}\right), \cdots, {\eta}\big(x_{B_{i}}\big),
\cdots,\Xi\big(x_{B_j}\big),\cdots, {\Phi}\big( x_{B_{|\pi|}}\big)\right).
\end{align*}
Fix $n >2$ and assume that $\Xi_k=0$ for all $1\leq k\leq n-1$.
Note that ${\dot{\Xi}}_n$ in the above equation depends only on $\Xi_1,\cdots, \Xi_{n-1}$ among the members of the family $\underline{\Xi}$. It follows that $\dot \Xi_n=0$, which implies that $\Xi_n=0$ since $\Xi_n\big|_{\t=0}=0$.
Therefore $\underline{\Phi}$ is a polynomial family of $sL_\infty$-morphisms parametrized by $\t\in[0,1]$.

Finally, we turn to the unit. Let $x_n=1_V$. Then the homotopy flow equation, for all $n\geq 1$, reduces to
\begin{align*}
\dot{\Phi}_n\big(x_1,&\cdots, x_n\big)
\\
=
&
\sum_{\substack{\pi \in P(n)\\ |\pi|\neq 1}}\sum_{i=1}^{|\pi|-1}\ep(\pi,i)
\ell^\pr_{|\pi|}\left({\Phi}\left( x_{B_1}\right), \cdots, {\Phi}\big(x_{B_{i-1}}\big)
 ,\eta\big(x_{B_i}\big), {\Phi}\big( x_{B_{i+1}}\big),\cdots, {\Phi}\big( x_{B_{|\pi|}}\big)\right)
\end{align*}
Assume $\underline{\Phi}(0)= \underline{\phi}$, where $\phi_1(1_V)= 1_{V^\pr}$ and $\phi_n\big(x_1,\cdots, x_{n-1},1_V\big)=0$ for all $n\geq 2$. We will show that $\underline{\Phi}(\t)$ has the same properties.

For $n=1$, we have $\dot{\Phi}_1\big(1_V\big) =0$. Hence $\Phi_1\big(1_V\big)=\phi_1\big(1_V\big)= 1_{V^\pr}$.
Now assume that ${\Phi}_k\big(x_1,\cdots, x_{k-1}, 1_V\big)=0$ for all $1\leq k\leq n$ for some $n>1$.
Then we have $\dot{\Phi}_{n+1}\big(x_1,\cdots, x_{n}, 1_V\big)=0$, which implies that 
\[
{\Phi}_{n+1}\big(x_1,\cdots, x_{n}, 1_V\big)={\phi}_{n+1}\big(x_1,\cdots, x_{n}, 1_V\big)=0.
\]
\hfill\qed\end{proof}
It can be checked that $\sim_\infty$ is an equivalence relation.
Forgetting the unit and unit conditions, we have the definition of the $sL_\infty$-homotopy of $sL_\infty$-morphisms. 
\begin{remark}\label{irem}
Assume that we have an $sL_\infty$-morphism $\underline{\phi}$ from $(V,\underline{\ell})$ to $(V^\pr, \underline{\ell^\pr})$ such that $\underline{\phi}=\phi_1,0,0,0,\dotsc$ Further assume that the cochain map $\phi_1:(V, d)\rightarrow (V^\pr, d^\pr)$ is cochain homotopic to zero, that is, $\phi_1= d^\pr \varsigma+\varsigma d$ for a linear map $\varsigma:V\rightarrow V^\pr$ of degree $-1$. This situation arises frequently. A natural question is whether the $sL_\infty$-morphism $\underline{\phi}=\phi_1,0,0,0,\dotsc$ is homotopic to the zero $sL_\infty$-morphism?
The answer in general is {\em no}. 

Consider the special case where the structure on $V^\pr$ is the zero structure. Then an $sL_\infty$-morphism of the special type is equivalent to a chain map $\phi_1$ which is chain homotopic to the zero map and annihilates $\ell_k$ for all $k$. Given a sequence of degree $-1$ maps $\underline{\eta}=\eta_1,\eta_2,\ldots,$ where $\eta_k:S^k V\to V^{\pr}$, the differential equations for the homotopy generated by $\underline{\eta}$ begin with:
\begin{align*}\dot{\Phi}_1(v_1)&=\eta_1 dv_1,
\\
\dot{\Phi}_2(v_1,v_2) &= \eta_2(dv_1,v_2)+(-1)^{|v_1|}\eta_2(v_1,dv_2)+\eta_1\ell_2(v_1,v_2),\ldots
\end{align*}

Consider the three dimensional $sL_\infty$-algebra $A$ spanned by $a$ and $b$ in degree $0$ and $c$ in degree $1$ with $da = \ell_2(b,b)=c$ and all other structure $0$. Then there is an $sL_\infty$-morphism from $A$ to $\fieldk$ of the form $\phi_1,0,0,0,\dotsc$ where $\phi_1(a)=1$ and $\phi_1(b)=0$. The map $\phi_1$ is clearly chain homotopic to the zero map. Note that $\dot{\Phi}_2(b,b)=\eta_1\ell_2(b,b)=\eta_1 da = \dot{\Phi}_1(a)$. This implies that 
\[(\Phi_2(1)-\Phi_2(0))(b,b)=(\Phi_1(1)-\Phi_1(0))(a).\] In order for $\Phi$ to be a homotopy between $\phi_1,0,0,\dotsc$ and the zero morphism, the left hand side of this must be $0$ while the right hand side must be $-1$.
\end{remark}

The homotopy category $\homotopycat\category{\UsL}_\infty{\overk}$ of unital $sL_\infty$-algebras can be defined to have the same objects as the the category $\category{\UsL}_\infty{\overk}$ and to have its morphisms be the homotopy types of unital $sL_\infty$-morphisms, provided that composition of homotopy types of morphisms is well-defined.
 
Consider the following diagram in the category $\category{\UsL}_\infty{\overk}$:
\[
\xymatrix{ 
V_{\Lie}
\ar@{..>}@/^0.5pc/[r]^{\underline{\w}}
\ar@{..>}@/_0.5pc/[r]_{\underline{\tilde\w}}
&
 V^\pr_{\Lie}
 \ar@{..>}@/^0.5pc/[r]^{\underline{\w}^\pr}
 \ar@{..>}@/_0.5pc/[r]_{\underline{\tilde\w}^{\hspace{1pt}\pr}}
& V^\ppr_{\Lie}},
\]
where $\underline{\tilde\w}$ is homotopic to $\underline{\w}$ and $\underline{\tilde\w}^{\hspace{1pt}\pr}$ is homotopic to $\underline{\w}^\pr$. 
From the assumptions, there are solutions $\underline{\Phi}(\t)$ and $\underline{\Phi}^\pr(\t)$ to the homotopy flow equation generated by some polynomial family $\underline{\eta}(\t)$ and $\underline{\eta}^\pr(\t)$ of unital $sL_\infty$-homotopies, respectively, such that
\begin{align*}
\left\{\begin{array}{l}
\underline{\Phi}(0)=\underline{\w},\\
\underline{\Phi}(1)=\underline{\tilde\w},
\end{array}\right.
\qquad
\left\{\begin{array}{l}
\underline{\Phi}^\pr(0)=\underline{\w}^\pr,\\
\underline{\Phi}^\pr(1)=\underline{\tilde\w}^{\hspace{1pt}\pr}.
\end{array}\right.
\end{align*}
\begin{lemma}\label{compho}
Suppose as above that $\underline{\Phi}(\t)$ and $\underline{\Phi^{\pr}}(\t)$ are time-dependent $sL_\infty$-homotopies such that the $\t=0$ evaluation of each one is an $sL_\infty$-morphism. Define $\underline{\Phi}^\ppr(\t)$ to be $\underline{\Phi}^\pr(\t)\bullet \underline{\Phi}(\t)$. Then $\underline{\Phi}^\ppr(\t)$ satisfies the homotopy flow equations generated by the polynomial family $\underline{\eta}^\ppr(\t)={\eta}_1^\ppr(\t),{\eta}_2^\ppr(\t),\dotsc$ of unital $sL_\infty$-homotopies, where
\begin{align*}
\eta^\ppr_n\big(v_1,\dotsc,& v_n) 
=
\sum_{\pi \in P(n)}\ep(\pi)\eta^\pr_{|\pi|}\left(\Phi(v_{B_1}),\dotsc,\Phi(v_{B_{|\pi|}})\right)
\\
&+\!\!\!\!\!\sum_{\substack{\pi \in P(n)}}\sum_{i=1}^{|\pi|}\ep(\pi,i)
\Phi^\pr_{|\pi|}\left( \Phi\big(v_{B_1}\big), \cdots,\Phi\big( v_{B_{i-1}}\big), \eta(v_{B_i}), \Phi\big(v_{B_{i+1}}\big),
\dotsc, \Phi\big( v_{B_{|\pi|}}\big)\right).
\end{align*}
\end{lemma}
The proof is similar to those above and is left to the reader. 

Note that $\underline{\Phi}^\ppr(\t)$ is a polynomial family of unital $sL_\infty$-morphisms from $V_{\Lie}$ to $V^\ppr_{\Lie}$ such that
\[
\left\{\begin{array}{l}
\underline{\Phi}^\ppr(0)=\underline{\w}^\pr\bullet \underline{\w},\\
\underline{\Phi}^\ppr(1)=\underline{\tilde\w}^{\hspace{1pt}\pr}\bullet\underline{\tilde\w}.\\
\end{array}\right.
\]
Hence, from {\bf Lemma \ref{compho}}, we conclude that
\begin{lemma}
If $\underline{\tilde\w}$ is homotopic to $\underline{\w}$ and $\underline{\tilde\w}^{\hspace{1pt}\pr}$ is homotopic to $\underline{\w}^\pr$, then $\underline{\tilde\w}^{\hspace{1pt}\pr}\bullet\underline{\tilde\w}$ is homotopic to $\underline{\w}^\pr\bullet \underline{\w}$ as unital $sL_\infty$-morphisms from $V_{\Lie}$ to $V^\ppr_{\Lie}$.
\end{lemma}
It follows that we have well-defined composition $\bullet_{\!h}$ of morphisms in the homotopy category $\homotopycat\category{\UsL}_\infty{\overk}$:
\[
\left[\underline{\w}^\pr\right]_\infty \bullet_{\!h} \left[\underline{\w}\right]_\infty\ceq 
\left[\underline{\w}^\pr\bullet\underline{\w}\right]_\infty,
\]
and the associativity of $\bullet$ implies that of $\bullet_{\!h}$.

\begin{definition}
The homotopy category 
 $\homotopycat\category{\UsL}_\infty{\overk}$
 of unital $sL_\infty$-algebras has the same objects as the the category $\category{\UsL}_\infty{\overk}$,
and its morphisms are the homotopy types of unital $sL_\infty$-morphisms with the composition operation $\bullet_{\!h}$.
 \end{definition}
 
Forgetting the unit and unit conditions, we have the homotopy category $\homotopycat\category{sL}_\infty{\overk}$ of $sL_\infty$-algebras.
Finally we can identify the structure of an $L_\infty$-algebra on $V$ as an $sL_\infty$-structure on $V[-1]$.
It is clear that morphisms and homotopies of $L_\infty$-algebras can be defined in a similar fashion. 

\subsubsection{Maurer--Cartan topics}

Let $\big(L, \underline{\ell}\big)$ be an $sL_\infty$-algebra.
The Maurer--Cartan (MC) equation of the $sL_\infty$-algebra is 
\[
d \g + \sum_{n=2}^\infty \Fr{1}{n!}\ell_n\big(\g,\dotsc, \g\big)=0
\]
where $d=\ell_1$ and $\g \in L^0$. Let $\g \in L^0$ be a solution to the MC equation. Then, it is a standard fact that 
\[
d_\g\ceq d + \ell_2\big(\g, \hbox{ }\big) +\sum_{n=3} \Fr{1}{(n-1)!}\ell_n\big(\g,\dotsc, \g, \hbox{ }\big):L^\bullet\rightarrow L^{\bullet+1}
\]
satisfies $d_\g^2=0$. If one defines the family $\underline{\ell_\g}=d_\g, \left(\ell_{\g}\right)_{2},\cdots$ such that, for all $n\geq 1$ and $x_1,\dotsc, x_n \in L$,
\[
 \left(\ell_{\g}\right)_{n}\big(x_1,\dotsc, x_n)
 \ceq \ell_n\big(x_1,\dotsc, x_n) +\sum_{k=1}^\infty \Fr{1}{k!}
 \ell_{n+k}\big(\g,\dotsc,\g, x_1,\dotsc, x_n),
 \]
then $\big(L, \underline{\ell_\g}\big)$ is also an $sL_\infty$-algebra.

\noindent{\em The deformation functor attached to $sL_\infty$-algebra:}
Recall that a Lie algebra can be tensored with a commutative and associative algebra to obtain a (current) Lie algebra. Similarly an $sL_\infty$-algebra and a graded super commutative and associate algebra can be combined into another $sL_\infty$-algebra by tensor product.
Let $(L,\underline{\ell})$ be an $sL_\infty$-algebra. 
Let $(\ma,\cdot)$ be a graded Artin local algebra with maximal ideal $\mm$ so that $\ma \simeq \fieldk\cdot 1_{\ma}\oplus \mm$.
Consider $\big(\mm\otimes L,\underline{\ell}\big)$ where 
\[
{\ell}_n\left(\a_1\otimes x_1, \cdots, \a_n\otimes x_n\right)
=\a_1\wedge\cdots \wedge\a_n \otimes \ell_n\left(x_1,\dotsc, x_n\right)(-1)^{\sum_{k=1}^n
|\a_k|(1+|x_1|+\cdots +|x_{k-1}|)}.
\]
Then it is easy to check that $\big(\mm\otimes L,\underline{\ell}\big)$ is a nilpotent $sL_\infty$-algebra.
Let $MC_{L}(\mm)$ denote the set of solutions to the corresponding MC equation;
\[
MC_{L}(\mm)=\left\{\g \in (\mm\otimes L)^0\big| d\g +\sum_{n=2}^\infty\Fr{1}{n!} \ell_n(\g,\dotsc,\g)=0\right\}.
\]
Then the follow lemma can be checked.
\begin{lemma}
Let $\g \in MC_{L}(\mm)$. For any $\l \in (\mm\otimes L)^0$, let $\G(\t)$ be a polynomial family of elements of $ (\mm\otimes L)^0$ parametrized by $\t \in [0,1]$ satisfying the following homotopy flow equation:
\[
\dot\G =d \mb{\l} 
+\sum_{n=2}^{\infty}\Fr{1}{(n-1)!}{\ell}_n\big(\G,\dotsc,\G,\mb{\l} \big),
\]
with initial condition $\G(0)=\g$. Then $\G(\t)$ is a uniquely defined family of elements of $MC_{L}(\mm)$.
\end{lemma}
The proof is by Artinian induction.

Let $\tilde\g = \G(1)$. We say $\tilde\g \ceq \G(1)$ is gauge equivalent to $\g$ by the gauge symmetry generated by $\l$. It can be checked that gauge equivalence is an equivalence relation. Define a set \[\hbox{Def}_L(\mm)\ceq MC(\mm)\big/\hbox{\text{gauge equivalence}}.\] Then it can be checked that $\hbox{Def}_L$ defines a (deformation) functor from the category $\category{Art}{\overk}$ of graded Artin local algebras to the category $\category{Set}$ of sets.

\noindent{\em MC solutions and $sL_\infty$-morphisms:}
Consider an $sL_\infty$-algebra $\big(V, \underline{\n}\big)$ such that $V$ is finite dimensional in each degree.
Let $V^*$ be the dual graded vector space.
Then the linear map $\n_n: S^n V\rightarrow V$ of degree $1$, $n\geq 1$, dualizes to a linear map $\d_n: V^* \rightarrow S^n V^*$ of degree $1$ which has a unique extension to a linear map $S(V^*)\rightarrow S(V^*)$ as a derivation of the free super-commutative and associative product $\wedge$ on $S(V^*)$. We will also use the notation $\d_n$ for this extension.
If, for example, we fix a basis $\{e_\a\}_{\a\in J}$ of $V$, the $sL_\infty$-structure $\underline{\n}$ is characterized by structure constants such that, $n\geq 1$,
\[
\n_n\big(e_{\a_1},\dotsc, e_{\a_n}\big)=\sum_\r f_{\a_1\cdots\a_n}{}^\r e_{\r}.
\]
Let $\{t^\a\}_{\a\in J}$ be the dual basis, $|e_\a|+|t^\a|=0$, and let $\fieldk[\![t^\a]\!]_{\a\in J}$ be the supercommutative formal power series ring (meaning the variables $t^\a$ are subject to the relations $t^{\a_1}t^{\a_2} =(-1)^{|t^{\a_1}||t^{\a_2}|}t^{\a_2}t^{\a_1}$).
Then $\fieldk[\![t^\a]\!]\simeq \fieldk\oplus \widehat{S}(V^*)$.
Let 
\[
\d_n = \Fr{1}{n!}\sum_{\a_1,\dotsc,\a_n,\r}(-1)^{|\a_1|+\cdots +|\a_n|}t^{\a_n}\cdots t^{\a_1}f_{\a_1\cdots\a_n}{}^\r\Fr{\rd}{\rd t^\r}
\]
and $\d\ceq {\d}_1+{\d}_2 +\cdots$, which is a derivation of degree $1$ on the free symmetric algebra $S(V^*)$.
Then $\underline{\n}$ being an $sL_\infty$-structure on $V$ implies that $\d^2=0$;
\[
\d_1^2=0,\quad \d_1\d_2+\d_2\d_1=0,\quad \d_2^2 +\d_1\d_3+\d_3\d_1=0,\quad \cdots
\]
The converse can be checked easily so that an $sL_\infty$-structure on $V$ is equivalent to a CDGA $\big(S(V^*), \d,\wedge\big)$.

Let $\big(W, \underline{\ell}\big)$ be another $sL_\infty$-algebra.
We consider the tensor product $\widehat S(V^*)\otimes W$, which is a graded vector space with the natural grading that is additive for the tensor product.
Both the derivation ${\d}$ and the $sL_\infty$-operations $\underline{\ell}$ have natural extensions to act on $\widehat S(V^*)\otimes W$ via
\begin{align*}
{\d}\big(\a\!\otimes\! x\big) &= {\d}a\!\otimes\! x
,\\
{\ell}_{\!n}\!\left(\a_1\!\otimes\! w_1, \dotsc, \a_n\!\otimes\! w_n\right)
&=\a_1\!\wedge\!\cdots \!\wedge\!\a_n \!\otimes \!\ell_{\!n}\!\left(w_1,\dotsc, w_n\right)(-1)^{\sum_{k=1}^n
|\a_k|(1+|w_1|+\cdots +|w_{k-1}|)},
\end{align*}
where $\a_1,\dotsc, \a_n$ and $w_1,\dotsc, w_n$ are elements in $\widehat S(V^*)$ and $W$, respectively. Consider the family of operations $\underline{\grave\ell}=\grave{\ell}_1,\grave{\ell}_2,\dotsc$, where $\grave{\ell}_1 = {\ell}_1- {\d}$ and $\grave{\ell}_n ={\ell}_n$ for all $n\geq 2$. It is trivial to show that $\underline{\grave\ell}$ gives $\left(\widehat S(V^*)\widehat\otimes W\right)$ the structure of an $sL_\infty$-algebra. We shall use the notation that $d=\ell_1$ and $\grave{d}=\grave{\ell}_1$.

Let $\underline{\phi}=\phi_1,\phi_2,\cdots$ be an $sL_\infty$-morphism from $(V,\underline{\n})$ to $(W,\underline{\ell})$.
Let $\mb{\phi}=\mb{\phi}_1+\mb{\phi}_2+\cdots$, where 
$$
\mb{\phi}_n=\Fr{1}{n!}\sum_{\a_1,\dotsc,\a_n}t^{\a_n}\cdots t^{\a_1}\phi_n(e_{\a_1},\dotsc, e_{\a_n}).
$$
Then, it is easy to check that $\mb{\phi}$ is in $\big(\widehat {S}(V^*)\widehat\otimes W\big)^0$ and satisfies
\eqn{mxa}{
d\mb{\phi}+\sum_{n=2}^\infty \Fr{1}{n!}{\ell}_n\big(\mb{\phi},\dotsc,\mb{\phi}\big) 
= {\d} \mb{\phi}.
}
The element $\mb{\phi}$ is sometimes called a \emph{twisting cochain}.

Let $ \underline{\tilde\phi}=\tilde\phi_1,\tilde\phi_2,\tilde\phi_3,\cdots$ be homotopic to $\underline{\phi}$, using the one-parameter family time-independent $sL_\infty$-homotopies $\underline{\eta}(\t)=\eta_1(\t),\eta_2(\t),\eta_3(\t),\dotsc$.
Recall that $\eta_n(\t)$, $n\geq 1$,
is a family of linear maps from $S^n(H)$ to $V$ of degree $-1$ so that 
$$
\mb{\etaup}_n \ceq \Fr{1}{n!}\sum_{\a_1,\dotsc,\a_n}(-1)^{|t^{\a_1}|+\cdots+|t^{\a_n}|}t^{\a_n}\cdots t^{\a_1}\eta_n(e_{\a_1},\dotsc, e_{\a_n})
$$
is an element in $(\widehat S^n(V^*)\widehat\otimes W)^{-1}$.
Let $\mb{\etaup} =\mb{\etaup}_1 +\mb{\etaup}_2+\mb{\etaup}_3+\cdots 
\in \big(\widehat {S}(V^*)\widehat\otimes W\big)^{-1}$.
Then, by {\bf Definition/Lemma \ref{l-hom} }and our notation, there is a smooth family $\mb{\Phi}(\t)$, parametrized by $\t \in [0,1]$, of elements in $\big(\widehat {S}(H^*)\hat\otimes V\big)^{0}$ uniquely defined by the following homotopy flow equation
\eqn{mxb}{
\dot{\mb{\Phi}}(\t) = d \mb{\etaup} 
+\sum_{n=2}^{\infty}\Fr{1}{(n-1)!}{\ell}_n\big(\mb{\Phi}(\t),\dotsc,\mb{\Phi}(\t),\mb{\etaup} \big)
-{\d}\mb{\etaup} 
}
such that $\mb{\Phi}(0) =\mb{\phi}$, $\mb{\Phi}(1)={\mb{\tilde\phi}}$ and $
{d}\mb{\Phi} + \sum_{n=2}^\infty \Fr{1}{n!}{\ell}_n\big(\mb{\Phi},\dotsc,\mb{\Phi}\big) 
= {\d} \mb{\Phi}$.
Equivalently, $\mb{\phi}$ solves the Maurer--Cartan equation of the $sL_\infty$-algebra $\big(\widehat{S}(V^*)\widehat{\otimes} W,\underline{\grave\ell}\big)$ and $\mb{\tilde\phi}$ is gauge equivalent to $\mb{\phi}$, since \eq{mxa} and \eq{mxb} can be rewritten as
\begin{align*}
\grave d\mb{\phi}+\sum_{n=1}^\infty \Fr{1}{n!}\grave{\ell}_n\big(\mb{\phi},\dotsc,\mb{\phi}\big) =0,
\\
\dot{\mb{\Phi}}(\t) =\grave d \mb{\etaup} 
+\sum_{n=2}^{\infty}\Fr{1}{(n-1)!}\grave{\ell}_n\big(\mb{\Phi}(\t),\dotsc,\mb{\Phi}(\t),\mb{\etaup} \big).
\end{align*}

\subsubsection{Minimal and Formal \texorpdfstring{$sL_\infty$}{sL-infinity}-algebras.}
The following well-known definitions and theorem have a long history in algebraic topology. One reasonable early reference is to the rational homotopy theory of Sullivan \cite{Su77}.
\begin{definition}[Minimal $sL_\infty$-algebras]
An $sL_\infty$-algebra $(V, \underline{\n})$ is called \emph{minimal} if $\n_1=0$.
\end{definition}
Assume that $(V, \underline{\n})$ is minimal and dualizable to the CDGA $\big(S(V^*), \d,\wedge\big)$.
It follows that $\d=\d_2+\d_3+\cdots$ and such a CDGA was called minimal in \cite{Su77}.

\begin{definition}[Formal $sL_\infty$-algebras]
An $sL_\infty$-algebra $(W, \underline{\ell})$ is called \emph{formal} if it is quasi-isomorphic to a graded sLie algebra.
\end{definition}
Assume that $(V, \underline{\n}=\n_2)$ is a graded sLie algebra which is dual to the CDGA $\big(S(V^*), \d=\d_2,\wedge\big)$. Recall that any $sL_\infty$-morphism $\underline{\phi}=\phi_1,\phi_2,\phi_3,\dotsc$ from $(V, \n_2)$ to $(W, \underline{\ell})$ correspond to a solution $\mb{\phi}$ to the MC equation
\[
d\mb{\phi}+\sum_{n=1}^\infty \Fr{1}{n!}{\ell}_n\big(\mb{\phi},\dotsc,\mb{\phi}\big) 
={\d}_{2}\mb{\phi}\Leftrightarrow
\left\{
\begin{array}{cl}
d\mb{\phi}_1 &=0,\\
d\mb{\phi}_2 +\Fr{1}{2}\ell_2\big(\mb{\phi}_1,\mb{\phi}_1\big)&={\d}_{2}\mb{\phi}_1,\\
\vdots &,\\
\end{array}\right.
\]
where $d=\ell_1$. 

Let $H$ denote the cohomology of the underlying cochain complex $(W, d)$. 
It follows that $\phi_1$, which is a cochain map from $(V,0)$ to $(W,d)$, induces a linear map $H(\phi_1)$ from the vector space $V$ to the vector space $H$. 

The map $H(\phi_1)$ can be more or less explicitly constructed as follows:
There always exists a cochain quasi-isomorphism $f:(H,0)\rightarrow (W,d)$, which induces the identity map $I_H$ on $H$,
constructed by linearly choosing representatives of each cohomology class. Note that $df=0$ and that for $x\in H$, the cohomology class of $f(x)\in W$ is $x$ by definition. 
Such a map is unique only up to cochain homotopy. Then there is a homotopy inverse $h: W\rightarrow H$ to $f$ such that 
\begin{align*}
h\circ f&=I_H
,\\
f\circ h&= I_W + d\b +\b d,
\end{align*}
for some cochain homotopy $\b$ from $W$ to $W$. 
Then $H(\phi_1)= h\circ \phi_1:V\rightarrow H$. 

Now assume there is an $sL_\infty$-quasi-isomorphism $\underline{\phi}$ from the graded sLie algebra $(V, \underline{\n}=\n_2)$ to $(W,\underline{\ell})$. Then $h\circ \phi_1$ is an isomorphism so that $V$ is isomorphic to $H$ as a graded vector space. Furthermore, the bracket $\ell_2^H$ gives $H$ the structure of a graded sLie algebra that is isomorphic to $(V,\n_2)$. It then follows that $(H, \ell_2^H)$ must be quasi-isomorphic to $(W, \underline{\ell})$ as an $sL_\infty$-algebra via $\underline{\phi}\bullet (h\circ\phi_1)^{-1}$.
% More specifically, let $\phi^H_n(x_1,\dotsc, x_n) = \phi_{n}\big(\p(x_1),\dotsc, \p(x_n)\big)$, where $\p$ is the inverse to the isomorphism $H(\phi_1)= h\circ \phi_1:V\rightarrow H$,
%then $\underline{\phi}^H=\phi^H_1, \phi^H_2,\cdots$ is such a quasi-isomorphism.

In general, it is easy to see that the cohomology $H$ of the $sL_\infty$-algebra $(W, \underline{\ell})$ has the uniquely induced structure $\ell_2^H$ of a graded sLie algebra since $d$ is a derivation of $\ell_2$. Define $L_2: S^2 H \rightarrow W$ as $L_2(x_1,x_2)= \ell_2\big(f(x_1),f(x_2)\big)$. Then 
\[\ell_2^H\ceq h\circ L_2:S^2 H\rightarrow H\]
 is independent of the choice of $f$. It follows that 
 \eqn{minb}{
 f\circ \ell_2^H(x_1,x_2) = f\circ h\circ L_2(x_1,x_2)= \ell_2(f(x_1),f(x_2)) + d(\b\circ L_2)(x_1,x_2)
 } 
 so that $f$ is an sLie-algebra homomorphism up to the homotopy $\phi_2^H\ceq \b\circ L_2:S^2 H\rightarrow W$. The above equation together with the Jacobi-identity of $\ell_2$ up to homotopy $\ell_3$ implies that
\eqnalign{minc}{
f&\left(\ell^H_2\big(\ell^H_2(x_1, x_2),x_3\big)
+(-1)^{|x_1|} \ell^H_2\big(x_1, \ell^H_2 (x_2, x_3)\big)
+(-1)^{(|x_1|+1)|x_2|} \ell^H_2\big(x_2, \ell^H_2 (x_1, x_3)\big)\right)
\\
&
=-dL_3(x_1,x_2,x_3),
}

where $L_3$ is a certain linear map from $S^3 H$ to $W$ of degree $-1$ which is given explicitly in the proof of Lemma~\ref{coL}.

Applying $h$ to \eq{minc} together with the identities $h\circ f=I_H$ and $hd=0$, we have the Jacobi identity for $\ell^H_2$
\eqn{mindd}{
\ell^H_2\big(\ell^H_2(x_1, x_2),x_3\big)
+(-1)^{|x_1|} \ell^H_2\big(x_1, \ell^H_2 (x_2, x_3)\big)
+(-1)^{(|x_1|+1)|x_2|} \ell^H_2\big(x_2, \ell^H_2 (x_1, x_3)=0.
}
Hence $(H, \ell_2^H)$ is always a graded sLie algebra. This, however, does not imply that $(W,\underline{\ell})$ is formal.

From \eq{mindd}, the relation \eq{minc} implies that $d L_3=0$. We hence can have a non-zero map $\ell_3^H: S^3 H\rightarrow H$ of degree $1$ defined by $\ell_3^H\ceq h\circ L_3$, which implies that 
\[
f\circ \ell_3^H(x_1,x_2,x_3) = \ell_3(f(x_1),f(x_2),f(x_3)) + d(\b\circ L_3)(x_1,x_2,x_3).
\] Once we define $\phi^H_3 = \b\circ L_3$, we notice that $0, \ell_2^H, \ell_3^H$ constitute the leading parts of a minimal $sL_\infty$-structure from $H$ and $\phi_1^H=f, \phi^H_2,\phi^H_3$ constitute the leading parts of an $sL_\infty$-quasi-isomorphism from this minimal $sL_\infty$-algebra to $(W,\underline{\ell})$; this is stated in Lemma~\ref{coL}.
It follows that formality breaks down if there is an $\ell_3^H$ which cannot be made $0$ by an $sL_\infty$-isomorphism. So a non-trivial $\ell_3^H$ (in this sense) is the first obstruction to formality of $(W,\underline{\ell})$.

%By continuing the above procedure one can construct a minimal $sL_\infty$-algebra $(H,\underline{\ell}^H=0,\ell_2^H,\ell_3^H,\cdots)$
%and quasi-isomorphism $\underline{\phi}^H=\phi^H_1,\phi^H_2,\phi^H_3,\cdots$ to $(W, \underline{\ell})$, which is a well-known lemma; 

\begin{lemma}\label{coL}
On the cohomology $H$ of an $sL_\infty$-algebra $(V, \underline{\ell})$, there is the structure of a minimal $sL_\infty$-algebra $\big(H, \underline{\ell}^H=0,\ell^H_2, \ell^H_3,\cdots\big)$ together with an $sL_\infty$-quasi-isomorphism $\underline{\phi}$ to $V$ (which induces the {\em identity} on cohomology).
\end{lemma}
\begin{proof}
Assume that we have constructed $\ell^H_{(n)}=0,\ell^H_2, \cdots, \ell^H_n$ and $\phi^H_{(n)}=\phi^H_1, \cdots, \phi^H_n$ satisfying the first $n$ relations to be an $sL_\infty$-algebra and morphism. For the base case we let $\phi^H_1=f$. Now define the linear map $L_{n+1}: S^{n+1} H\rightarrow \sC$ of degree $1$ as
\begin{align*}
L_{n+1}\big(x_1,\dotsc, x_{n+1}\big)
=
&
\sum_{\substack{\pi \in P(n+1)\\\color{red}|\pi|\neq 1}}
\ep(\pi)
\ell_{|\pi|}\Big(
{\phi}^H\big(x_{B_1}\big),
\cdots,{\phi}^H\big( x_{B_{|\pi|}}\big)
\Big)
\\
 &
 -\!\!\!\!\!
 \sum_{\substack{\pi \in P(n+1)\\|B_i|=n-|\pi|+2\\ \color{red}|\pi|\neq 1}}\!\!\!\!\ep(\pi,i) 
 \phi^H_{|\pi|}\Big(
x_{B_1},\dotsc, x_{B_{i-1}},\ell^H\big( x_{B_i}\big), x_{B_{i+1}}, \cdots, x_{B_{|\pi|}}
\Big).
\end{align*}
Define $\ell^H_{n+1}\ceq h\circ L_{n+1}$ and $\phi^H_{n+1}\ceq \beta\circ L_{n+1}$. 
Then one can show the following identity (the details are an exercise left to the reader):
\eqnalign{minf}{
dL_{n+1}&\big( x_1,\dotsc,x_{n+1}\big)
\\
=-&\!\!\!\!
\sum_{\substack{\pi \in P(n+1)\\|B_i|=n -|\pi|+2}}\!\!\!\!
\ep(\pi,i) f\left(
\ell^H_{|\pi|}\Big(
x_{B_1},\dotsc, x_{B_{i-1}},\ell^H\big( x_{B_i}\big), x_{B_{i+1}}, \cdots, x_{B_{|\pi|}}\Big)
\right).
}
Applying the map $h:\sC\rightarrow H$ to the above, we have
\eqnalign{ming}{
0=\!\!\!\!
\sum_{\substack{\pi \in P(n+1)\\|B_i|=n -|\pi|+2}}\!\!\!\!
\ep(\pi,i) \ell^H_{|\pi|}\Big(
x_{B_1},\dotsc, x_{B_{i-1}},\ell^H\big( x_{B_i}\big), x_{B_{i+1}}, \cdots, x_{B_{|\pi|}}\Big),
}
which implies that $\ell^H_{(n+1)}=0,\ell^H_2, \cdots, \ell^H_{n+1}$ satisfy the first $(n+1)$ relations to be an $sL_\infty$-algebra.
%It also follows that $d L_{n+1}=0$ so that $\Im L_{n+1} \subset \Ker d\subset V$.
%Define the linear map $\ell^H_{n+1}:S^{n+1} H\rightarrow H$ of degree $1$ as
%\[
%\ell^H_{n+1}\ceq h\circ L_{n+1}.
%\]
Next note that by definition,
%
%Then $f\circ\ell^H_{n+1}= L_{n+1}+d\left( \b\circ L_{n+1}\right)$.
%Set $\phi^H_{n+1}= \b\circ L_{n+1}: S^{n+1} H\rightarrow \sC$,
%which is a linear map of degree $0$, to obtain
\eqn{minh}{
L_{n+1}+d\phi^H_{n+1} = \phi^H_1\circ\ell^H_{n+1},
}
which is equivalent to the following;
\begin{align*}
\sum_{\substack{\pi \in P(n+1)}}
\ep(\pi)
\ell_{|\pi|}&\Big(
{\phi}^H\big(x_{B_1}\big),
\cdots,{\phi}^H\big( x_{B_{|\pi|}}\big)
\Big)
\\
&
=\!\!\!\!\!
\sum_{\substack{\pi \in P(n+1)\\|B_i|=n-|\pi|+2}}\!\!\!\!\ep(\pi,i) 
\phi^H_{|\pi|}\Big(
x_{B_1},\dotsc, x_{B_{i-1}},\ell^H\big( x_{B_i}\big), x_{B_{i+1}}, \cdots, x_{B_{|\pi|}}
\Big).
\end{align*}
Hence $\phi^H_{n+1}=\phi^H_1,\phi^H_2, \cdots, \phi^H_n$ satisfies the first $(n+1)$ relations to be an $sL_\infty$-morphism.

%Note that many sources make the additional assumptions that $\beta\circ f = \beta\circ\beta = h\circ\beta = 0$; in this case the second sum can be left out of the definition of $L_n$.
\hfill\qed\end{proof}
\begin{remark}
Many sources make the equivalent choice to define $L_n$ without the second summand; in general the two definitions for $L_n$ yield different formulas beginning with $\ell_3$ and $\phi_3$. However, the overall transfered structures $\underline{\ell}^H$ obtained using the two choices of formula are isomorphic. The situation is less clear for the morphism $\underline{\phi}^H$; see the following remark. Many sources make the additional assumptions that $\beta\circ f = \beta\circ\beta = h\circ\beta = 0$. Under these assumptions, the two choices for $L_n$ yield the same formulas.
\end{remark}

\begin{remark}
The $sL_\infty$-structure $\underline{\ell}^H$ is unique up to homotopy. Because $\ell^H_1=0$, this implies that the structure is unique up to $sL_\infty$-isomorphism. However, there are two points to note about this fact, both related to the fact that the automorphism group is almost never trivial.

First, the {\em structure constants} of $\underline{\ell}^H$ are not uniquely determined after $\ell_2^H$ because a nontrivial automorphism can perturb them.

Second, the homotopy type of the $sL_\infty$-morphism $\underline{\phi}^H$ is not determined uniquely. Let us be more precise. The map $\phi_2^H\ceq \b\circ L_2$ depends on $\b$ and $f$, but this dependence can be absorbed by a suitably chosen $sL_\infty$-homotopy. This is not the type of non-uniqueness we mean.

On the other hand, consider $W$ with the zero $sL_\infty$-structure $\underline{\ell}=\underline{0}$. Then $H=W$, again with the zero $sL_\infty$-structure. {\em Any} sequence $\underline{\phi}^H$ beginning with the identity map is an $sL_\infty$-morphism which induces the identity on cohomology between $H$ and $W$; however, two such $sL_\infty$-morphisms are $sL_\infty$-homotopic if and only if they are equal. One can construct counterexamples in less degenerate situations as well.
\hfill$\natural$
\end{remark}

\begin{remark}
Lemma \ref{coL} is generally attributed to Kadeishvili \cite{KadeA}, where the $A_\infty$-version of it was proven. By that time, various versions of this lemma had already appeared, if implicitly, in \cite{Su77,Chen,Hain}. 
\hfill$\natural$.
\end{remark}

\begin{remark}
The definitions of the category and the homotopy category of unital $sA_\infty$-algebras and associated Maurer--Cartan topics can be obtained by replacing symmetric products with tensor products, classical partitions with interval partitions, and $1/n!$ with one everywhere in section~\ref{subs:Linfty}.
\hfill$\natural$.
\end{remark}

\begin{remark}
The following is a philosophical remark. A mathematical model of physical phenomena may be deduced from a set of data, collected coherently and completely (in some sense), after a systematic process of removing redundancy such that every possible ambiguity caused by the removal of redundancy is under control. This process is clearly not unique, though the final outcome (``the law'') should be invariant of the choices involved. Then, this infinite sequence of ambiguities may be organizable by introducing an infinite sequence of homotopies such that ``the law'' is an invariant of the homotopy type.
\hfill$\natural$
\end{remark}

\section{Infinitesimal symmetries of the expectation and their applications}\label{sec:infsym}

By a {\em symmetry} of a classical algebraic probability space we mean a symmetry of the expectation morphism $\iota$.
In general, the expectation morphism $\iota$ of a classical algebraic probability space $\iota: A_{\classical} \rightarrow \fieldk$ has a large kernel. As a result, this expectation morphism has non-trivial symmetries. Exploiting the structure of these symmetries leads to a natural enrichment of the notion of a classical algebraic probability space to a (homotopy) probability space.
Ideally, one would like to factorize the expectation $\iota$ through the quotient map from $A$ to $A/\Ker \iota$. But this is not viable in practice since constructing this quotient map is equivalent to computating the expectation value of every random variable. 
Alternatively, one may work with a certain null-space $\sN \subset \Ker\iota$ with a more controllable algebraic structure on it. In this approach, one factorizes $\iota$ through a ``bigger'' quotient $A/\sN$, rather than $A/\Ker \iota$. One of the results of this paper is a method to determine the law of random variables up to finite ambiguity in the case that the quotient space $A/\sN$ is finite dimensional. The major technical and conceptual difficulty to carry out this program is that in general, neither $A/\Ker \iota$ nor the more reasonable quotient $A_\mg=A/\sN$ has a canonically defined algebraic structure induced from the algebra $A_{cl}$. This is because $\Ker \iota$ is an ideal of $A$ if and only if the expectation morphism $\iota$ is an algebra homomorphism, meaning that there are absolutely no correlations! We shall resolve these difficulties with a little help from $\infty$-homotopy theory after revising some of the basic notions in classical algebraic probability theory to rephrase the law of random variables as a homotopy invariant.

The choice of $\sN$ will not be arbitrary and will be based on certain carefully chosen kinds of symmetries of the expectation morphism. A particularly useful notion of symmetry for our purpose is that of an \emph{infinitesimal symmetry} of the expectation.
Let $L\!\Diff_\fieldk(A)$ denote the Lie algebra of linear algebraic differential operators on $A$. Then an infinitesimal symmetry is defined as a Lie algebra representation $\vr:\mg\rightarrow L\!\Diff_\fieldk(A)$ such that $\Im\vr(\mg) \in \Ker \iota$, that is, $\iota\big(\vr(g)(x) \big)=0$ for all $g\in \mg$ and $x\in A$. 
Then the corresponding null-space is $\sN =\vr(\mg)A$ so that $A/\sN$ is the coinvariants $A_\mg$ of the $\mg$-module $A$.
The presence of such an infinitesimal symmetry naturally leads to a generalization of an classical algebraic probability space in cochain complexes. This generalization will be the prototype of binary \padj{} probability spaces, and will recover the desired coinvariants as the degree $0$ cohomology.

 \subsection{Symmetries of the expectation}\label{subs:symmetries}
 Fix a classical algebraic probability space $\xymatrix{A_{\classical} \ar[r]^\iota &\fieldk}$.
 
Let $\End_\fieldk(A)$ denote the space of all linear endomorphisms of $A$. This space has the structure of a unital associative algebra via composition $\circ$ and the identity endomorphism $I_A$. Denote by $L\End_\fieldk(A)$ the same space viewed as a Lie algebra with the commutator bracket $[\hbox{ },\hbox{ }]_\circ$.
Let $\Aut_\fieldk(A)$ denote the space of all linear automorphisms of $A$---the space of all invertible linear endomorphisms, which forms a group with composition $\circ$ as the binary operation and the identity map $I_A$ as its unit. We will consider the left actions of $\End_\fieldk(A)$ and $\Aut_\fieldk(A)$ on $A$, which make it both a left $L\End_\fieldk(A)$-module and a left $\Aut_\fieldk(A)$-module. 

Recall that for $G$ a group, the coinvariants $V_G$ of a left $G$-module $V$ are defined to be the quotient $V/DV$, where $DV$ denotes the submodule of $V$ generated by $g.v- v$, for all $g\in G$ and $v\in V$.
Let $\mg$ be a Lie algebra. Then the coinvariants $V_\mg$ of a left $\mg$-module $V$ are defined to be the quotient $V/\mg.V$, where $\mg.V$ denotes the submodule of $V$ generated by $g.v$ for $g\in \mg$ and $v\in V$.

The symmetry group of the expectation morphism $\iota$, denoted by $\Aut_\fieldk(\iota)$, is defined to be the space of all linear automorphisms of $A$ fixing the expectation:
\[
\Aut_\fieldk(\iota)=\{\ma \in \Aut_\fieldk(A)\big| \iota\circ \ma =\iota\}.
\]
The symmetry group is a subgroup of $\Aut_\fieldk(A)$. 

The symmetry algebra of the expectation morphism $\iota$, denoted by $\End_\fieldk(\iota)$, is defined to be the space of all linear endomorphisms of $A$ whose images are contained in $\Ker\iota \in A$:
\[
\End_\fieldk(\iota)\ceq \{\r \in \End_\fieldk(A)\big| \iota\circ \r =0\}.
\]
The symmetry algebra is an associative algebra without unit. We denote by $L\!\End_\fieldk(\iota)$ the commutator Lie algebra of the symmetry algebra, a subalgebra of $L\!\End_\fieldk(A)$.

Consider $\mathfrak{e}\in \End_\fieldk(A)$ defined as $\mathfrak{e}(x)= x -\iota(x)\cdot 1_A$ for all $x\in A$.
Then $\mathfrak{e} \in \End_\fieldk(\iota)$ since $\iota\big(\mathfrak{e}(x)\big)= \iota(x) -\iota(x)\cdot\iota(1_A)=0$.
In general,
\[
\iota\Big( \prod_{j=1}^n \mathfrak{e}(x_j)\Big)
=\begin{cases}
0& n=1,
\\
\k_n(x_1,\dotsc, x_n)&n= 2,3.
\end{cases}
\]

Any element $x \in A$ can be decomposed as 
\[
x = \iota(x)\cdot 1_A + \mathfrak{e}(x),
\]
which implies that $x \sim \iota(x) \cdot 1_A$ in the quotient $A_{L\End_\fieldk(\iota)}$. It follows that $A_{L\End_\fieldk(\iota)}\simeq \fieldk$. 

Now we turn to $\Aut_\fieldk(\iota)$.
Let $\mathfrak{a}\in \Aut_\fieldk(\ker\iota)$. Then $\mathfrak{a}$ acts on the image of $\mathfrak{e}$. Moreover, $\mathfrak{e}$ restricts to the identity $I_{\ker\iota}$ on $\ker\iota$ so $\mathfrak{e}\mathfrak{a}=\mathfrak{a}$ for all $\mathfrak{a}$.

Now define $\vr(\mathfrak{a})\in \End_\fieldk(A)$ as 
\[\vr(\mathfrak{a}) = I_A + (\mathfrak{a}-I_{\ker\iota})\mathfrak{e}.\]
This operator preserves $\iota$. Now for $\mathfrak{a}$, $\mathfrak{a}'\in \Aut_\fieldk(\ker\iota)$, since $\mathfrak{e}\mathfrak{a}'=\mathfrak{a}'$, a quick calculation shows that
\[\vr(\mathfrak{a})\circ\vr(\mathfrak{a}') = \vr(\mathfrak{a}\circ\mathfrak{a}').\]
Note also that $\vr(\mathfrak{I_{\ker\iota}})=I_A$. Then $\vr(\mathfrak{a})\in\Aut_\fieldk(\iota)$ and we achieve the following.
\begin{lemma}
The map $\vr$ defined as above is a linear and faithful representation 
\[
\vr : \Aut_\fieldk(\ker\iota) \rightarrow \Aut_\fieldk(A)
\] 
which satisfies $\iota\circ\Im \vr=\iota$.
\end{lemma}
Moreover
\[
x = \iota(x)\cdot 1_A + \rho(-I_{\ker\iota})(y) -y \qquad y=-\Fr{x}{2}\in A, 
\]
which implies that $x \sim \iota(x) \cdot 1_A$ in the quotient $A_{\Aut_\fieldk(\iota)}$.
It follows that $A_{\Aut_\fieldk(\iota)}\simeq \fieldk$.  We thus have shown that both $\End_\fieldk(\iota)$ and $\Aut_\fieldk(\iota)$ are non-trivial and their coinvariant modules are isomorphic to $\fieldk$.

In fact, we have shown more. If $\ker\iota\ne 0$, the element $\rho(-I_{\ker\iota})$ generates a $\Z_2$ subgroup of $\Aut_\fieldk(\iota)$, which we call the fundamental symmetry of the classical algebraic probability space $A$, and the coinvariant module even of this fundamental symmetry subgroup is isomorphic to $\fieldk$.

On the other hand, it is extremely difficult to work with this representation since the explicit construction of the automorphism $\vr(-I_{\ker\iota})$ means that we must know the expectation value $\iota(x)$ of every element $x \in A$ to calculate it, since $\vr(-I_{\ker\iota})(x) = -x +2\iota(x)\cdot 1_A$. Nevertheless, we have learned an important lesson: {\em the expectation morphism of any classical algebraic probability space can be completely characterized by its symmetry}. 
% \begin{remark}
% Our $\Z_2$-symmetry of expectation is a part of general $\fieldk^\times$-symmetry.\footnote{This was pointed out
% by G.\ Drummond-Cole after reading this manuscript.}
% Consider the family $\{\r_a\}_{a \in \fieldk^\times} \in \End_\fieldk(A)$,
% where, for all $x \in A$,
% \[
% \r_a(x) \ceq a\cdot x + (1-a)\cdot \iota(x).
% \]
%  Then, we have $\iota\circ \r_a = \iota$ for all $a \in \fieldk^\times$.
%  Note that $\r_1 = I_A$ and $\r_a\circ \r_b = \r_{ab}$, for all $a,b \in \fieldk^\times$,
% such that $\r^{-1}_a = \r_{1/a}$. Hence we have a faithful representation $\vr :\fieldk^{\times}\rightarrow \Aut_\fieldk(A)$ of 
% the multicative group $\fieldk^{\times}$ defined by $\vr(a) \ceq \r_a$, for all $a\in \fieldk^\times$,
% which is a symmetry of the expectation. Note that $\Z_2$ is isomorphic to the subgroup $\{1,-1\}$ of $\fieldk^\times$ 
% and $\vr_{1}=I_A$ and $\vr_{-1}=\vr(-I_{\ker\iota})$.
% \hfill$\natural$.
% \end{remark}

In practice, we shall deal with a specific subalgebra $L\!\Diff_\fieldk(\iota)$ 
of Lie algebra symmetries $L\End_\fieldk(\iota)$ which does not contain $\mathfrak{e}$ and bypass the symmetry group $\Aut_\fieldk(\iota)$ of our classical algebraic probability spaces. We work with the Lie algebra than with the group because infinite-dimensional cases are easier to handle with Lie algebras.

\subsection{Infinitesimal symmetries of classical algebraic probability spaces}\label{subs:infsyms}

Now we are ready to consider infinitesimal symmetries of a classical algebraic probability space $\xymatrix{A_{\classical} \ar[r]^\iota &\fieldk}$. 

Consider an element $\r \in \End_\fieldk(A)$ as an operator acting on the left on the unital commutative algebra $A$. Then define a family $\underline{\ell}^\r=\ell^\r_1,\ell^\r_2,\cdots$ by
\[
\ell^\r_{j}(x_1,\dotsc, x_j) \ceq \left[\left[\left[\cdots\left[\r, L_{x_1}\right], L_{x_2}\right],\cdots\right],L_{x_j}\right](1_A)
\] 
where $L_x$ denotes the operator of left multiplication by $x \in A$. 
For example
\begin{align*}
\ell^\r_1(x)&= \r(x)- x\cdot \r(1_A)
,\\
\ell^\r_2(x,y)&=\r(x\cdot y) - x\cdot \r(y) - y\cdot\r(x) + y\cdot x \cdot \r(1_A).
\end{align*}
We say, following Grothendieck, that $\r$ is an algebraic differential operator of order $n$ if $\ell^\r_{n}\neq 0$ and $\ell^\r_{n+1}=0$.
In particular the condition $\ell_1^\r=0$ implies that $\r$ is right multiplication by $\r(1_A) \in A$. 

The space $\Diff_\fieldk(A)$ of all linear algebraic differential operators forms a subalgebra of the associative algebra $\End_\fieldk(A)$, and has the natural structure of a filtered algebra. Let $F_n\Diff_\fieldk(A)$ denote the space of all elements with order less or equal to $n\geq 0$. Then we have a increasing filtration 
\[
F_0\Diff_\fieldk(A) \subset F_1\Diff_\fieldk(A) \subset F_2\Diff_\fieldk(A) 
\subset\cdots
\]
which satisfies 
\begin{align*}
F_j\Diff_\fieldk(A)\circ F_k\Diff_\fieldk(A) &\subset F_{j+k}\Diff_\fieldk(A)
,\\
\left[F_j\Diff_\fieldk(A), F_k\Diff_\fieldk(A)\right]_\circ & \subset F_{j+k-1}\Diff_\fieldk(A).
\end{align*}
We denote by $L\!\Diff_\fieldk(A)$ the Lie subalgebra $\big( \Diff_\fieldk(A), [\hbox{ },\hbox{ }]_\circ\big)$ of $L\!\End_\fieldk(A)$ which inherits the structure of a filtered Lie algebra from the filtration $F_\bullet \Diff_\fieldk(A)$. Note that the distinguished endomorphism $\mathfrak{e}$ is {\em not} an element of $\Diff_\fieldk(A)$. 

The \emph{infinitesimal symmetry algebra} of the expectation morphism $\iota$, denoted by $\Diff_\fieldk(\iota)$, is the space of all linear algebraic differential operators on $A$ with image contained in $\Ker\iota \in A$:
\[
\Diff_\fieldk(\iota)\ceq \{\r \in \Diff_\fieldk(A)\big| \iota\circ \r =0\},
\]
The infinitesimal symmetry algebra is an associative algebra without unit via the composition product.
We denote by $L\!\Diff_\fieldk(\iota)$ the associated commutator Lie algebra, which is a subalgebra of $L\End_\fieldk(\iota)$.

\begin{definition}[Infinitesimal symmetry]
An \emph{infinitesimal symmetry} of a classical algebraic probability space $\xymatrix{A_{\classical} \ar[r]^\iota &\fieldk}$ is a Lie algebra $\mg$ equipped with a linear representation $\vr:\mg\rightarrow L\!\Diff_\fieldk(A)$ such that $\iota\circ\Im\vr=0$, i.e., $\Im\vr \subset \Diff_\fieldk(\iota)$.
\end{definition}

Any linear representation $\vr:\mg\rightarrow L\!\Diff_\fieldk(A)$ of a Lie algebra $\mg$ such that $\iota\circ\Im\vr=0$ induces a Lie algebra homomorphism $\check{\vr}:\mg \rightarrow L\!\Diff_\fieldk(\iota)$.

\begin{definition}[Faithful infinitesimal symmetry]
An infinitesimal symmetry $\vr:\mg\rightarrow L\!\Diff_\fieldk(A)$ of the classical algebraic probability space $\xymatrix{A_{\classical} \ar[r]^\iota &\fieldk}$ is called \emph{faithful} if the map $\check{\vr}$ is injective.
\end{definition}
We shall implicitly assume that all symmetries are faithful.

Assume that the classical algebraic probability space $\xymatrix{A_{\classical} \ar[r]^\iota &\fieldk}$ has an infinitesimal symmetry $\vr:\mg\rightarrow L\hbox{Diff}_\fieldk(A)$.
Then the expectation morphism $\iota$ induces a unique map $\iota_\mg$ from the coinvariants $A_{\mg}=A/\mg. A$ of the $\mg$-module $A$ to $\fieldk$ such that the following diagram commutes
\[
\xymatrix{
A \ar[r]^\iota \ar[d]_\mq & \fieldk
\\
A_{\mg}\ar[ru]_{\iota_\mg}&
}.
\]

Now we give four examples which demonstrate the usefulness of considering infinitesimal symmetries. 

\begin{example}\label{gaussian-ex1}
Set $A= \R[s]$ with $1_{A}=1$. Fix some $\s\in \R_+$. Define the expectation morphism $\iota: \R[s]\rightarrow \R$ for $O\in \R[s]$ as
\[
\iota\big(O\big)=\int^{+\infty}_{-\infty} O\cdot e^{-\fr{s^2}{2\s^2}} \text{d}s\left/
\int^{+\infty}_{-\infty} e^{-\fr{s^2}{2\s^2}} \text{d}s\right.
\]
The normalization ensures that $\iota(1_A)=1$. This expectation morphism corresponds to the Gaussian distribution with zero mean and variance $\s^2$. Note that the linear map 
\[
\r\ceq -\s^2\Fr{d}{ds} + L_s: \R[s]\rightarrow \R[s],
\]
is in the symmetry algebra of the expectation $\iota$ since, for all $O\in A$, we have
\[
\iota\circ\r(O)= \iota\left(-\s^2\Fr{d O}{ds} + s\cdot O\right)
\propto \int^{s=+\infty}_{s=-\infty}\text{d}\left( O\cdot e^{-\fr{s^2}{2\s^2}} \right)=0.
\]
Define $\mg$ to be the Lie algebra $\R\cdot e$ over $\R$. Define the linear map $\vr:\mg \rightarrow \Diff_{\R}(A)$ by $\vr(e)=\r$. Then we have an $\R$-linear representation $\vr:\R \rightarrow L\!\Diff_{\R}(A)$ which satisfies $\iota\circ\Im \vr=0$.
From
\begin{align*}
\r(1_A)=& s
,\\
\r(s^n)=& -\s^2 n s^{n-1} + s^{n+1}, \quad n\geq 1,
\end{align*}
we have $\mq\big(s^{2k}\big) =\s^{2k} (2k-1)!! \mq(1_A)$ and $\mq\big(s^{2k-1}\big)=0$. Thus $A_{\mg} =\mathbb{R}\cdot \mq(1_A)$. It follows that
\[
\iota(s^n) =
\begin{cases}
0 &\text{$n$ odd},\\
\s^{n}\cdot (n-1)!! &\text{$n$ even}.
\end{cases}
\]
Hence ``the law'' is determined completely by a symmetry argument alone.
\hfill$\natural$
\end{example}

\begin{example}\label{gaussian-ex2}
Consider the same Gaussian example.
From the identity $(2k-1)!! =\Fr{(2k)!}{2^k k!}$, $k\geq 1$, we obtain that
\[
Z(t)_s \ceq \iota\left(e^{t s}\right)= e^{\fr{1}{2}t^2 \s^2}.
\]
It follows that $\k_1(s)=0$, $\k_2(s,s) =\s^2$ and $\k_n(s,\dotsc, s)=0$ for all $n\geq 3$.
Note that $Z(t)_s$ satisfies the following differential equation;
\begin{align*}
\left(\Fr{d^2}{dt^2} - \s^2\big(1+ t^2\s^2\big)\right)Z(t)_s=0
\quad\hbox{\em where}\quad
Z(0)_s =1.
\end{align*}
This equation can also be determined directly by symmetry arguments alone.

Define a family of $\C$-linear maps $\iota_{t}: A \rightarrow \R[\![t]\!]$ as $\iota_{t}(O)\ceq \iota\left(e^{t s}\cdot O\right)$ so that $
Z(t)_s = \iota_t(1_A)$.
Then we have the one parameter family of representations $\vr_{t}: \mg\rightarrow L\!\Diff_\R( A)$:
\[
\vr_t(e) 
\ceq \r_t
\ceq e^{-ts}\r e^{ts}= -\s^2\Fr{d}{ds} + L_{s -\s^2 t}.
\]
The family $\r_t$ satisfies $\iota_t\circ \r_t =0$, since
\[
\iota_t\circ \r_t(O) =\iota\left(e^{ts}\cdot e^{-ts}\r e^{ts}(O)\right)
=\iota\left(\r \left(e^{ts}\cdot O\right)\right)=\iota\circ\r\left(e^{ts}\cdot O\right)=0.
\]
From $\r_t(1) = s - \s^2 t$ and $\r_t(s) = -\s^2 +s^2 -\s^2 t \cdot s$, we obtain that
\[
s^2 = \s^2(1+t^2\s^2)\cdot 1_A +\r_t(t \s^2 + s).
\]
From $\Fr{d^2}{dt^2}\iota_t(1_A) =\iota_t\left(s^2\right)$ and $\iota_t\circ\r_t=0$, it follows that
\[
\Fr{d^2}{dt^2} \iota_t(1_A)= \s^2(1+t^2\s^2)\cdot\iota_t(1_A)
\]
which is the desired differential equation for $Z(t)_s \equiv \iota_t(1_A)$.
\hfill$\natural$
\end{example}

\begin{example}\label{period-ex}
Consider $A \ceq \C[p, \underline{s}]$, where $\underline{s} = s_0,s_1,\dotsc, s_n$, which has the structure $A_{\classical} =(A,1_A,\cdot)$ of unital associative commutative $\C$-algebra with unit $1_A=1$.
Let 
\[S(p,\underline{s}) = p\cdot G(\underline{s}),\] where $G(\underline{s})$ is a generic homogeneous polynomial of degree $n+1$. Let $X_G$ be the smooth Calabi--Yau hypersurface of dimension $n-1$ in $\mathbb{CP}^n$ with homogeneous coordinates $\underline{s}$ defined by $G(\underline{s})$. 
Pick a middle dimensional homology cycle $\g$ in $X_G$ and define a linear map $\iota: A \rightarrow \C$ as
\[
\iota(O) =\int_{T(\g)}\left(\int_{0}^\infty O\cdot e^{-S} dp\right) \Omega\left/
\int_{T(\g)}\left(\int_{0}^\infty e^{-S} dp\right) \Omega,\right.
\]
where $T(\g)$ is a tubular neighborhood of $\g$ in $\mathbb{CP}^n-X_G$ and $ \Omega =\sum_{i=0}^n (-1)^i s_i ds_0\wedge\cdots\wedge\widehat{ds_i}\wedge\cdots ds_n$.
Note that 
\[
\int_{T(\g)}\left(\int_{0}^\infty e^{-S} dp\right) \Omega = \int_{T(\g)}\Fr{\Omega}{G(\underline{s})} =\C^*\cdot \int_\g \omega,
\]
where $\omega$ is the holomorphic $(n-1)$-form, unique up to multiplication by non-zero complex numbers, on the Calabi--Yau hypersurface $X_G$. For the second equality in the above we have used the residue map of Griffiths in the classical paper \cite{Gr}. It follows that $\iota(1_A)=1$ and $\xymatrix{A_{\classical}\ar[r]^\iota &\C}$ is a classical algebraic probability space.

Now we consider an infinitesimal symmetry of the expectation morphism $\iota$ and associated coinvariant module.
Let $\{e_\m\}$, $\m=-1,0,1,\dotsc, n$, be the standard basis on $\C^{n+2}$, $z_{-1}=p$ and $z_i = s_i$ for $i=0,1,\dotsc,n$.
Then the following lemma is proven in \cite{PP}.
\begin{lemma}
Let $\mg$ be the Abelian Lie algebra $\C^{n+2}$ let $\vr: \mg \rightarrow L\!\Diff_\C(A)$ be the linear representation given by
\[
\vr(e_\m)\ceq \r_\m = -\Fr{\rd}{\rd z^\m} +\Fr{\rd S}{\rd z^\m}.
\]
Then $\iota\circ \r_\m=0$ for all $\m$ and the coinvariant module $A_\mg$ is isomorphic to the middle dimensional primitive cohomology $PH^{n-1}(X_G)$ of $X_G$ as a complex vector space.
\end{lemma}

Let $\{e_\a\}$ be a basis of $PH^{n-1}(X_G)$ and $\{\tilde e_\a\}$ be the corresponding basis of $A_\mg$. 
Let $\{\omega_\a\}$ be a set of representatives of $\{e_\a\}$ in the de Rham complex of $X_G$. Let $\{\varpi_\a\}$ be a set of representatives of $\{\tilde e_\a\}$ in $A$, i.e., $\mathfrak{q}(\varpi_\a)=\tilde e_\a$.
Then $\iota(\varpi_\a) = \int_\g \omega_\a$. It follows that the laws of any random variable can be determined in terms of
(up to the finite ambiguities given by)
the period integrals $\left\{ \int_\g \omega_\a\right\}$. For example consider a random variable $x \in A$ given by
\[
x = - p\cdot F(\underline{s})
\]
where $F(\underline{s})$ is a generic homogeneous polynomial of degree $n+1$. Then the moment generating function $Z_x(t)$ of the random variable $x$ is given by
\[
Z_x(t) =\iota\left(e^{t x}\right) \equiv \int_{T(\g)}\Fr{\Omega}{G(\underline{s}) + t F(\underline{s})}\left/\int_{T(\g)}\Fr{\Omega}{G(\underline{s})}\right.
.
\]
Hence $Z_x(t)$ governs the variation of the period integrals of the holomorphic $(n-1)$-forms of the $1$-parameter family of Calabi--Yau hypersurfaces $X_{G+ t F}$ in $\C\mathbb{P}^n$ defined by $G(\underline{s}) + t F(\underline{s})$. Then the finite dimensionality of $A_\mg\simeq PH^{n-1}(X_G)$ implies that $Z_x(t)$ satisfies an ordinary differential equation, known as the Picard--Fuchs equation, which can be used to determine the moment generating function of $x$ completely.
We refer to \cite{PP} for further discussion of the details and implications of this example.
\hfill$\natural$
\end{example}

\begin{example}
\label{matrix-ex}
In this example, we begin with an infinitesimal symmetry rather than with an explicitly defined expectation morphism. Fix a natural number $N$ and consider the complex valued polynomial algebra $A=\C[X_{i\!j}]_{i,j\in N}$ of $N^2$ indeterminates $X_{i\!j}$. Let $\mg\ceq \C^{N^2}$ be an Abelian Lie algebra with basis $\{e_{ij}\}$. Consider the representation $\vr: \mg=\C^{N^2}\rightarrow L\hbox{Diff}_\C( A)$ defined by
\[
\vr(e_{i\!j}) = -\Fr{\rd}{\rd X_{i\!j}} +\Fr{\rd S}{\rd X_{i\!j}}
,\qquad
S= -\Fr{\sqrt{-1}}{3!} \tr\left(X^3\right),
\]
where $X$ denotes a $N\times N$ matrix with $i\!j$ components $X_{ij}$. Assume that there is a $\C$-linear map $\iota: A \rightarrow \C$ such that $\iota\circ \vr(e_{i\!j}) =0$ for all $i\!j$. Consider a family of random variables $X_{11},\dotsc, X_{NN}$ and the associated generating function
\[
Y(T)=\iota\left(
e^{\fr{\sqrt{-1}}{2} \tr\left(T^2 X\right)}
\right),\qquad 
T=\left(\begin{array}{ccc}t_1&&0\\&\ddots\\0&& t_N\end{array}\right).\] 
Define a family of $\C$-linear maps $\iota_{\underline{t}}: A \rightarrow \C[\![t_1,\dotsc, t_N]\!]$ by, for $O\in A$,
\[
\iota_{\underline{t}}(O)= \iota\left(e^{\fr{\sqrt{-1}}{2} \tr\left(T^2 X\right)}\cdot O\right).
\] 
We also have an $N$-parameter family of representations $\vr_{\underline{t}}: \mg\rightarrow L\hbox{Diff}_\C( A)$ of the Abelian Lie algebra $\mg$ defined by
\[
\vr_{\underline{t}}(e_{i\!j}) = -\Fr{\rd}{\rd X_{i\!j}} +\Fr{\rd S_{\underline{t}}}{\rd X_{i\!j}}
,\qquad
S_{\underline{t}}=-\Fr{\sqrt{-1}}{2} \tr\left(\Fr{1}{3}X^3 + T^2 X\right),
\]
which satisfies $\iota_{\underline{t}}\circ \vr_{\underline{t}}(e_{i\!j}) =0$ for all $i\!j$, since
\begin{align*}
\vr_{\underline{t}}(e_{i\!j}) &= e^{-\fr{\sqrt{-1}}{2} \tr\left(T^2 X\right)}\vr(e_{i\!j})e^{\fr{\sqrt{-1}}{2} \tr\left(T^2 X\right)}
,\\
\iota_{\underline{t}}\left( \vr_{\underline{t}}(e_{i\!j})( O)\right)
&= \iota\left(\vr(e_{ij})\left(e^{\fr{\sqrt{-1}}{2} \tr\left(T^2 X\right)}\cdot O\right)\right).
\end{align*}

We claim without proof that the coinvariant module $A_\mg$, considered as a $\C$-vector space for generic $\underline{t}$, is spanned by the class of $1$, the classes of $X_{i\!i}$ for $i=1,\dotsc, N$, and the classes $X_{i\!i}\cdot X_{j\!j}$ for $i,j=1,\dotsc, N$ and $i\neq j$.
We also claim that 
\eqn{exafaf}{
\mq\left(X_{kk}\cdot X_{kk}\right)= 
-t_k^2\mq\left(1\right) + 2\sqrt{-1} \sum_{\ell\neq k}\Fr{1}{t_k^2 -t_\ell^2}\left(\mq(X_{kk}) -\mq(X_{\ell\ell})\right).
}
We note that
\[
 \sqrt{-1}\Fr{1}{t_k}\Fr{\rd }{\rd  t_k} Y(T)
 =-\iota\left(
e^{\fr{\sqrt{-1}}{2} \tr\left(T^2 X\right)}\cdot X_{kk}\right)
=-\iota_{\underline{t}}\left(X_{kk}\right).
\]
It follows by applying $\iota_\mg$ to \eq{exafaf} that $Y(T)$ satisfies the following system of differential equations for all $k=1,\dotsc, N$:
\[
\left(t_k^2 -\left(\Fr{1}{t_k}\Fr{\rd}{\rd t_k}\right)^2 
- 2 \sum_{\ell\neq k}\Fr{1}{t_k^2 -t_\ell^2}\left(\Fr{1}{t_k}\Fr{\rd}{\rd t_k} 
-\Fr{1}{t_\ell}\Fr{\rd}{\rd t_\ell}\right)\right)Y(T)=0.
\]
This is known as the matrix Airy differential equation.

In fact, there is an underlying measure theory for the family $\iota_{\underline{t}}$ of expectation morphisms.
Let $X$ be an $N\times N$ Hermitian matrix, let $T$ as above, and let $dX$ denote the Haar measure---$dX =\prod_{i=1}^N d X_{ii} \prod_{i < j}d(\hbox{Re} (X_{ij}))d(\hbox{Im}(X_{ij}))$.
Then
\[
\iota_{\underline{t}}(O)\ceq \int O\cdot e^{-S_{\underline{t}}}dX,
\]
so that $Y(T) = \iota_{\underline{t}}(1) =\int e^{\fr{\sqrt{-1}}{2} \tr\left(\fr{1}{3}X^3 +T^2 X\right)}dX$,
which is known to satisfy the matrix Airy differential equation~\cite{KM}.
\hfill$\natural$
\end{example}

In all of our examples so far the expectation maps have been translation invariant measures twisted by exponential functions.
Our last example is of a different kind.

\begin{example}\label{semicircle}
Set $A= \R[s]$ with $1_{A}=1$, and define an expectation morphism $\iota: \R[s]\rightarrow \R$ by, for all $O\in A$,
\[
\iota\big(O\big)=\int^{+2}_{-2} O(s)\cdot\sqrt{4-s^2} \,\text{d}s\left/
\int^{+2}_{-2} \sqrt{4-s^2}\, \text{d}s\right.
\]
This formula corresponds to the semi-circular distribution with zero expectation. Consider $\mg=\R\cdot e$ as a Lie algebra generated by one element $e$ and the representation $\vr:\mg \rightarrow L\!\Diff_{\R}(A)$ given by
\[
\vr(e)\ceq\r= (4-s^2)\Fr{d}{ds} -3 L_s: \R[s]\rightarrow \R[s].
\]
It can be checked that $\iota\circ \r =0$: 
for all $O\in A$
\[
\iota\circ\r(O)= \iota\left((4-s^2)\Fr{d O}{ds} -3 s\cdot O\right)
\propto \int^{s=2}_{s=-2}\text{d}\left( O\cdot (4 -s^2)^{\fr{3}{2}} \right)=0.
\]
Hence we have a $\R$-linear representation $\vr:\R \rightarrow L\!\Diff_{\R}(A)$ which satisfies $\iota\circ\Im \vr=0$.
For any $n\geq 1$, we have
\[
\r(s^{n-1})
= 4(n-1) s^{n-2} - (n+2) s^{n},
\]
so that $\iota(s)=0$ and, for all $n\geq 2$,
\[
\iota(s^n) =4 \Fr{(n-1)}{(n+2)}\iota(s^{n-2}).
\]
It follows that
\[
\iota(s^n) =
\begin{cases}
0 &\text{$n$ odd},\\
2^{n+1}\cdot \Fr{(n-1)!!}{(n+2)!!} &\text{$n$ even}.
\end{cases}
\]
Note that $2^{2k+1}\Fr{(2k-1)!!}{(2k+2)!!}=\Fr{(2k)!}{(k+1)! k!}=C_{k}$ where $C_k$ is the Catalan number.
It follows that the moment generating function is given by
\[
Z(t) = \sum_{n\geq 0}\Fr{t^{2n}}{(2n)!}C_n= \sum_{n\geq 0}\Fr{t^{2n}}{(n+1)!n!}
\]
Consider the family of representations
\[
\r_t = e^{-ts}\r e^{t s}= 
 (4-s^2)\Fr{d}{ds} -3 s + t(4-s^2)
\]
We can proceed as in the previous examples, defining $\iota_t$ and showing that $\iota_t\r_t=0$. Then we calculate that $\r_t(1)= -3 s + t(4-s^2) = -t s^2 -3s + 4t$. Writing $\iota_t(s)$ in terms of $\iota_t(1)=Z(t)$ we obtain the following differential equation:
\[
\left( t\Fr{d^2}{dt^2} + 3\Fr{d}{dt} - 4t\right)Z(t)=0.
\]
which is equivalent to the recursive relation  
$(n+2)C_{n+1}= 2(2n+1)C_n$ of the Catalan numbers.
With the initial condition $Z(0)=1$, the differential equation has the unique solution in $\R[\![t]\!]$ given above.
\hfill$\natural$.
\end{example}

\subsection{The homotopical realization of a classical algebraic probability space}\label{subs:realization}

Consider a classical algebraic probability space $\xymatrix{A_{\classical} \ar[r]^\iota &\fieldk}$ with a faithful infinitesimal symmetry $\vr:\mg\to L\hbox{Diff}_{\fieldk}(A)$ such that $\iota\circ \Im\vr=0$.
We recall that coinvariants is a right-exact functor from $\mg$-modules to vector spaces and the corresponding left derived functor is the {\em Lie algebra homology} $H_\bullet(\mg;A)$. The Lie algebra homology begins $H_0(\mg;A) = A_{\mg}$, and can be computed in general via the Chevalley--Eilenberg--Koszul chain complex $\big(A\otimes \wedge^\bullet \mg, K\big)$, where
\begin{align*}
K(a\otimes g_1\wedge\cdots\wedge g_n)
=
& \sum_{i=1}^n (-1)^{i-1}\vr(g_i)(a)\otimes g_1\wedge\cdots \wedge\widehat{g_i}\wedge\cdots\wedge g_n
\\
&+\sum_{i < j}(-1)^{i+j}a\otimes [g_i, g_j]\wedge g_1\wedge \cdots\wedge \widehat{g_i}\wedge \cdots \widehat{g_j}\wedge\cdots\wedge g_n.
\end{align*} 

By reversing the sign of the degree, we obtain a cochain complex $(\sA, K)$ with degree concentrated in non-positive integers:
\[
\xymatrix{\cdots \ar[r]^K & \sA^{-2}\ar[r]^K & \sA^{-1} \ar[r]^K & \sA^{0}\ar[r]^K&0,
}
\]
where $\sA^{-k} = A\otimes S^k(\mg[-1])$.
The notation $\mg[-1]$ refers to the graded vector space which is the the Lie algebra $\mg$ after shifting degree by $-1$ and $S^\ell(\mg[-1])$ is the $\ell$-fold graded (super)-symmetric product, which is isomorphic to the exterior algebra $\wedge^\ell \mg$.
The $-\ell$-th cohomology $H^{-\ell}_K$, $\ell=0,1,2,\cdots$ of the cochain complex $(\sA, K)$ is isomorphic to the $\ell$-th Lie algebra homology $H_{\ell}(\mg, A)$.
Hence the \nth{0} cohomology $H^0$ is isomorphic to the coinvariants $A_\mg$. We note that the graded vector space 
\[
\sA= A\otimes S(\mg[-1])
\] 
has a super-commutative and associative binary product, denoted by $\cdot$, with unit $1_\sA =1_A$, which is annihilated by $K$, i.e., $K 1_\sA=0$. 
We hence obtain the quartet $\sA_{\binarycomm}\ceq (\sA, 1_\sA, \cdot, K)$ such that the trio $(\sA, 1_\sA, \cdot)$ is a unital graded commutative algebra and the another trio $(\sA, 1_\sA,K)$ is a pointed cochain complex.
\begin{example}

Let $\{e_j\}_{j\in J}$ be a basis of the Lie algebra $\mg$ with structure constants determined by $[e_i, e_j]_\mg=\sum_k f_{ij}{}^k e_k$ and let $\r_i = \vr(e_i) \in \Diff_\fieldk(A)$. Introduce the corresponding basis $\{\eta_j\}$ of $\mg[-1]$ where $\eta_i$ has degree $-1$, $\eta_i \eta_j =-\eta_j \eta_i$, and $a\cdot \eta_j=\eta_j \cdot a$. 
Then $S(\mg[-1])\simeq \fieldk[\eta_j]_{j\in J}$,
$\sA=A\otimes S(\mg[-1])\simeq A[\eta_j]_{j\in J}$ and
\[
K = \sum_{j\in J} \r_j\otimes \Fr{\rd}{\rd\eta_i} 
+\Fr{1}{2} \sum_{i,j,k\in J}I_A\otimes f_{ij}^k\eta_k\Fr{\rd^2}{\rd\eta_i \rd\eta_j}.
\]
\hfill$\natural$
\end{example}

The original expectation morphism $\iota: A \rightarrow \fieldk$ can be uniquely extended to a pointed cochain map $\mc:(\sA, 1_\sA, K)\rightarrow (\fieldk, 1, 0)$, where the ground field is considered as a pointed cochain complex concentrated in degree zero. The extension is $\iota$ on $\sA^0=A$ and zero in negative degrees. Then $\mc(1_\sA)=\iota(1_A)=1$ and the condition $\iota \circ Im\vr= 0$ implies that $\mc K=0$.

The resulting algebraic structure on $\sA$, denoted by $\sA_{\binarycomm}=(\sA, 1_\sA, \cdot, K)$ is a prototype for our notion of \emph{binary \padj{} probability algebras}, which will be defined shortly. The map $\mc$ will be a morphism of binary \padj{} probability algebras to the initial object $\fieldk$ in the category $\category{\HProb}_{\binarycomm}{\overk}$ of binary \padj{} probability algebras. Such a morphism will be called a \emph{binary \padj{} probability space}. This morphism
\[
\xymatrix{\sA_{\binarycomm}\ar[r]^\mc &\fieldk},
\]
will be called a \emph{homotopical realization} of the classical algebraic probability space $\xymatrix{A_{\classical} \ar[r]^\iota &\fieldk}$ with infinitesimal symmetry $\vr:\mg\rightarrow \hbox{Diff}_{\fieldk}(A)$.

\begin{example}\label{hgaussian-ex2}
Consider the classical algebraic probability space from Example \ref{gaussian-ex1}.
Then $\sA = \R[s,\eta]=\sA^{-1}\oplus \sA^0=\R[s]\cdot\eta \oplus R[s]$, where $\eta$ has degree $-1$ with $\eta^2=0$ and $s\cdot \eta=\eta\cdot s$, and
\[
K = \r\Fr{\rd}{\rd\eta}= -\s^2\Fr{\rd^2}{\rd s \rd\eta} + s\Fr{\rd}{\rd\eta}.
\]
The map $c:\sA\to\fieldk$ is zero on $\sA^{-1}$ and $\iota$ on $\sA^{0}=A$.
It is obvious that $K^2=K1=0$ and $c\circ K=0$ since, for any $\La = \l(s)\eta \in \sA^{-1}$, we have $K\La=\r(\l(s))$.
From 
\begin{align*}
K\eta =&s
,\\
K\left(s^n\eta\right)=& -n\s^2 s^{n-1} +s^{n+1}\qquad n\geq 1
,
\end{align*} 
it follows that $\left[s^{2k}\right]_K= (2k-1)!! \left[ 1_{\sA}\right]_K$ and $\left[s^{2k-1}\right]_K=0$.
Hence
\[
c(s^{2k})\equiv\iota(s^{2k})= (2k-1)!!,\quad c(s^{2k-1})\equiv\iota(s^{2k-1})= 0.
\]
It can be shown that $H=H^0= \mathbb{R}\cdot [1_\sA]$.
 \hfill$\natural$
\end{example}

The homotopical realizations of examples \ref{period-ex}, \ref{matrix-ex} and \ref{semicircle} are also straightforward. Note that if there is an underlying translation invariant measure twisted by an exponential function such that the expectation morphism is integration against this measure, as in examples \ref{gaussian-ex1}, \ref{period-ex} and \ref{matrix-ex}, we have an Abelian infinitesimal symmetry of the expectation and obtain a Batalin--Vilkovisky algebra \cite{BV1,BV2} as a homotopical realization. 

\subsection{Binary \padj{} probability spaces}\label{subs:binary prob space}

\subsubsection{The homotopy category of binary \padj{} probability algebras}

To set the stage, we define a binary \padj{} probability space as a cochain enhancement of a classical algebraic probability space.

\begin{definition}
\label{bhpa}
A \emph{binary \padj{} probability algebra} is a tuple $\sC_{\binarycomm} = \big(\sC, 1_\sC, \cdot, K\big)$, where 
\begin{itemize}
\item the trio $\big(\sC, 1_\sC, \cdot\big)$ is a graded unital commutative associative algebra and 
\item the trio $(\sC, 1_\sC, K)$ is a pointed cochain complex.
\end{itemize}
A morphism between two binary \padj{} probability algebra $\sC_{\binarycomm}$ and $\sC^\pr_{\binarycomm}$ is a pointed cochain map $f: \sC\rightarrow \sC^\pr$. 
Two morphisms $f$ and $\tilde f$ are homotopic if they are pointed cochain homotopic.
\end{definition}
\begin{corollary}
A binary \padj{} probability algebra concentrated in degree zero is a unital commutative associative algebra, and vice versa.
\end{corollary}
We denote by $\category{\HProb}_{\binarycomm}{\overk}$ the category of binary \padj{} probability algebras.
Morphisms in the category $\category{\HProb}_{\binarycomm}{\overk}$ come with the notion of homotopy defined above, which induces an equivalence relation on morphisms.
We define the homotopy category $\homotopycat\category{\HProb}_{\binarycomm}{\overk}$ of binary \padj{} probability algebra to have binary \padj{} probability algebras as objects and homotopy types of $\category{\HProb}_{\binarycomm}{\overk}$-morphisms. A morphism of binary \padj{} probability algebra is called a quasi-isomorphism if it induces an isomorphism on the cohomology of the underlying cochain complex.

We can consider the classical algebraic probability algebra $\fieldk$ as a binary \padj{} probability algebra by concentrating it in degree zero. We again omit decoration and call this object simply $\fieldk$. It is clear that $\fieldk$ is the initial object in both categories $\category{\HProb}_{\binarycomm}{\overk}$ and $\homotopycat\category{\HProb}_{\binarycomm}{\overk}$. 

\begin{definition}
A \emph{binary \padj{} probability space} is a diagram $\xymatrix{\sC_{\binarycomm} \ar[r]^{[c]} &\fieldk}$ in the homotopy category $\homotopycat\category{\HProb}_{\binarycomm}{\overk}$.
\end{definition}
In other words, a binary \padj{} probability space is an object $\sC_{\binarycomm}$ and a morphism $c$, defined up to homotopy, to the initial object $\fieldk$ in $\category{\HProb}_{\binarycomm}{\overk}$, that is, a linear map from $\sC$ to $\fieldk$ satisfying $c(1_\sC)=1$ and $c \circ K =0$. We call an element $x$ in $\sC$ a {\em cohomological random variable} if $K x =0$ and the value $c(x)$ is called the \emph{expectation} of $x$. It follows that $c(x)$ depends only on the homotopy type of $c$ and the cohomology class of $x$. 

\begin{corollary}
A classical algebraic probability space $\xymatrix{A_{\classical} \ar[r]^\iota &\fieldk}$ is a binary \padj{} probability space concentrated in degree zero, and vice versa.
\end{corollary} 

\begin{corollary}
The homotopical realization $\xymatrix{\sA_{\binarycomm} \ar[r]^\mc &\fieldk}$ of a classical algebraic probability space $\xymatrix{A_{\classical} \ar[r]^\iota &\fieldk}$ equipped with an infinitesimal symmetry is a binary \padj{} probability space concentrated in non-positive degree.
\end{corollary} 

\subsubsection{Homotopy functors to the category of unital \texorpdfstring{$sL_\infty$}{sL-infinity}-algebras}

In {\bf Definition~\ref{bhpa}}, neither the differential $K$ nor a morphism $f$ are required to satisfy any specific relation with respect to multiplication. In particular, the differential $K$ may fail to be a derivation of the multiplication and a morphism $f$ may fail to be an algebra map. These properties are crucial to capturing {\em non-trivial correlations} among random variables.
We organize these and successive failures by adopting the notion of classical independence from classical algebraic probability theory, which leads us to construct a family of functors $\Des_{\La}$ from the category $\category{\HProb}_{\binarycomm}{\overk}$ of binary \padj{} probability algebras to the category $\category{sL}_\infty{\overk}$ of unital $sL_\infty$-algebras.

\begin{definition}\label{bdesc}
(Descendant functor)
\begin{itemize}
\item
Let $\sC_{\binarycomm}=\left(\sC, 1_\sC, \cdot, {K}\right)$ be a binary \padj{} probability algebra. The {\em descendant} of $\sC_{\binarycomm}$ is the triple $\sC_{\Lie} = \left(\sC, 1_\sC, \underline{\ell}^{K}\right)$, where the family $\underline{\ell}^{K}= {\ell}^K_1, \ell^K_2,\cdots$ is recursively defined, for $n \geq 1$ and elements $x_1,\dotsc, x_n \in \sC$, by the equation 
\[
K (x_1\cdot \ldots \cdot x_n)=\sum_{\substack{\pi \in P(n)\\ |B_i|=n-|\pi|+1}}\!\!\!\! \epsilon(\pi,i)
 x_{B_1}\cdot\ldots\cdot x_{B_{i-1}}\cdot \ell^K(x_{B_i})\cdot x_{B_{i+1}}\cdot \ldots\cdot x_{B_{|\pi|}}.
\]
\item
Let $\xymatrix{\sC_{\binarycomm} \ar[r]^f & \sC^\pr_{\binarycomm}}$ be a morphism in $\category{\HProb}_{\binarycomm}{\overk}$. Let $\underline{\La}=\La_1,\La_2,\cdots$ be a sequence of linear maps $\La_n: S^n(\sC)\rightarrow \sC^\pr$ of degree $-1$ which satisfy the relations $\La_1(1_\sC)=0$ and $\La_{n+1}\big(x_1,\dotsc,x_{n}, 1_\sC\big)=\La_{n}\big(x_1,\dotsc,x_{n}\big)$ for $n\geq 1$. The {\em descendant of $f$ up to the homotopy $\underline{\La}$} is defined to be the family $\underline{\phi}^{f,\underline{\La}}= \phi_1^{f,\underline{\La}}, \phi_2^{f,\underline{\La}},\cdots$ determined for $n \geq 1$ by the following recursive relation valid for elements $x_1,\dotsc, x_n \in \sC$:
\begin{align*}
f(x_1\cdot \ldots\cdot x_n)=
&\sum_{\pi \in P(n)}\ep(\pi)\phi^{f,\underline{\La}}\!\left(x_{B_1}\right)\cdot^\pr\ldots \cdot^\pr 
\phi^{f,\underline{\La}}\!\left(x_{B_{|\pi|}}\right)
\\
&
+K^\pr \La_n(x_1,\dotsc, x_n)
\\
&
+\!\!\!\!\!\sum_{\substack{\pi \in P(n)\\ |B_i| = n-|\pi|+1}}\!\!\!\!\!\!\!\!
\ep(\pi,i)\La_{|\pi|} \left( x_{B_1},\dotsc, x_{B_{i-1}},\ell^K\left(x_{B_i}\right), x_{B_{i+1}},\dotsc, x_{B_{|\pi|}}\right)
.
\end{align*}
\end{itemize}

\end{definition}

Encoded in the above combinatorial and recursive definitions of the families $\underline{\ell}^K$ and $\underline{\phi}^{f,\underline{\La}}$ are the failures and higher failures of $K$ being a derivation of the multiplication and $f$ being an algebra homomorphism, respectively. 
For example:
\begin{align*}
\ell_1^K=
& K
,\\
\ell_2^K(x_1,x_2) = 
&K(x_1\cdot x_2)- Kx_1\cdot x_2 -(-1)^{|x_1|}x_1\cdot K x_2
,\\
\ell_3^K(x_1,x_2,x_3) =
&K(x_1\cdot x_2\cdot x_3)
\\
&
- Kx_1\cdot x_2 \cdot x_3
-(-1)^{|x_1|}x_1\cdot K x_2\cdot x_3
-(-1)^{|x_1|+|x_2|}x_1\cdot x_2\cdot Kx_3
\\
&
-\ell_2^K(x_1,x_2)\cdot x_3
-(-1)^{|x_1|}x_1\cdot\ell_2^K(x_2,x_3)
-(-1)^{(|x_1|+1)|x_2|}x_2\cdot\ell_2^K(x_1,x_3).
\end{align*}
Note that $\ell_3^K$ can be expanded entirely in terms of $K$; in this format it looks like the seven-term Batalin--Vilkovisky relation between $K$ and the product.
The first terms in the descendant of $f$ are:
\begin{align*}
\phi^{f,\underline{\La}}_1 
= & f - K^\pr \La_1 - \La_1 K
,\\
\phi^{f,\underline{\La}}_2(x_1,x_2) 
= 
& \phi^{f,\underline{\La}}_1(x_1\cdot x_2) - \phi^{f,\underline{\La}}_1(x_1)\cdot\phi^{f,\underline{\La}}_1(x_2)
\\
&+K'\Lambda_1(x_1\cdot x_2) + \Lambda_1 K(x_1\cdot x_2)
\\
&
-K^\pr\La_2(x_1, x_2) - \La_2(Kx_1, x_2) -(-1)^{|x_1|}\La_2(x_1, Kx_2)
,\\
\phi^{f,\underline{\La}}_3(x_1,x_2,x_3) 
= &
\phi^{f,\underline{\La}}_1(x_1\cdot x_2\cdot x_3) 
- \phi^{f,\underline{\La}}_1(x_1)\cdot \phi^{f,\underline{\La}}_2(x_2, x_3)
\\
&
- \phi^{f,\underline{\La}}_2(x_1, x_2)\cdot \phi^{f,\underline{\La}}_1(x_3)
 -(-1)^{|x_1||x_2|} \phi^{f,\underline{\La}}_1(x_2)\cdot \phi^{f,\underline{\La}}_2(x_1, x_3)
 \\
&
 -K^\pr\La_3(x_1, x_2,x_3) 
- \La_2(Kx_1, x_2,x_3)
-(-1)^{|x_1|}\La_2(x_1, Kx_2,x_3)
\\
&
-(-1)^{|x_1|+|x_2|}\La_2(x_1, x_2,Kx_3)
-\La_2(\ell_2(x_1, x_2),x_3)
\\
&
-(-1)^{|x_1|}\La_2(x_1, \ell_2(x_2,x_3))
-(-1)^{(|x_1|+1)(|x_2|}\La_2(x_2, \ell_2(x_1,x_3))
\\&+K'\Lambda_1(x_1\cdot x_2\cdot x_3) + \Lambda_1 K(x_1\cdot x_2\cdot x_3).
\end{align*}

Let $\Des_{\La}\big(\sC_{\binarycomm} \big)= \sC_{\Lie}$ and $\Des_{\La}\big(f) = \underline{\phi}^{f,\underline{\La}}$.
It can be checked that $\sC_{\Lie}$ is a unital $sL_\infty$-algebra\footnote{This $sL_\infty$-structure has been constructed before in \cite{BDA,Ak,Krav}.}. It can be also checked that $\underline{\phi}^{f,\underline{\La}}$ is an $sL_\infty$-morphism\footnote{See~\cite{DPT2,RS} for other equivalent characterizations of the descendent morphism with $\underline{\La}=\underline{0}$.} from $\sC_{\Lie}$ to $\sC^\pr_{\Lie}$ for any choice of $\underline{\La}$ and that $\underline{\phi}^{\tilde f,\underline{\tilde\La}}$ is homotopic to $\underline{\phi}^{f,\underline{\La}}$ if $\tilde f$ is cochain homotopic to $f$.\footnote{See \cite{PP} for a proof when $\underline{\La}=\underline{0}$.} 
\begin{theorem}\label{xmaina}
The assignment of descendants to objects and morphisms in the category $\category{\HProb}_{\binarycomm}{\overk}$ defines a family of functors $\Des_{\La}: \category{\HProb}_{\binarycomm}{\overk}\rightarrow \category{sL}_\infty{\overk}$, which induce a well defined functor $\text{ho}\!\Des :\homotopycat\category{\HProb}_{\binarycomm}{\overk}\rightarrow \homotopycat\category{sL}_\infty{\overk}$.
\end{theorem}
We will not prove this theorem here; it is a special case of {\bf Theorem \ref{maina}}.
\begin{example}
Consider a classical algebraic probability space $\xymatrix{A_{\classical} \ar[r]^\iota &\fieldk}$,
where $A_{\classical}=(A,1_A,\cdot)$. Then both $A_{cl}$ and $\fieldk$ are binary \padj{} probability algebras with zero differential. It follows that the descendants of $A_{cl}$ and $\fieldk$ are unital $sL_\infty$-algebras with zero $sL_\infty$-structure. It also follows that the descendant of $\iota$ is independent of $\Lambda$; explicitly,
\[
\iota(x_1\cdot x_2\cdots x_n) 
=\sum_{\pi \in P(n)}
 \phi^\iota\big(x_{B_1}\big) \phi^\iota\big(x_{B_2}\big)\cdots\phi^{\iota}\big( x_{B_{|\pi|}}\big).
\]
Hence the descendant morphism $\underline{\phi}^\iota$ is exactly the cumulant morphism $\underline{\k}$.
 \hfill$\natural$
\end{example}

\begin{example} 
Consider an classical algebraic probability space $\xymatrix{A_{\classical} \ar[r]^\iota &\fieldk}$ with a faithful infinitesimal symmetry, that is, a representation $\vr:\mg \rightarrow L\hbox{Diff}_{\fieldk}(A)$ which satisfies $\iota\circ \Im\vr=0$. Assume further for this example that $\vr(\mg)$ annihilates the unit $1_A$.
Let $\{e_j\}_{j\in J}$ be a basis of the Lie algebra $\mg$ with structure constants $[e_i, e_j]_\mg=\sum_k f_{ij}{}^k e_k$ and let $\r_i$ denote $\vr(e_i) \in \End_{\fieldk}(A)$. Introduce the corresponding basis $\{\eta_j\}$ of $\mg[-1]$.
Let $\sA = A\otimes S(\mg[-1])\simeq A[\eta_j]_{j\in J}$. Equip $\sA$ with the unit $1_\sA = 1_A$, the product $(a\otimes \a)\cdot (a^\pr\otimes \a^\pr)\ceq (aa^\pr)\otimes \a\a^\pr$ and the differential
\[
K \ceq \sum_{j\in J} \r_j\otimes \Fr{\rd}{\rd\eta_j} 
+\Fr{1}{2} \sum_{i,j,k\in J}I_A\otimes f_{ij}^k\eta_k\Fr{\rd^2}{\rd\eta_j \rd\eta_i}.
\]
Then $\sA_{\binarycomm}\ceq \big(\sA, 1_\sA, \cdot, K\big)$ is a binary \padj{} probability algebra concentrated in nonpositive degrees.

We note that there is the natural structure of a current Lie algebra on $\sA^{-1}=A\otimes \mg[-1]$ with degree $1$ bracket $[a\otimes \a, a^\pr\otimes \a^\pr]\ceq (a\cdot a^\pr)\otimes [\a, \a^\pr]_\mg$.
On the other hand, let $\l_1 = \l_1^j \otimes\eta_j$ and likewise for $\l_2$. Then the descendant bracket $\ell_2^K$ on $\sA^{-1}$ is given by
\begin{align*}
\ell^K_2(\l_1,\l_2)
= &\Fr{1}{2}\left(\l_1^i \l_2^j -\l_2^i\l_1^j\right)\otimes f_{ij}^k \eta_k 
\\
&
+\left(\l^i_1\cdot\r_i(\l_2^j) -\l_2^i\cdot \r_i(\l_1^j) 
+\ell_2^{\r_i}(\l_1^i, \l_2^j) -\ell_2^{\r_i}(\l_1^j, \l_2^i)\right)\otimes\eta_j,
\\
\end{align*}
where the first term in the right hand side above is the current bracket $[\l_1,\l_2]$.
Assume that $\r_i$ is at most a \nth{1} order linear differential operator for each $i$, so that $\ell^{\r_i}_2=0$, for all $i$. Then $K$ is at most a second order differential, meaning that $\ell^K_{3}=\ell^K_{4}=\cdots=0$. 
It follows that the descendant algebra is a unital sDGLA $\left(\sA, 1_A,K, \ell_2^K\right)$; that is, $\ell_2^K$ is a derivation of the product. Note that
\begin{align*}
\ell^K_2(\l_1,\l_2)
=&
\Fr{1}{2}\left(\l_1^i \l_2^j -\l_2^i\l_1^j\right) \otimes f_{ij}^k \eta_k 
+\left(\l^i_1\cdot\r_i(\l_2^j) -\l_2^i\cdot \r_i(\l_1^j) 
\right)\eta_j.
\end{align*}
Then $\sA^{-1}$ is a Lie algebra with bracket $\ell^K_2$ which acts on $\sA^0=A$ as a derivation; 
\begin{align*}
&\sA^{-1}\times A \rightarrow A
&
[\d_{\!\l_1},\d_{\!\l_2}]=& \d_{\!\ell^K_2(\l_1,\l_2)},
\\
&(\l, x) \mapsto \d_{\!\l}(x)\ceq \ell_2^K(\l,x)
&
\d_{\!\l}(x\cdot y)= & \d_{\!\l} x\cdot y +x\cdot \d_{\!\l} y.
\end{align*}

Recall that $\sA^0$ is an commutative algebra. $\sA^0$ acts on $\sA^{-1}$;
\[
A\times \sA^{-1} \rightarrow \sA^{-1},
\qquad
(x,\l)\mapsto x\cdot \l.
\]
Note also that $\ell_2^K(\l_1, x\cdot \l_2) =x\cdot \ell_2^K(\l_1, \l_2) 
+\d_{\!\l_1}x\cdot \l_2$, which implies that the pair $\left(\sA^{-1}, \sA^0\right)$ is a Lie algebroid.
 \hfill$\natural$
\end{example}

 Recall that an $L_\infty$-morphism $\underline{\phi}=\phi_1,\phi_2,\cdots$ is a quasi-isomorphism if $\phi_1$ is a cochain quasi-isomorphism. The following  theorem, which is a special case of
 {\bf Theorem} \ref{lemc}, shall play an important role in what follows.
\begin{theorem}
Any descendant algebra is quasi-isomorphic to a unital Abelian graded sLie algebra as a unital $sL_\infty$-algebra.
\end{theorem}
In particular the cohomology $H$ of the cochain complex $(\sC,K)$ in $\sC_{\binarycomm}$ is a unital Abelian graded Lie algebra, or equivalently a unital $L_\infty$-algebra $H_{\Lie}=(H, 1_H,\underline{0})$ with the zero $sL_\infty$-structure $\underline{0}$ and the unit $1_H=[1_\sC]_K$. 
Furthermore,  $H_{\Lie}$
is $sL_\infty$-quasi-isomorphic to $\sC_{\Lie}$.

\subsubsection{Homotopical random variables and their laws}
 
Fix a binary \padj{} probability space $\xymatrix{\sC_{\binarycomm} \ar[r]^\mc &\fieldk}$ and let $\sC_{\Lie}=\big(\sC, 1_\sC,\underline{\ell}^K\big)$ and $\underline{\phi}^{\mc,\underline{\La}}$ be the descendants of $\sC_{\binarycomm}$ and $\mc$ up to the homotopy $\underline{\La}$, respectively. Thus we have a unital $sL_\infty$-morphism ${\underline{\phi}^{\mc,\underline{\La}}}:\xymatrix{\sC_{\Lie} \ar@{.>}[r] & \fieldk}$.
Forgetting the unit from $\sC_{\Lie}$, we have an $sL_\infty$-algebra $\big(\sC,\underline{\ell}^K\big).$

Recall that an element $x \in \sC$ is called a cohomological random variable if $Kx =0$. Let $x^\pr$ be another cohomological random variable in the same cohomology class. Then $x$ and $x^\pr$ have the same expectation $\mc(x)= \mc(x^\pr)$, which depends only on the cochain homotopy type of $\mc$. 
On the other hand, we {\em cannot} say that these two random variables $x$ and $x^\pr$ are equivalent. In particular, since $K$ is not a derivation of the product, in general $x^{\pr n} - x^{n} \notin \Im K$ so the distributions $\mc(e^{t x})$ and $\mc(e^{t x^\pr})$ may be different.
Even worse, the distribution $\mc(e^{t x})$ of the random variable $x$ may not even well defined,\footnote{
This phenomena is reminiscent of classical measure theoretic probability theory where the random variables are measurable functions, which do not always form an algebra.
This failure of algebraic structure in important classes of the examples is a common criticism of classical algebraic probability theory from the viewpoint of measure theoretic probability theory. 
But our homotopical version reflects precisely this issue and thus the criticism is not applicable here. This phenomenon indeed poses an interesting challenge to the theory in defining correlations (joint moments) among random variables, which will be addressed later.} since there is no guarantee that $K x^n =0$ for $n\geq 2$, i.e., $Kx =0$ does not imply that $K e^{tx}=0$. Hence the $K$-cohomology itself can not characterize random variables since there is no natural algebra structure on it as $K$ is not a derivation.
The failure, however, of $K$ being a derivation is crucial for existence of correlations; for example, the condition $x\in \Im K$, which implies $\mc(x)=0$, does not necessarily imply that $\mc(x^2)=0$ as well since $\mc(x^2) -\mc(x)^2 =\phi^{\mc}_2(x,x)$ (the variance of $x$). 
 
To further the discussion, we shall need the following important definition and proposition:
\begin{definition}\label{coradef}
Let $V$ be a graded vector space regarded as an $sL_\infty$-algebra $\big(V, \underline{0}\big)$ with zero $sL_\infty$-structure. 
For any $sL_\infty$-morphism $\underline{\w}:\xymatrix{\big(V, \underline{0}\big)\ar@{.>}[r]&\big(\sC,\underline{\ell}^K\big)}$ and any $\underline{\Lambda}$ as in Definition~\ref{bdesc}, we will associate the following three families of operators:
\begin{enumerate}
\item the operators $\underline{\Pi}^{\underline{\w}} =\Pi^{\underline{\w}} _1, \Pi^{\underline{\w}} _2,\Pi^{\underline{\w}} _3,\dotsc$, where $\Pi^{\underline{\w}} _n$ is a degree zero operator $S^nV\to\sC$ defined by
\begin{align*}
\Pi^{\underline{\w}} _n(v_1,\dotsc, v_n)&\ceq 
\sum_{\pi \in P(n)}\ep(\pi)\w\big(v_{B_1}\big)\cdots \w\big(v_{B_{|\pi|}}\big)
,
\end{align*}
\item the operators $\underline{\m} =\m _1, \m _2,\m_3,\dotsc$, where $\m _n$ is a degree zero operator $S^nV\to \fieldk$ defined by:
\begin{align*}
\m_n(v_1,\dotsc, v_n)&\ceq \mc\left(\Pi^{\underline{\w}} _n\big(v_1,\dotsc, v_n\big)\right),
\end{align*}
and
\item the operators $\underline{\k} =\k _1, \k _2,\k_3,\dotsc$, where $\k _n$ is a degree zero operator $S^nV\to\fieldk$ defined by:
\begin{align*}
{\k} _n(v_1,\dotsc, v_n)&\ceq \left( \underline{\phi}^{\mc,\underline{\La}}\bullet \underline{\w}\right)_n\big(v_1,\dotsc,v_n\big)
\end{align*}
\end{enumerate}
\end{definition}
\begin{proposition}\label{cora}
\begin{enumerate}
\item \label{baby Pi item: K} The collection of operators $\Pi^{\underline{\w}} _n$ satisfies $K\Pi^{\underline{\w}} _n=0$.

\item \label{baby Pi item: Sigma} Suppose $\underline{\w}$ and $\underline{\tilde\w}$ are $sL_\infty$-homotopic $sL_\infty$-morphisms $\xymatrix{\big(V, \underline{0}\big)\ar@{.>}[r]&\big(\sC,\underline{\ell}^K\big)}$. Then $\Pi^{\underline{\tilde\w}} _n -\Pi^{\underline{\w}} _n = K \S_n$ for some degree $-1$ map $\S_n$ from $S^n V$ to $\sC$.

\item \label{baby Pi item: mk} The operators $\m_n$ and $\k_n$ depend only on the cochain homotopy type of $\mc$ and the $sL_\infty$-homotopy type of $\underline{\w}$;

\item \label{baby Pi item: relation}the two families $\underline{\m}$ and $\underline{\k}$ satisfy the relation
\[
\m_n(v_1,\dotsc, v_n)=
\sum_{\pi \in P(n)}\ep(\pi)\k\big(v_{B_1}\big)\cdots \k\big(v_{B_{|\pi|}}\big).
\]
\end{enumerate}
\end{proposition}

We shall not prove this proposition here, as it is a corollary of a more general proposition in Section~\ref{sec:maintheory}. Instead we will make several remarks about it (proving some but not all of it).

\begin{remark}
We have the following commutative diagram in the category of $sL_\infty$-algebras:
\[
\xymatrix{
\big(V,\underline{0}\big)\ar@{.>}[r]^{\underline{\w}}
\ar@{.>}@/_2pc/[rr]_{\underline{\k}=\underline{\phi}^{\mc,\underline{\La}}\bullet \underline{\w}}
&
\big(\sC,\underline{\ell}^K\big)\ar@{.>}[r]^{\underline{\phi}^{\mc,\underline{\La}}}&\big(\fieldk,\underline{0}\big)
}.
\]
Recall that the $sL_\infty$-homotopy type of the composition of two $sL_\infty$-morphisms depends only on the $sL_\infty$-homotopy types of each constituent $sL_\infty$-morphism. Hence the $sL_\infty$-homotopy type of $\underline{\k}$ depends only on the homotopy types of $\underline{\phi}^{\mc,\underline{\La}}$ and $\underline{\w}$. Recall also that the $sL_\infty$-homotopy type of the descendant $sL_\infty$-morphism $\underline{\phi}^{\mc,\underline{\La}}$ is independent of $\underline{\La}$ and depends only on the cochain homotopy type of $\mc$. It follows that the $sL_\infty$-homotopy type of $\underline{\k}$ depends only on the $sL_\infty$-homotopy type of $\underline{\w}$ and the cochain homotopy type of $\mc$.
On the other hand $\underline{\k}$ is an $sL_\infty$-morphism between trivial $sL_\infty$-algebras $(V,\underline{0})$ and $(\fieldk, \underline{0})$ so that two $sL_\infty$-morphisms are homotopic if and only if they are equal. Hence $\underline{\k}$ itself depends only on the $sL_\infty$-homotopy type of $\underline{\w}$ and the cochain homotopy type of $\mc$.
 \hfill$\natural$
\end{remark}

\begin{remark}
The explicit form of $\Pi^{\underline{\w}} _n$ for $n=1,2,3$ are
\begin{align*}
\Pi^{\underline{\w}} _1(v)=&\w_1(v)
,\\
\Pi^{\underline{\w}} _2(v_1,v_2)=&\w_1(v_1)\cdot \w_1(v_2) + \w_2(v_1,v_2)
,\\
\Pi^{\underline{\w}}_3(v_1,v_2,v_3)=&
\w_1(v_1)\cdot\w_1(v_2)\cdot\w_1(v_3) 
+\w_2(v_1, v_2)\cdot\w_1(v_3) 
+\w_1(v_1)\cdot\w_2(v_2, v_3)
\\
&
+(-1)^{|v_1||v_2|}\w_1(v_2)\cdot \w_2(v_1, v_3)
+\w_3(v_1, v_2, v_3).
\end{align*} \hfill$\natural$
\end{remark}
\begin{remark}
Let $\big(V, \underline{0}\big)$ be a finite dimensional graded vector space with zero $sL_\infty$-structure. Let $\{e_i\}_{i\in J}$ be a basis of $V$ and $t_V=\{t^i\}_{i\in J}$ be the dual basis. For any $sL_\infty$-morphism $\underline{\w}=\w_1,\w_2,\cdots$ from $V_{\Lie}$ to $\sC_{\Lie}$, set
\[
\G(\underline{\w}) = \sum_{n=1}^\infty \sum_{j_1,\dotsc, j_n}\Fr{1}{n!} t^{j_n}\cdots t^{j_1}\w_n\big(e_{j_1},\dotsc, e_{j_n}\big) 
\in\big( \fieldk[\![t_V]\!]\otimes \sC\big)^0.
\]
In this context, $e^{\G(\underline{\w})}$ is the generating function for the family $\underline{\Pi}^{\underline{\w}}$:
\[
e^{\G(\underline{\w})} = 1 +\sum_{n=1}^\infty \sum_{j_1,\dotsc, j_n}\Fr{1}{n!} t^{j_n}\cdots t^{j_1}\Pi^{\underline{\w}}_n(e_{j_1},\dotsc,e_{j_n})\in\big( \fieldk[\![t_V]\!]\otimes \sC\big)^0.
\]
Applying $\mc$ we get $Z(t_V)\ceq \mc \left(e^{\G(\underline{\w})}\right)\in \fieldk[\![t_V]\!]^0$, which generates the family $\underline{\m}$:
\begin{align*}
Z(t_V) &= 1 +\sum_{n=1}^\infty \sum_{j_1,\dotsc, j_n}\Fr{1}{n!} t^{j_n}\cdots t^{j_1}\m_n(e_{j_1},\dotsc,e_{j_n}).
\end{align*}
Similarly, we can define $F(t_V)$ as the generating function of the family $\underline{\k}$ by replacing $\m$ with $\k$ in the above. Now items~\ref{baby Pi item: K},~\ref{baby Pi item: mk}, and~\ref{baby Pi item: relation} of {\bf Proposition \ref{cora}} are equivalent to the following statements:
\begin{enumerate}
\item we have $K e^{\G(\underline{\w})}=0$, which is equivalent to the Maurer--Cartan equation of the descendant $sL_\infty$-algebra $\big(\sC,\underline{\ell}^K\big)$:
\[
K \G(\underline{\w}) \
+\sum_{n=2}^\infty \Fr{1}{n!}\ell_n^K\left(\G(\underline{\w}),\dotsc,\G(\underline{\w})\right)=0,
\]
\item the generating functions $Z(t_V)$ and $F(t_V)$ depend only on the cochain homotopy type of $\mc$ and the $sL_\infty$-homotopy type of $\underline{\w}$, and
\item the generating functions satisfy the relation $ Z(t_V) = e^{F(t_V)}$.
\end{enumerate}
 \hfill$\natural$
\end{remark}
\begin{remark}
Consider a binary \padj{} probability space $\xymatrix{\sC_{\binarycomm} \ar[r]^\mc &\fieldk}$ concentrated in non-positive degree. Consider an arbitrary set $\{x_1,x_2,\dotsc, x_k\}$ of $k$ elements in $\sC^0$. Any element in $\fieldk[x_1,x_2,\dotsc,x_k]$ belongs to $\Ker K$ so the joint moments are well-defined and the generating function $\mc\left(\exp\left(\sum_{i=1}^k t^i x_i\right)\right)$ depends only on the cochain homotopy type of $\mc$.
Consider a $k$-dimensional vector space $V$ with basis $\{e_i\}_{i=1,2,\dotsc,k}$ and define a linear map $\w_1:V\rightarrow \sC^0$ by $\w_1(e_i)=x_i$. Then $\underline{\w}=\w_1,0,0,0,\dotsc$ is an $sL_\infty$-morphism from $V_{\Lie}=(V,\underline{0})$ to $\sC_{\Lie}$ which satisfies
\[
\exp\left(\sum_{i=1}^k t^i x_i\right)=e^{\G(\underline{\w})}.
\] 
Note that any family $\underline{\tilde{\w}}=\tilde\w_1,\tilde\w_2,\cdots$, of linear maps $\tilde{\w}_n: S^nV\to\sC^0$, define an $sL_\infty$-morphism from $V_{\Lie}$ to $\sC_{\Lie}$.
Now the essential content of Proposition \ref{cora} is the following:
let $\underline{\tilde{\w}}=\tilde\w_1,\tilde\w_2,\cdots$ be an $sL_\infty$-morphism that is $sL_\infty$-homotopic to $\underline{\w}=\w_1$, then we have
\[
\mc\left(\exp\left(\sum_{i=1}^k t^i x_i \right) \right)=\mc\left(\exp\left(\sum_{n=1}^\infty\sum_{i_1,\dotsc, i_n}\Fr{1}{n!}
t^{i_n}\cdots t^{i_1}\tilde\w_n\big(e_{i_1},\dotsc,e_{i_n}\big)\right)\right).
\]
By this means we have obtained an infinite family of identities whose utility shall be demonstrated in the following example. 
 \hfill$\natural$
\end{remark}

\begin{example}\label{gaussian-ex3}
Consider the Gaussian Example \ref{hgaussian-ex2}: $\sC_{\binarycomm}=(\sC,K,\cdot)$, where $\sC=\R[s,\eta]$ and $K={-\s^2\cdot\Fr{\rd^2}{\rd s\rd\eta}} + L_s \Fr{\rd}{\rd \eta}$, with the descendant $sL_\infty$-algebra $\sC_{\Lie}=(\sC, K, \ell^K_2)$, where 
\[
\ell_2^K(\eta,\eta)=0,\quad \ell_2^K(\eta,s)=\ell_2^K(s,\eta)=-\s^2,\quad \ell_2^K(s,s)=0,
\] 
and $\ell_2^K$ is a derivation of the product. 
Let $V =\R=\R\cdot e$, regarded as a trivial $sL_\infty$-algebra $V_{\Lie}=(\R, \underline{0})$. 
Let $t$ be the dual basis of $V$. 

Consider the $sL_\infty$-morphism $\underline{\w}=\w_1,0,0,0,\dotsc$ from $V_{\Lie}$ to $\sC_{\Lie}=(\sC, K, \ell^K_2)$ defined by $\w_1(e) =s$.
Then the moment generating density is
\[
e^{\G(\underline{\w})}= e^{t \w_1(e)}= e^{ts}.
\]
We claim that the $sL_\infty$-morphism $\underline{\w}$ is $sL_\infty$-homotopic to $\underline{\tilde\w}=0,\tilde\w_2,0,0,0,\dotsc$,
where $\tilde\w_2(e,e)= \s^2\in \fieldk$. Then the corresponding moment generating density is 
\[
e^{\G(\underline{\tilde\w})}= e^{\fr{t^2}{2} \w_2(e,e)}= e^{\fr{t^2}{2}\s^2},
\]
and, from $\mc\left(e^{\G(\underline{\w})}\right)=\mc\left(e^{\G(\underline{\tilde\w})}\right)$, we conclude that
\[
 Z(t)_x=\mc\left(e^{ts}\right) = \mc\left( e^{\fr{t^2}{2\s^2}}\right) = e^{\fr{t^2}{2}\s^2}.
\]

Now we check the claim. Note that $s = K\eta$. Define a linear map $\varsigma: V\rightarrow \sC$ of degree $-1$ such that $\varsigma(e)=\eta$. Then $\w_1 = K\varsigma$.
Hence the cochain map $\w_1:(V,0)\rightarrow (\sC, K)$ is cochain homotopic to zero.
Recall that this does not necessarily imply that the $sL_\infty$-morphism $\underline{\w}=\w_1,0,0,0,\dotsc$ is $sL_\infty$-homotopic to zero
(See {\em Remark} \ref{irem}).
Now consider the $sL_\infty$-homotopy $\underline{\l}= \l_1,0,0,0,\dotsc$, where $\l_1=-\varsigma$.
Let $\underline{\tilde\w}=\tilde\w_1,\tilde\w_2,\tilde\w_3,\dotsc$ be the $sL_\infty$-morphism that is homotopic to $\underline{\w}$ by the homotopy $\underline{\l}$. By definition, there exists a family $\underline{\Phi}=\Phi_1,\Phi_2,\dotsc$, polynomial in the variable $\t$, such that
 \begin{align*}
\underline{\Phi}\big|_{\t=0}&=\underline{\w}
,\\
\underline{\Phi}\big|_{\t=1}&=\underline{\tilde\w}
,\\
\Fr{1}{n!}\dot{\Phi}_{n}(e,\dotsc,e)
&=
\Fr{1}{n!}K\l_{n}(e,\dotsc,e) +\sum_{j=1}^{n-1}\Fr{1}{j!(n-j)!}\ell^{K^\pr}_{2}\left(\Phi_{j}(e,\dotsc,e),\l_{n-j}(e,\dotsc,e)\right).
\end{align*}
From $\underline{\l}= \l_1,0,0,0,\dotsc$, we have
\eqn{oneho}{
\dot{\Phi}_{1}(e)=K\l_{1}(e),
}
and for all $n\geq 2$
\eqnalign{twoho}{
\Fr{1}{n!}\dot{\Phi}_{\ell}(e,\dotsc,e)
&=
\Fr{1}{(n-1)!}\ell^{K^\pr}_{2}\left(\Phi_{n-1}(e,\dotsc,e),\l_{1}(e)\right).
}
From the initial condition ${\Phi}_1\big|_{\t=0}=\w_1= K\varsigma$ and \eq{oneho}, we have $\Phi_{1}= \w_1 + \t K\l_{1}=(1-\t) K\varsigma$ so that the $n=2$ case of \eq{twoho} becomes
\[
\fr{1}{2}\dot{\Phi}_{2}(e,e)
=-(1-\t)\ell^{K^\pr}_{2}\left(K\varsigma(e),\varsigma(e)\right)
=-(1-\t)\ell^{K^\pr}_{2}\left(s,\eta\right)
=(1-\t)\s^2.
\]
Combined with the initial condition $\Phi_2\big|_{\t=0} =0$,
we obtain $\Phi_{2}(e,e)= 2\s^2\t -\s^2\t^2 \in \fieldk[\t]$. Since $\fieldk[\t]$ is annihilated by $\ell^{K^\pr}_2$, it follows that $\fr{1}{n!}\dot\Phi_{n}(e,\dotsc,e)=0$ for all $n\geq 3$. In turn, this implies $\Phi_n =0$ for all $n\geq 3$, since $\Phi_n\big|_{\t=0} =0$ for $n\geq 3$. 
Finally, we have
\begin{align*}
&\Phi_{1}(e)=(1-\t) s
,\\
&\Phi_{2}(e,e)= \s^2\left(2\t -\t^2\right)
,\\
&\Phi_{n}(e,\dotsc,e)= 0\qquad\qquad n\geq 3.
\end{align*}
Evaluation at $\t=1$ proves the claim.
 \hfill$\natural$
\end{example}

Proposition \ref{cora} together above remarks lead us to the following definitions.

\begin{definition}
Given a binary \padj{} probability space $\xymatrix{\sC_{\binarycomm}\ar[r]^c &\fieldk}$, a space $\sV$ of \emph{homotopical random variables} in $\sC$ is a finite dimensional graded vector space $|\sV|=V$ regarded as an $sL_\infty$-algebra $(V,\underline{0})$ equipped with a homotopy type of $sL_\infty$-morphisms $[\underline{\w}^V]$ between $(V,\underline{0})$ and $(\sC,\underline{\ell}^K)$.% whose homotopy type is denoted $[\sV]$.

The \emph{moment} and \emph{cumulant} morphisms $\underline{\m}^\sV$ and $\underline{\k}^\sV$ on $\sV$ are $\underline{\m}^\sV\ceq \mc\circ\underline{\Pi}^{\underline{\w}^V}$ and $\underline{\k}^\sV\ceq \underline{\phi}^{\mc,\underline{\La}}\bullet\underline{\w}^{\!V}$ for any representative $\underline{\w}^{\!V}$ of $[\underline{\w}^{\!V}]$. 

A \emph{set of homotopical random variables} is a set $v_1,\dotsc, v_k\in V$ and given such a set its families of \emph{joint moments} and \emph{joint cumulants} are 
\begin{align*}
\left\{ \m^\sV _n\left(v_{j_1},\dotsc, v_{j_n}\right)
\big| n\geq 1, 1\leq j_1,\dotsc, j_n \leq k\right\}
,\\
\left\{ \k^\sV _n\left(v_{j_1},\dotsc, v_{j_n}\right)
\big| n\geq 1, 1\leq j_1,\dotsc, j_n \leq k\right\}
.
\end{align*}
The \emph{joint distributions} (the \emph{law}) of the homotopical set of random variables $\{v_1,\dotsc, v_k\}$ is the linear map 
\[
\hat\m^\sV: \fieldk[t^1,\dotsc, t^k]\rightarrow \fieldk,
\]
where $\{t^1,\dotsc, t^k\}$ is a set of super-commutative indeterminates with $|t^i|=-|v_i|$, such that, for a polynomial $P =\sum a_{j_1\cdots j_{n}}t^{j_1}\cdots t^{j_n}\in \fieldk[t_1,\dotsc, t_k]$,
\[
\hat\m^\sV(P) = \sum_{j_1,\dotsc, j_n} a_{j_1\cdots j_{n}}\m^V_n\big(v_{j_1},\dotsc, v_{j_n}\big).
\]
\end{definition}
By construction, for $v_1,\dotsc, v_n \in V$,
\[ 
 \m^\sV_n(v_1,\dotsc, v_n)=\!\sum_{\pi \in P(n)}\!\ep(\pi)\k^{\sV}\!\!\big(v_{B_1}\big)\cdots \k^{\sV}\!\!\big(v_{B_{|\pi|}}\big).
\]

Also by construction, the law of a set of homotopical random variables is itself an invariant of the homotopy types of $\mc$ and is independent of the choice of representative $\underline{\w}^{\!V}$. 

The space $V$ of random variables is assumed, in general, to have no algebraic structure beyond being a graded vector space. 

Note that the descendant unital $sL_\infty$-algebra $\sC_{\Lie}$ of a binary \padj{} probability space has the following special property.
\begin{lemma}\label{lemma:baby complete}
The cohomology $H$ of $\sC_{\Lie}$ has the structure of a unital $sL_\infty$-algebra $H_{\Lie}=(H,\underline{0})$ with zero $sL_\infty$-structure $\underline{0}$ and unit $1_H$ which is the cohomology class of $1_\sC$. There is a unital $sL_\infty$-quasi-isomorphism $H_{\Lie}\to \sC^{\Lie}$ that is not necessarily defined uniquely, even up to homotopy.
\end{lemma}

This observation is a special case of {\bf Theorem \ref{lemc}} and will not be proven here. It will motivate the definition of the following special type of space of homotopical random variables with more structure, which we will consider for the remainder of the subsection.

\begin{definition}
Given a binary \padj{} probability space $\xymatrix{\sC_{\binarycomm}\ar[r]^c &\fieldk}$, a \emph{complete space $\comprandvars$ of homotopical random variables} in $\sC$ is a pointed graded vector space $\sS$ regarded as a unital $sL_\infty$-algebra $(\sS,\underline{0})$ equipped with a homotopy type of unital $sL_\infty$-quasi-isomorphisms $[\underline{\w}^V]$ between $(\sS,\underline{0})$ and $(\sC,\underline{\ell}^K)$
\end{definition}

Lemma~\ref{lemma:baby complete} ensures that complete spaces of homotopical random variables exist and shows that the graded vector space isomorphism type of such a complete space must be that of $H$. Then we can identify the moduli space of all complete spaces of homotopical random variables with the space of homotopy types of unital $sL_\infty$-quasi-isomorphisms from $H_{\Lie}$ to $\sC_{\Lie}$.

There are two descriptions above of the moment morphism. It is defined in terms of the family $\underline{\Pi}^{\underline{\w}^\comprandvars}$ and it satisfies an identity in terms of $\underline{\k}^{\comprandvars} = \underline{\phi}^{\mc}\bullet \underline{\w}^{\sS}$.

There is a third description that will also be useful.
\begin{theorem}
Let $\comprandvars$ be a complete space of homotopical random variables in a binary probability space. Define $\iota^{\comprandvars}:\sS\rightarrow\fieldk$ as $\m_1^{\comprandvars}=\mc\circ \w_1^{\sS}$.

Then there exists a family $\underline{M}^{\comprandvars}=M_1^{\comprandvars}, M_2^{\comprandvars},\dotsc$ of linear degree zero maps $M_n^{\comprandvars}:S^n\sS\rightarrow\sS$ which satisfy the following properties:
\begin{itemize}
\item
$\m^{\comprandvars}_n = \iota^{\comprandvars}\circ M^{\comprandvars}_n,
$
\item $M_1^{\comprandvars}$ is the identity, and
\item for all $n\ge 1$ and $s_1,\dotsc, s_n \in \sS$, we have the identity
$M^{\comprandvars}_{n+1}\big(s_1,\dotsc, s_n, 1_\sS\big)=M^{\comprandvars}_n\big(s_1,\dotsc, s_n\big)$.
\end{itemize}
\end{theorem}

Again, we will not prove this theorem, which is a special case of a theorem in Section~\ref{sec:maintheory}. Instead, we briefly discuss an implication of the theorem.

Assume that $\sS\simeq H$ is finite dimensional. Choose a basis $\{e_\a\}_{\a\in J}$ containing $e_0 = 1_\sS$. Let $\{t^\a\}_{\a \in J}$ be the dual basis so that we have the moment generating function of $\comprandvars$ given by
\[
Z_\comprandvars\ceq 1 +\sum_{n=1}^\infty \Fr{1}{n!}\sum_{\a_1,\dotsc,\a_n \in J} t^{\a_n}\cdots t^{\a_1}
\m^\comprandvars_n\big(e_{\a_1},\dotsc, e_{\a_n}\big) \in \fieldk[\![t^\a]\!]_{\a\in J}
\]
Define a $1$-tensor $\{T_\comprandvars^\g\}$ in the formal power series ring $\fieldk[\![t^\a]\!]_{\a\in J}$ by
\[
T_\comprandvars^\g = \sum_{n=1}^\infty \Fr{1}{n!}\sum_{\a_1,\dotsc,\a_n \in J} t^{\a_n}\cdots t^{\a_1} M^\comprandvars_{\a_1\cdots\a_n}{}^\g
\]
where $\{M^\comprandvars_{\a_1\cdots\a_n}{}^\g\}$, $n\geq 1$, are the structure constants of $M^\comprandvars_n$:
\[
M^\comprandvars_n\big(e_{\a_1},\dotsc, e_{\a_n}\big) =\sum_{\g \in J} M^\comprandvars_{\a_1\cdots\a_n}{}^\g e_\g.
\]
Then we have 
\[
Z_\comprandvars= 1 +\sum_{\g\in J} T_\comprandvars^\g \iota^\comprandvars(e_\g).
\]
From $M_1^\comprandvars$ being the identity map on $\sS$, we have $\Fr{\rd}{\rd t^\b}T_\comprandvars^\g = \d_\b{}^\g$ up to the variables $t^\a$, which implies that the matrix $\sG$ with $\b\g$-entry $\sG_\b{}^\g : =\Fr{\rd}{\rd t^\b}T_\comprandvars^\g$ is invertible.
Define a $(2,1)$-tensor $\{A^\comprandvars_{\a\b}{}^\g\}$ in the formal power series ring $\fieldk[\![t^\a]\!]_{\a\in J}$ by
\[
A^\comprandvars_{\a\b}{}^\g\ceq \sum_{\r\in J}\Fr{\rd\sG_{\b}{}^\r}{\rd t^\a} \sG^{-1}_\r{}^\g.\]
We see directly that 
\[
\left(\Fr{\rd^2}{\rd t^\a\rd t^\b}- \sum_{\r\in J} A^\comprandvars_{\a\b}{}^\r\Fr{\rd}{\rd t^\r}\right) T_\comprandvars^\g=0.
\]
Then, it can easily be checked that the $(2,1)$-tensor $\{A^\comprandvars_{\a\b}{}^\g\}$ has the following properties:
\begin{align*}
A^\comprandvars_{0\b}{}^\g -\d_{\b}{}^\g &=0
,\\
A^\comprandvars_{\a\b}{}^\g - (-1)^{|e_\a||e_\b|} A^\comprandvars_{\b\a}{}^\g &=0
,\\
\Fr{\rd A^\comprandvars_{\b\g}{}^\s}{\rd  t^\a} -(-1)^{|e_\a||e_\b|}\Fr{\rd A^\comprandvars_{\a\g}{}^\s}{\rd t^\b}
+\sum_{\r\in J} \left(A^\comprandvars_{\b\g}{}^\r A^\comprandvars_{\a\r}{}^\s-(-1)^{|e_\a||e_\b|}A^\comprandvars_{\a\g}{}^\r A^\comprandvars_{\b\r}{}^\s\right)
&=0.
\end{align*}
One immediate consequence is that the moment generating function $Z_\comprandvars$ satisfies the following system of differential equations:
\begin{align*}
\left(\Fr{\rd^2}{\rd  t^\a \rd t^\b} - \sum_{\g\in J} A^\comprandvars_{\a\b}{}^\g \Fr{\rd}{\rd t^\g}\right)Z_\comprandvars=0
.\\
\left(\Fr{\rd}{\rd t^0}-1\right)Z_\comprandvars=0.\\
\end{align*}

\begin{remark}\label{remark: doesnt have to be finite superselection sector}
It should noted that a space $\sV$ of homotopical random variables does not need to be complete or a finite super-selection sector of a complete space.
Such incomplete spaces of homotopical random variables are also an important ingredient of homotopy probability theory. In general the moment generating function of the space $\sV$ satisfies a system of higher order, $n\geq 2$, formal PDEs if the cohomology of the target binary \padj{} probability space is finite dimensional (ODEs if $\sV$ is one dimensional). For examples of this phenomenon, see \cite{PP}. In some cases, we may even obtain a system of differential equations with infinite dimensional cohomology.
 \hfill$\natural$
\end{remark}

\section{Intermission}\label{sec:intermission}
\begin{quote}\it\footnotesize
 If we begin with certainties, we will end in doubts, but
 \\
 if we begin with doubts and bear them patiently
 \\
 we may end in certainty. --Francis Bacon.
 \end{quote}

We begin this section with some doubts about a fundamental assumption underlying the definition of classical algebraic probability spaces, namely that there is a binary multiplication $\cdot$ that determines every joint moment via the expectation morphism.
A family of joint moments of random variables could be viewed as a collection of observational data of joint events.
The corresponding family of joint cumulants could be viewed as the same collections of observational data but is organized coherently with a particular notion of independence of events in mind so that we may be able to figure out the underlying probability law of the system we are interested in. It would have been a ``big discovery'' if there were always a binary multiplication $\cdot$ in the space $A$ of random variables such that every joint moment could be determined by applying a single expectation morphism to a suitable iteration of binary products of random variables.

In Sect.~\ref{subs: pre-alg}, we will define a pre-algebraic probability space. The defining data of such a space shall be a pointed vector space $A$ and a family $\underline{\m}=\m_1,\m_2,\dotsc$ of linear maps $\m_n:S^n A \rightarrow \fieldk$, so that $1_A$ ``behaves as a unit''. Here we regard $A$ as the space of random variables and $\m_n(x_1,\cdots, x_n)$ as $n$-th joint moments. We then introduce the cumulant morphism $\underline{\k}$ defined in terms of the moment morphism $\underline{\m}$ and the binary product in $\fieldk$ via a combinatorial formula involving classical partitions. Independence of random variables is defined as usual in terms of cumulants and any classical algebraic probability space is an example of a pre-algebraic probability space.
We then show that our theory has the same central limit as its classical counterpart.

In Sect.~\ref{subs: when assoc}, we add an assumption to our pre-algebraic probability spaces---that there exist a family $\underline{M}=M_1,M_2,\dotsc$,
such that $\m_n =\m_1\circ M_n$. We then examine conditions which imply that there is an underlying classical algebraic probability space by repackaging $\underline{M}$ as a new family $\underline{m}=m_2,m_3,\dotsc$, via a combinatorial formula involving classical partitions. The family $\underline{m}$ will be a sequence of obstructions that will vanish after $m_2$ when $M_2$ is a commutative associative multiplication and $M_n$ is iterated multiplication using $M_2$.
%The initial condition shall be $m_2 \ceq M_2$ such that $(A, 1_A, \cdot)$ is a unital commutative algebra (without associativity), where $x\cdot y \ceq m_2(x,y)$, for all $x,y\in A$. The next operation $m_3:T^3 A\rightarrow A$ is defined so that it is both the obstructions if $(A, 1_A, \cdot)$ is associative and if $M_3(x,y,z)= x\cdot y\cdot z$. Assuming $m_3=0$,
%$m_4:T^4A\rightarrow A$ is the obstruction if $M_4(x_1,x_2,x_3, x_4)=x_1\cdot x_2\cdot x_3\cdot x_4$, etc.

In Sect.~\ref{subs: comm prob space}, we turn the table around by defining {\CorAnames{}} and \algebraic{} probability spaces as the natural generalization of unital commutative associative algebras and classical algebraic probability spaces. Then, in Sect ~\ref{subs:formal geo I}, we examine the formal geometric picture of finite dimensional probability spaces in this context.

\subsection{Commutative pre-algebraic probability spaces and the central limit theorem}\label{subs: pre-alg}

It is reasonable to assume that the space of all relevant random variables forms a vector space over the ground field $\fieldk$ in which every joint moment has a value.
\begin{definition}\label{defiprealg}
A \emph{pre-algebraic probability space} is a pair $\big(A, \underline{\m}\big)$ where $A$ is a pointed vector space and $\underline{\m}=\m_1,\m_2,\dotsc$ is a family of linear maps $\m_n:S^n A\to \fieldk$ called the \emph{moment morphism} which satisfy the unit constraints  
$\m_1(1_A)=1$ and for $n\geq 1$ and $x_1,\dotsc, x_n \in A$,
\[
\m_{n+1}\big(x_1,\dotsc, x_{n}, 1_A\big)= \m_{n}\big(x_1,\dotsc, x_{n}\big).
\] 
From the family $\underline{\m}$ and {\em the multiplication in $\fieldk$}, we define another family $\underline{\k}=\k_1,\k_2,\dotsc$, called the \emph{cumulant morphism}, recursively
 by the following relation valid for $n\geq 1$ and $x_1,\dotsc, x_n \in A$:
\[
\k_n\big(x_1,\dotsc, x_n\big)\ceq \m_n\big(x_1,\dotsc, x_n\big)-
\sum_{\substack{\pi \in P(n)\\ \color{red}|\pi|\neq 1}} \k\big(x_{B_1}\big) \cdots\k\big( x_{B_{|\pi|}}\big).
\]
\end{definition}

It is trivial to check that
\begin{itemize}
\item $\k_n$ is a linear map from $S^n(A)$ to $\fieldk$ for all $n\geq 1$,
\item $\k_1(1_A)=1$ and, for $n\geq 2$ and $x_1,\dotsc, x_n \in A$,
$\k_{n}\big(x_1,\dotsc, x_{n}, 1_A\big)=0$. 
\end{itemize}
Equivalently, we have
\eqn{zutr}{
\m_n\big(x_1,\dotsc, x_n\big)=\sum_{\pi \in P(n)} \k\big(x_{B_1}\big) \cdots\k\big( x_{B_{|\pi|}}\big),
}
so that the family $\underline{\k}$ also determines the family $\underline{\m}$ completely using {\em the multiplication in $\fieldk$}.

For a set of random variables $\{x_1,\dotsc, x_k\}$, one associates the family of joint moments
\[
\left\{ \m_n\big(x_{j_1},\dotsc, x_{j_n}\big)\big| n\geq 1, 1\leq j_1,\dotsc, j_n \leq k\right\},
\]
and the family of joint cumulants
\[
\left\{ \k_n\big(x_{j_1},\dotsc, x_{j_n}\big)\big| n\geq 1, 1\leq j_1,\dotsc, j_n \leq k\right\}.
\]

Let $\g =\sum_{j=1}^k t_j x_j$ in $A[\![t_1,\dotsc, t_k]\!]$. Then we define joint moment and cumulant \emph{generating functions} $Z(t_1,\dotsc, t_k)$ and $F(t_1,\dotsc, t_k)$, respectively, by
\begin{align*}
Z(t_1,\dotsc, t_k) &= 1 + \m_1(\g) +\sum_{n=2}^\infty \Fr{1}{n!} \m_n(\g,\dotsc, \g)\in\fieldk[\![t_1,\dotsc, t_k]\!]
,\\
F(t_1,\dotsc, t_k) &= \k_1(\g) +\sum_{n=2}^\infty \Fr{1}{n!} \k_n(\g,\dotsc, \g)\in\fieldk[\![t_1,\dotsc, t_k]\!]/\fieldk
,\\
\end{align*}
Now $F(t_1,\dotsc, t_k)$ is the formal logarithm of $Z(t_1,\dotsc, t_k)$.

The \emph{joint distributions} (the \emph{law}) of the set $\{x_1,\dotsc, x_k\}$ of random variables is the linear map 
\[
\hat\m: \fieldk[t_1,\dotsc, t_k]\rightarrow \fieldk
\]
such that, for a polynomial $P =\sum a_{j_1\cdots j_{n}}t_{j_1}\cdots t_{j_n}$,
\[
\hat\m(P) = \sum a_{j_1\cdots j_{n}}\m_n\big(x_{j_1},\dotsc, x_{j_n}\big).
\]
It is clear that the law of the set $\{x_1,\dotsc, x_k\}$ of random variables is determined completely by the family of joint moments, and vice versa.

Now we are going to show that pre-algebraic probability spaces belong to the same universality class (have the same limiting behavior) as classical algebraic probability spaces, adopting essentially the same arguments.\footnote{This should be obvious to the experts since the various central limit theorems are consequences
of  the combinatorics of  relevant cumulants (independence), defined from moments and multiplication in the ground field,
and are independent to how the moments are defined -- although this trivial observation may never has pronounced before.}

Let $t$ be a formal parameter. For any $z\in A$, let 
\begin{align*}
Z(t)_z &\ceq 1 + \sum_{n=1}^\infty\Fr{1}{n!} t^n \m_n(z,\dotsc,z) \in\fieldk[\![t]\!]
,\\
F(t)_z &\ceq \sum_{n=1}^\infty\Fr{1}{n!} t^n \k_n(z,\dotsc,z) \in\fieldk[\![t]\!]
\end{align*}
Then as before $F(t)_z$ is the formal logarithm of $Z(t)_z$.
It is obvious, by definition, that $Z(t)_{az}=Z(at)_z$ and $F(t)_{az}=F(at)_z$ if $a \in\fieldk$.

\begin{corollary}\label{crucib}
Let $x,y \in A$ such that $\k_n(x+y,\dotsc, x+y)=\k_n(x,\dotsc, x) + \k_n(y,\dotsc,y)$ for all $n\geq 1$.
Then 
\[
\left\{
\begin{array}{rl}
F(t)_{x+y}&= F(t)_x + F(t)_y
,\\
Z(t)_{x+y}&= Z(t)_x Z(t)_y.
\end{array}
\right.
\]
\end{corollary}
\begin{proof}
Trivial.
\hfill\qed\end{proof}

We arrive at the following definition for independence.
\begin{definition} Two random variables $x$ and $y$ are independent if $\k_n(x+y,\dotsc, x+y)=\k_n(x,\dotsc, x) + \k_n(y,\dotsc,y)$ for all $n\geq 1$.
\end{definition}

In the remainder of this subsection, assume that $\sqrt{N} \in \fieldk$ for every natural number $N$.

\begin{proposition}\label{procent}
Consider a set of jointly independent random variables $\{x_1,\dotsc, x_N\}$ such that $\m_n(x_i,\dotsc,x_i)=\m_n(x_j,\dotsc,x_j)$ for all $n\geq 1$ and $i,j\in 1,\dotsc, N$. 
Assume further that  $\m_1(x_i)=0$ and that $\m_2(x_i, x_i)$ is $\s^2$. 
Let
\[
X = \Fr{x_1+\cdots + x_N}{\sqrt{N}}.
\]
Then, for all $N\geq 1$, we have
\begin{align*}
&\k_1(X)=0,\\
&\k_2(X,X)=\s^2,\\
&\k_n(X,\dotsc, X)=\k_n(x_1,\dotsc,x_1)\cdot N^{1-\fr{n}{2}}.%, \quad \text{for all }n\geq 3.
\end{align*}
\end{proposition}
\begin{proof}
By the assumption that each $x_i$ has the same law,
$\m_n(x_i,\dotsc, x_i) =\m_n(x_j,\dotsc, x_j)$ for all $n=1,2,3\cdots$ and $i,j=1,\dotsc, N$, we also have $\k_n(x_i,\dotsc, x_i) =\k_n(x_j,\dotsc, x_j)$. %Set $\m_2(x_i,x_i)=\s^2$.
From {\bf Corollary \ref{crucib}},
we have
\[
F(t)_X = F_{x_1}\left(\Fr{t}{\sqrt{N}}\right)+\cdots + F_{x_N}\left(\Fr{t}{\sqrt{N}}\right) \in \fieldk[\![t]\!].
\]
By taking the $n$-th derivative with respect to $t$ and evaluating at $t=0$, we obtain for all $n\geq 1$ that
\[
\k_n(X,\dotsc, X) = \Fr{\k_n(x_1,\dotsc, x_1)+\cdots +\k_n(x_N,\dotsc, x_N)}{N^{\fr{n}{2}}}=\k_n(x_1,\dotsc,x_1)\cdot N^{1-\fr{n}{2}}.
\]
For $n=1$, we have $\k_1(X)=0$ since $\k_1(x_i)=\m_1(x_i)=0$ by assumption.
For $n=2$, we have $\k_2(X,X) =\s^2$ since $\k_2(x_i, x_i)=\m_2(x_i,x_i)=\s^2$ for all $i=1,\dotsc, N$.
\hfill\qed\end{proof}

The law of a random variable $x$ in a \algebraic{} probability space is called normal with zero mean and variance $\s^2 \in \fieldk$ if $\iota(x)\equiv\m_1(x)\equiv\k_1=0$, $\k_2(x,x)=\s^2$ and $\k_n(x,\dotsc,x)=0$ for all $n\geq 3$. Equivalently, $\m_{2k+1}(x,\dotsc, x)=0$, for all $k\geq 0$, and $\m_{2k}=\s^{2k}(2k-1)!!$,
for all $k\geq 1$. Then the moment generating function is $Z(t)_x = e^{\fr{t^2}{2}\s^2}$. 

The central limit theorem in classical algebraic probability spaces has a straightforward generalization to any pre-algebraic probability space. From {\bf Proposition \ref{procent}}, we have

\begin{theorem}\label{centrallimit}
Let $x_1,x_2,x_3,\cdots \in A$ be a sequence of identically distributed and jointly independent random variables such that $\m_1(x_i)=0$ and $\m_2(x_i, x_i)=\s^2$ for all $i\geq 1$.
Let
\[
X = \Fr{x_1+\cdots + x_N}{\sqrt{N}}.
\]
Then, the law of $X$ converges to normal distribution with zero mean and the the variance $\s^2$ in the large $N$ limit:
\[
\lim_{N\rightarrow \infty} \k_1(X)=0,\qquad \lim_{N\rightarrow \infty} \k_2(X,X)=\s^2,\qquad \lim_{N\rightarrow \infty} \k_n(X,\dotsc, X)=0
\hbox{ for }n\geq 3.
\]
\end{theorem}

\subsection{When does a space of random variables form an associative algebra?}\label{subs: when assoc}
\begin{definition}\label{deficomplete}
We say that a pre-algebraic probability space $\big(A; \underline{\m}\big)$ is \emph{complete} if there is a family $\underline{M}=M_1, M_2, \cdots$, where $M_n:S^n(A)\to A$ such that $M_1$ is the identity map on $A$ and for $x_1,\dotsc, x_n \in A$,
\begin{align*}
\m_n(x_1,\dotsc, x_n) &= \m_1\big(M_n(x_1,\dotsc, x_n)\big)
,\\
 M_{n+1}\big(x_1,\dotsc, x_{n}, 1_A\big)
&=M_{n}\big(x_1,\dotsc, x_{n}\big).
\end{align*}
\end{definition}
The idea motivating this definition of completeness is that in a complete space, any joint event is itself also an event represented by a random variable. In that case, we can perform the following reorganization.

% The following proposition summarizes a general conclusion we can make using the assumption that our space is complete.
\begin{definition}\label{cruciadef}
Let $\big(A;\underline{\m}\big)$ be a complete pre-algebraic probability space. We define the family $\underline{m}=m_2, m_3, \cdots$ of linear maps $m_n:T^n(A)\to A$ 
by $m_2=M_2$ and, for $n\geq 3$ and $x_1,\dotsc, x_n \in A$, recursively,
\begin{align*}
m_n(x_1,\dotsc, x_n) \ceq 
& M_n(x_1,\dotsc, x_n)
-\sum_{\substack{\pi \in P(n)\\ |B_{|\pi|}|= n-|\pi|+1\\ n-1\sim_\pi n\\\color{red}|\pi|\neq 1}}
M_{|\pi|}\left(x_{B_1},\dotsc, x_{B_{|\pi|-1}}, m\big(x_{B_{|\pi|}}\big)\right)
.\\
\end{align*}
\end{definition}

\begin{proposition}\label{crucia}
Let $\big(A;\underline{\m}\big)$ be a complete pre-algebraic probability space. Then for $n\geq 2$ the cumulant morphism satisfies the relation
\begin{align*}
\k_{n}\big(x_1,\dotsc, x_{n}\big)=
&
\!\!\!\!\sum_{\substack{\pi \in P(n)\\ |B_{|\pi|}|= n-|\pi|+1\\ n-1\sim_\pi n}}\!\!\!\!
\k\left(x_{B_1},\dotsc, x_{B_{|\pi|-1}}, m\big(x_{B_{|\pi|}}\big)\right)
-\sum_{\substack{\pi \in P(n)\\|\pi|=2\\ n-1\nsim_\pi n}}
\k\left(x_{B_1}\right)\k\left(x_{B_2}\right).
\end{align*}
\end{proposition}
The proof is an easy induction using Definitions~\ref{defiprealg},~\ref{deficomplete}, and~\ref{cruciadef}.
\begin{remark}
This proposition implies that obstructions for a complete pre-algebraic probability space $\big(A, \underline{\m}\big)$ to be a classical algebraic probability space are $m_3,m_4,m_5,\dotsc$.

More specifically, note that $m_2=M_2$ and
\[
m_3(x,y,z) = M_3(x,y,z) - m_2(x, m_2(y,z)).
\]
Assume that $m_3=0$. Then $M_3(x,y,z) = m_2(x, m_2(y,z))$, which implies that $(A,\cdot)$ is an commutative associative algebra, where $x\cdot y \ceq m_2(x,y)$. Hence $m_3$ is the obstruction to $(A, 1_A, m_2)$ being a unital commutative associative algebra. 
Fix $n\geq 4$ and assume that $m_3=\cdots = m_{n-1}=0$. Then by induction the sum in the definition of $m_n$ is equal to $x_1\cdot x_2\cdots x_n$, so we have
\[
m_n(x_1,\dotsc, x_n) = M_n(x_1,\dotsc, x_n) 
- x_1\cdot x_2\cdots x_n.
\]
Thus $m_n$ is the obstruction to $M_n(x_1,\dotsc, x_n)$ being  equal to $x_1\cdot x_2\cdots x_n$. That is, \[\m_n(x_1, x_2,\dotsc, x_n) = \m_1(x_1\cdot x_2\cdots x_n\big)\] if $m_n=0$.
\hfill$\natural$
\end{remark}

\begin{corollary}
Assume that $\underline{m}=m_2,0,0,0,\dotsc$. Set $x\cdot y = m_2(x,y)$ and $\iota =\m_1$. Then the quartet $\big(A, 1_A, \cdot, \iota\big)$ is a classical algebraic probability space, where for $n\geq 2$ and $x_1,\dotsc, x_n$ in $A$, we have:
\begin{align*}
\m_n(x_1,\dotsc, x_n)&=\iota(x_1\cdots x_n)
,\\
\k_{n}\big(x_1,\dotsc, x_{n}\big)&=
\k_{n-1}\left(x_1,\dotsc, x_{n-2}, x_{n-1}\cdot x_n\right)
-\sum_{\substack{\pi \in P(n)\\|\pi|=2\\ n-1\nsim_\pi n}}
\k\left(x_{B_1}\right)\k\left(x_{B_2}\right).
\end{align*}
\end{corollary}

\begin{proof}
Recall that $M_2(x,y) =m_2(x,y)$,
which implies that $x\cdot y=y\cdot x$ and $x\cdot 1_A=1_A\cdot x=x$ for all $x,y \in A$.
Note that $M_3(x,y,z) = m_2(x, m_2(y,z))$ for all $x,y,z \in A$, since $m_3=0$,
which implies that $m_2(x, m_2(y,z))= m_2(x, m_2(z,y))=m_2(z, m_2(x,y))= m_2(m_2(x, y), z)$.
Then $x\cdot(y\cdot z)= (x\cdot y)\cdot z$.
Hence $\big(A, 1_A, \cdot\big)$ is a unital commutative associative algebra. Note that the symmetric properties of $M_n$ for $n\geq 3$, impose no further condition on $m_2$. Then we have, for $n\geq 1$, 
 \[
 M_n\big(x_1,\dotsc, x_n) = x_1\cdots x_n,
\]
since $m_k=0$ for $k\geq 3$. 
Hence, for $n\geq 1$,
\[
\m_n \big(x_1,\dotsc, x_n) = \m_1\big(x_1\cdots x_n\big).
\]
Using \eq{zutr}, this relation implies that for $n\geq 1$, we have
\begin{align*}
 \m_1\big(x_1\cdots x_n) &= \sum_{\pi\in P(n)}\k\big(x_{B_1}\big)\cdots \k\big(x_{B_{|\pi|}}\big)
 .
\end{align*}
From $\iota=\m_1$, we have $\iota(1_A)=0$ so that $\big(A, 1_A, \cdot, \iota\big)$ is a classical algebraic probability space and the family $\underline{\k}$ is its cumulant morphism. The second relation follows from {\bf Proposition \ref{crucia}}.
\hfill\qed\end{proof}

\subsection{\Algebraic{} probability spaces}\label{subs: comm prob space}
We begin to axiomatize the explorations begun in the previous subsection. For the remainder of this section, all the lemmas and propositions are either straightforward combinatorial inductions or special cases of lemmas or propositions in Section~\ref{sec:maintheory}. In all cases the proofs are omitted.

\begin{definition}
A \emph{\CorAname{}} is a pair $(V, \underline{M})$, where $V$ is a pointed vector space and $\underline{M}=M_1, M_2, \cdots$ is a family of linear maps $M_n:S^nV\to V$ where $M_1$ is the identity and for $n>1$, $M_{n+1}\big(v_1,\dotsc, v_n, 1_V\big)=M_n\big(v_1,\dotsc, v_n\big)$.

We recursively define the family $\underline{m}=m_2,m_3,\ldots$ of linear maps $m_n:T^n V\to V$ by the following relation:
\begin{align*}
M_n(v_1,\dotsc, v_n) =
&
\sum_{\substack{\pi \in P(n)\\|B_{|\pi|}|=n-|\pi|+1\\ n-1\sim_\pi n}} 
M_{|\pi|}\left( v_{B_1}, \cdots, v_{B_{|\pi|-1}},
m\big(v_{B_{|\pi|}}\big)\right).
\end{align*}
\end{definition}
\begin{remark}
The condition $|B_{|\pi|}| = n-|\pi|+1$ in the sum implies that the blocks $B_1, \cdots, B_{|\pi|-1}$ are singletons. The other condition $ n-1\sim_\pi n$ in the sum implies that both $n-1$ and $n$ belong to the last block $B_{|\pi|}$.
\hfill$\natural$
\end{remark}
\begin{remark}
It would also be possible (and in fact more natural from a geometric point of view) to define a \CorAname{} as the pair $(V,\underline{m})$ and construct the family $\underline{M}$ recursively. However, the relations necessary to ensure that $\underline{M}$ is symmetric and satisfies the appropriate unit constraints are not as easily expressed. Nevertheless, we will sometimes implicitly take this approach in what follows.
\end{remark}

\begin{lemma}\label{oflato}
Let $(V, \underline{M})$ be a \CorAname{}. Then the family $\underline{m}=m_2,m_3,\cdots$ has the following properties:
\begin{enumerate}
\item for all $x \in V$, $m_2(x, 1_V)=x$;
\item for all $n\geq 3$, $m_n(x_1,\dotsc, x_n)=0$ whenever $x_j =1_V$ for some $j$;
\item \label{property:invariant}for $n\geq 0$ the operation $m_{n+2}:T^{n+2}V\to V$ is invariant under the action of $\hbox{Perm}_n\times\hbox{Perm}_2$ on $T^{n+2}V$.
\end{enumerate}
\end{lemma}

\begin{corollary} A unital commutative associative algebra $(A, 1_A,\cdot)$ is a \CorAname{} $(A,\underline{m})$ with $\underline{m}=m_2,0,0,0,\dotsc$ and $x\cdot y= m_2(x,y)$, for all $x,y\in A$, and vice versa.
\end{corollary}

\begin{definition}
An \emph{\algebraic{} probability space} is a \CorAname{} $A_{\Comm}=(A, \underline{M})$ together with a linear map $\iota: A\rightarrow \fieldk$ such that $\iota(1_A)=1$.
A \emph{random variable} is an element of $A$.
The \emph{moment morphism} $\underline{\m}=\m_1,\m_2,\cdots$ is defined as 
\[
\m_n=\iota\circ M_n.
\]
The \emph{cumulant morphism} $\underline{\k}=\k_1,\k_2,\cdots$ is defined recursively by the relation
\[
\iota\big(M_n(x_1,\dotsc, x_n)\big)=\sum_{\pi \in P(n)} \k\big(x_{B_1}\big) \cdots\k\big( x_{B_{|\pi|}}\big).
\]
We say that two random variables $x$ and $y$ are \emph{independent} if $\k_n(x+y,\dotsc,x+y)=\k_n(x,\dotsc,x)+\k_n(y,\dotsc, y)$ for all $n\geq 1$.
\end{definition}

\begin{remark} 
It is obvious that this definition coincides with that of a complete pre-algebraic probability space. It follows that the central limit theorem, {\bf Theorem \ref{centrallimit}}, also holds of \algebraic{} probability spaces over an appropriate ground field $\fieldk$.
\end{remark}
% \begin{remark}
% We shall see later see that there is a geometrical analogue such that $\underline{M}$ correspond to an affinely flat coordinates of smooth manifold, while $\underline{m}$ corresponds to a torsion-free flat connection on the tangent space of the manifold.
% It is well-known that those two structures are equivalent \cite{KN}.
% \hfill$\natural$
% \end{remark}

\begin{example}
Consider a binary \padj{} probability space $\xymatrix{\sC_{\binarycomm}\ar[r]^c &\fieldk}$ concentrated in non-positive degree. The only linear map of negative degree from $S^n\sC$ to $\fieldk$ is zero, so there is a unique descendant $\underline{\phi}^\mc=\phi^\mc_1,\phi^\mc_2,\cdots$ of the expectation $\mc$. We have $\phi^\mc_1=\mc$ and $\phi^\mc_n:S^n\sC^0\rightarrow \fieldk$. 
Consider a complete space $\comprandvars$ of homotopical random variables.
A priori the underlying vector space $\sS$ is also concentrated in non-positive degree, while both the moment morphism $\underline{\m}^\comprandvars$ and the cumulant morphisms $\underline{\k}^\comprandvars$ are non-zero maps only on $S^n(\sS^0)$ for all $n\geq 1$. Hence we have the structure of an \algebraic{} probability algebra on $\sS^0\simeq H^0$.
Note that every element in $\sS\simeq H$ with negative degree is independent of everything and invisible to both moments and cumulants. Nevertheless, such elements form an integral part of our definition of $\comprandvars$ because they are used in the unital $sL_\infty$-quasi-isomorphisms between $\sS$ and $\sC_{\binarycomm}$. If our binary \padj{} probability space is a homotopical realization of the classical algebraic probability space $\xymatrix{A_{\classical}\ar[r]^\iota &\fieldk}$ with infinitesimal symmetry $\vr:\mg \rightarrow L\!\Diff_\fieldk(A)$ then the \algebraic{} probability space structure on $\sS^0$ is that on $A_\mg\simeq \sS^0$.
\hfill$\natural$
\end{example}

Denote by $\category{\Prob}_{\Comm}{\overk}$ the category of {\CorAnames{}} whose objects are {\CorAnames{}} and whose morphisms are unit preserving linear maps. Then the ground field $\fieldk$ is the initial object of $\category{\Prob}_{\Comm}{\overk}$. 
Define the descendant $\underline{\phi}^f$ of a morphism $f:A_{\Comm}\rightarrow A^\pr_{\Comm}$ in $\category{\Prob}_{\Comm}{\overk}$ by the recursive relation
\[
f\big(M_n(x_1,\dotsc,x_n)\big)=\sum_{\pi\in P(n)}M^\pr_{|\pi|}\left(\phi^f(x_{B_1}),\dotsc, \phi^f(x_{B_{|\pi|}})\right),
\]
There is also the following equivalent exercise, which is proven by a straightforward, tedious induction.
\begin{lemma}\label{altmor}
For all $n\geq 2$ and for all $x_1,\dotsc,x_n \in A$,
\begin{align*}
\phi^f_{n}\big(&x_1,\dotsc, x_{n}\big)
\\
=
&
\!\!\!\!\!\!\sum_{\substack{\pi \in P(n)\\ |B_{|\pi|}|= n-|\pi|+1\\ n-1\sim_\pi n}}\!\!\!\!\!\!\!
\phi^f\!\!\left(x_{B_1},\dotsc, x_{B_{|\pi|-1}}, m\big(x_{B_{|\pi|}}\big)\right)
-\sum_{\substack{\pi \in P(n)\\ n-1\nsim_\pi n}}\!\!
m^\pr_{|\pi|}\left(\phi^f(x_{B_1}),\dotsc,\phi^f(x_{B_{|\pi|}})\right).
\end{align*}
\end{lemma}

Recall the category $\category{\Prob}_{\Lie}{\overk}$, which was defined in Sect.~\ref{subs:cats for algprob}. We can express $\Des$ as a functor from $\category{\Prob}_{\Comm}{\overk}$ to $\category{\Prob}_{\Lie}{\overk}$. The functor $\Des$ takes the object $A_{\classical}=\big(A,1_{A},\underline{m}\big)$ to $A_{\Lie}=\big(A, 1_{A}\big)$ and takes the morphism $f$ to its descendant $\underline{\phi}^f$.

Now we may regard \algebraic{} probability theory as the study of the category $\category{\Prob}_{\Comm}{\overk}$ and the functor $\Des$ into the category $\category{\Prob}_{\Lie}{\overk}$, where a morphism $\iota$ to the initial object $\fieldk$ in $\category{\Prob}_{\Comm}{\overk}$ has a special interpretation as an expectation map whose descendant $\underline{\phi}^\iota$ is the cumulant morphism $\underline{\k}$.

\subsection{Formal geometry of \algebraic{} probability spaces}\label{subs:formal geo I}

The purpose of this section is to examine some geometrical aspects of finite dimensional \algebraic{} probability spaces.

Let $V_{\Comm}=(V,\underline{m})$ be a finite dimensional \CorAname{}.
Choose a basis $\{e_\a\}_{\a \in J}$ of $V$ with the distinguished element $e_0= 1_{V}
\in V$. Let $t_{V}=\{t^\a\}_{\a\in J_{}}$ be the dual basis which defines a coordinate system on $V$. 
Let $\rd_\a =\fr{\rd}{\rd  t^\a}$ be the linear derivation on $\fieldk[\![t_{V}]\!]$ defined by $\rd_\a t^\b=\d_\a{}^\b$ and $\rd_\a (X\cdot Y) = \rd_\a X\cdot Y +X\cdot \rd_\a Y$ for all $X,Y \in \fieldk[\![t_{V}]\!]$.
Consider the family of structure constants $m_{\a_1\a_2}{}^\b, m_{\a_1\a_2\a_3}{}^\b,\dotsc$ for $\underline{m}=m_2,m_3,\dotsc$
 so that, for all $n\geq 2$,
\[
m_n\left(e_{\a_1},\dotsc,e_{\a_n}\right) =\sum_{\b\in J} m_{\a_1\cdots \a_n}{}^\b e_\b.
\]
From property~\ref{property:invariant} in {\bf Lemma \ref{oflato}}, we can form a formal power series $A_{\a\b}{}^\g \in \fieldk[\![t_{V}]\!]$ defined by 
\[
 A_{\a\b}{}^\g \ceq m_{\a\b}{}^\g +\sum_{n\geq 1}\Fr{1}{n!}\sum_{\r_1,\dotsc,\r_n \in J}
 t^{\r_n}\cdots t^{\r_1}m_{\r_1\cdots \r_n\a\b}{}^\g,
\]
representing a $(2,1)$-tensor field on ${V}$. 
\begin{lemma}\label{oflatcoor}
The formal power series $A_{\a\b}{}^\g \in \fieldk[\![t_{V}]\!]$ satisfy the following equations for all $\a,\b,\g,\s \in J$:
\begin{align*}
A_{0\b}{}^\g &=\d_\b{}^\g;
\\
A_{\a\b}{}^\g &=A_{\b\a}{}^\g;
\\\rd_{\a}{A}_{\b\g}{}^\s -\rd_{\b}{A}_{\a\g}{}^\s 
+\sum_{\r\in J}\left({A}_{\b\g}{}^\r A_{\a\r}{}^\s-{A}_{\a\g}{}^\r A_{\b\r}{}^\s\right)
&=0.
\end{align*}
\end{lemma}
This lemma is a special case of {\bf Lemma} \ref{finitec}.

Let $M_{\a_1}{}^\b, M_{\a_2\a_2}{}^\b,\dotsc$ be the family of structure constants such that, for all $n\geq 1$,
\[
M_n\left(e_{\a_1},\dotsc,e_{\a_n}\right) =\sum_{\b\in J} M_{\a_1\cdots \a_n}{}^\b e_\b.
\]
Recall that the families $\underline{m}$ and $\underline{M}$ uniquely determine one another.
\begin{lemma}\label{gfff}
There is a $1$-tensor field $T^\g$ in the power series ring $\fieldk[\![t_{V}]\!]$ uniquely determined by the following system of formal differential equations, valid for all $\b,\g,\s \in J$:
\begin{align*}
\left(\rd_\b\rd_\g -\sum_{\r\in J} A_{\b\g}{}^\r\rd_\r\right)T^\s&=0;
\\
\rd_\g T^\s\big|_{\underline{t}=\underline{0}}&=\d_\g{}^\s;
\\
T^\s\big|_{\underline{t}=\underline{0}}&=0;
\\
\left(\Fr{\rd}{\rd t^0} -1\right)T^\s &=\d_0{}^\s.
\end{align*}
More specifically, for all $n\geq 1$ and $\a_1,\dotsc,\a_n,\s \in J$,
\[
\rd_{\a_1}\cdots\rd_{\a_n} T^\s \Big|_{\underline{t}=\underline{0}}
=M_{\a_1\cdots\a_{n}}{}^\s.
\]
\end{lemma}
This lemma is a special case of {\bf Lemma} \ref{flatcoor} and \ref{finited}.

\begin{remark}\label{remark:intro to affine}
A smooth real manifold $X$ is called \emph{affinely flat} or simply \emph{affine} if there is a torsion-free flat affine connection on the tangent bundle $TX$. On a coordinate neighborhood $U$ with local coordinates $\{x^1,\dots,x^n\}$, an affine connection $\nabla$ is represented by
\[
\nabla_{\!\rd_i} \rd_j =\sum_k \G_{ij}{}^k \rd_k,
\]
where we have set $\rd_j=\fr{\rd}{\rd x^j}$.
The connection is torsion-free if $\G_{ij}{}^k=\G_{ji}{}^k$, and flat if
\[
\rd_i \G_{jk}{}^\ell -\rd_j \G_{ik}{}^\ell -\sum_m\left(\G_{ik}{}^m \G_{jm}{}^\ell -\G_{jk}{}^m \G_{im}{}^\ell\right)=0. 
\]
Then, one basic result of Cartan geometry \cite{KN} is that there is a flat coordinate system $\{y^1(\underline{x}),\dots, y^n(\underline{x})\}$ unique up to translation such that
\[
\nabla_{\tilde\rd_j} \tilde\rd_k =0,
\]
where $\tilde\rd_j=\fr{\rd}{\rd y^j}$. Then let $g_{j}{}^k =\fr{\rd y^k}{\rd x^j}$. We have $\tilde\rd_j=\sum_k (g^{-1})_j{}^k\rd_k$ such that
\[
\nabla_{\tilde\rd_i}\tilde\rd_j = \sum_m (g^{-1})_i{}^m\left( \rd_m (g^{-1})_j{}^k +\sum_{\ell} (g^{-1})_j{}^\ell \G_{m\ell}{}^k\right)\rd_k
\]
Hence $\nabla_{\tilde\rd_i}\tilde\rd_j =0$ is equivalent to the following.
\[
\rd_m (g^{-1})_j{}^k + \sum_\ell (g^{-1})_j{}^\ell \G_{m\ell}{}^k=0.
\]
Using matrix multiplication to move coefficients around and distributing using the Leibniz rule, we get
\[
\G_{m\ell}{}^k=-\sum_j g_{\ell}{}^j\rd_m(g^{-1})_j{}^k=\sum_j \left(\rd_m g_\ell{}^j\right)(g^{-1})_j{}^k
\]
which implies by inverting $(g^{-1})_j{}^k$ that
\[ 
\left(\rd_m \rd_\ell -\sum_k\G_{m\ell}{}^k \rd_k\right) y^j =0.
\]
\hfill$\natural$
\end{remark}

Hence a finite dimensional \CorAname{} has the following geometric interpretation. Consider a formal based manifold $\sM$ whose algebra of functions is the topological algebra $\fieldk[\![t_V]\!]$ and whose tangent space at the base point $o$ is isomorphic to the affine space $V$ after forgetting the origin.
Then the formal $(2,1)$-tensor $\{A_{\a\b}{}^\g\}$ can be interpreted as the connection one-form for a flat torsion-free affine connection $\nabla$ on $T\sM$ in a coordinate neighborhood of $o \in \sM$, i.e., $\nabla_{\rd_\a} \rd_\b =\sum_\g A_{\a\b}{}^\g\rd_\g$ and $A_{0\b}{}^\g =\d_{\b}{}^\g$. The tensor field $\{T^\g\}$ constitutes affine flat coordinates satisfying the last three conditions in {\bf Lemma \ref{gfff}}.

That is, define $\sG_{\b}{}^\g\ceq \rd_\b T^\g$. Then $\sG$, regarded as a square matrix, is invertible and we can define $\tilde\rd_\a = \sum_\g (\sG^{-1})_\a{}^\g\rd_\g$. Then $\nabla_{\tilde\rd_\a}\tilde \rd_\b =0$.

\section{Homotopy Probability Spaces}\label{sec:maintheory}

We have seen that \algebraic{} probability spaces constitute more than an idle generalization of classical algebraic probability spaces. The theory has the same universality as its classical counterpart and, on the space of coinvariants $A_\mg$ of a classical algebraic probability space with an infinitesimal symmetry, there is the structure of an \algebraic{} probability space. In this final section we revisit the theory in full generality, including the homotopy theory of probability spaces.

\subsection{Correlation algebras}\label{subs:corralg}

A \GCCorAname{} is the natural graded version of a \CorAname{}.

\begin{definition}\label{GCCorA}
A \emph{\GCCorAname{}} is a pair $(V, \underline{M})$, where $V$ is a pointed graded vector space
and $\underline{M}$ is a family $\underline{M}=M_1, M_2, \dotsc$,
where $M_n$ is a linear map from $S^n V$ to $V$ of degree $0$ for all $n\geq 1$ such that $M_{n+1}\big(v_1,\dotsc, v_n, 1_V\big)=M_n\big(v_1,\dotsc, v_n\big)$ for all $n\geq 1$ and $M_1$ is the identity on $V$.

We recursively define the family of operators $\underline{m}=m_2,m_3,\ldots$, where $m_n:T^n V\to V$, as follows:
\eqnalign{iterm}{
M_n(v_1,\dotsc, v_n) =
&
\sum_{\substack{\pi \in P(n)\\|B_{|\pi|}|=n-|\pi|+1\\ n-1\sim_\pi n}} 
\ep(\pi) M_{|\pi|}\left( v_{B_1}, \cdots, v_{B_{\pi-1}},
m\big(v_{B_{|\pi|}}\big)\right).
}
\end{definition}

\begin{remark}\label{Faad}
The conditions on the summation imply that the blocks $B_1, \cdots, B_{|\pi|-1}$ are singletons and that both $n-1$ and $n$ belong to the last block $B_{|\pi|}$. It follows that
\begin{align*}
M_n(v_1\otimes\cdots\otimes v_n) =
&
\sum_{\varsigma \subset [n-2]}\ep(\varsigma)
M_{|\varsigma|+1}\left( v_{\varsigma}\otimes m_{|\varsigma|+2}\big(v_{\varsigma^c}\otimes v_{n-1}\otimes v_n\big)\right),
\end{align*}
where the sum is over every subset $\varsigma$ of the ordered set $[n-2]=\{1,2,\cdots, n-2\}$ and 
$v_{\varsigma} = v_{j_1}\otimes\cdots\otimes v_{j_{|\varsigma|}}$
 if $\varsigma=\{j_1,\dotsc, j_{|\varsigma|}\}$. We use the notation
 $\varsigma^c$ for the subset of $[n-2]$ complimentary to $\varsigma$
and $\e(\varsigma)$ for the Koszul sign for $v_{\varsigma}\otimes v_{\varsigma^c}$ with respect to $v_1\otimes\cdots\otimes v_{n-2}$.
\hfill$\natural$
\end{remark}

\begin{example} Here are the explicit relations for $n=2,3,4$:
\begin{align*}
M_2(v_1,v_2) =&m_2(v_1,v_2)
,\\
M_3(v_1, v_2, v_3)=
&m_3(v_1,v_2,v_3)+ M_2(v_1, m_2(v_2, v_3))
\\
=
&m_3(v_1,v_2,v_3)+ m_2(v_1, m_2(v_2, v_3))
,\\
M_4(v_1, v_2, v_3,v_4)=&
m_4(v_1, v_2, v_3,v_4) 
\\
&
+M_2(v_1, m_3(v_2,v_3,v_4)
+(-1)^{|v_1||v_2|}M_2(v_2, m_3(v_1,v_3,v_4)
\\
&
+M_3(v_1,v_2,m_2(v_3,v_4)
\\
=&
m_4(v_1, v_2, v_3,v_4) 
\\
&
+m_2\big(v_1, m_3(v_2, v_3,v_4)\big)
+(-1)^{|v_1||v_2|}m_2\big(v_2, m_3(v_1, v_3,v_4)\big)
\\
&
+m_3\big(v_1, v_2, m_2(v_3, v_4)\big)
+m_2\big(v_1, m_2(v_2, m_2(v_3, v_4))\big)
.
\end{align*}
\end{example}

\begin{remark}
Note that the families $\underline{m}$ and $\underline{M}$  determine one another uniquely. That is, the recurrence relation above is equivalent to the following:
\begin{itemize}
\item $m_2=M_2$,
\item for $n\geq 3$ and $x_1,\dotsc, x_n\in \sC$,
\[
m_n(x_1,\dotsc, x_n)
=
M_n(v_1,\dotsc, v_n) 
-\sum_{\substack{\pi \in P(n)\\|B_{|\pi|}|=n-|\pi|+1\\ n-1\sim_\pi n\\|\pi|\neq 1}} 
\ep(\pi) M_{|\pi|}\left( v_{B_1}, \dotsc, v_{B_{\pi-1}},
m\big(v_{B_{|\pi|}}\big)\right).
\]
\end{itemize}
\hfill$\natural$
\end{remark}

\begin{lemma}\label{flato}
Let $(V, \underline{M})$ be a \GCCorAname{}. Then the family $\underline{m}=m_2,m_3,\cdots$ has the following properties:
\begin{enumerate}
\item\label{flato:unit2} for any $x \in V$, $m_2(x, 1_V)=x$;
\item\label{flato:unithigh} for all $n\geq 3$, $m_n(x_1,\dotsc, x_n)=0$ whenever $x_j =1_V$ for some $j$, $1\leq j\leq n$;
\item\label{flato:S2} for all $n\geq 0$ and $x_1,\dotsc, x_n, y,z \in V$, we have
\[
m_{n+2}\big(x_1,\dotsc, x_n, y,z\big)= (-1)^{|y||z|}m_{n+2}\big(x_1,\dotsc, x_{n}, z, y\big);
\]
\item\label{flato:Sn} for all $n\geq 2$ and $x_1,\dotsc, x_n, y,z \in V$,
we have
\[
m_{n+2}\big(x_1,\dotsc, x_n, y,z\big)= \e(\s)m_{n+2}\big(x_{\s(1)},\dotsc, x_{\s(n)}, y, z\big)
\]
for any $\s \in \hbox{Perm}_n$.
\end{enumerate}
\end{lemma}

\begin{proof}
Exercise.
\hfill\qed\end{proof}

Here is the standard example of a \GCCorAname{}.
\begin{lemma}
Let $V$ be a pointed graded vector space and $m_2:V\otimes V\to V$ an operator. The following are equivalent.
\begin{enumerate} 
\item $(V,m_2)$ is a unital graded commutative associative algebra.
\item Let $M_n(v_1,\dotsc,v_n)$ be $m_2(v_1,\dotsc m_2(v_{n-1},v_n))$. Then $(V,\underline{M})$ is a \GCCorAname{}.
\item Let $\underline{m}=m_2,0,0,\dotsc$ and define $\underline{M}$ recursively in terms of $\underline{m}$ as above. Then $M_n$ is an operator $S^nV\to V$ (not just $T^n V\to V$) and $(V,\underline{M})$ is a \GCCorAname{}.
\end{enumerate}
% is a unital graded commutative associative algebra $(V, 1_V, \cdot)$ with the binary multiplication $v_1\cdot v_2\ceq m_2(v_1,v_2)$, and vice versa. This is true if and only if the family $\underline{M}$ satisfies
\end{lemma}

\begin{proof}
Exercise.
\hfill\qed\end{proof}

Before leaving this subsection, we examine finite dimensional \GCCorAnames{}.

\subsubsection{The finite dimensional case.}\label{subsubs:fdcase}
The purpose of this section is to examine some aspects of finite dimensional \GCCorAnames.
% $({\sS}, \underline{m}^{\sS})$.

Fix notation first. Let $\sS$ be a finite dimensional pointed vector space.
Choose a basis $\{e_\a\}_{\a \in J_{\dia}}$ of $\sS$ with the unit $e_0= 1_{\sS}
\in \sS^0$. Let $t_{\sS}=\{t^\a\}_{\a\in J_{\dia}}$ be the dual basis, which defines a coordinate system on $\sS$. 
Let $\rd_\a =\fr{\rd}{\rd  t^\a}$ be the graded linear derivation on $\fieldk[\![t_{\sS}]\!]$ determined by $\rd_\a t^\b=\d_\a{}^\b$ and $\rd_\a (X\cdot Y) = \rd_\a X\cdot Y +(-1)^{|X||t^\a|}X\cdot \rd_\a Y$, for all $X,Y \in \fieldk[\![t_{\sS}]\!]$.
(Recall that $t^\a t^\b = (-1)^{|e_\a||e_\b|}t^\b t^\a$, for all $\a,\b \in J$).

Now suppose $({\sS}, \underline{M}^{\sS})$ is a \GCCorAname{}. Consider the family of structure constants $m_{\a_1\a_2}{}^\b, m_{\a_1\a_2\a_3}{}^\b,\dotsc$ for $\underline{m}^{\sS}=m^{\sS}_2,m^{\sS}_3,\dotsc$
 such that, for all $n\geq 2$,
\[
m^{\sS}_n\left(e_{\a_1},\dotsc,e_{\a_n}\right) =\sum_{\b\in J} m_{\a_1\cdots \a_n}{}^\b e_\b.
\]
From property~\ref{flato:Sn} in {\bf Lemma \ref{flato}}, we can form a formal power series $A_{\a\b}{}^\g \in \fieldk[\![t_{\sS}]\!]$ defined by 
\eqn{flatcon}{
 A_{\a\b}{}^\g \ceq m_{\a\b}{}^\g +\sum_{n\geq 1}\Fr{1}{n!}\sum_{\r_1,\dotsc,\r_n \in J}
 t^{\r_n}\cdots t^{\r_1}m_{\r_1\cdots \r_n\a\b}{}^\g,
}
representing a $(2,1)$-tensor field on ${\sS}$. 
Let $\underline{M}^{\sS}$ be the family of iterated multiplications generated by $\underline{m}^{\sS}$.
Similarly, let $M_{\a_1}{}^\b, M_{\a_2\a_2}{}^\b,\dotsc$ be the family of structure constants such that, for all $n\geq 1$,
\[
M^{\sS}_n\left(e_{\a_1},\dotsc,e_{\a_n}\right) =\sum_{\b\in J} M_{\a_1\cdots \a_n}{}^\b e_\b.
\]
%Recall that the family $\underline{m}^{\sS}$ uniquely determines the family $\underline{M}^{\sS}$ .
\begin{lemma}\label{flatcoor}
There is a $1$-tensor field $T^\g$ in the formal power series ring $\fieldk[\![t_{\sS}]\!]$ uniquely determined by the following system of formal differential equations:
for all $\b,\g,\s \in J$,
\begin{align*}
\left(\rd_\b\rd_\g -\sum_{\r\in J} A_{\b\g}{}^\r\rd_\r\right)T^\s&=0
,\\
\rd_\g T^\s\big|_{t_{\sS}=\underline{0}}&=\d_\g{}^\s
,\\
T^\s\big|_{t_{\sS}=\underline{0}}&=0
.
\end{align*}
Moreover, for all $n\geq 1$ and $\a_1,\dotsc,\a_n,\s \in J_{\dia}$,
\[
\rd_{\a_1}\cdots\rd_{\a_n} T^\s \Big|_{t_{\sS}=\underline{0}}
=M_{\a_1\cdots\a_{n}}{}^\s.
\]

\end{lemma}

\begin{proof}
We are looking for a solution $\check T^\s\in \fieldk[\![t_{\sS}]\!]$ of the first equation:
for all $\b,\g \in J$,
\eqn{accm}{
\rd_\b\rd_\g \check T^\s 
=\sum_{\r\in J} A_{\b\g}{}^\r \rd_\r \check T^\s
}
with initial conditions given by the second and third equation. It follows that
\[
\check T^\s = t^\s +\sum_{n\geq 2}\Fr{1}{n!}\sum_{\a_1,\dotsc,\a_n\in J}t^{\a_n}\cdots t^{\a_1}\check M_{\r_1\cdots \r_n}{}^\s
\]
for some family $\check M_{\a_1}{}^\s,\check M_{\a_2\a_2}{}^\s, \check M_{\a_1\a_2\a_3}{}^\s,\dotsc \in\fieldk$,
which are totally graded symmetric for all the lower indices.
Our goal is to show that $\check T^\g$ is determined uniquely and $\check T^\g =T^\g$, for all $\g \in J$, i.e.,
$\check M_{\a_1\dotsc\a_n}{}^\s= M_{\a_1\dotsc\a_n}{}^\s$,
for all $n\geq 1$ and $\a_1,\dotsc,\a_n,\s\in J$.

By setting $t_{\sS}=\underline{0}$ in \eq{accm}, we have
\[
\check M_{\b\g}{}^\s ={m}_{\b\g}{}^\s = M_{\b\g}{}^\s,
\]
where we have used $\rd_\r \check T^\s\big|_{t_{\sS}=\underline{0}}=\d_\r{}^\s$.
Fix $n\geq 3$ and assume that $\check{M}_{\a_1\dotsc\a_k}{}^\b = {M}_{\a_1\dotsc\a_k}{}^\b$, for $2\leq k \leq n-1$.
By applying $\rd_{\a_1}\cdots \rd_{\a_{n-2}}$ to \eq{accm},
 we obtain that 
\eqnalign{aconnx}{
\rd_{\a_1}\cdots\rd_{\a_n} \check T^\s 
=&\sum_{\r\in J}\rd_{\a_1}\cdots\rd_{\a_{n-2}}
\left( A_{\a_{n-1}\a_n}{}^\r \cdot\rd_\r \check T^\s\right)
\\
=&\sum_{\r\in J}
\sum_{\varsigma \subset [n-2]} \ep(\varsigma)
\left(D_{\varsigma} A_{\a_{n-1}\a_n}{}^\r \right)\cdot\left(D_{\varsigma^c}\rd_\r\check T^\s\right),
}
where the second summation in the last line is the sum over every subset $\varsigma$ of the ordered set $[n-2]=\{1,2,\dotsc,n-2\}$; $D_\varsigma =\rd_{\a_{j_1}}\cdots \rd_{\a_{j_{|\varsigma|}}}$ if $\varsigma=\big\{j_1,\dotsc, j_{|\varsigma|}\big\}$; $\varsigma^c$ is the subset of $[n-2]$ that is complimentary to $\varsigma$; and $\e(\varsigma)$ is the Koszul sign for $e_{\varsigma}\otimes e_{\varsigma^c}$ with respect to $e_1\otimes\cdots\otimes e_{n-2}$, where $e_{\varsigma}
= e_{j_1}\otimes\cdots\otimes e_{j_{|\varsigma|}}$ if $\varsigma=\big\{j_1,\dotsc, j_{|\varsigma|}\big\}$. 
For the second equality, we have used the graded version of multivariate Fa\'{a} di Bruno's formula (see Proposition 5 in \cite{MH}).
Recall that
\[
\check M_{\a_1\cdots\a_{n}}{}^\s
=\rd_{\a_1}\cdots\rd_{\a_n} \check T^\s \Big|_{t_{\sS}=\underline{0}},
\]
while, if $\varsigma=\big\{j_1,\dotsc, j_{|\varsigma|}\big\} \subset [n-2]$ such that $\varsigma^c=\big\{k_1,\dotsc, k_{|\varsigma^c|}\big\}$,
\begin{align*}
 m_{\a_{j_1}\cdots\a_{j_{|\varsigma|}}\a_{n-1}\a_n}{}^\r
&=D_{\varsigma} A_{\a_{n-1}\a_n}{}^\r\Big|_{t_{\sS}=\underline{0}}
\text{ and}\\
M_{\a_{k_1}\cdots\a_{k_{|\varsigma^c|}}\r}{}^\s
&=D_{\varsigma^c}\rd_\r \check T^\s\Big|_{t_{\sS}=\underline{0}}.
\end{align*}
Then, by setting $ t_{\sS}=\underline{0}$ in \eq{aconnx}, we have the following relations for all $n\geq 2$ and $\a_1,\dotsc,\a_n,\s \in J$:
\[
\check M_{\a_1\cdots\a_{n}}{}^\s
=
\sum_{\r\in J}
\sum_{\varsigma=\big\{j_1,\dotsc, j_{|\varsigma|}\big\} \subset [n-2]} 
\e(\varsigma)
m_{\a_{j_1}\cdots\a_{j_{|\varsigma|}}\a_{n-1}\a_n}{}^\r
M_{\a_{k_1}\cdots\a_{k_{|\varsigma^c|}}\r}{}^\s.
\]
Equivalently, we have
\begin{align*}
\check M_n(v_1\otimes\cdots\otimes v_n) =
&
\sum_{\r\in J}
\sum_{\varsigma \subset [n-2]}\ep(\varsigma)
M^{\sS}_{|\varsigma|+1}
\left( v_{\varsigma}\otimes m^{\sS}_{|\varsigma|+2}\big(v_{\varsigma^c}\otimes v_{n-1}\otimes v_n\big)\right).
\end{align*}
From {\bf Definition \ref{GCCorA}} (see {\em Remark \ref{Faad}}),
it follows that $\check M_{\a_1\cdots\a_{n}}{}^\s= M_{\a_1\cdots\a_{n}}{}^\s$. We have shown by induction that $\check{T}^\g = T^\g$, for all $\g \in J$.
\hfill\qed\end{proof}

\begin{lemma}\label{finitec}
The formal power series $A_{\a\b}{}^\g \in \fieldk[\![t_{\sS}]\!]$ has the following properties:% that, for all $\a,\b,\g,\s \in J_{\dia}$:
\begin{enumerate}

\item\label{finitec:unit} $
A_{0\b}{}^\g =\d_\b{}^\g;
$ 

\item\label{finitec:comm} $
A_{\a\b}{}^\g =(-1)^{|t^\a||t^\b|}A_{\b\a}{}^\g;
$

\item\label{finitec:derivs} $\rd_{\a}{A}_{\b\g}{}^\s -(-1)^{|t^\a||t^\b|}\rd_{\b}{A}_{\a\g}{}^\s 
+\sum_{\r\in J}\left({A}_{\b\g}{}^\r A_{\a\r}{}^\s-(-1)^{|t^\a||t^\b|} {A}_{\a\g}{}^\r A_{\b\r}{}^\s\right)
=0$.
\end{enumerate}
\end{lemma}

\begin{proof}
It is straightforward to prove~(\ref{finitec:unit}) and~(\ref{finitec:comm}) from properties~(\ref{flato:unit2})--(\ref{flato:S2}) in {\bf Lemma \ref{flato}}.
We now proceed to prove~(\ref{finitec:derivs}).
From
\[
\left(\rd_\b\rd_\g -\sum_{\r\in J} A_{\b\g}{}^\r\rd_\r\right)T^\s=0,
\]
we have
\begin{align*}
\rd_\a\rd_\b \rd_\g T^\s 
&=\sum_{\r\in J}\rd_\a A_{\b\g}{}^\r\rd_\r T^\s
+\sum_{\r\in J} A_{\b\g}{}^\r\rd_\a\rd_\r T^\s 
\\
&=\sum_{\r\in J}\left(\rd_\a A_{\b\g}{}^\r
+\sum_{\r^\pr\in J} A_{\b\g}{}^{\r^\pr} A_{\a\r^\pr}{}^{\r}\right)\rd_\r T^\s.
\end{align*}
Define a formal series $\sG_{\b}{}^\g\ceq \rd_\b T^\g\in \fieldk[\![t_{\sS}]\!]$ representing a $(1,1)$-tensor on ${\sS}$.
Regard $\sG$ as a square matrix with entries in $\fieldk[\![t_{\sS}]\!]$. By the third equation in {\bf Lemma \ref{flatcoor}} 
we have $\sG_\b{}^\g\big|_{t_{\sS}=0}=\d_\b{}^\g$, so $\sG$ is invertible with inverse $\{\sG^{-1}\}_{\b}{}^\g=\sG^{-1}_{\b}{}^\g$, i.e., $\sum_\r \sG_\b{}^\r\sG^{-1}_\r{}^\g =\d_\b{}^\g$.
Hence, the above equation is equivalent to the following:
\begin{align*}
\sum_{\r\in J}\left(\rd_\a\rd_\b \sG_\g{}^\r \right)\sG^{-1}_\r{}^\s
&=\rd_\a A_{\b\g}{}^\s
+\sum_{\r\in J} A_{\b\g}{}^{\r} A_{\a\r}{}^{\s}.
\end{align*}
Then property~(\ref{finitec:derivs}) follows from $\rd_\a\rd_\b -(-1)^{|e_\a||e_\b|}\rd_\b\rd_\a=0$.
\hfill\qed\end{proof}

\begin{lemma}\label{finited}
For all $\g\in J$, we have $
\left(\fr{\rd}{\rd t^0} -1\right)T^\g =\d_0^\g.
$
\end{lemma}
\begin{proof}
From the $1$st relation in {\bf Lemma \ref{flatcoor}}, we have
\[
\rd_0\rd_\g T^\s=\rd_\g\rd_0 T^\s=\sum_{\r\in J} A_{0\g}{}^\r\rd_\r T^\s = \rd_\g T^\s,
\]
where we have used $A_{0\g}{}^\r=\d_\g{}^\r$ (the $1$st relation in {\bf Lemma \ref{finitec}}) for the second equality.
It follows that $\rd_0 T^\s = T^\s + a^\s$, where $a^\s$ are some constants. From $\rd_0 T^\s|_{t_{\sS}=\underline{0}}=\d_0{}^\s$ and $T^\s|_{t_{\sS}=\underline{0}}=0$,
we have $a^\s =\d_0{}^\s$, so that $\rd_0 T^\s = T^\s + \d_0^\s$.
\hfill\qed\end{proof}

 \subsubsection{Property Q}
 In this subsection, we briefly consider \GCCorAnames{} with a special property.

 \begin{definition}
 A \GCCorAname{} $(V, \underline{M})$ has \emph{property Q} if the family $\underline{m}=m_2,m_3,\cdots$ is a collection of symmetric maps. That is to say, the $\GCCorAname{}$ has property Q if each map $m_n:T^n V\to V$ descends to the quotient by the symmetric group to a map $S^n V\to V$.
 \end{definition}

 Assume that $(\sS, \underline{M}^{\sS})$ is a finite dimensional \GCCorAname{} with property $Q$. Borrowing notation from Section~\ref{subsubs:fdcase}, consider the two formal power series $A_{\b\g}{}^\s, T^\g \in \fieldk[\![t_{\sS}]\!]$:
 \begin{align*}
 A_{\b\g}{}^\s &\ceq m_{\b\g}{}^\s 
 +\sum_{n\geq 1}\sum_{\r_1,\dotsc,\r_n\in J}\Fr{1}{n!}t^{\r_n}\cdots t^{\r_1}m_{\r_1\cdots \r_n\b\g}{}^\s
 ,\\
 T^\g &\ceq t^\g +\sum_{n\geq 2}\Fr{1}{n!}\sum_{\r_1,\dotsc,\r_n\in J}t^{\r_n}\cdots t^{\r_1}M_{\r_1\cdots \r_n}{}^\g
 ,
 \end{align*}
 representing a $(2,1)$-tensor and a $1$-tensor field, respectively, on ${\sS}$. Then it is trivial to check the following.
 \begin{lemma}\label{finitef}
 The formal power series $A_{\a\b}{}^\g \in \fieldk[\![t_{\sS}]\!]$ and $T^\g \in \fieldk[\![t_{\sS}]\!]$ satisfy the following properties for all $\a,\b,\g,\s \in J$:
 \begin{enumerate}

 \item\label{item: A0} 
 $
 A_{0\b}{}^\g =\d_\b{}^\g;
 $ 

 \item\label{item: Aabg} $
 A_{\a\b}{}^\g =(-1)^{|t^\a||t^\b|}A_{\b\a}{}^\g;
 $

 \item\label{item: T0} $
 \left(\Fr{\rd}{\rd t^0} -1\right)T^\g =\d_0^\g;
 $

 \item\label{item: TbgA} $
  \left(\rd_\b\rd_\g -\sum_{\r\in J} A_{\b\g}{}^\r\rd_\r\right)T^\s=0;
 $

 \item\label{item:Asym} $\rd_{\a}{A}_{\b\g}{}^\s -(-1)^{|t^\a||t^\b|}\rd_{\b}{A}_{\a\g}{}^\s =0$;

 \item\label{item:Aass} $\sum_{\r\in J}\left({A}_{\b\g}{}^\r A_{\a\r}{}^\s-(-1)^{|t^\a||t^\b|} {A}_{\a\g}{}^\r A_{\b\r}{}^\s\right)=0$.

 \end{enumerate}
 \end{lemma}
 \begin{proof}
 Properties~(\ref{item: A0}) and~(\ref{item: Aabg}) follow from those in Proposition \ref{finitec}. Property~(\ref{item: T0}) follows from Lemma \ref{finited} and Property~(\ref{item: TbgA}) follows from Lemma~\ref{flatcoor}. Property~\ref{item:Asym} can be checked directly as follows:
 \begin{align*}
 \rd_\a A_{\b\g}{}^\s 
 &=\sum_{n\geq 0}\sum_{\r_1,\dotsc,\r_n\in J}\Fr{1}{n!}t^{\r_n}\cdots t^{\r_1}m_{\r_1\cdots \r_n{\color{red}\a\b}\g}{}^\s
 \\
 &=(-1)^{|t^\a||t^\b|}\sum_{n\geq 0}\sum_{\r_1,\dotsc,\r_n\in J}\Fr{1}{n!}t^{\r_n}\cdots t^{\r_1}m_{\r_1\cdots \r_n{\color{red}\b\a}\g}{}^\s
 \\
 &=(-1)^{|t^\a||t^\b|}\rd_{\b}{A}_{\a\g}{}^\s.
 \end{align*}
 Then property~(\ref{item:Aass}) also follows from property~(\ref{finitec:derivs}) in Proposition \ref{finitec}.
 \hfill\qed\end{proof}

 \begin{lemma}
 Let ${\sS}$ be finite dimensional as a graded vector space. 
 Choose a basis $\{e_\a\}_{\a \in J}$ of $\sS$ with the unit $e_0= 1_{\sS}
 \in \sS^0$. Let $t_{\sS}=\{t^\a\}_{\a\in J}$ be the dual basis, which defines a coordinate system on $\sS$. 
 Let $A_{\a\b}{}^\g \in \fieldk[\![t_{\sS}]\!]$,
 be a formal power series representing a $(2,1)$-tensor field on ${\sS}$ with the following properties for all $\a,\b,\g,\s \in J_{\sS}$:
 \begin{enumerate}

 \item $A_{0\b}{}^\g =\d_\b{}^\g$;

 \item $A_{\a\b}{}^\g =(-1)^{|t^\a||t^\b|}A_{\b\a}{}^\g$;

 \item $\rd_{\a}{A}_{\b\g}{}^\s -(-1)^{|t^\a||t^\b|}\rd_{\b}{A}_{\a\g}{}^\s =0$;

 \item $\sum_{\r\in J}\left({A}_{\b\g}{}^\r A_{\a\r}{}^\s-(-1)^{|t^\a||t^\b|} {A}_{\a\g}{}^\r A_{\b\r}{}^\s\right)
 =0$.
 \end{enumerate}
 Then $\sS$ has the structure of a \GCCorAname{} with property Q.
 \end{lemma}
 \begin{proof}
 Exercise.
 \hfill\qed\end{proof}

\begin{remark}
Property Q will be relevant in the context of quantum field theory~\cite{P}, where many of the results in this paper will be background material.
\hfill$\natural$
\end{remark}

\subsection{The homotopy category of \padj{} probability algebras}\label{subs:prob alg theory}
%A \padj{} probability algebra is a simple cochain enhancement of \CorAname{}:

\begin{definition}
A\emph{ \padj{} probability algebra} is a trio $\sC_{\Comm}=(\sC, \underline{M}, K)$,
where the pair $(\sC, \underline{M})$ is a \GCCorAname{} and the pair $(\sC, K)$ is a pointed cochain complex. A \emph{morphism} of \padj{} probability algebras from $\sC_{\Comm}$ to $\sC^\pr_{\Comm}$ is a pointed cochain map and two morphisms $f$ and $\tilde f$ are \emph{homotopic} if they are pointed cochain homotopic.
\end{definition}

Let $\category{\HProb}_{\Comm}{\overk}$ be the category of \padj{} probability algebras and let $\homotopycat\category{\HProb}_{\Comm}{\overk}$ be the homotopy category of $\category{\HProb}_{\Comm}{\overk}$: its objects are the same as the those of $\category{\HProb}_{\Comm}{\overk}$ and its morphisms are homotopy types of morphisms in $\category{\HProb}_{\Comm}{\overk}$.

\begin{example}
A \padj{} probability algebra concentrated in degree zero is a \CorAname{}, and vice versa.
 A \padj{} probability algebra with zero differential is a \GCCorAname{}, and vice versa.
 \hfill$\natural$
\end{example}

\begin{example}
A binary \padj{} probability algebra $(\sC, \cdot, K)$ is a \padj{} probability algebra $(\sC, \underline{M}, K)$ such that $\underline{m}=m_2,0,0,0,\dotsc$ with $m_2(x,y) = x\cdot y$, and vice versa. The category $\category{\HProb}_{\binarycomm}{\overk}$ of binary \padj{} probability algebras is a full subcategory of $\category{\HProb}_{\Comm}{\overk}$.
 \hfill$\natural$
\end{example}

\begin{example}
The ground field $\fieldk$ is a \padj{} probability algebra, which is the initial object in both the category $\category{\HProb}_{\Comm}{\overk}$ and the homotopy category $\homotopycat\category{\HProb}_{\Comm}{\overk}$ of \padj{} probability algebras.
 \hfill$\natural$
\end{example}
\begin{definition}
A \emph{\padj{} probability space} is an object $\sC_{\Comm}$ together with a morphism $[\mc]$ to the initial object $\fieldk$ in the homotopy category $\homotopycat\category{\HProb}_{\Comm}{\overk}$;
\[
\xymatrix{\sC_{\Comm}\ar[r]^{[\mc]} &\fieldk},
\] 
\end{definition}
In practice, we study a \padj{} probability space in the category $\category{\HProb}_{\Comm}{\overk}$ by choosing a representative $\mc$ of the homotopy type $[\mc]$;
\[
\xymatrix{\sC_{\Comm}\ar[r]^\mc &\fieldk}
\] 
so that $\mc$ is a pointed cochain map from $\sC$ to $\fieldk$. The morphism $\mc$ shall be called the expectation and is only defined up to pointed cochain homotopy. Note that another representative $\tilde \mc$ of $[\mc]$ is homotopic to $\mc$ if $\tilde \mc = \mc+ r K$ for some pointed cochain homotopy $r$. In what follows, then, we are only allowed to consider structures and quantities on $\sC$ that are constant on all expectation morphisms in the same pointed homotopy type.

\begin{example} 
A binary \padj{} probability space $\xymatrix{(\sC, \cdot, K)\ar[r]^-\mc&\fieldk}$ is the same thing as a \padj{} probability space $\xymatrix{(\sC, \underline{M}, K)\ar[r]^-\mc&\fieldk}$ such that 
\begin{align*}
&\underline{m}=m_2,0,0,0,\dotsc
, \\
&m_2(x,y) = x\cdot y \quad \text{for all }x,y\in\sC.
\end{align*}
\hfill$\natural$
\end{example}

\begin{example} 
A \algebraic{} probability space is a \padj{} probability space concentrated in degree zero, and vice versa.
A classical algebraic probability space is a binary \padj{} probability space concentrated in degree zero, and vice versa.
\hfill$\natural$
\end{example}

\subsection{Homotopy functors to the category of unital \texorpdfstring{$sL_\infty$}{sL-infinity}-algebras}\label{subs:descendant functor}

In this subsection we construct a family of functors $\Des_{\La} : \category{\HProb}_{\Comm}{\overk}\Longrightarrow \category{\UsL}_\infty{\overk}$ from the category $\category{\HProb}_{\Comm}{\overk}$ of \padj{} probability algebras to the category $\category{\UsL}_\infty{\overk}$ of unital $sL_\infty$-algebras, each of which induces a well-defined functor $\text{ho}\!\Des_{\La} : \homotopycat\category{\HProb}_{\Comm}{\overk}\Longrightarrow \homotopycat\category{\UsL}_\infty{\overk}$ from the homotopy category $\homotopycat\category{\HProb}_{\Comm}{\overk}$ to the homotopy category $\homotopycat\category{\UsL}_\infty{\overk}$.

\begin{definition}[Descendant Algebra]
\label{GDesA}
The \emph{descendant} of an object $\sC_{\Comm}=\left(\sC, \underline{M}, {K}\right)$ in the category $\category{\HProb}_{\Comm}{\overk}$ is the pair $\sC_{\Lie} = \left(\sC, \underline{\ell}^{K}\right)$,
where $\underline{\ell}^{K}= {\ell}^K_1, \ell^K_2, \ell^K_3,\dotsc$ is a family of linear maps $\ell^K_n:S^n\sC\to \sC$ recursively defined by the relations
\[
K M(x_1, \dotsc, x_n)=\!\!\!\! \sum_{\substack{\pi \in P(n)\\ |B_i|=n-|\pi|+1}}\!\!\!\! \epsilon(\pi,i)
M_{|\pi|}\left(
 x_{B_1},\dotsc, x_{B_{i-1}},\ell^K(x_{B_i}), x_{B_{i+1}}, \dotsc, x_{B_{|\pi|}}\right).
\]
\end{definition}

\begin{remark}
The sequence $\underline{\ell}^K=\ell^K_1, \ell^K_2,\ell^K_3,\dotsc$ is well-defined since its defining equation can be rewritten as follows:
\begin{align*}
\ell^K_n(x_1,\dotsc,x_n)
=& KM_n(x_1,\dotsc,x_n)
\\
&
-\sum_{\substack{\pi \in P(n)\\ |B_i| = n-|\pi|+1\\ {\color{red}|\pi|\neq 1}}}\ep(\pi,i)
M_{|\pi|}\left( x_{B_1}, \dotsc,x_{B_{i-1}}, \ell^K(x_{B_i}), x_{B_{i+1}},\dotsc, x_{B_{|\pi|}}\right).
\end{align*}
We record the calculation of $\ell^K_1$,$\ell^K_2$, and $\ell^K_3$ explicitly: $\ell^K_1 =K$,
and
\begin{align*}
\ell^K_2(x_1,x_2)=& KM_2(x_1,x_2)
- M_2(Kx_1,x_2)-(-1)^{|x|_1}M_2(x_1,Kx_2)
,\\
\ell^K_3(x_1,x_2,x_3)=&
KM_3(x_1,x_2,x_3)
- M_3(Kx_1,x_2,x_3)
-(-1)^{|x_1|}M_3(x_1,Kx_2,x_3)
\\
&
-(-1)^{|x_1|+|x_2|}M_3(x_1,x_2,Kx_3)
- M_2\big(\ell^K_2(x_1,x_2),x_3\big)
\\
&
-(-1)^{|x_1|}M_2\big(x_1,\ell^K_2(x_2,x_3)\big)
-(-1)^{(|x_1|+1)|x_2|}M_2\big(x_2,\ell^K_2(x_1,x_3)\big)
.
\end{align*}\hfill$\natural$
\end{remark}

\begin{definition}[Descendant Morphisms]
\label{GDesM}
Let $\sC_{\Comm}=\left(\sC, \underline{M}, {K}\right)$ and $\sC^\pr_{\Comm}=\left(\sC^\pr, \underline{M}^\pr, {K}^\pr\right)$ be two \padj{} probability algebras and let $f$ be a morphism $\xymatrix{\sC_{\Comm} \ar[r]^f & \sC^\pr_{\Comm}}$ between them. Let $\underline{\La}=\La_1,\La_2,\La_3,\dotsc$ be a family of degree $-1$ linear maps $\La_n: S^n(\sC)\rightarrow \sC^\pr$ which satisfy $\La_1(1_\sC)=0$ and, for all $n\geq 1$, \[
\La_{n+1}\big(x_1,\dotsc,x_{n}, 1_\sC\big)=\La_{n}\big(x_1,\dotsc,x_{n}\big).
\]
Then the \emph{descendant of $f$ up to the homotopy $\underline{\La}$} is the sequence $\underline{\phi}^{f,\underline{\La}}= \phi_1^{f,\underline{\La}}, \phi_2^{f,\underline{\La}}, \phi_3^{f,\underline{\La}},\dotsc$ of linear maps $\phi_n^{f,\underline{\La}}:S^n \sC\to \sC^\pr$ defined by the recursive relation
\begin{align*}
f\left(M(x_1, \cdots, x_n)\right)=
&\sum_{\pi \in P(n)}\ep(\pi)M^\pr\left(\phi^{f,\underline{\La}}\!\left(x_{B_1}\right),\dotsc,
\phi^{f,\underline{\La}}\!\left(x_{B_{|\pi|}}\right)\right)
\\
&
+K^\pr \La_n(x_1,\dotsc, x_n)
\\
&
+\!\!\!\!\!\sum_{\substack{\pi \in P(n)\\ |B_i| = n-|\pi|+1}}\!\!\!\!\!\!\!\!
\ep(\pi,i)
\La_{|\pi|} \left( x_{B_1},\dotsc, x_{B_{i-1}},\ell^K\left(x_{B_i}\right), x_{B_{i+1}},\dotsc, x_{B_{|\pi|}}\right)
.
\end{align*}
\end{definition}

\begin{remark}
It is clear that sequence $\underline{\phi}^{f,\underline{\La}}={\phi}^{f,\underline{\La}}_1, {\phi}^{f,\underline{\La}}_2, {\phi}^{f,\underline{\La}}_3,\dotsc$ in the above definition is determined uniquely.
For example,
\begin{align*}
\phi^{f,\underline{\La}}_1=&{f} -K^\pr \La_1 -\La_1 K
,
\\
\phi^{f,\underline{\La}}_2(x_1,x_2)
= 
&
f\big(M_2(x_1,x_2)\big)
- M_2^\pr\big(\phi^{f,\underline{\La}}_1(x_1),\phi^{f,\underline{\La}}_1(x_2)\big)
\\
&
-K^\pr\La_2(x_1,x_2) - \La_2(Kx_1,x_2) -(-1)^{|x_1|}\La_2(x_1, Kx_2)
\\
&-\La_1(\ell_2^K(x_1,x_2)).
\\
=
&
\phi^{f,\underline{\La}}_1\big(M_2(x_1,x_2)\big)
- M_2^\pr\big(\phi^{f,\underline{\La}}_1(x_1),\phi^{f,\underline{\La}}_1(x_2)\big)
\\
&
-K^\pr\La_2(x_1,x_2) - \La_2(Kx_1,x_2) -(-1)^{|x_1|}\La_2(x_1, Kx_2)
\\
&
+K^\pr\La_1(M_2(x_1,x_2))
+\La_1(M_2(Kx_1,x_2))
+(-1)^{|x_1|}\La_1(M_2(x_1,Kx_2)).
\end{align*}
\hfill$\natural$
\end{remark}

We are going to show that the descendant of a \padj{} probability algebra is a unital $sL_\infty$-algebra, that the descendant of a morphism of \padj{} probability algebra up to arbitrary homotopy is a unital $sL_\infty$-morphism between descendant unital $sL_\infty$-algebras, and that the descendants of a morphism of \padj{} probability algebra up to different homotopies are homotopic as unital $sL_\infty$-morphisms.
For each object $\sC_{\Comm}$ in the homotopy category $\homotopycat\category{\HProb}_{\Comm}{\overk}$, let $\text{ho}\!\Des\big(\sC_{\Comm}\big)$ be its descendant algebra.
For each morphism $[f]$ in the homotopy category $\homotopycat\category{\HProb}_{\Comm}{\overk}$, let $\text{ho}\!\Des\big(\sC_{\Comm}\big)$ be the homotopy type of its descendant morphism up to arbitrary homotopy. We check the functoriality of $\text{ho}\!\Des$ to arrive at the following main theorem of this subsection:
\begin{theorem}\label{maina}
$\text{ho}\!\Des$ is a functor from the homotopy category $\homotopycat\category{\HProb}_{\Comm}{\overk}$ of \padj{} probability algebras to the homotopy category $\homotopycat\category{\UsL}_\infty{\overk}$ of unital $sL_\infty$-algebras. 
\end{theorem}
Our proof of the above theorem consists of several lemmas, which show that
\begin{itemize}
\item the descendant algebra $\sC_{\Lie}=\Des_\La \big(\sC_{\Comm}\big)$ is a unital $sL_\infty$-algebra,
\item the descendant morphism $\underline{\phi}^{f,\underline{\La}}$, up to homotopy $\underline{\La}$, of a \padj{} probability algebra morphism $f$ is a unital $sL_\infty$-morphism between the respective descendant algebras,
\item the homotopy type of the descendant morphism does not depend on $\underline{\La}$,
\item if $f$ and $\tilde{f}$ are homotopic probability algebra morphisms, their descendants up to arbitrary homotopy $\underline{\phi}^{\tilde f,\underline{ \La}}$ and $\underline{\phi}^{f,\underline{\La}}$ are homotopic as unital $sL_\infty$-morphisms, and
\item if $f$ and $f^\pr$ are composable morphisms of \padj{} probability algebras, then composition of the descendant morphisms of $f$ and $f^\pr$, each up to homotopy, is homotopic to the descendent of $f^\pr\circ f$ up to homotopy.
\end{itemize}
For the purposes of proving these lemmas, we shall repeatedly use the following construction. Fix $x_1,\ldots, x_n$, and let $t_i$ be a formal graded commutative parameter such that $|t_i| = -|x_i|$. Let $\g =\sum_{i=1}^n t_i x_i$. Rewriting various formulas dependent on the $x_i$ in terms of $\g$ will make the simplifications easier.

\begin{lemma}\label{bna}
Let $\sC_{\Lie} = \left(\sC, \underline{\ell}^{K}\right)$ be the descendant of the \padj{} probability algebra $\sC_{\Comm}=(\sC, \underline{M}, K)$.
Then $\sC_{\Lie}$ is a unital ${sL}_\infty$-algebra.
\end{lemma}

\begin{proof}
We first show that $\left(\sC, \underline{\ell}^{K}\right)$ is a ${sL}_\infty$-algebra.
From the symmetric property of $\underline{M}$, it is trivial that $\ell^K_n$ is a linear map from $S^n(\sC)$ to $\sC$ of degree $1$. Define another family $\underline{F}=F_1, F_2,\dotsc$ of degree $2$ linear maps $F_n :S^n(\sC)\rightarrow \sC$ by
\eqn{xxy}{
{F}_n(x_1,\dotsc, x_n)\ceq 
\sum_{\substack{\pi \in P(n)\\ |B_i| = n-|\pi|+1}}\ep(\pi,i)
\ell^K_{|\pi|}\left( x_{B_1}, \dotsc, x_{B_{i-1}}, \ell^K(x_{B_i}), x_{B_{i+1}},\dotsc, x_{B_{|\pi|}}\right).
}
We claim that, for all $n\geq 1$,
\eqn{xxz}{
\sum_{\substack{\pi \in P(n)\\ |B_i| = n-|\pi|+1}}\ep(\pi)
M_{|\pi|}\left( x_{B_1}, \dotsc, x_{B_{i-1}}, {F}(x_{B_i}), x_{B_{i+1}},\dotsc, x_{B_{|\pi|}}\right)
=0,
}
which can be rewritten as follows;
\begin{align*}
{F}_n(x_1,\dotsc,x_n)
=
-\sum_{\substack{\pi \in P(n)\\ |B_i| = n-|\pi|+1\\{\color{red}|\pi|\neq 1}}}\ep(\pi)
M_{|\pi|}\left( x_{B_1}, \dotsc,x_{B_{i-1}}, {F}(x_{B_i}), x_{B_{i+1}},\dotsc, x_{B_{|\pi|}}\right).
\end{align*}
Note that $F_1=K^2=0$. Note also that $F_n=0$ if $F_1=F_2=\cdots = F_{n-1}=0$.
By induction, we deduce that ${F}_n=0$ for all $n\geq 1$, which implies that $\left(\sC, \underline{\ell}^{K}\right)$ is an ${sL}_\infty$-algebra.

Now we prove the claimed identity \eq{xxz}. The formula for the descendant algebra in {\bf Definition} \ref{GDesA} is equivalent to the following:
\eqnalign{xxa}{
 K M_n(\g,\dotsc,\g)
=\sum_{k=1}^n\binom{n}{k-1}M_{k}\big(\g,\dotsc,\g, \ell_{n-k+1}(\g,\dotsc,\g)\big).
}
Similarly, the formula \eq{xxy} is equivalent to the following:
\begin{align*}
 {F}_n (\g,\dotsc,\g)
=\sum_{k=1}^n\binom{n}{k-1} \ell^K_{k}\big(\g,\dotsc,\g, \ell_{n-k+1}(\g,\dotsc,\g)\big).
\end{align*}
By applying $K$ to \eq{xxa}, we have
\eqn{xxb}{
K^2 M_n(\g,\dotsc,\g)
=\sum_{k=1}^n\binom{n}{k-1}KM_{k}\big(\g,\dotsc,\g, \ell_{n-k+1}(\g,\dotsc,\g)\big).
}
From \eq{xxa}, we also deduce that, for any $\b$,
\eqnalign{xxc}{
K M_{k}(\b,\g,\dotsc,\g)
=
&
(-1)^{|\b|}\sum_{j=2}^{k}\binom{k-1}{j-2} M_{j}\big(\b,\g,\dotsc,\g, \ell^K_{k-j+1}(\g,\dotsc,\g)\big)
\\
&
+\sum_{j=1}^{k}\binom{k-1}{j-1}M_{j}\big(\g,\dotsc,\g, \ell^K_{k-j+1}(\b,\g,\dotsc,\g)\big).
}
Combining \eq{xxb} with \eq{xxc}, we have
\begin{align*}
&K^2 M_{n}(\g,\dotsc,\g)
\\
=&-\sum_{k=1}^n\binom{n}{k-1}
\sum_{j=2}^{k}\binom{k-1}{j-2}M_{j}\big(\ell^K_{n-k+1}(\g,\dotsc,\g),\g,\dotsc,\g, \ell^K_{k-j+1}(\g,\dotsc,\g)\big)
\\
+&
\sum_{k=1}^n\binom{n}{k-1}
\sum_{j=1}^{k}\binom{k-1}{j-1}M_{j}\big(\g,\dotsc,\g, \ell^K_{k-j+1}(\ell^K_{n-k+1}(\g,\dotsc,\g),\g,\dotsc,\g)\big).
\end{align*}
Note that $\ell_k(\g,\dotsc,\g)$ has degree $1$ for all $k\geq 1$.
Hence the terms in the first line on the right hand side of the above equation cancel due to the graded commutativity of $M_k$ for $k\geq 2$. After a resummation of the remaining terms, we obtain
\eqnalign{xxd}{
K^2 & M_{n}(\g,\dotsc,\g)
=\sum_{k=1}^n\binom{n}{j-1}M_{j}\big(\g,\dotsc,\g, {F}_{n-j+1}(\g,\dotsc,\g)\big),
}
which is equivalent to the claimed identity \eq{xxz} since $K^2=0$.

Now we are going to show that $\ell_n(1_\sC, x_2,\dotsc, x_n)=0$ for all $n\geq 1$ and for any $ x_2,\dotsc, x_n\in\sC$. From \eq{xxc},
we have
\begin{align*}
K M_{k}(1_\sC,\g,\dotsc,\g)
=
&
\sum_{j=2}^{k}\binom{k-1}{j-2}M_{j}\big(1_\sC,\g,\dotsc,\g, \ell^K_{k-j+1}(\g,\dotsc,\g)\big)
\\
&
+\sum_{j=1}^{k}\binom{k-1}{j-1}M_{j}\big(\g,\dotsc,\g, \ell_{k-j+1}(1_\sC,\g,\dotsc,\g)\big).
\end{align*}
Using the property that $M_n(1_\sC, x_2,\dotsc, x_n)=M_{n-1}(x_2,\dotsc, x_n)$ for all $n\geq 2$ and the identity \eq{xxa}, we have
\[
\sum_{j=1}^{k}\binom{k-1}{j-1}M_{j}\big(\g,\dotsc,\g, \ell_{k-j+1}(1_\sC,\g,\dotsc,\g)\big)=0.
\]
This relation together with the initial condition $\ell^K_1(1_\sC)=0$, since $K=\ell^K_1$, imply by induction that $\ell^K_n(1_\sC, x_2,\dotsc, x_n)=0$ for all $n\geq 1$ and any $ x_2,\dotsc, x_n\in\sC$.
\hfill\qed\end{proof}

\begin{lemma}\label{bnb}
Let $\sC_{\Comm}=(\sC, \underline{m}, K)$ and $\sC^\pr_{\Comm}=(\sC^\pr, \underline{m^\pr}, K^\pr)$ be \padj{} probability algebras with descendant algebras $\sC_{\Lie} = \left(\sC, \underline{\ell}^{K}\right)$ and $\sC^\pr_{\Lie} = \left(\sC^\pr, \underline{\ell}^{K^\pr}\right)$, respectively.
Let $f$ be a morphism from $\sC_{\Comm}$ to $\sC^\pr_{\Comm}$ in the category $\category{\HProb}_{\Comm}{\overk}$. Then the descendant $\underline\phi^{f,\underline{\La}}$ of $f$ up to an arbitrary homotopy $\underline{\La}$ is a unital ${sL}_\infty$-morphism from $\sC_{\Lie}$ to $\sC^\pr_{\Lie}$.
 \end{lemma}

\begin{proof}
From the symmetric property of $\underline{M}$ and $\underline{M}^\pr$, it is trivial that $\phi^f_n$ is a linear map from $S^n(\sC)$ to $\sC^\pr$ of degree $0$ for all $n\geq 1$. Define another family $\underline{{E}}={E}_1, {E}_2,\dotsc$ of linear maps ${E}_n :S^n(\sC)\rightarrow \sC^\pr$ of degree $1$ for all $n\geq 1$ by
\eqnalign{kxy}{
{E}_n(x_1,\dotsc, x_n)\ceq &
\sum_{{\pi \in P(n)}}\ep(\pi)
\ell_{|\pi|}^{K^\pr} \left(\phi^{f,\underline{\La}}\left( x_{B_1}\right), \dotsc, \phi^{f,\underline{\La}}\big( x_{B_{|\pi|}}\big)\right)
\\
&
-\!\!\!\!\sum_{\substack{\pi \in P(n)\\ |B_i| = n-|\pi|+1}}\ep(\pi,i)
\phi^{f,\underline{\La}}_{|\pi|}\left( x_{B_1}, \dotsc, x_{B_{i-1}}, \ell^K(x_{B_i}), x_{B_{i+1}},\dotsc, x_{B_{|\pi|}}\big)\right).
}
We claim that, for all $n\geq 1$,
\eqnalign{zzw}{
\sum_{\substack{\pi \in P(n)}}&\ep(\pi)
M_{|\pi|}\left( \phi^{f,\underline{\La}}(x_{B_1}), \dotsc, \phi^{f,\underline{\La}}(x_{B_{i-1}}), {E}(x_{B_i}),
\phi^{f,\underline{\La}}(x_{B_{i+1}}),\dotsc, \phi^{f,\underline{\La}}(x_{B_{|\pi|}})\right)
\\
&=0,
}
which can be rewritten as follows:
\begin{align*}
E_n&(x_1,\dotsc, x_n)
\\
=
&
-\!\!\!\sum_{\substack{ \pi \in P(n)\\ {\color{red}|\pi|\neq 1}}}\ep(\pi)
M_{|\pi|}\left( \phi^{f,\underline{\La}}(x_{B_1}), \dotsc, \phi^{f,\underline{\La}}(x_{B_{i-1}}), {E}(x_{B_i}),
\phi^{f,\underline{\La}}(x_{B_{i+1}}),\dotsc, \phi^{f,\underline{\La}}(x_{B_{|\pi|}})\right).
\end{align*}
Note that $E_1=K^\pr\circ {f}- 
{f}\circ K=0$. Note also that $E_n=0$ if we assume that $E_1=E_2=\cdots = E_{n-1}=0$.
By induction we deduce that $E_n=0$ for all $n\geq 1$, which implies that $\underline{\phi}^f$ is an ${sL}_\infty$-morphism.
It remains to prove the claim \eq{zzw}.

We begin by noting that the formula in {\bf Definition} \ref{GDesM} is equivalent to the following;
\eqnalign{zza}{
 {f}\big( M_n(\g,&\dotsc,\g)\big)
 =
 \sum_{j_1+\cdots+j_r=n}\Fr{n!}{j_1! \cdots j_r! r!}M^\pr_{r}\left(
\phi^{f,\underline{\La}}_{j_1}(\g,\dotsc,\g), \dotsc, \phi^{f,\underline{\La}}_{j_r}(\g,\dotsc,\g)\right)
\\
&
+K^\pr \La_n(\g,\dotsc, \g)
+\sum_{k=1}^n\binom{n}{k-1}\La_{k}\left(\g,\dotsc,\g, \ell^K_{n-k+1}(\g,\dotsc,\g)\right),
}
Similarly, the formula \eq{kxy} is equivalent to the following;
\begin{align*}
 {E}_n (\g,\dotsc,\g)
=
&
\sum_{j_1+\cdots+j_r=n}\Fr{n!}{j_1! \cdots j_r! r!}\ell_{r}^{K^\pr}\left(
\phi^{f,\underline{\La}}_{j_1}(\g,\dotsc,\g), \dotsc, \phi^{f,\underline{\La}}_{j_r}(\g,\dotsc,\g)\right)
\\
&
-\sum_{k=1}^n\binom{n}{k-1} \phi^{f,\underline{\La}}_{k}\left(\g,\dotsc,\g, \ell^K_{n-k+1}(\g,\dotsc,\g)\right).
\end{align*}
By applying $K^\pr$ to \eq{zza} and ${f}$ to \eq{xxa}, we have
\eqnalign{zzb}{
 K^\pr {f}\left(M_n(\g,\dotsc,\g)\right)
 =
 &
 \sum_{j_1+\cdots+j_r=n}\Fr{n!}{j_1! \cdots j_r!r!}
 K^\pr M_{r}^\pr\left(
\phi^{f,\underline{\La}}_{j_1}(\g,\dotsc,\g), \dotsc, \phi^{f,\underline{\La}}_{j_r}(\g,\dotsc,\g)\right)
\\
&
+\sum_{k=1}^n\binom{n}{k-1} K^\pr \La_{k}\big(\g,\dotsc,\g, \ell^K_{n-k+1}(\g,\dotsc,\g)\big);
\\
{f}\left( K M_n(\g,\dotsc,\g)\right)
=
&\sum_{k=1}^n\binom{n}{k-1}{f}\left(M_{k}\big(\g,\dotsc,\g, \ell^K_{n-k+1}(\g,\dotsc,\g)\big)\right).
}
Hence
\begin{align*}
(K^\pr\circ {f}- &{f}\circ K)\big(M_n(\g,\dotsc, \g_n)\big)
\\
 =
 &
 \sum_{j_1+\cdots+j_r=n}\Fr{n!}{j_1! \cdots j_r!r!}
 K^\pr M_{r}^\pr\left(
\phi^{f,\underline{\La}}_{j_1}(\g,\dotsc,\g), \dotsc, \phi^{f,\underline{\La}}_{j_r}(\g,\dotsc,\g)\right)
\\
&
+\sum_{k=1}^n\binom{n}{k-1} K^\pr \La_{k}\big(\g,\dotsc,\g, \ell^K_{n+k-1}(\g,\dotsc,\g)\big)
\\
&
-\sum_{k=1}^n\binom{n}{k-1} {f}\left(M_{k}\big(\g,\dotsc,\g, \ell^K_{n+k-1}(\g,\dotsc,\g)\big)\right).
\end{align*}
Applying the formulas in {\bf Definition} \ref{GDesA} and {\bf Definition} \ref{GDesM} to the above, we obtain that
\begin{align*}
(K^\pr\circ {f}- &{f}\circ K)\big(M_n(\g,\dotsc, \g_n)\big)
\\
&=
\sum_{j_1+\cdots+j_r=n}\Fr{n!}{j_1! \cdots j_r!r!}
M_{r}^\pr\left({E}_{j_1}(\g,\dotsc,\g),
\phi^{f,\underline{\La}}_{j_2}(\g,\dotsc,\g), \dotsc, \phi^{f,\underline{\La}}_{j_{r}}(\g,\dotsc,\g)\right),
\end{align*}
which is equivalent to the claimed identity \eq{zzw}, since $K^\pr{f}- {f} K=0$.

Recall that $\phi^{f,\underline{\La}}_1=f -K^\pr \La_1 -\La_1 K$, which implies that $\phi^{f,\underline{\La}}_1(1_\sC)= 1_{\sC^\pr}$.
It remains to show that $\phi^{f,\underline{\La}}_{n+1}(1_\sC, x_1,\dotsc, x_n)=0$ for all $n\geq 1$ and for any $ x_1,\dotsc, x_n\in \sC$, which is left as an exercise. 
\hfill\qed\end{proof}

\begin{lemma}\label{bnc}
The descendant $\underline{\phi}^{f,\underline{\La}}$ of $f$ up to a homotopy $\underline{\La}=\La_1,\La_2,\La_3,\dotsc$ is $sL_\infty$-homotopic to the descendant $\underline{\phi}^{f,\underline{\tilde\La}}$ of $f$ up to another homotopy $\underline{\tilde\La}=\tilde\La_1,\tilde\La_2,\tilde\La_3,\dotsc$.
\end{lemma}
\begin{proof}
Let $\underline{\Theta(\t)}=\Theta(\t)_1,\Theta(\t)_2,\dotsc$, be a polynomial family in $\t$ of linear maps from $\Theta(\t)_n: S^n(\sC)$ to $\sC^\pr$ of degree $-1$ satisfying $\Theta(\t)_1\big(1_\sC\big)=0$ and 
 $\Theta(\t)_{n+1}\big(x_1,\dotsc,x_{n}, 1_\sC\big)=\Theta(\t)_n\big(x_1,\dotsc,x_{n}\big)$ for all $n\geq 1$, such that
\[
\Theta(0)_n=\La_n,\qquad \Theta(1)_n=\tilde\La_n.
\]
It follows that the descendant $\underline{\phi}^{f,\underline{\Theta(\t)}}$ of $f$ up to the one-parameter family of homotopies $\underline{\Theta(\t)}$ is a one-parameter family of $sL_\infty$-morphisms $\underline{\Phi}(\t)\ceq \underline{\phi}^{f,\underline{\Theta(\t)}}$ such that, for all $n\geq 1$,
\eqn{zxza}{
\left\{
\begin{array}{lr}
\Phi(0)_n &=\phi^{f,\underline{\La}}_n
,\\
\Phi(1)_n &=\phi^{f,\underline{\tilde\La}}_n.
\end{array}\right.
}
By the definition of a descendant morphism, we have
\begin{align*}
f\big(M_n(x_1, \dotsc, x_n)\big)=
\!\!\sum_{\pi \in P(n)}\!\!\ep(\pi)M^\pr_{|\pi|}\left(\Phi\!\big(x_{B_1}\big),\dotsc,
\Phi\!\big(x_{B_{|\pi|}}\big)\right)
+K^\pr \Theta_n(x_1,\dotsc, x_n)
\\
+\!\!\!\!\!\sum_{\substack{\pi \in P(n)\\ |B_i| = n-|\pi|+1}}\!\!\!\!\!\!\!\!
\ep(\pi,i)
\Theta_{|\pi|} \left( x_{B_1},\dotsc, x_{B_{i-1}},\ell^K\left(x_{B_i}\right), x_{B_{i+1}},\dotsc, x_{B_{|\pi|}}\right),
\end{align*}
for all $n\geq 1$ and and $x_1,\dotsc, x_n \in \sC$.
Applying $\Fr{d}{d\t}$ to the above, we obtain
\eqnalign{zzzx}{
\dot{\Phi}_n(x_1,\dotsc,x_n)
=&
-\!\!\sum_{\substack{\pi \!\in P(n)\\ \color{red}|\pi|\neq 1}}
\ep(\pi)\fr{d}{d\t} M^\pr_{|\pi|}\left(\Phi\!\big(x_{B_1}\big),\dotsc,\Phi\!\big(x_{B_{|\pi|}}\big)\right)
\\
&
-K^\pr \dot{\Theta}_n(x_1,\dotsc, x_n)
\\
&
-\!\!\!\!\!\sum_{\substack{\pi \in P(n)\\ |B_i| = n-|\pi|+1}}\!\!\!\!\!\!\!\!
\ep(\pi,i)
\dot{\Theta}_{|\pi|} \left( x_{B_1},\dotsc, x_{B_{i-1}},\ell^K\big(x_{B_i}\big), x_{B_{i+1}},\dotsc, x_{B_{|\pi|}}\right)
.}
Define the family $\underline{\eta}(\t)=\eta(\t)_1,\eta(\t)_2,\dotsc$ such that, for all $n\geq 1$ and $x_1,\dotsc,x_n \in \sC$,
\eqnalign{zxzc}{
\dot{\Theta}_n(&x_1,\dotsc, x_n)
\\
=
&
\!\!\!\!\sum_{\substack{\pi \in P(n)}}\sum_{i=1}^{|\pi|}\ep(\pi,i)
M^\pr_{|\pi|}\!\left({\Phi}\big( x_{B_1}\big), \dotsc, {\Phi}\big(x_{B_{i-1}}\big)
 ,\eta\big(x_{B_i}\big), {\Phi}\big( x_{B_{i+1}}\big),\dotsc, {\Phi}\big( x_{B_{|\pi|}}\big)\right).
}
Equivalently
\begin{align*}
\eta_n(x_1,&\dotsc, x_n)
=
\dot{\Theta}_n(x_1,\dotsc, x_n)
\\
&
-\sum_{\substack{\pi \in P(n)\\\color{red}|\pi|\neq 1}}\sum_{i=1}^{|\pi|}\ep(\pi,i)
M^\pr_{|\pi|}\!\left({\Phi}\big( x_{B_1}\big), \dotsc, {\Phi}\big(x_{B_{i-1}}\big)
 ,\eta\big(x_{B_i}\big), {\Phi}\big( x_{B_{i+1}}\big),\dotsc, {\Phi}\big( x_{B_{|\pi|}}\big)\right).
\end{align*}
It is trivial to check that $\eta(\t)_n$ is a polynomial family in $\t$ of degree $-1$ linear maps $\eta(\t)_n:S^n(\sC)\to\sC^\pr$ satisfying $\eta(\t)_n\big(x_1,\dotsc,x_{n-1}, 1_\sC\big)=0$ for all $n\geq 1$ and for all $\t \in[0,1]$.
\begin{claim}
The identity \eq{zzzx} together with \eq{zxzc} imply that, for all $n\geq 1$ and $x_1,\dotsc,x_n \in \sC$,
\begin{align*}
&\dot\Phi_n \big(x_1,\dotsc, x_n\big)
\\
=
&
-\!\!\!\!\!\sum_{\substack{\pi \in P(n)}}\sum_{i=1}^{|\pi|}\ep(\pi,i)
\ell^{K^\pr}_{|\pi|}\!\left({\Phi}\big( x_{B_1}\big), \dotsc, {\Phi}\big(x_{B_{i-1}}\big)
 ,\eta\big(x_{B_i}\big), {\Phi}\big( x_{B_{i+1}}\big),\dotsc, {\Phi}\big( x_{B_{|\pi|}}\big)\right)
\\
&
-\!\!\!\!\!\!\!\!\sum_{\substack{\pi \in P(n)\\ |B_i| = n-|\pi|+1}}\ep(\pi,i)
\eta_{|\pi|}\left( x_{B_1}, \dotsc, x_{B_{i-1}}, \ell^K(x_{B_i}), x_{B_{i+1}},\dotsc, x_{B_{|\pi|}}\big)\right)
.\\
\end{align*}
\end{claim}
From {\bf Definition \ref{l-hom}} for homotopic $sL_\infty$-morphisms, this claim, along with the boundary condition \eq{zxza} means that the unital $sL_\infty$-morphism $\underline{\phi}^{f,\underline{\La}}$ is homotopic to the $sL_\infty$-morphism $\underline{\phi}^{f,\underline{\tilde\La}}$.

What remains is to check the claim, which is equivalent to the following identity, for all $n\geq 1$:
\eqnalign{zxa}{
\Fr{d}{d\t}\hat\Phi_{n}
=&
-\sum_{j=0}^{n-1}\sum_{k^\pr_1+\cdots+k^\pr_s=j }
\Fr{1}{s!}\ell^{K^\pr}_{s+1}\left(\hat\Phi_{k^\pr_1}, \dotsc, \hat\Phi_{k^\pr_{s}},\hat\eta_{n-j}\right)
\\
&
-\sum_{k=0}^{n-1}\Fr{1}{k!} \eta_{k+1}\left(\g,\dotsc,\g, \hat\ell^K_{n-k}\right).
}
where we have introduced the following shorthand notation:
\[
\hat\Phi_j =\Fr{1}{j!} \Phi_j(\g,\dotsc,\g)
,\qquad \hat\eta_j =\Fr{1}{j!} \eta_j(\g,\dotsc,\g)
,\qquad\hat \ell^K_k =\Fr{1}{k!}\ell^K_k(\g,\dotsc,\g)
.
\]
Using the same notation, we note that the identity \eq{zzzx} is equivalent to the following:
\eqnalign{zxb}{
 % \Fr{d}{d\t}{\hat\Phi}_{n}
% =&-\sum_{i=1}^{n-1} \sum_{k_1+\cdots + k_r=i}\Fr{1}{r!}M^\pr_{r+1}
&\sum_{i=0}^{n-1} \sum_{k_1+\cdots + k_r=i}\Fr{1}{r!}M^\pr_{r+1}
\left(\hat\Phi_{k_1},\dotsc, \hat\Phi_{k_r}, \Fr{d}{d\t}{\hat\Phi}_{n-i}\right)
\\
&
=-\Fr{1}{n!}K^\pr\dot\Theta_n(\g,\dotsc,\g)
-\sum_{k=0}^{n-1}\Fr{1}{k!} \dot\Theta_{k+1}\left(\g,\dotsc,\g, \hat\ell^K_{n-k}	\right).
}
Also the defining formula \eq{zxzc} of $\underline{\eta}$ is equivalent to the following;
\eqn{zxc}{
\frac{1}{n!}\dot\Theta_n(\g,\dotsc,\g)
=\sum_{j=0}^{n-1}\sum_{k^\pr_1+\cdots+k^\pr_s=j }
\Fr{1}{s!}M^\pr_{s+1}\left(\hat\Phi_{k^\pr_1}, \dotsc, \hat\Phi_{k^\pr_{s}},\hat\eta_{n-j}\right).
}

To prove the claim, we will show that the identities \eq{zxb} and \eq{zxc} imply \eq{zxa} by induction.

For $n=1$, \eq{zxb} becomes $\Fr{d}{d\t}\hat\Phi_1 = -K^\pr \dot\Theta_1(\g) - \dot\Theta_1(K\g)$. From \eq{zxc} we know that $\dot\Theta_1=\eta_1$, so we obtain the $n=1$ case of \eq{zxa}. Then we must show that \eq{zxa} is true for a particular value of $n$, assuming that it holds for all lower indices.

Substituting \eq{zxa} into \eq{zxb}, we obtain
\eqnalign{zxe}{
\Fr{d}{d\t}\hat\Phi_{n}&
-\sum_{\substack{1\le i; 0\le j\\i+j\le n-1}}
%{i=0}^{n-1} \sum_{j=0}^{n-i-1}
\sum_{\substack{k_1+\cdots + k_r = i\\ k^\pr_1+\cdots + k^\pr_s = j}}
\Fr{1}{r!s!}M^\pr_{r+1}\left(\hat\Phi_{k_1},\dotsc, \hat\Phi_{k_r}, \ell^K_{s+1}
\left(\hat\Phi_{k^\pr_1}, \dotsc, \hat\Phi_{k^\pr_{s}},\hat\eta_{n-i-j}\right)\right)
\\
&
-\sum_{\substack{1\le i;0\le k\\i+k\le n-1}}
\sum_{i=0}^{n-1} \sum_{j_1+\cdots + j_r=i}
%\sum_{k=1}^{n-i-1}\Fr{1}{k!r!}
M^\pr_{r+1}\left(\hat\Phi_{j_1},\dotsc, \hat\Phi_{j_r}, \eta_{k+1}\left(\g,\dotsc,\g, \hat\ell^K_{n-i-k}\right)\right)
\\
&
=-\Fr{1}{n!}K^\pr\dot\Theta_n(\g,\dotsc,\g)
-\sum_{k=0}^{n-1}\Fr{1}{k! } \dot\Theta_{k+1}\left(\g,\dotsc,\g, \hat\ell^K_{n-k}\right).
}
Consider the terms in the last line of the above.
By \eq{zxc} and Definition~\ref{GDesA}, we obtain that
\eqnalign{zxf}{
\Fr{1}{n!}&K^\pr\dot\Theta_n(\g,\dotsc,\g) 
\\
%=&\sum_{\substack{i=0}}^{n-1}\sum_{j=0}^{n-i-1}
=&\sum_{\substack{0\le i,j\\i+j\le n-1}}
\sum_{\substack{k_1+\cdots + k_r = i\\ k^\pr_1+\cdots + k^\pr_s = j}}
\Fr{1}{r!s!}M^\pr_{r+1}\left(\hat\Phi_{k_1},\dotsc,\hat\Phi_{k_r},
\ell^K_{s+1}\left(\hat\Phi_{k^\pr_1},\dotsc,\Phi_{k^\pr_s}, \hat\eta_{n-i-j}\right)\right)
\\
%&+\sum_{\substack{i=0}}^{n-2}\sum_{j=0}^{n-i-1}
&\sum_{\substack{0\le i,1\le j\\i+j\le n-2}}
\sum_{\substack{k_1+\cdots + k_r = i\\ k^\pr_1+\cdots + k^\pr_s = j}}
\Fr{1}{r!s!}M^\pr_{r+2}\left(\ell^K_s\left(\hat\Phi_{k^\pr_1},\dotsc,\hat\Phi_{k^\pr_s}\right)
,\hat\Phi_{k_1},\dotsc,\hat\Phi_{k_r}, \hat\eta_{n-i-j}\right).
\\
}
Recall that $\underline{\Phi}$ is a one parameter family of ${sL}_\infty$-morphisms from $\sC_{\Lie}=\big(\sC,1_\sC, \underline{\ell}^K\big)$ to $\sC^\pr_{\Lie}=\big(\sC^\pr,1_{\sC^\pr}, \underline{\ell}^{K^\pr}\big)$ so that we have the following identity;
\begin{align*}
\sum_{k^\pr_1+\cdots+k^\pr_s=j}\Fr{1}{s!}\ell^{K^\pr}_{s}\left(
\hat\Phi_{k^\pr_1}, \dotsc, \hat\Phi_{k^\pr_s}\right)
=
\sum_{k=0}^{j-1}\Fr{1}{k!} \Phi_{k+1}\left(\g,\dotsc,\g, \hat\ell^K_{j-k}\right).
\end{align*}
It follows that \eq{zxf} becomes (after reparameterizing the sum)
\eqnalign{zxg}{
\Fr{1}{n!}&K^\pr\dot\Theta_n(\g,\dotsc,\g) 
\\
=&
\sum_{\substack{0\le i,j\\i+j\le n-1}}\sum_{\substack{k_1+\cdots + k_r = i\\ k^\pr_1+\cdots + k^\pr_s = j}}
\Fr{1}{r!s!}M^\pr_{r+1}\left(\hat\Phi_{k_1},\dotsc,\hat\Phi_{k_r},
\ell^K_{s+1}\left(\hat\Phi_{k^\pr_1},\dotsc,\Phi_{k^\pr_s}, \hat\eta_{n-i-j}\right)\right)
\\
&+\!\!\!\!
%\sum_{\substack{i=0}}^{n-2}\sum_{j=0}^{n-i-1}\sum_{k=0}^{j-1}
\sum_{\substack{0\le i,k\\1\le j\\i+j+k\le n-1}}\!\!\!\!
\sum_{\substack{k_1+\cdots + k_r = i}}
\Fr{1}{r!k!}
% \\
% &\qquad\qquad\qquad
% \times 
M^\pr_{r+2}\left(\Phi_{k+1}\left(\g,\dotsc,\g, \hat\ell^K_{j}\right)
,\hat\Phi_{k_1},\dotsc,\hat\Phi_{k_r}, \hat\eta_{n-i-j-k}\right).
\\
}
For the other term involving $\dot\Theta$ in \eq{zxe}, the defining formula \eq{zxzc} of $\underline{\eta}$ implies that
\eqnalign{zxh}{
&\Fr{1}{m!} \dot\Theta_{m+1}\left(\g,\dotsc,\g, \hat\ell^K_{j}\right)
\\
=
&
-
\sum_{\substack{0\le i,k\\i+k\le m-1}}
% \sum_{j=0}^{k-2}\sum_{i=0}^{k-j-2}
\!\!\sum_{k_1+\cdots+k_r=i }\!\!
\Fr{1}{r!k!}M^\pr_{r+2}\left(\Phi_{k+1}(\g,\dotsc,\g,\hat\ell^K_{j}),\hat\Phi_{k_1}, \dotsc, \hat\Phi_{k_{r}},\hat\eta_{m-i-j}\right)
\\
&
+\sum_{i=0}^{m}\sum_{k_1+\cdots+k_s=i }
\Fr{1}{s!m!}M^\pr_{s+1}\left(\hat\Phi_{k_1}, \dotsc, \hat\Phi_{k_{s}},\eta_{m-i+1}(\g,\dotsc,\g,\hat\ell^K_{j})\right)
.
}
Adding together \eq{zxg} and \eq{zxh} after rearranging the summations, we obtain that
\begin{align*}
\Fr{1}{n!}&K^\pr\dot\Theta_n(\g,\dotsc,\g) +\sum_{k=0}^{n-1}\Fr{1}{k!} \dot\Theta_{k+1}\left(\g,\dotsc,\g, \hat\ell^K_{n-k}\right)
\\
=
% &\sum_{\substack{i=0}}^{n-1}\sum_{j=0}^{n-i-1}\sum_{\substack{k_1+\cdots + k_r = i\\ k^\pr_1+\cdots + k^\pr_s = j}}
% \Fr{1}{r!s!}M^\pr_{r+1}\left(\hat\Phi_{k_1},\dotsc,\hat\Phi_{k_r},
% \ell^K_{s+1}\left(\hat\Phi_{k^\pr_1},\dotsc,\Phi_{k^\pr_s}, \hat\eta_{n-i-j}\right)\right)
% \\
% &
% +\sum_{k=0}^{n-1}\sum_{i=0}^{k}\sum_{k_1+\cdots+k_s=i }
% \Fr{1}{s!k!}M^\pr_{s+1}\left(\hat\Phi_{k_1}, \dotsc, \hat\Phi_{k_{s}},\eta_{k-i+1}(\g,\dotsc,\g,\hat\ell^K_{n-k})\right)
% \\
% =
% &\sum_{\substack{i=0}}^{n-1}\sum_{j=0}^{n-i-1}
&\sum_{\substack{0\le i,j\\i+j\le n-1}}
\sum_{\substack{k_1+\cdots + k_r = i\\ k^\pr_1+\cdots + k^\pr_s = j}}
\Fr{1}{r!s!}M^\pr_{r+1}\left(\hat\Phi_{k_1},\dotsc,\hat\Phi_{k_r},
\ell^K_{s+1}\left(\hat\Phi_{k^\pr_1},\dotsc,\Phi_{k^\pr_s}, \hat\eta_{n-i-j}\right)\right)
\\
&
% +\sum_{i=0}^{n-1}
% \sum_{k=1}^{n-i} 
+\sum_{\substack{0\le i,k\\i+k\le n-1}}
\sum_{j_1+\cdots + j_r=i}\Fr{1}{r!k!}
M^\pr_{r+1}\left(\hat\Phi_{j_1},\dotsc, \hat\Phi_{j_r}, \eta_{k+1}\left(\g,\dotsc,\g, \hat\ell^K_{n-i-k}\right)\right).
\end{align*}
The two sums on the right hand side are almost identical to the two sums on the left hand side of \eq{zxe}. The only difference is that in one case $i=0$ is allowed and in the other case it is not. Thus, subtracting, we obtain:
\begin{align*}
 \Fr{d}{d\t}{\hat\Phi}_{n}
=
&-\sum_{j=0}^{n-1}\sum_{ k^\pr_1+\cdots + k^\pr_s = j}
\Fr{1}{s!}M^\pr_{1}\left(
\ell^K_{s+1}\left(\hat\Phi_{k^\pr_1},\dotsc,\Phi_{k^\pr_s}, \hat\eta_{n-j}\right)\right)
\\
&
-
\sum_{k=0}^{n-1} \Fr{1}{k!}
M^\pr_{1}\left( \eta_{k+1}\left(\g,\dotsc,\g, \hat\ell^K_{n-k}\right)\right).
\end{align*}
Using the definition that $M_1^\pr$ is the identity map, we conclude that \eq{zxa} is true.
\hfill\qed\end{proof}

\begin{lemma}
Let $f$ and $\tilde f$ be morphisms from $\sC_{\Comm}$ to $\sC^\pr_{\Comm}$ in the same homotopy type. Then the descendant $\underline{\phi}^{f,\underline{\La}}$ of $f$ up to the homotopy $\underline{\La}$ and the the descendant $\underline{\phi}^{\tilde f,\underline{\tilde \La}}$ of $\tilde f$ up to the homotopy $\underline{\La}$ are homotopic as unital $sL_\infty$-morphisms.
\end{lemma}
\begin{proof}
Let $\varsigma$ be a pointed cochain homotopy between $f$ and $\tilde{f}$. Write $\underline{\check\La}=\varsigma\circ M_1,\varsigma\circ M_2,\varsigma \circ M_3,\dotsc$ Then Lemma~\ref{bnc} says that $\underline{\phi}^{\tilde f,\underline{0}}\sim_\infty \underline{\phi}^{\tilde f,\underline{\tilde\La}}$ and $\underline{\phi}^{f,\underline{\check\La}}\sim_\infty \underline{\phi}^{f,\underline{\La}}$. But using the equations of Definition~\ref{GDesM} for the definition of $\underline{\phi}^{f,\underline{\check\La}}$, and simplifying using the equations of Definition~\ref{GDesA}, we find that the defining equations for $\underline{\phi}^{\tilde f,\underline{0}}$ and $\underline{\phi}^{f,\underline{\check\La}}$ are identical.
\hfill\qed\end{proof}

Now we are going to check the functoriality of our construction.
\begin{lemma}
The descendant of the identity morphism $I$ of a \padj{} probability algebra up to the homotopy $\underline{\La}$ is unital $sL_\infty$-homotopic to the identity $\underline{I}=I,0,0,\dotsc$
\end{lemma}
\begin{proof}
By {\bf Lemma \ref{bnc}} it suffices to consider the case $\underline{\La}=\underline{0}$. But the descendant of the identity up to the zero homotopy is equal to the $\underline{I}$ by an easy induction.
\end{proof}

\begin{lemma}
Fix the following diagram in the category $\category{\HProb}_{\Comm}{\overk}$;
\[
\xymatrix{\sC_{\Comm}\ar[r]^f&\sC^\pr_{\Comm}\ar[r]^{f^\pr}&\sC^\ppr_{\Comm}}.
\]
Let $\underline{\phi}^{f,\underline{\La}}$ be the descendant of $f$ up to the homotopy $\underline{\La}$ and let $\underline{\phi}^{f^\pr,\underline{\La}^\pr}$ be the descendant of $f^\pr$ up to the homotopy $\underline{\La}^\pr$.
Then for any family $\underline{\La}^\ppr=\La_1^\ppr,\La_2^\ppr,\dotsc$ where $\La^\ppr_n$ is a linear map $\La^\ppr_n: S^n(\sC)\rightarrow \sC^\ppr$ of degree $-1$ satisfying $\La^\ppr_1\big( 1_\sC\big)=0$ and $\La^\ppr_{n+1}\big(x_1,\dotsc,x_{n}, 1_\sC\big)=\La^\ppr_{n}\big(x_1,\dotsc,x_{n}\big)$ for all $n\geq 1$,
we have
\[
\underline{\phi}^{f^\pr,\underline{\La}^\pr}\bullet
\underline{\phi}^{f,\underline{\La}}
\sim_\infty\underline{\phi}^{f^\pr\circ f,\underline{\La}^\ppr}.
\]
\end{lemma}

\begin{proof}
By {\bf Lemma \ref{bnc}} it suffices to check the Lemma for $\underline{\La}=\underline{\La}^\pr=\underline{\La}^{\ppr}=\underline{0}$. In this case the equation is true on the nose by another easy induction.
\hfill\qed\end{proof}

\subsection{Spaces of homotopical random variables and correlation functions}\label{subs:homotopical random variables-theory}

This subsection contains three fundamental lemmas that will be used to define spaces of homotopical random variables and their laws as invariants of the homotopy types of certain $sL_\infty$-morphisms.

Consider a \padj{} probability space $\xymatrix{\sC_{\Comm} \ar[r]^\mc &\fieldk}$. 
Let $\sC_{\Lie}=\big(\sC, \underline{\ell}^K\big)$ be the descendant of $\sC_{\Comm}=\big(\sC, \underline{M}, K\big)$ and let $\underline{\phi}^{\mc,\underline{\La}}$ be the descendant of $\mc$ up to the homotopy $\underline{\La}$.
Consider a graded vector space $V$ regarded as an $sL_\infty$-algebra $\big(V, \underline{0}\big)$ with zero $sL_\infty$-structure. Also consider $S(V)=V\oplus S^2V\oplus S^3V\oplus\cdots$ as a a cochain complex $\big(S(V),0\big)$ with zero differential. 
%Denoted by $\imath_n$, for all $n\geq 1$, the canonical injection $\jmath_n:S^n V \rightarrow S(V)$.

In the first lemma, given an $sL_\infty$-morphism $\underline{\w}:\big(V, \underline{0}\big)
\rightarrow \big(\sC, \underline{\ell}^K\big)$, we construct a cochain map $\Pi^{\underline{\w}}:\big(S(V), 0)\rightarrow \big(\sC, K\big)$ such that homotopic $sL_\infty$-morphisms produce homotopic cochain maps. The second lemma is converse to the first one. In both constructions, we shall only use the product structure $\underline{M}$ on $\sC$. The third lemma states, firstly, that both the compositions $\underline{\k}=\underline{\phi}^{\mc,\underline{\La}}\bullet \underline{\w}$ and $\m=\mc\circ \Pi^{\underline{\w}}$ depend only on the homotopy types of the constituent morphisms so that both commutative diagrams
\[
\xymatrix{
\big(V,\underline{0}\big)\ar@{..>}[r]^{\underline{\w}}
\ar@{..>}@/_1.5pc/[rr]_{\underline{\k}}
&\big(\sC,\underline{\ell}^K\big)\ar@{..>}[r]^{\underline{\phi}^{\mc, \La}}
&\big(\fieldk,\underline{0}\big)
}
\quad{\color{blue}\Big|}\quad
\xymatrix{
\big(S(V),0\big)\ar@{->}[r]^{\underline{\Pi}^{\underline{\w}}}
\ar@{->}@/_1.5pc/[rr]_{\underline{\m}}
&\big(\sC,K\big)\ar@{->}[r]^{\mc}
&\big(\fieldk,0\big)
}
\]
induce commutative diagrams in the appropriate homotopy categories, and (ii) the two morphisms $\underline{\k}$ and $\underline{\m}$ are related by 
the moment/cumulant formula that, for every $n\geq 1$ and $v_1,v_2,\dotsc, v_n \in V$,
we have
\eqn{morcum}{
\m_n(v_1,\dotsc, v_n)=
\sum_{\pi \in P(n)}\ep(\pi)\k\big(v_{B_1}\big)\cdots \k\big(v_{B_{|\pi|}}\big),
}

Then we define a space $\sV$ of (homotopical) random variables as a graded vector space $|\sV|=V$, regarded as an $sL_\infty$-algebra with zero $sL_\infty$-structure, together with the homotopy type $[\sV]$ of an $sL_\infty$-morphism into $\big(\sC,\underline{\ell}^K\big)$. Such a space comes with canonically defined moment and cumulant morphisms $\underline{\m}^\sV$ and $\underline{\k}^\sV$. This also leads to a natural definition of the law $\hat\m: S(V^*)\rightarrow \fieldk$ for $\sV$.

\begin{lemma}\label{cobra}
Let $V$ be a graded vector space regarded as a trivial $sL_\infty$-algebra $\big(V, \underline{0}\big)$. 
For any $sL_\infty$-morphism $\underline{\w}=\w_1,\w_2,\dotsc$ from $\big(V, \underline{0}\big)$ to $\big(\sC, \underline{\ell}^K\big)$, associate the family $\underline{\Pi}^{\underline{\w}} =\Pi^{\underline{\w}} _1, \Pi^{\underline{\w}} _2,
\dotsc$ of degree zero operators $\Pi^{\underline{\w}} _n: S^n V\to \sC$ defined by 
\[
\Pi^{\underline{\w}} _n(v_1,\dotsc, v_n)=
\sum_{\pi \in P(n)}\ep(\pi)M_{|\pi|}\left(\w\big(v_{B_1}\big),\dotsc ,\w\big(v_{B_{|\pi|}}\big)\right),
\]
for every $n\geq 1$ and $v_1,v_2,\dotsc, v_n \in V$.
Then, for all $n\geq 1$,
\begin{enumerate}
\item the family satisfies $K\Pi^{\underline{\w}} _n=0$ and
\item whenever $\underline{\tilde\w}$ is homotopic to $\underline{\w}$ as an $sL_\infty$-morphism, then $\Pi^{\underline{\tilde\w}} _n$ is chain homotopic to $\Pi^{\underline{\w}} _n$.
\end{enumerate}
\end{lemma}

\begin{proof}
It is obvious that $\Pi^{\underline{\w}} _n$ is a linear map from $S^n V$ to $\sC$ of degree $0$.
By definition, we also have
\begin{align*}
K\Pi^{\underline{\w}} _n(v_1,\dotsc, v_n)=&
\sum_{\pi \in P(n)}\ep(\pi)KM_{|\pi|}\left(\w\big(v_{B_1}\big),\dotsc ,\w\big(v_{B_{|\pi|}}\big)\right).
\end{align*}
Applying the formula in {\bf Definition} \ref{GDesA} to the above, we obtain that
\eqnalign{gaf}{
K&\Pi^{\underline{\w}} _n(v_1,\dotsc, v_n)
\\
&=
\sum_{\pi \in P(n)}\sum_{i=1}^{|\pi|}\ep(\pi,i)M_{|\pi|}\left(\w\big(v_{B_1}\big),\dotsc \w\big(v_{B_{i-1}}\big) ,
E\big(v_{B_{i}}\big),\w\big(v_{B_{i+1}}\big),\dotsc,\w\big(v_{B_{|\pi|}}\big)\right),
}
where $E(v_B)= E_r(v_{j_1},\dotsc, v_{j_r})$ if $B=\{j_1,j_2,\dotsc, j_r\}$ and
\[
E_r(y_1,\dotsc,y_r)\ceq \sum_{\substack{\pi \in P(r)}}\ep(\pi)
\ell^K_{|\pi|}\left(\w\left( y_{B_1}\right), \dotsc, \w\big( y_{B_{|\pi|}}\big)\right).
\]
Recall that any $sL_\infty$-morphism $\underline{\w}$ from $(V,\underline{0})$ to $(\sC, \underline{\ell}^K)$ satisfies $E_r(y_1,\dotsc,y_r)=0$ for all $r\geq 1$ and $y_1,\dotsc, y_r \in V$.
Hence \eq{gaf} implies that $K\Pi^{\underline{\w}} _n(v_1,\dotsc, v_n)=0$.
This proves the first property.

Now assume that $\underline{\w}$ and $\underline{\tilde\w}$ are $L_\infty$-homotopic. That is, let $\underline{\Phi}=\Phi_1,\Phi_2,\dotsc$ be a smooth polynomial family parametrized by $\t\in [0,1]$, of degree $0$ maps $\Phi_n:S^n V\to V^\pr[\t]$ satisfying the homotopy flow equations generated by the $sL_\infty$-homotopy $\underline{\eta}$;
\begin{align*}
\dot\Phi_n&\big(v_1,\dotsc, v_n\big)
\\
=
&
-\sum_{\substack{\pi \in P(n)}}\sum_{i=1}^{|\pi|}\ep(\pi,i)
\ell^K_{|\pi|}\left({\Phi}\left( v_{B_1}\right), \dotsc, {\Phi}\big(v_{B_{i-1}}\big)
 ,\eta\big(v_{B_i}\big), {\Phi}\big( v_{B_{i+1}}\big),\dotsc, {\Phi}\big( v_{B_{|\pi|}}\big)\right)
\end{align*}
with the initial and the final conditions that $\underline{\Phi}\big|_{\t=0}= \underline{\w}$ and $\underline{\Phi}\big|_{\t=1}=\underline{\tilde\w}$. Recall that the family $\underline{\Phi}$ is a family of $sL_\infty$-morphisms from the $sL_\infty$-algebra $(V,\underline{0})$ with zero $sL_\infty$-structure to the $sL_\infty$-algebra $\big(\sC, \underline{\ell}^K\big)$. It follows that, for all $n\geq 1$ and every $v_1,\dotsc, v_n \in V$,
\eqn{gfrb}{
\sum_{{\pi \in P(n)}}\ep(\pi)
\ell_{|\pi|}^{K^\pr} \left(\Phi\left( v_{B_1}\right), \dotsc, \Phi\big( v_{B_{|\pi|}}\big)\right)
=0.
}
Define
\[
\Pi^{\underline{\Phi}} _n(v_1,\dotsc, v_n)=
\sum_{\pi \in P(n)}\ep(\pi)M_{|\pi|}\left(\Phi\big(v_{B_1}\big),\dotsc ,\Phi\big(v_{B_{|\pi|}}\big)\right),
\]
which is a one parameter family $\Pi^{\underline{\Phi}} _n$ of linear maps from $S^n V$ to $\sC$ which satisfy $\Pi^{\underline{\Phi}} _n\big|_{\t=0}= \Pi^{\underline{\w}} _n$ and $\Pi^{\underline{\Phi}} _n\big|_{\t=1}= \Pi^{\underline{\tilde\w}} _n$.
We claim that there exists, for all $n\geq 1$, a one parameter family $\Xi_n$ of linear maps from $S^n V$ to $\sC$ which satisfy 
\[
\dot \Pi^{\underline{\Phi}} _n
=K \Xi_n 
\]
which implies that
\[
\Pi^{\underline{\tilde\w}} _n -\Pi^{\underline{\w}} _n = 
\int^1_0 \dot\Pi^{\underline{\Phi}} _n \;d\t = K \int^1_0 \Xi_n\; d\t.
%K\chi_n,\quad \chi_n\ceq \int^1_0 \Xi_n \;d\t.
\]
This proves the second part of the lemma, since $V$ has no differential.

It remains to prove the claim. As before, set $\g=\sum_{i=1}^n t_i v_i$ and write: 
\[
\hat\Pi^{\underline{\Phi}} _n=\Fr{1}{n!}\Pi^{\underline{\Phi}} _n(\g,\dotsc, \g)
,\quad
\hat\Phi_{j}=\Fr{1}{j!}\Phi_{j}(\g,\dotsc,\g)
,\quad
\hat\l_{j}=\Fr{1}{j!}\l_{j}(\g,\dotsc,\g)
.
\]
We then have
\eqnalign{gfra}{
\hat\Pi^{\underline{\Phi}} _n=
&\sum_{j_1+\cdots + j_r=n}\Fr{1}{r!}M_{r}\left(\hat\Phi_{j_1},\dotsc ,\hat\Phi_{j_r}\right)
,\\
\Fr{d}{d\t}\hat\Phi_{\ell}
=&
-\sum_{j=0}^{\ell-1}\sum_{k^\pr_1+\cdots+k^\pr_s=j }
\Fr{1}{s!}\ell^K_{s+1}\left(\hat\Phi_{k^\pr_1}, \dotsc, \hat\Phi_{k^\pr_{s}},\hat\l_{\ell-j}\right)
,
}
It will then suffice to show that $\Fr{d}{d\t}\hat\Pi^{\underline{\Phi}} _n$ is $K$-exact to prove the claim.
By applying the derivative $\Fr{d}{d\t}$ to the first relation in \eq{gfra}, we obtain that
\begin{align*}
\Fr{d}{d\t}\hat\Pi^{\underline{\Phi}} _n
&
=\sum_{i=0}^{n-1} \sum_{k_1+\cdots + k_r=i}
\Fr{1}{r!}M_{r+1}\left(\hat\Phi_{k_1},\dotsc, \hat\Phi_{k_r}, \Fr{d}{d\t}{\hat\Phi}_{n-i}\right)
\\
&=
-\sum_{\substack{i=0}}^{n-1}\sum_{j=0}^{n-i-1}\sum_{\substack{k_1+\cdots + k_r = i\\ k^\pr_1+\cdots + k^\pr_s = j}}
\Fr{1}{r!s!}M_{r+1}\left(\hat\Phi_{k^\pr_1},\dotsc,\hat\Phi_{k^\pr_r},
\ell^K_{s+1}\left(\hat\Phi_{k^\pr_1},\dotsc,\hat\Phi_{k^\pr_s}, \hat\l_{n-i-j}\right)\right),
\end{align*}
where we have used the second relation in \eq{gfra} for the second equality.
Now we consider the following quantity:
\[
\hat\Xi_n
=\sum_{s=0}^{n-1}\sum_{k_1+\cdots + k_r = s}\Fr{1}{r!}M_{r+1}\left(\hat\Phi_{k_1},\dotsc,\hat\Phi_{k_r}, \hat\l_{n-s}\right).
\]
By a direct computation, we have
\begin{align*}
K\hat\Xi_n =
&\sum_{\substack{i=0}}^{n-1}\sum_{j=0}^{n-i-1}\sum_{\substack{k_1+\cdots + k_r = i\\ k^\pr_1+\cdots + k^\pr_s = j}}
\Fr{1}{r!s!}M_{r+1}\left(\hat\Phi_{k^\pr_1},\dotsc,\hat\Phi_{k^\pr_r},
\ell^K_{s+1}\left(\hat\Phi_{k^\pr_1},\dotsc,\Phi_{k^\pr_s}, \hat\l_{n-i-j}\right)\right)
\\
&+\sum_{\substack{i=0}}^{n-2}\sum_{j=0}^{n-i-1}\sum_{\substack{k_1+\cdots + k_r = i\\ k^\pr_1+\cdots + k^\pr_s = j}}
\Fr{1}{r!s!}M_{r+2}\left(\ell^K_s\left(\hat\Phi_{k^\pr_1},\dotsc,\hat\Phi_{k^\pr_s}\right)
,\hat\Phi_{k_1},\dotsc,\hat\Phi_{k_r}, \hat\l_{s}\right)
\\
=
&\sum_{\substack{i=0}}^{n-1}\sum_{j=0}^{n-i-1}\sum_{\substack{k_1+\cdots + k_r = i\\ k^\pr_1+\cdots + k^\pr_s = j}}
\Fr{1}{r!s!}M_{r+1}\left(\hat\Phi_{k^\pr_1},\dotsc,\hat\Phi_{k^\pr_r},
\ell^K_{s+1}\left(\hat\Phi_{k^\pr_1},\dotsc,\hat\Phi_{k^\pr_s}, \hat\l_{n-i-j}\right)\right)
,\\
\end{align*}
where we have used the identity \eq{gfrb}, which implies that, for all $j=1,2,3,\dotsc$,
\[
\sum_{\substack{k^\pr_1+\cdots + k^\pr_s = j}}
\Fr{1}{s!}\ell^K_s\left(\hat\Phi_{k^\pr_1},\dotsc,\hat\Phi_{k^\pr_s}\right)=0.
\]
Hence we have confirmed that $\Fr{d}{d\t}\hat\Pi^{\underline{\Phi}} _n=K \hat \Xi_n$.
\hfill\qed\end{proof}

\begin{remark}
The above lemma is the natural generalization of {\bf Proposition \ref{cora}}.
\hfill$\natural$
\end{remark}

Assume that we have a family $\underline{\Pi}=\Pi_1,\Pi_2,\dotsc$, where $\Pi _n$ is a linear map from $S^n V$ to $\sC$ of degree $0$ such that $K\Pi _n=0$. Then the family $\underline{\Pi}$ can be naturally regarded as a cochain map $\Pi$ from the cochain complex $\big(S(V), 0\big)$ to $\big(\sC,K\big)$ 

We therefore interpret the above proposition as the construction of a cochain map $\Pi^{\underline{\w}}:(S(V), 0)\rightarrow (\sC, K)$ from the $sL_\infty$-morphism $\underline{\w}:(V,\underline{0})\rightarrow \big(\sC,\underline{\ell}^K\big)$ which preserves homotopy types. The following lemma reverses this construction.

\begin{lemma}\label{cobrax}
For a cochain map $\Pi: \big(S(V),0)\rightarrow \big(\sC, K\big)$,
let $\underline{\w}^{\Pi} =\w^{\Pi} _1, \w^{\Pi} _2, \dotsc$ be a family such that, for every $n\geq 1$ and $v_1,v_2,\dotsc, v_n \in V$,
\[
\Pi _n(v_1,\dotsc, v_n)=
\sum_{\substack{\pi \in P(n)}}\ep(\pi)M_{|\pi|}\left(\w^{\Pi}\big(v_{B_1}\big),\dotsc ,\w^{\Pi}\big(v_{B_{|\pi|}}\big)\right).
\]
Then
\begin{enumerate}
\item $\underline{\w}^{\Pi}$ is an $sL_\infty$-morphism from $\big(V,\underline{0}\big)$ to $\big(\sC,\underline{\ell}^K\big)$ and
\item $\underline{\w}^{\tilde\Pi}$ is $sL_\infty$-homotopic to $\underline{\w}^{\Pi}$ if $\tilde\Pi$ is cochain homotopic to $\Pi$.
\end{enumerate}
\end{lemma}

\begin{proof}
Let \[
E^\Pi_r(y_1,\dotsc,y_r)\ceq \sum_{\substack{\pi \in P(r)}}\ep(\pi)
\ell^K_{|\pi|}\left(\w^\Pi\left( y_{B_1}\right), \dotsc, \w^\Pi\big( y_{B_{|\pi|}}\big)\right)
\]
and define $E^\Pi(v_B)= E^\Pi_r(v_{j_1},\dotsc, v_{j_r})$ if $B=\{j_1,j_2,\dotsc, j_r\}$.
Define a family of operators $\underline{L}=L_1, L_2,\dotsc$ by
\begin{align*}
L_n(&v_1,\dotsc,v_n)
\\
\ceq 
&\sum_{\pi \in P(n)}\sum_{i=1}^{|\pi|}\ep(\pi,i)M_{|\pi|}\!\!\left(\w^\Pi\big(v_{B_1}\big),\dotsc \w^\Pi\big(v_{B_{i-1}}\big) ,
E^\Pi\big(v_{B_{i}}\big),\w^\Pi\big(v_{B_{i+1}}\big),\dotsc,\w^\Pi\big(v_{B_{|\pi|}}\big)\right)
\\
=
&E^\Pi_n(v_1,\dotsc, v_n)
\\
&+\!\!\!\!\sum_{\substack{\pi \in P(n)\\|\pi\neq 1}}\sum_{i=1}^{|\pi|}\ep(\pi,i)M_{|\pi|}\!\!\left(\w^\Pi\big(v_{B_1}\big),\dotsc \w^\Pi\big(v_{B_{i-1}}\big) ,
E^\Pi\big(v_{B_{i}}\big),\w^\Pi\big(v_{B_{i+1}}\big),\dotsc,\w^\Pi\big(v_{B_{|\pi|}}\big)\right)
.
\end{align*}
Then, by definition, we have $K\Pi_n= L_n$. 
From the condition that $K\Pi_n =0$, we have $L_n=0$ for all $n\geq 1$.
Note that $L_1=E_1^\Pi$ so that $E^\Pi_1=0$. If $E^\Pi_1=\dotsc=E^\Pi_{n-1}=0$, then $0=L_n =E^\Pi_n$. Thus by induction we have $E^\Pi_n=0$ for all $n\geq 1$, which means that $\underline{\phi}^\Pi$ is an $sL_\infty$-morphism.

Now consider a chain map $\tilde\Pi$ homotopic to $\Pi$. Since $S(V)$ has zero differential, this means that $\tilde\Pi = \Pi + K \chi$, where $\chi$ is a linear map from $S(V)$ to $\sC$ of degree $-1$. %Let $\underline{\chi}=\chi_1,\chi_2,\dotsc$, where $\chi_n$ is the component of $\chi$ that goes from $S^n V$ to $\sC$.
Set
\[
\S(\t) \ceq \Pi + \t K\chi \Longrightarrow %\left\{\begin{array}{l}\S(0)=\Pi\\ \S(1)=\tilde\Pi\end{array}\right.
\] 
so that $\S(0) = \Pi$ and $\S(1)=\tilde \Pi$, and define a polynomial family $\underline{\Phi}(\t)$ by the following recursive relation:
\eqn{borri}{
\S _n(v_1,\dotsc, v_n)=
\sum_{\substack{\pi \in P(n)}}\ep(\pi)M_{|\pi|}\left(\Phi\big(v_{B_1}\big),\dotsc ,\Phi\big(v_{B_{|\pi|}}\big)\right).
}
Then it is obvious that $\underline{\Phi}(0)=\underline{\phi}^\Pi$ and $\underline{\Phi}(1)=\underline{\phi}^{\tilde\Pi}$. To conclude, we will show that $\underline{\Phi}(\t)$ also satisfies the homotopy flow equation.

By applying $\fr{d}{d\t}$ to \eq{borri}, we have
\eqnalign{boria}{
\dot\Phi_n(v_1,\dotsc, v_n) 
= 
&
K \chi_n(v_1,\dotsc, v_n)
\\
&
-\sum_{\substack{\pi \in P(n)\\|\pi|\neq 1}}\sum_{i=1}^{|\pi|}
\ep(\pi)M_{|\pi|}\left(\Phi\big(v_{B_1}\big),\dotsc ,\dot\Phi\big(v_{B_i}\big),\dotsc,\Phi\big(v_{B_{|\pi|}}\big)\right).
}
%where we have used $\dot{\S}=\Fr{d\S}{d\t}=K\chi$.
Define a polynomial family $\underline{\eta}(\t)=\eta(\t)_1,\eta(\t)_2,\dotsc$, parametrized by $\t$, by the recursive equation
%such that, for all $n\geq 1$ and $v_1,\dotsc,v_n \in \sC$,
\eqnalign{biux}{
\chi_n&(x_1,\dotsc, x_n)
\\
&
=
\!\!\!\!\sum_{\substack{\pi \in P(n)}}\sum_{i=1}^{|\pi|}\ep(\pi,i)
M_{|\pi|}\!\left({\Phi}\big( x_{B_1}\big), \dotsc, {\Phi}\big(x_{B_{i-1}}\big)
 ,\eta\big(x_{B_i}\big), {\Phi}\big( x_{B_{i+1}}\big),\dotsc, {\Phi}\big( x_{B_{|\pi|}}\big)\right).
}
For example 
\begin{align*}
\eta_1 &=\chi_1
,\\
\eta_2(v_1,v_2)&=\chi_2(v_1,v_2) -M_2\big(\eta_1(v_1), \Phi_1(v_2)\big) 
-(-1)^{|v_1|}M_2\big(\Phi_1(v_1), \eta_1(v_2)\big).
\end{align*}
Note that $\dot{\Phi}_1 = K \eta_1$, so that $\Phi_1$ satisfies the homotopy flow equation generated by $\eta_1$.
We claim that the family $\Phi_1,\Phi_2,\dotsc,\Phi_{n}$ satisfy the homotopy flow equation generated by $\eta_1,\eta_2,\dotsc,\eta_{n}$ if $\Phi_1,\Phi_2,\dotsc,\Phi_{n-1}$ satisfy the homotopy flow equation generated by $\eta_1,\eta_2,\dotsc,\eta_{n-1}$, $n \geq 2$.
(This claim can be checked by computing $K\chi_n$ using \eq{biux} and by substituting the result into \eq{boria}.)
By induction it follows that $\underline{\Phi}$ satisfies the homotopy flow equation generated by $\underline{\eta}(\t)$. 
\hfill\qed\end{proof}

\begin{lemma}\label{lemb}
For any $sL_\infty$-morphism $\w:(V,\underline{0})\rightarrow (\sC,\underline{\ell}^K)$,
define the families $\underline{\m}=\m_1, \m_2,\dotsc$ and $\underline{\k}=\k_1, \k_2,\dotsc$ of degree zero linear maps $\m_n,\k_n:S^n V\to \fieldk$, by
\begin{align*}
\m_n&\ceq \mc\circ \Pi^{\underline{\w}} _n
,\\
{\k}_n& \ceq \left(\underline{\phi}^{\mc,\underline{\La}}\bullet \underline{\w}\right)_n.
\end{align*}
Then, for all $n\geq 1$,
\begin{enumerate}
\item both $\m_n$ and $\k_n$ depend only on the homotopy types of the expectation morphism $\mc$ and the $sL_\infty$-morphism $\underline{\w}$ and
\item for $v_1,v_2,\dotsc, v_n \in V$
\[
\m_n(v_1,\dotsc, v_n)=
\sum_{\pi \in P(n)}\ep(\pi)\k\big(v_{B_1}\big)\cdots \k\big(v_{B_{|\pi|}}\big).
\]
\end{enumerate}
\end{lemma}

\begin{proof}
Consider $\m_n=\mc\circ \Pi^{\underline{\w}} _n$. Let $\tilde\mc= \mc + r\circ K$, where $r$ is a pointed cochain homotopy. 
Let $\underline{\w}$ be an $sL_\infty$-morphism that is homotopic to $\underline{\w}$.
Then, from {\bf Proposition \ref{cobra}}, we have, for all $n\geq 1$,
$\tilde\mc\circ \Pi^{\underline{\tilde\w}} _n = \tilde\mc\circ \Pi^{\underline{\w}} _n
= \mc\circ \Pi^{\underline{\tilde\w}} _n= \mc\circ \Pi^{\underline{\w}} _n$, where we have used $\mc\circ K=K\circ \Pi^{\underline{\w}} _n=K^2=0$.
Hence the family $\underline{\m}$ depends only on the homotopy types of the expectation morphism $\mc$ and the $sL_\infty$-morphism $\underline{\w}$. 

It is obvious that the family $\underline{\k}= \underline{\phi}^{\mc,\underline{\La}}\bullet \underline{\w}$,
is an $sL_\infty$-morphism from $(V,\underline{0})$ to $(\fieldk,\underline{0})$. Note that there $\underline{\k}$ is alone in its $sL_\infty$-homotopy type since it is a morphism between trivial $sL_\infty$-algebras. It follows that $\underline{\k}$ depends only on the homotopy types of the $sL_\infty$-morphisms $\underline{\phi}^{\mc,\underline{\La}}$ and $\underline{\w}$. Recall that the homotopy type of $\underline{\phi}^{\mc,\underline{\La}}$ depends only on the cochain homotopy type of $\mc$ and is independent of $\underline{\La}$. 
Hence the family $\underline{\k}^V$ depends only on the homotopy types of the expectation morphism $\mc$ and the $sL_\infty$-morphism $\underline{\w}$.

It remains to show the second property. By definition, we have
\begin{align*}
\m_n(v_1,\dotsc, v_n)
&= \mc\left( \Pi^{\underline{\w}} _n(v_1,\dotsc, v_n)\right)
\\
&=\sum_{\pi \in P(n)}\ep(\pi)\mc\left(M_{|\pi|}\left(\w\big(v_{B_1}\big),\dotsc ,\w\big(v_{B_{|\pi|}}\big)\right)\right)
\\
&=\sum_{\pi \in P(n)}\ep(\pi)\mc\left(M_{|\pi|}\left(y_1,\dotsc ,y_{|\pi|}\right)\right)
,
\end{align*}
where we have set $y_i=\w\big(v_{B_i}\big)$ for each $\pi=B_1\sqcup\cdots\sqcup B_{|\pi|} \in P(n)$ for the third equality.
Then we have
\begin{align*}
\m_n(v_1,\dotsc, v_n)
&=\sum_{\pi \in P(n)}\ep(\pi)\sum_{\pi^\pr \in P(|\pi|)}\e(\pi^\pr)\phi^\mc\big(y_{B^\pr_1}\big)\cdots \phi^\mc\big(y_{B^\pr_{|\pi^\pr|}}\big)
\\
&=\sum_{\pi \in P(n)}\ep(\pi)
\left( \underline{\phi}^\mc\bullet \underline{\w}\right)_{|B_1|}\big(v_{B_1}\big)\cdots \left( \underline{\phi}^\mc\bullet \underline{\w}\right)_{|B_{|\pi|}|}\big(v_{B_{|\pi|}}\big)
\\
&=\sum_{\pi \in P(n)}\ep(\pi)
\k\big(v_{B_1}\big)\cdots \k\big(v_{B_{|\pi|}}\big)
.
\end{align*}
\hfill\qed\end{proof}

\begin{definition}
A space $\sV$ of homotopical random variables is a finite dimensional graded vector space $|\sV|=V$ regarded as an $sL_\infty$-algebra $(V,\underline{0})$ equipped with the homotopy type $[\sV]$ of an $sL_\infty$-morphism $\underline{\w}^V:(V,\underline{0})\rightarrow (\sC,\underline{\ell}^K)$.
The moment and cumulants morphism $\underline{\m}^\sV$ and $\underline{\k}^\sV$ on $\sV$ are $\underline{\m}^\sV\ceq \mc\circ\underline{\Pi}^{\underline{\w}^V}$ and $\underline{\k}^\sV\ceq \underline{\phi}^{\mc,\underline{\La}}\bullet\underline{\w}^{\!V}$ and satisfy, for all $n\geq 1$ and $v_1,\dotsc, v_n \in V$,
\[ 
 \m^\sV_n(v_1,\dotsc, v_n)=\!\sum_{\pi \in P(n)}\!\ep(\pi)\k^{\sV}\!\big(v_{B_1}\big)\cdots \k^{\sV}\!\big(v_{B_{|\pi|}}\big).
\]
\end{definition}

\begin{remark}\label{remark:unique and intrinsic}
Recall that both $\m^\sV_n = \mc\circ \Pi^{\underline{\w}^V} _n$ and $\k^\sS_n = \underline{\phi}^{\mc,\underline{\La}}\bullet \underline{\w}^V$ depend only on 
 the cochain homotopy type of $\mc$, which determine the $sL_\infty$-homotopy type of 
 $\underline{\phi}^{\mc,\underline{\La}}$, and the $sL_\infty$-homotopy type $[\sV]$ of $\underline{\w}^{V}$.
Hence, both families $\underline{\m}^\sV$ and $\underline{\k}^\sV$ are unique and intrinsic structures of the space $\sV$. On the other hand, there are infinitely many different ways of getting the same result. This phenomenon could be viewed as (i) a wide-ranging computational flexibility, which can be exploited by choosing the simplest or easiest representatives of the relevant homotopy types, or (ii) a guarantee of the independence of the construction of any choice of suitable regularization scheme, where such a scheme can be viewed as a kind of deformation of a bad choice of representatives. 
\hfill$\natural$
\end{remark}
In what follows, we shall demonstrate the power of this philosophy by considering a scenario which is extreme but comes up in applications.

\subsection{Homotopically completely integrable \padj{} probability spaces}\label{subs:completely integrable}

\begin{definition}
We say a space $\sV$ of homotopical random variables is \emph{homotopically completely integrable} if there is a representative $sL_\infty$-morphism $\underline{\w}$ in its homotopy type $[\sV]$
 such that $\Im {\w}_n \in \fieldk \cdot 1_\sC$ for all $n$. A \padj{} probability space is called homotopically completely integrable if every space of homotopical random variables is so.
\end{definition}
%Let $\underline{\w}: \big(V,\underline{0}\big) \rightarrow \big(\sC, \underline{\ell}^K\big)$ be a representative $sL_\infty$-morphism of the homotopy type $[\sV]$,so that $\underline{\k}^\sV = \underline{\phi}^\mc\bullet \underline{\w}$.
Consider a homotopically completely integrable space of homotopical random variables $\sV$. Then there is a representative $sL_\infty$-morphism $\underline{\w}$ of the homotopy type $[\sV]$ and a family of functions $\underline{\vk}=\vk_1,\vk_2,\ldots$, where $\vk_n:S^n V\to \fieldk$ such that, for all $n\geq 1$ and $v_1,\dotsc, v_n\in V$,
\[
\w _n(v_1,\dotsc, v_n) = \vk_n(v_1,\dotsc, v_n)\cdot 1_\sC
%\quad\hbox{{\em where}}\quad \vk_n(v_1,\dotsc, v_n) \in \fieldk.
\]
Recall that $\underline{\phi}^\mc\bullet \underline{\w}$ and
%=\underline{\phi}^\mc\bullet \underline{\tilde\w}$ and
\begin{align*}
\left(\underline{\phi}^\mc\bullet \underline{\w}\right)_n (v_1,\dotsc, v_n)
= 
&\mc\left({\w}_n (v_1,\dotsc, v_n)\right)
+\sum_{k=2}^n\sum_{\substack{\pi \in P(n)\\ |\pi|=k}}\ep(\pi)\phi^\mc_k\left(\w(v_{B_1}),\dotsc, \w(v_{B_k})\right).
\end{align*}
Furthermore, for $n\ge 2$, $\phi^\mc_n$ annihilates $1_\sC$.
It follows that 
\begin{align*}
\k^\sV_n(v_1,\dotsc, v_n)\equiv\left(\underline{\phi}^\mc\bullet \underline{\w}\right)_n (v_1,\dotsc, v_n)
=\mc\left( \vk_n(v_1,\dotsc, v_n)\cdot 1_\sC\right)
= \vk_n(v_1,\dotsc, v_n).
\end{align*}
Hence, the law of family of random variables can be determined completely by purely algebraic/homotopy theoretic methods. 
Recall that we have already encountered such phenomena in {\em Example \ref{gaussian-ex3}}.

\begin{definition}
A \padj{} probability space is \emph{non-degenerated} if every random variable of degree $0$ with zero expectation belongs to the image of its differential.
\end{definition}

Let $\xymatrix{\sC_{\Comm}\ar[r]^{\mc}&\fieldk}$ be a non-degenerated \padj{} probability space. Recall that the unit $1_\sC$ has non-trivial cohomology class, denoted $1_H$.
For any $y\in \sC^0$, $x\ceq y -\mc(y)\cdot 1_\sC \in \Ker \mc$. By assumption, there exists $\l \in \sC^{-1}$
 such that $y-\mc(y)\cdot 1_\sC= K\l$. it follows that $Ky=0$ for any $y\in \sC^0$. It also
 follows that $y$ and $\mc(y) \cdot 1_\sC$ belong to the same cohomology class for all $y\in \sC^0$.
Hence, $\sC^0 \subset \Ker K$ and $H^0 =\fieldk\cdot 1_H$.

\begin{theorem} 
Every non-degenerated \padj{} probability space is homotopically completely integrable.
\end{theorem}

\begin{proof}
We begin by considering the relevant homotopy flow equation:
\begin{align*}
\dot\Phi_n&\big(v_1,\dotsc, v_n\big)
\\
=
&
\sum_{\substack{\pi \in P(n)}}\sum_{i=1}^{|\pi|}\ep(\pi,i)
\ell^K_{|\pi|}\left({\Phi}\left( v_{B_1}\right), \dotsc, {\Phi}\big(v_{B_{i-1}}\big)
 ,\eta\big(v_{B_i}\big), {\Phi}\big( v_{B_{i+1}}\big),\dotsc, {\Phi}\big( v_{B_{|\pi|}}\big)\right).
\end{align*}
For $n\ge 2$, define a polynomial family $L(\t)$ in $\t$ of degree zero linear maps from $S^{n}V$ to $\sC$ by the equation
\begin{align*}
L_{n}\big(&v_1,\dots,v_{n}\big)
\\
&
\ceq \sum_{\substack{\pi \in P(n)\\\color{red}|\pi|\neq 1}}\sum_{i=1}^{|\pi|}\ep(\pi,i)
\ell^K_{|\pi|}\left({\Phi}\left( v_{B_1}\right), \dotsc, {\Phi}\big(v_{B_{i-1}}\big)
 ,\eta\big(v_{B_i}\big), {\Phi}\big( v_{B_{i+1}}\big),\dotsc, {\Phi}\big( v_{B_{|\pi|}}\big)\right).
 \end{align*}
It is obvious that $L(\t)_{n}$ depends only on $\Phi(\t)_1,\dotsc, \Phi(t)_{n-1}$ and $\eta(\t)_1,\dotsc, \eta(\t)_{n-1}$.
The $n$-th flow equation, for $n\geq 2$, can be written as: 
\[
\dot\Phi(\t)_n =K \eta(\t)_n + L(\t)_{n}.
\] 

Pick any space $\sV$ of homotopical random variables.
Given an arbitrary representative $sL_\infty$-morphism $\underline{\w}:(V, \underline{0})\rightarrow \big(\sC, \underline{\ell}^K\big)$ in the homotopy type $[\sV]$, our goal is to construct another representative $sL_\infty$-morphism $\underline{\tilde\w}$ such that ${\tilde\w}_n= {\k}_n\cdot 1_\sC$, where $\k_n: S^n V\rightarrow \fieldk$, for all $n\geq 1$.

By assumption we have $\Im \varphi_n \subset \sC^0\subset \ker K$, for all $n\geq 1$, so that the degree zero linear map $\vk_n\ceq \mc\circ \varphi_n: S^n V\rightarrow \fieldk$ depends only the homotopy type of $\mc$. It also follows that for $n\ge 1$ there exists a linear map $\varsigma_n:S^n V\rightarrow \sC$ of degree $-1$ such that
\[
\w_n = \vk_n\cdot 1_\sC - K \varsigma_n.
\]
We shall construct $\underline{\tilde\w}$ and $\underline{\k}$ by induction.

Set $\eta_1 =\varsigma_1$, $\k_1=\vk_1$, and $\Phi(\t)_1 \ceq \w_1 + \t K\eta_1$.
Then, we have $\Phi(0)_1 =\w_1$, $\tilde{\w}_1\ceq \Phi(1)_1=\k_1\cdot 1_\sC$, and
\[
\dot\Phi_1 = K \eta_1.
\]
Fix $n\geq 2$ and assume that we have a sequence $\underline{\Phi(\t)}^{(n)}=\Phi(\t)_1,\cdots,\Phi(\t)_n$ of polynomial families, parametrized by $\t$, and another sequence $\underline{\eta(\t)}^{(n)}=\eta(\t)_1,\cdots, \eta(\t)_n$ of polynomial families such that, for all $1 \leq j\leq n$,
we have $\Phi_j(0)=\w_j$, $\tilde\w_j\ceq \Phi_j(1)= \k_j\cdot 1_A$, where $\Im \k_j \in \fieldk$, and 
\[
\dot\Phi(\t)_j = K\eta(\t)_j + L(\t)_j.
\]
Consider $L(\t)_{n+1}$, which is a degree $0$ polynomial family in $\t$ of linear maps from $S^{n+1}V$ to $\sC$ determined by $\underline{\Phi(\t)}^{(n)}$ and $\underline{\eta(\t)}^{(n)}$.
Then there is a polynomial family of linear maps $\upsilon(\t)_{n+1}$ from $S^{n+1}V$ to $\fieldk$ of degree $0$ and another polynomial family of linear maps $\xi(\t)_{n+1}$ from $S^{n+1}V$ to $\sC$ of degree $-1$ such that
\[
L(\t)_{n+1} = \upsilon(\t)_{n+1}\cdot 1_\sC - K \xi(\t)_{n+1}.
\] 
Now we would like to find two polynomial families, $\Phi(\t)_{n+1}$ and $\eta(t)_{n+1}$ such that $\Phi(0)_{n+1}=\w_{n+1}$, $\Phi(1)_{n+1}=\k_{n+1}\cdot 1_\sC$ for some linear map $\k_{n+1}:S^{n+1}V\rightarrow \fieldk$,
and
\[
\dot\Phi(\t)_{n+1} = K\eta(\t)_{n+1} + L(\t)_{n+1}. 
\]
Set $\eta(\t)_{n+1} =\xi(\t)_{n+1} + \varsigma_{n+1}$; then we have
\[
\dot\Phi(\t)_{n+1} = \upsilon(\t)_{n+1}\cdot 1_\sC +K\varsigma_{n+1}.
\]
It follows that
\begin{align*}
\Phi(\t)_{n+1} 
&=\w_{n+1}+ \t K\varsigma_{n+1} + \int^\t_0 \upsilon(\s)_{n+1}d\s\cdot 1_\sC 
\\
&=\left( \vk_n + \int^\t_0 \upsilon(\s)_{n+1}d\s\right)\cdot 1_\sC +(\t-1)K\varsigma_{n+1}.
\end{align*}
Then $\tilde \phi_{n+1}\ceq \Phi(1)_{n+1} = \k_{n+1}\cdot 1_\sC$,
where $\k_{n+1}=
\left( \vk_{n+1} + \int^\t_0 \upsilon(\s)_{n+1}d\s\right)\cdot 1_\sC$.
\hfill\qed\end{proof}

\subsection{Formality of descendant \texorpdfstring{$sL_\infty$}{sL-infinity}-algebras}\label{subs:formality}

Recall that a \padj{} probability space $\xymatrix{\sC_{\Comm} \ar[r]^c &\fieldk}$, where $\sC_{\Comm} =(\sC, \underline{M}, K)$,
contains as part of its data the pointed cochain complex $(\sC, K)$. Let $H$ be the cohomology, which is a pointed graded vector space of possibly infinite dimension, since $1_\sC \in \Ker K$ and $1_H\ceq [1_\sC] \neq 0$ since $c(1_\sC)=1$. 
Note that $H$ is also the cohomology of the unital $sL_\infty$-algebra $\sC_{\Lie}$. Recall too that there is the structure of a minimal $sL_\infty$-algebra on the cohomology of any $sL_\infty$-algebra.

\begin{theorem}\label{lemc}
On the cohomology $H$ of the descendant $sL_\infty$-algebra $\sC_{\Lie}$ of a \padj{} probability space, there is the structure of a unital $sL_\infty$-algebra $H_{\Lie}=\big(H, \underline{0}\big)$ with zero $sL_\infty$-structure $\underline{0}$ that is quasi-isomorphic to $\sC_{\Lie}$.
\end{theorem}
\begin{proof}
We will prove the theorem by exhibiting a unital $sL_\infty$-quasi-isomorphism $\underline{\w}$ from $H_{\Lie}=\big(H, \underline{0}\big)$ to $\sC_{\Lie}$. That is, we will construct a family $\underline{\w}=\w_1,\w_2,\dotsc$, where $\w_k$ is a linear map from $S^k(H)$ to $\sC$ of degree $0$ for all $k\geq 2$ such that $\w_1$ is a cochain quasi-isomorphism from $(H, 0)$ to $(\sC, K)$ and, for all $n\geq 2$ and $a_1,\dotsc, a_n$ of $H$, $\underline{\w}$ satisfies
\begin{align*}
\sum_{\pi \in P(n)} 
\ep(\pi)
{\ell}^{K}_{|\pi|}\Big(
{\w}\big(a_{B_1}\big), 
\dotsc,{\w}\big( a_{B_{|\pi|}}\big)
\Big)
&=0,
\qquad \phi_{n}\big(a_1,\dotsc, a_{n-1}, 1_H\big)=0.
\end{align*}
We shall first construct $\w_1$ and then proceed by induction.

Define a linear map $\w_1:H\rightarrow \sC$ of degree $0$ as follows. First fix a basis $\{e_\a\}_{\a\in J}$ of $H$ with a distinguished element $e_0=1_H$. Then arbitrarily choose representatives $\w_1(e_\a)$ of each element in the basis, choosing them so that $\w_1(1_H)= 1_\sC$. Finally extend $\w_1$ linearly to all of $H$. Then $\w_1$ is a pointed cochain map from $(H, 0)$ to $(\sC, K)$ which induces the identity map on $H$;
that is
\[
K\w_1=0,\qquad \w_1(1_H)=1_\sC,
\]
and the cohomology class of $f(a)$ is $a$ for all $a\in H$.
Hence $\w_1$ is is a cochain quasi-isomorphism from $(H, 1_H, 0)$ to $(\sC,1_\sC, K)$.
Let $h:\sC\rightarrow H$ be a homotopy inverse to $\w_1$;
\begin{align*}
h\circ \w_1 &=I_H
,\\
\w_1\circ h &= I_\sC + K\b+\b K,
\end{align*}
where $\b:\sC\rightarrow \sC$ is a linear map of degree $-1$.
It follows that 
\[
hK=0,\qquad h(1_\sC)=1_H.
\]
From $h(1_\sC)=1_H$ and $\w_1(1_H)=1_\sC$ we can choose $\b$ so that $\b(1_\sC)=0$.

% Define a bi-linear map $L_2: S^2 H\rightarrow \sC$ of degree $1$ such that, for elements $a_1,a_2 \in H$,
% \[
% L_2\big(a_1,a_2\big)=\ell_2^K\big(\w_1(a_1),\w_1(a_2)\big).
% \]
% Then $\Im L_2 \subset \Ker K\cap \sC$ since $K$ is a derivation of $\ell_2^K$ and $K\w_1=0$.
% Define another bi-linear map $\n_2: S^2 H\rightarrow H$ of degree $1$ by
% \[
% \n_2\ceq h\circ L_2.
% \]
% On the other hand, the following identity
% \begin{align*}
% L_2(a_1, a_2) \equiv&\ell_2^K\big(\w_1(a_1),\w_1(a_2)\big)
% \\
% =&K M_2\big(\w_1(a_1),\w_1(a_2)\big)
% -M_2\left(K\w_1(a_1),\w_1(a_2)\right)-(-1)^{|a_1|}M_2\left(\w_1(a_1), K\w_1(a_2)\right)
% \\
% =&K M_2\big(\w_1(a_1),\w_1(a_2)\big),
% \end{align*}
% implies that $h\circ L_2=hK\circ M_2=0$. It follows that 
% \[
% \n_2=0.
% \]
% From $h\circ L_2=0$, we also have $f\circ h\circ L_2\equiv L_2+ K \w_2=0$, where we have defined 
% \[
% \w_2\ceq \b\circ L_2, 
% \]
% which is a bi-linear map of degree $0$ from $S^2 H$ to $\sC$. Note that $\w_2\big(a,1_H\big)=0$ for all $a \in H$ since $L_2(a, 1_H)=0$.
% Hence we have, for all $a_1,a_2 \in H$,
% \eqn{teb}{
% K \w_2\big(a_1,a_2\big)+\ell_2^K\big(\w_1(a_1),\w_1(a_2)\big)=0,\qquad \w_2\big(a_1, 1_H)=0.
% }

%Fix $n \ge1$ and assume that $\n_2,\n_3,\dotsc,\n_{n}=0$. Also 
Fix $n\ge 1$ and assume that there is a family $\mb{\w}^{[n]}=\w_1,\w_2, \dotsc, \w_n$,
where $\w_k:S^k(H)\rightarrow \sC$ is a linear map of degree $0$ for $1\leq k\leq n$ satisfying $\w_{k}\big(a_1,\dotsc, a_{k-1},1_H\big)=0$ for $2\leq k\leq n$,
such that, for $1\le k\le n$ and $v_1,\ldots, v_k\in H$,
\eqnalign{tec}{
\sum_{\pi \in P(k)}
\ep(\pi)
{\ell}^{K}_{|\pi|}\Big(
\mb{\w}^{[n]}\big(a_{B_1}\big), \dotsc,\mb{\w}^{[n]}\big( a_{B_{|\pi|}}\big)
\Big)
&=0.
}
Define a linear map $L_{n+1}: S^{n+1} H\rightarrow \sC$ of degree $1$ by
\begin{align*}
L_{n+1}\big(&a_1,\dotsc, a_{n+1}\big)
\ceq 
\sum_{\substack{\pi \in P(n+1)\\|\pi|\neq 1}}
\ep(\pi)
{\ell}^{K}_{|\pi|}\Big(
\mb{\w}^{[n]}\big(a_{B_1}\big), \dotsc,\mb{\w}^{[n]}\big( a_{B_{|\pi|}}\big)
\Big).
\end{align*}
% Define a linear map $\n_{n+1}:S^{n+1} H\rightarrow H$ of degree $1$ such that
% \[
% \n_{n+1}\ceq h\circ L_{n+1}.
% \]
From the assumption in \eq{tec}, together with the definition of the descendant $sL_\infty$-structure $\underline{\ell}^K$, it is direct that
\[
L_{n+1}\big(a_1,\dotsc, a_{n+1}\big)=
K\sum_{\substack{\pi \in P(n+1)\\|\pi|\neq 1}}
\ep(\pi)M_{|\pi|}\left(
\mb{\w}^{[n]}\big(a_{B_1}\big),\dotsc,\mb{\w}^{[n]}\big( a_{B_{|\pi|}}\big)\right).
\]
The above identity implies that $KL_{n+1}=0$ and that $h\circ L_{n+1}=0$. Define a degree $0$ linear map $\w_{n+1}:S^{n+1} H\to \sC$ by
\[
\w_{n+1}\ceq \b\circ L_{n+1}.
\]
% 
% \[
% \n_{n+1}=0.
% \] 
% It follows that $\n_2,\n_3,\dotsc,\n_{n+1}=0$.
% From $h\circ L_{n+1}=0$ we also have 
Then $0 = \w_1\circ h\circ L_{n+1} -\beta\circ KL_{n+1}= L_{n+1}+ K \w_{n+1}$. Hence, for $a_1,\dotsc, a_{n+1}\in H$, we have
\eqn{tecca}{
 L_{n+1}\big( a_1,\dotsc, a_{n+1}\big)
+K\w_{n+1}\big( a_1,\dotsc, a_{n+1}\big)=0.
}
Set $\mb{\w}^{[n+1]}= \w_1,\w_2,\dotsc,\w_n, \w_{n+1}$ so that the identity \eq{tecca} becomes the desired formula:
\[
\sum_{\substack{\pi \in P(n+1)}}
\ep(\pi)
{\ell}^{K}_{|\pi|}\Big(
\mb{\w}^{[n+1]}\big(a_{B_1}\big),
\dotsc,\mb{\w}^{[n+1]}\big( a_{B_{|\pi|}}\big)
\Big)=0.
\]
This completes the induction. 

For unitality, the definition of $L_{n+1}$ along with the inductive assumptions related to the unit on $\w_k$ for $1\le k\le n$ and the unital conditions on the descendant $sL_\infty$-structure $\underline{\ell}^K$ imply that $L_{n+1}\big(a_1,\dotsc, a_{n}, 1_H\big)=0$ so that $\w_{n+1}(a_1,\dotsc, a_{n}, 1_H)=0$. 
\hfill\qed\end{proof}

\begin{remark}
It is important to note that the $sL_\infty$-quasi-isomorphism $\underline{\w}$ from $H_{\Lie}$ to $\sC_{\Lie}$ is not unique.
For example we may replace $\w_1$ and $\b$, in the above proof, with $\tilde\w_1$ and $\tilde\b$ such that $\tilde\w_1$ is cochain homotopic to $\w_1$ and $\tilde\b:\sC\rightarrow \sC$ is another linear map of degree $-1$ such that $\tilde\b(1_\sC)=0$ satisfying the appropriate conditions. Then the resulting $sL_\infty$-quasi-isomorphism, say, $\underline{\tilde\w}$ can be shown to be unital $sL_\infty$-homotopic to $\underline{\w}$. On the other hand, we may, for example, replace $\w_2\ceq \beta\circ L_2$ with $\w_2^\pr = \w_2 + \chi_2$, where $\chi_2$ is a bi-linear map from $S^2H$ to $\Ker K \subset \sC$ of degree $0$ such that $\chi_2(v, 1_H)=0$ for all $v\in H$ and $\chi_2(v_1,v_2) \notin \Im K$ for some fixed pair $v_1,v_2\in H$. Then the construction in the above proof still works but the resulting $sL_\infty$-quasi-isomorphism is by no means homotopic to $\underline{\w}$.
\hfill$\natural$
\end{remark}

\subsection{Complete spaces of homotopy random variables}\label{subs:complete space}
Fix a \padj{} probability space $\xymatrix{\sC_{\Comm} \ar[r]^\mc &\fieldk}$. 
Let $\sC_{\Lie}$ be the descendant of $\sC_{\Comm}$ and $\underline{\phi}^{\mc,\underline{\La}}$ be the descendant of $\mc$ up to the homotopy $\underline{\La}$. From {\bf Theorem \ref{lemc}}, the cohomology $H$ of $\sC_{\Lie}$ has the unique structure $H_{\Lie}=\big(H, 1_H,\underline{0}\big)$ of unital $sL_\infty$-algebra which is unital $sL_\infty$-quasi-isomorphic to $\sC_{\Lie}$. On the other hand, the $sL_\infty$-quasi-isomorphism $\underline{\w}$ from $H_{\Lie}$ to $\sC_{\Lie}$ is not unique even up to homotopy. To sidestep this ambiguity, we introduce the notion of complete sample spaces.
\begin{definition}
A \emph{complete space} $\comprandvars$ of homotopical random variables is a unital $sL_\infty$-algebra $\sS_{\Lie}=(\sS, 1_\sS,\underline{0})$ with zero $sL_\infty$-structure together with the homotopy type of a unital $sL_\infty$-quasi-isomorphism into $\sC_{\Lie}$. This homotopy type is denoted $[\comprandvars]$,
\end{definition}
The underlying graded vector space $\sS$ of $\comprandvars$ is isomorphic to the cohomology $H$, but the space $\comprandvars$ has additional structures which depend only on the attached homotopy type $[\comprandvars]$ of unital $sL_\infty$-quasi-isomorphisms to $\sC_{\Lie}$. 
\begin{definition}
The moment morphism and cumulant morphism on $\sS$ are the families of maps $S^n \sS\to\fieldk$ denoted $\underline{\m}^\sS={\m}^\comprandvars_1,{\m}^\comprandvars_2,\dotsc$ and $\underline{\k}^\comprandvars={\k}^\comprandvars_1,{\k}^\comprandvars_2,\dotsc$, respectively. They are defined for $n\geq 1$ as
\begin{align*}
\m^\comprandvars_n &\ceq \mc\circ \Pi^{\underline{\w}^\sS} _n :S^n(\sS)\rightarrow \fieldk,
\\
\k^\comprandvars_n &\ceq \left( \underline{\phi}^{\mc,\underline{\La}}\bullet \underline{\w}^\sS\right)_n:S^n(\sS)\rightarrow \fieldk,
\end{align*}
where $\underline{\w}^\sS:\xymatrix{\sS_{\Lie} \ar@{..>}[r]&\sC_{\Lie}}$ is an arbitrary representative of the homotopy type $[\comprandvars]$.
\end{definition}

\begin{remark}
As in Remark~\ref{remark:unique and intrinsic}, both families $\underline{\m}^\comprandvars$ and $\underline{\k}^\comprandvars$ are unique and intrinsic structures of the space $\comprandvars$.
\hfill$\natural$
\end{remark}

Now the following two theorems shall complete a circle.
\begin{theorem}
On a complete space $\comprandvars$ of homotopical random variables, there is the canonical structure $\sS_{\Comm}=\big(\sS, \underline{M}^{\comprandvars}, 0\big)$ of a \padj{} probability algebra with zero differential that is quasi-isomorphic to the \padj{} probability algebra $\sC_{\Comm}$ such that
\begin{enumerate}
\item $\sS_{\Lie}$ is the descendant algebra of $\sS_{\Comm}$, and
\item
any $sL_\infty$-quasi-morphism of the homotopy type $[\comprandvars]$ from $\sS_{\Lie}$ to $\sC_{\Lie}$ is a descendant morphism.
\end{enumerate}
\end{theorem}

\begin{proof}
Pick an arbitrary representative 
 $\underline{\w}^\sS:\xymatrix{\sS_{\Lie} \ar@{..>}[r]&\sC_{\Lie}}$ of the homotopy type $[\comprandvars]$ and set $f=\w_1^\sS$, which is a cochain quasi-isomorphism from $(\sS, 0)$ to $(\sC, K)$ such that $f(1_\sS)= 1_\sC$.
Let $h:\sC\rightarrow \sS$ be a homotopy inverse to $f$:
\eqnalign{chw}{
h\circ f &=I_\sS
,\\
f\circ h &= I_\sC + K\b+\b K,
}
where $\b:\sC\rightarrow \sC$ is a linear map of degree $-1$.
It follows that 
\eqn{chwa}{
hK=0,\qquad h(1_\sC)=1_\sS.
}
Since $h(1_\sC)=1_\sS$ and $f(1_\sS)=1_\sC$, without loss of generality we can choose $\b$ so that $\b(1_\sC)=0$.
We will construct a family $\underline{\Pi}^{\underline{\w}^\sS}$ of \emph{correlation densities} where $\Pi^{\underline{\w}^\sS}_n$ is a degree $0$ linear map $S^n\sS\to\sC$. These densities will be defined in terms of $\underline{\w}^\sS$ and the family $\underline{M}$ of multiplications in $\sC_{\Comm}$ by the following formula:
\eqn{chwb}{
\Pi^{\underline{\w}^{\sS}} _n(s_1,\dotsc, s_n)=
\sum_{\pi \in P(n)}\ep(\pi)M_{|\pi|}\left(\w^\sS\big(s_{B_1}\big),\dotsc, \w^\sS\big(s_{B_{|\pi|}}\big)\right).
}
\begin{property}\label{propertyPi}
The family $\underline{\Pi}^{\underline{\w}^{\sS}}=\Pi^{\underline{\w}^{\sS}}_1,\Pi^{\underline{\w}^{\sS}}_2,\dotsc$ has the following properties;
\begin{enumerate}[label=$(\alph*)$]
% \item\label{Pi item Sn}
% $\Pi^{\underline{\w}^{\sS}} _n$ is a linear map of degree zero from $S^n (\sS)$ to $\sC$ for all $n\geq 1$.
\item\label{Pi item Pi1}
$\Pi^{\underline{\w}^{\sS}} _1 = \w_1^\sS=f$.
\item\label{Pi item unit}
$\Pi^{\sS} _{n+1}(s_1,\dotsc, s_n, 1_\sS)=\Pi^{\underline{\w}^{\sS}} _{n}(s_1,\dotsc, s_n)$ for all $n\geq 1$ and $s_1,\dotsc, s_n \in \sS$.
\item \label{Pi item K}$K \Pi^{\underline{\w}^{\sS}} _n=0$ and the chain homotopy type of $\Pi^{\underline{\w}^{\sS}}_n$ depends only on the homotopy type $[\sS]$ for all $n\geq 1$.
\end{enumerate}
\end{property}
Note that the first two properties are direct consequences of \eq{chwb}, while the last property is from {\bf Lemma \ref{cobra}}.

Define a family $\underline{M}^\comprandvars=M^\comprandvars_1, M^\comprandvars_2,\dotsc$ of degree zero operators $S^n(\sS)\to \sS$ by
\eqn{chwc}{
M^\comprandvars_n\ceq h\circ \Pi^{\underline{\w}^{\sS}} _n.
}
From Property~\ref{propertyPi}\ref{Pi item Pi1}, we have $M^\comprandvars_1 = h\circ \w_1^\sS = h\circ f = I_\sS$. From Property~\ref{propertyPi}\ref{Pi item unit}, we also have $M^{\comprandvars} _{n+1}(s_1,\dotsc, s_n, 1_\sS)=h\circ \Pi^{\underline{\w}^\sS} _{n+1}(s_1,\dotsc, s_n, 1_\sS)
=h\circ \Pi^{\underline{\w}^{\sS}} _{n}(s_1,\dotsc, s_n)
=M^{\comprandvars} _{n}(s_1,\dotsc, s_n)$. Finally Property~\ref{propertyPi}\ref{Pi item K} implies that $M_n^\sS$ is a canonical structure on $\sS$.
Hence, we conclude the following:
\begin{property}
On $\sS$, there is a canonically defined family $\underline{M}^\comprandvars=M^\comprandvars_1, M^\comprandvars_2,\dotsc$ such that
\begin{itemize}
\item $M^\comprandvars_1= I_\sS$ and $M^\comprandvars_n$ is a linear map from $S^n(\sS)$ to $\sS$ of degree $0$ for all $n\geq 2$.
\item $M^{\comprandvars} _{n+1}(s_1,\dotsc, s_n, 1_\sS)=M^{\comprandvars} _{n}(s_1,\dotsc, s_n)$ for all $n\geq 1$ and $s_1,\dotsc, s_n \in \sS$.
\end{itemize}
\end{property}

% Now we define another family $\underline{m}^\comprandvars=m_2^\comprandvars, m_3^\comprandvars,\dotsc$ such that $m_2^\comprandvars\ceq M_2^\comprandvars$ and, for $n\geq 3$ and $x_1,\dotsc, x_n\in \sC$,
% \[
% m_n^\comprandvars(x_1,\dotsc, x_n)
% \ceq 
% M_n^\comprandvars(v_1,\dotsc, v_n) 
% -\sum_{\substack{\pi \in P(n)\\|B_{|\pi|}|=n-|\pi|+1\\ n-1\sim_\pi n\\|\pi|\neq 1}} 
% \ep(\pi) M^\comprandvars_{|\pi|}\left( v_{B_1}, \dotsc, v_{B_{\pi-1}},
% m^\comprandvars\big(v_{B_{|\pi|}}\big)\right).
% \]
% Then $\big(\sS, \underline{m}^{\comprandvars}\big)$ is a structure of GCCorA with the iterated multiplications $\underline{M}^\comprandvars=M^\comprandvars_1, M^\comprandvars_2,\dotsc$. Such the structure is uniquely defined on the space $\comprandvars$ by definition.
We have given $\sS$ the structure of a \padj{} probability algebra with zero differential: $\sS_{\Comm}=\big(\sS, \underline{M}^{\comprandvars}, 0\big)$. Also observe that $f:\sS\rightarrow \sC$ is a quasi-isomorphism of \padj{} probability algebras:
\[
\xymatrix{\sS_{\Comm}\ar[r]^f &\sC_{\Comm}}.
\]
%which is a quasi-isomorphism.

It is clear that $\sS_{\Lie}=\big(\sS, \underline{0}\big)$ is the descendant of 
 $\sS_{\Comm}=\big(\sS, \underline{M}^{\comprandvars}, 0\big)$.
To conclude, we shall prove that any unital $sL_\infty$-morphism $\underline{\w}^\sS$ with the homotopy type $[\comprandvars]$ from $\sS_{\Lie}$ to $\sC_{\Lie}$ is a descendant of $f$.
To begin with we define a family of degree $-1$ linear operators for $n\geq 1$,
\eqn{chwe}{
\La_n \ceq \b\circ \Pi^{\underline{\w}^{\sS}} _n:S^n(\sS)\to \sC.
}
It is trivial to check the following;

\begin{property}\label{property Lambda}
The family $\underline{\La}=\La_1, \La_2,\dotsc$ has the following properties:
\begin{itemize}
\item $\La_1(1_\sS)=0$ and 
\item $\La_{n+1}(s_1,\dotsc, s_n, 1_\sS)=\La_{n}(s_1,\dotsc, s_n)$ for all $n\ge 1$ and $s_1,\dotsc, s_n \in \sS$.
\end{itemize}
\end{property}

Applying $f$ to \eq{chwc}, we have
\begin{align*}
f\circ M^\comprandvars_n
&=f\circ h\circ \Pi^{\underline{\w}^{\sS}} _n
\\
&=\Pi^{\underline{\w}^{\sS}} _n + K\left(\b\circ \Pi^{\underline{\w}^{\sS}} _n\right) + \b\left(K \Pi^{\underline{\w}^{\sS}} _n\right)
\\
&=\Pi^{\underline{\w}^{\sS}} _n + K\left(\b\circ \Pi^{\underline{\w}^{\sS}} _n\right),
\end{align*}
where we have used \eq{chw} for the second equality and Property~\ref{propertyPi}\ref{Pi item K} for the last equality.
Using \eq{chwe}, we obtain that, for all $n\geq 1$,
\eqn{chwd}{
f\circ M^\comprandvars_n=\Pi^{\underline{\w}^{\sS}} _n + K\La _n.
}
It follows that, for every $n\geq 1$ and $s_1,\dotsc, s_n \in \sS$,
we have
\eqnalign{chwee}{
f\left(M^\comprandvars_n(s_1,\dotsc, s_n)\right)=&
\sum_{\pi \in P(n)}\ep(\pi)M_{|\pi|}\left(\w^\sS\big(s_{B_1}\big),\dotsc, \w^\sS\big(s_{B_{|\pi|}}\big)\right)
\\
&
+K\La_n(s_1,\dotsc, s_n).
}
From {\em Property~\ref{property Lambda}}, we conclude that $\underline{\w}^\sS$ is the descendant of $f$ up to the homotopy $\underline{\La}=\La_1,\La_2,\dotsc$
\hfill\qed\end{proof}

Let's summarize what we have so far before stating the second theorem.
For a given \padj{} probability space $\xymatrix{\sC_{\Comm}\ar[r]^{\mc}&\fieldk}$, we consider a complete space $\comprandvars$ of homotopical random variables, which has the canonical structure $\sS_{\Comm}=\big(\sS, \underline{M}^{\comprandvars}, 0\big)$ of a \padj{} probability algebra together with a quasi-isomorphism $\xymatrix{\sS_{\Comm}\ar[r]^f &\sC_{\Comm}}$ of \padj{} probability algebras.
% \[
%  \xymatrix{\sS_{\Comm}\ar[r]^f &\sC_{\Comm}\ar[r]^{\mc}&\fieldk}.
% \]

Define a linear map $\iota^\comprandvars: \sS\rightarrow \fieldk$ by $\iota^\comprandvars\ceq \mc\circ f$. Recall that the morphism $\mc$ is unique up to pointed cochain homotopy. Also the quasi-isomorphism $f$ is the first component of the unital $sL_\infty$-morphism $\underline{\w}^\sS$, which is a representative of the homotopy type $[\comprandvars]$. It follows that the first component of another representative of unital $sL_\infty$-morphism differs at most by a pointed cochain homotopy. Hence $f$ is also unique up to pointed cochain homotopy. Then $\iota^\sS$ depends only on the cochain homotopy types of $f$ and $\mc$ and satisfies $\iota^\comprandvars(1_\sS) =1$, since 
\[
\iota^\comprandvars(1_\sS)=\mc\big(f(1_\sS)\big)=\mc\big(1_\sC\big)=1.
\]
Now consider the definition of moment morphism $\m^\comprandvars_n = \mc\circ \Pi^{\underline{\w}^\sS} _n :S^n(\sS)\rightarrow \fieldk$ on $\sS$.
From the relation $\Pi^{\underline{\w}^{\sS}} _n =f\circ M^\sS_n + K\La _n$ in \eq{chwd}, we have
\[
\m^\sS_n = \mc\circ \Pi^{\underline{\w}^\sS} _n =\mc\circ \big( f\circ M^\comprandvars_n\big) +\mc\circ( K\La _n)
\]
which implies that $\m^\comprandvars_n= \iota^\comprandvars\circ M^\comprandvars_n$,
since $\mc\circ K=0$.
We hence have the following.
\begin{theorem}\label{chra}
On a complete space $\comprandvars$ of homotopical random variables, there is the canonical structure $\xymatrix{\sS_{\Comm}\ar[r]^{\iota^\comprandvars}&\fieldk}$ of a \algebraic{} probability space, where $\sS_{\Comm}=\big(\sS, \underline{M}^{\comprandvars}, 0\big)$,
with moment and cumulant morphisms $\underline{\m}^\comprandvars$ and $\underline{\k}^\comprandvars$ such that, for $n\geq 1$ and $s_1,\dotsc, s_n\in \sS$, we have
\begin{align*}
\iota^\comprandvars\left(M^\comprandvars_n(s_1,\dotsc, s_n)\right)
&=\m^\comprandvars_n(s_1,\dotsc, s_n)
\\
&=\sum_{\pi \in P(n)}\ep(\pi)\k^\comprandvars\big(s_{B_1}\big)\dotsc \k^\comprandvars\big(s_{B_{|\pi|}}\big).
\end{align*}

\end{theorem}

Combining the previous two theorems, we have the following.
\begin{theorem}
We have $\text{ho}\!\Des\left(\sS_{\Comm}\right)=\sS_{\Lie}$,
$\text{ho}\!\Des(\iota_\comprandvars)=\underline{\k}^{\comprandvars}$, and $\text{ho}\!\Des([f])= [\comprandvars]$ such that the following diagram commutes:
% \newpage
% \[
% \xymatrix{
% &&\fieldk&&
% \\
% \sS_{\Lie}\ar@{..>}@/_1.5pc/[rrrr]_{\text{ho}\!\Des([f])}\ar@{..>}@/^0.5pc/[rru]^{\underline{\k}^\comprandvars}
%  &\ar@{=>}[l] \sS_{\Comm}\ar[ru]_{\iota^{\comprandvars}}\ar[rr]_{[f]}&&\sC_{\Comm}\ar@{=>}[r] \ar[lu]^{[c]}
%  &\sC_{\Lie}\ar@{..>}@/_0.5pc/[llu]_{\text{ho}\!\Des([c])}
%  \\
% }.
% \]
\[
\xymatrix{
&&&\fieldk&&
\\
&&&\fieldk\ar@{=>}[u]&&
\\
& &\ar@{=>}[dll] \sS_{\Comm}\ar[ru]^{\iota_{\!\comprandvars}}\ar[rr]_{[f]}&&\sC_{\Comm}\ar@{=>}[drr] \ar[lu]_{[c]}&
 \\
\sS_{\Lie}\ar@{-->}[rrrrrr]_{\text{ho}\!\Des([f])}\ar@{-->}[rrruuu]^{\underline{\k}^\comprandvars}
&&&&&&
\sC_{\Lie}\ar@{-->}[llluuu]_{\text{ho}\!\Des([c])}
}.
\]
Here the inner diagram commutes in the homotopy category $\homotopycat\category{\HProb}_{\Comm}{\overk}$ of \padj{} probability algebras and the outer diagram commutes in the homotopy category $\homotopycat\category{\UsL}_\infty{\overk}$ of unital $sL_\infty$-algebras.
\end{theorem}

\subsection{Geometry of finite dimensional complete spaces of homotopical random variables.}\label{subs:geometry II}
Now we consider the case that the cohomology $H$ of the \padj{} probability algebra $\sC_{\Comm}$ is finite dimensional. 
From {\bf Theorem \ref{lemc}}, it follows that
\begin{lemma}
The deformation functor associated with descendant $sL_\infty$-algebra $\sC_{\Lie}$ of the \padj{} probability algebra $\sC_{\Comm}$ is pro-representable by the complete symmetric algebra $\widehat{S(H^*)}$.
\end{lemma} 
It is, then, standard that we have a formal based super manifold $\sM$ whose algebra of functions is the topological algebra $\widehat{S(H^*)}$ and whose tangent space $T_o\sM$ at the base point $o\in \sM$ is isomorphic to $H$. (See \cite{BK} for details on all of this). 
Note that the space of all complete spaces of homotopical random variables can be identified with the space of all gauge equivalence classes of versal solutions to the Maurer--Cartan equation of $\sC_{\Lie}$ parametrised by $\widehat{S(H^*)}$.
Hence the space $\sM$ is the moduli space of complete spaces of homotopical random variables and each such complete space of homotopical random variables $\comprandvars$ gives coordinates at the base point $o \in \sM$.

Consider a complete space of homotopical random variables $\comprandvars$ and let $\xymatrix{\sS_{\Comm}\ar[r]^{\iota^\comprandvars}&\fieldk}$ be the structure of a \padj{} probability space on it, where $\sS_{\Comm}=\big(\sS, \underline{M}^{\comprandvars}, 0\big)$. 

Choose a basis $\{e_\a\}_{\a\in J}$ on $\sS\simeq H$ with distinguished element $e_0=1_\sS$. 
Let $t_{\sS}=\{t^\a\}_{\a\in J}$ be the dual basis, which gives an affine coordinate system on the based moduli space $\sM$ via a unital $sL_\infty$-quasi-isomorphism 
of the homotopy type $[\comprandvars]$ from $\sS_{\Lie}=(\sS, \underline{0})$ to $\sC_{\Lie}$. Then $\left\{\rd_\a \ceq \fr{\rd}{\rd t^\a}\right\}$ is a formal frame field on $T_o\sM$,
and we also use $\rd_\a$ to denote the graded derivation of the formal power-series ring $\fieldk[\![t_\sS]\!]$ satisfying $\rd_\a t^\b = \d_\a{}^\b$.

Define a $(2,1)$-tensor $\{A^\comprandvars_{\a\b}{}^\g\}$ and $1$-tensor $T_\comprandvars^\g$ in $\fieldk[\![t_\sS]\!]$ by
\begin{align*}
A^\comprandvars_{\a\b}{}^\g &\ceq m^{\comprandvars}_{\a\b}{}^\g +\sum_{n=1}^\infty\Fr{1}{n!}\sum_{\r_1,\dotsc,\r_n \in J} t^{\r_n}\cdots t^{\r_1} m^{\comprandvars}_{\r_1\cdots\r_n \a\b}{}^\g,
\\
T_\comprandvars^\g &\ceq \sum_{n=1}^\infty\Fr{1}{n!}\sum_{\r_1,\dotsc,\r_n \in J} t^{\r_n}\cdots t^{\r_1} M^{\comprandvars}_{\r_1\cdots\r_n}{}^\g
,
\end{align*}
where $\{m^{\comprandvars}_{\a_1\cdots\a_n}{}^\g\}$ and $\{M^{\comprandvars}_{\a_1\cdots\a_n}{}^\g\}$ are the structure constants of $m^{\comprandvars}_n$ and $M^{\comprandvars}_n$, respectively.
Then the following is a corollary of 
{\bf Lemmas \ref{flatcoor}}, {\bf\ref{finitec}}, and {\bf\ref{finited}}:
\begin{theorem}
The formal $(2,1)$-tensor $\{A^\comprandvars_{\a\b}{}^\g\}$ is the connection one-form for a graded flat and torsion-free affine connection $\nabla$ on $T\sM$ in the coordinates neighborhood $\{t^\a\}_{\a\in J}$.  That is, define $\nabla_{\rd_\a} \rd_\b =\sum_{\g\in J} A^\comprandvars_{\a\b}{}^\g\rd_\g$; then we have the following:
\begin{align*}
A^\comprandvars_{\a\b}{}^\g - (-1)^{|e_\a||e_\b|} A^\comprandvars_{\b\a}{}^\g =0
;\\
\rd_\a A^\comprandvars_{\b\g}{}^\s -(-1)^{|e_\a||e_\b|}\rd_\b A^\comprandvars_{\a\g}{}^\s
+\sum_{\r\in J} \left(A^\comprandvars_{\b\g}{}^\r A^\comprandvars_{\a\r}{}^\s
-(-1)^{|e_\a||e_\b|}A^\comprandvars_{\a\g}{}^\r A^\comprandvars_{\b\r}{}^\s\right)
=0
,\\
\end{align*}
along with the unit condition $
A^\comprandvars_{0\b}{}^\g =\d_{\b}{}^\g$. 
The formal $1$-tensor $\{T_\comprandvars^\g\}$ satisfies
\begin{align*}
T_\comprandvars^\g\big|_{\underline{t}=\underline{0}} &=0
,\\
\rd_\b T_\comprandvars^\g\big|_{\underline{t}=\underline{0}} &=\d_\b{}^\g
,\\
\end{align*}
with the unit condition $\rd_0 T_\comprandvars^\g= T_\comprandvars^\g +\d_0^\g$.

This means that $\{T_\comprandvars^\g\}$ are the unique affine flat coordinates of this connection. That is, we can write $\rd_\a = \sum_{\g\in J} \rd_\a T_\comprandvars^\g\tilde\rd_\g$ and then due to the above conditions on $T_\comprandvars^\g$, this can be taken as a definition of $\tilde{\rd}_\g$ which satisfies $\nabla_{\tilde\rd_\a}\tilde \rd_\b =0$.
\end{theorem}

% Let $Z_\comprandvars$ be the moment generating function of $\comprandvars$:
% \[
% Z_\comprandvars\ceq 1 + \sum_{n=1}^\infty\Fr{1}{n!} \sum_{\r_1,\dotsc,\r_n \in J} 
% t^{\r_n}\cdots t^{\r_1} \m^{\comprandvars}_{n}\big(e_{\r_1},\dotsc, e_{\r_n}\big)
% \in \fieldk[\![t_\sS]\!]
% .
% \]
\begin{theorem}
The affine flat coordinates $\{T_\comprandvars^\g\}$ determine the moment generating function of $\comprandvars$
\[
Z_\comprandvars\ceq 1 + \sum_{n=1}^\infty\Fr{1}{n!} \sum_{\r_1,\dotsc,\r_n \in J} 
t^{\r_n}\cdots t^{\r_1} \m^{\comprandvars}_{n}\big(e_{\r_1},\dotsc, e_{\r_n}\big)
\in \fieldk[\![t_\sS]\!]
\]
up to the finite unknowns $\{\iota^\comprandvars(e_\g)\}$ in $\fieldk$ by the equation
\[
Z_\comprandvars= 1+\sum_{\g\in J} T_\comprandvars^\g\; \iota^\comprandvars(e_\g).
\]
Furthermore, $Z_\comprandvars$ satisfies the following system of formal differential equations: for $\a,\b \in J$,
\begin{align*}
\left(\rd_\a\rd_\b- \sum_{\g\in J} A^\comprandvars_{\a\b}{}^\g \rd_\g\right)Z_\comprandvars=0,\\
\left(\rd_0-1\right)Z_\comprandvars=0.\\
\end{align*}
\end{theorem}
\begin{proof}
From {\bf Theorem \ref{chra}}, we have $\m^{\comprandvars}_{n}\big(e_{\a_1},\dotsc, e_{\a_n}\big)=
\sum_{\r\in J} M^\comprandvars_{\a_1\cdots \a_n}{}^\r \iota^\comprandvars(e_\r)$. It follows directly that
\[
Z_\comprandvars= 1+\sum_{\r\in J} T_\comprandvars^\r \iota^\comprandvars(e_\r).
\]
This clearly implies that 
\begin{align*}
\rd_\g Z_\comprandvars&= \sum_{\r\in J} \rd_\g T_\comprandvars^\r \iota^\comprandvars(e_\r)
,\\
\rd_\a\rd_\b Z_\comprandvars&= \sum_{\g\in J} \rd_\a\rd_\b T_\comprandvars^\r \iota^\comprandvars(e_\r)
.
\end{align*}
Note that the condition $\nabla_{\tilde\rd_\a}\tilde \rd_\b =0$ is equivalent to the following (see Remark~\ref{remark:intro to affine}):
\[
\rd_\a\rd_\b T_\comprandvars^\g =\sum_{\r\in J} A^\comprandvars_{\a\b}{}^\r \rd_\r T_\comprandvars^\g.
\]
It follows that $\left(\rd_\a\rd_\b- \sum_{\g\in J} A^\comprandvars_{\a\b}{}^\g \rd_\g\right)Z_\comprandvars=0$.
Apply $\rd_0$ to $Z_\comprandvars$ to obtain that
\[
\rd_0 Z_\comprandvars=\sum_{\g\in J}\rd_0 T_\comprandvars^\g\; \iota^\comprandvars(e_\g) 
= \sum_{\g\in J} T_\comprandvars^\g \; \iota^\comprandvars(e_\g) + \d_0^\g \; \iota^\comprandvars(e_\g) 
=1+\sum_{\g\in J} T_\comprandvars^\g \; \iota^\comprandvars(e_\g) = Z_\comprandvars,
\]
where we have used the unit condition $\rd_0 T_\comprandvars^\g= T_\comprandvars^\g +\d_0^\g$ for the second equality.
\qed
\end{proof}

Note that the cohomology $H$ of a \padj{} probability algebra $\sC_{\Comm}$ can be an infinite dimensional graded vector space---recall that the underlying graded vector space of every complete space of homotopical random variables is isomorphic to $H$.
Define a {\em finite super-selection sector} $\comprandvars^\dia$ of a complete space of homotopical random variables $\comprandvars$ as a finite dimensional subalgebra of the \GCCorAname{}
$\big(\sS, \underline{M}^\sS\big)$.

Let $\comprandvars^\dia\subset \comprandvars$ be a finite super-selection sector with the \GCCorAname{} $\big(\sS^{\dia}, \underline{M}^{\sS^{\dia}}\big)$. 
Equivalently $\sS^\dia$ has the structure of a \padj{} probability algebra 
$\sS_{\dia\color{blue}C}=\big(\sS^\dia, \underline{M}^{S^\dia},0\big)$ with zero differential 
whose descendant algebra is a unital $sL_\infty$-algebra 
$\sS^{\dia}_{\color{blue}L}=\big(\sS^\dia, \underline{0}\big)$ with zero $sL_\infty$-structure.
It is obvious that $\sS^{\dia}_{\color{blue}L}$ is a subalgebra of $\sS_{\Lie}=(\sS, \underline{0})$. 
Furthermore any unital $sL_\infty$-morphism $\underline{\w}^\sS$ of the homotopy type 
$[\comprandvars]$ from $\sS_{\Lie}$ to $\sC_{\Lie}$ induces a unital $sL_\infty$-morphism 
$\underline{\w}^{\sS^\dia}$ from $\sS^\dia_{\Lie}$ to $\sC_{\Lie}$. 
In fact $\underline{\w}^{\sS^{\dia}}$ is a descendant of the morphism 
$\w_1^\sS\big|_{\sS^{\dia}}:\sS^{\dia}_{\color{blue}C}\rightarrow \sC_{\Comm}$ when $\underline{\w}^\sS$ is 
a descendant of the quasi-isomorphism $\w_1^\sS:\sS_{\Comm}\rightarrow \sC_{\Comm}$. It follows that
 we have the structure of a \padj{} probability space on $\sS^{\dia}$:
\[
\xymatrix{\sS^{\dia}_{\color{blue}C}\ar[r]^{\iota^{\comprandvars^\dia}}&\fieldk}
\]
where $\iota^{\comprandvars^\dia}\ceq \mc\circ \w_1^{\sS^{\dia}}$.

Choose a basis $\{e_\a\}_{\a \in J^{\dia}}$ of $\sS^{\dia}$ such that $e_0 = 1_{\sS^{\dia}}=1_\sS$. 
Let $t_{\sS^{\dia}}=\{t^\a\}_{\a\in J^{\dia}}$ be the dual basis and 
let $\rd_\a =\fr{\rd}{\rd t^\a}$ be the corresponding graded linear derivation on $\fieldk[\![t_{\sS^{\dia}}]\!]$.
Consider the two formal power series $A^\dia_{\b\g}{}^\s$ and $T_\dia^\g \in \fieldk[\![t_{\sS^{\dia}}]\!]$, defined as follows:
\begin{align*}
A^\dia_{\b\g}{}^\s &\ceq m^\dia_{\b\g}{}^\s 
+\sum_{n\geq 1}\sum_{\r_1,\dotsc,\r_n}\Fr{1}{n!}t^{\r_n}\cdots t^{\r_1}m^\dia_{\r_1\cdots \r_n\b\g}{}^\s
,\\
T_\dia^\g &\ceq t^\g +\sum_{n\geq 2}\Fr{1}{n!}\sum_{\r_1,\dotsc,\r_n}t^{\r_n}\cdots t^{\r_1}M^\dia_{\r_1\cdots \r_n}{}^\g
,
\end{align*}
representing a $(2,1)$-tensor and a $1$-tensor on $\sS^{\dia}$, respectively. 
Then the moment generating function $Z_{\comprandvars^{\dia}}$ of the finite super-selection sector is given by 
\[
Z_{\comprandvars^{\dia}}
= 1 +\sum_{\g\in J^{\dia}}T_{\!\dia}^\g\; \iota^{\comprandvars_{\!\dia}}(e_\g).
\]
Hence the moment generating function is determined by the $1$-tensor $\{T_\dia^\g\}$ up to the finite set 
$\left\{\iota^{\comprandvars_{\!\dia}}(e_\g)\right\}_{\g\in J^{\dia}}$ of unknown numbers. Of course 
$c_0=\iota^{\sS_{\!\dia}}(e_0)=1$. 
Then the following is obvious:
\begin{corollary}
On a finite dimensional super-selection sector $\sS^\dia$, there is a formal power series 
$A^{\dia}_{\a\b}{}^\g\in \fieldk[\![t_{\sS^\dia}]\!]$ representing a $(2,1)$-tensor with the following properties:
\begin{align*}
A^\dia_{0\b}{}^\g =\d_{\b}{}^\g
,\\
A^\dia_{\a\b}{}^\g = (-1)^{|e_\a||e_\b|} A^\dia_{\b\a}{}^\g
,\\
\rd_\a A^\dia_{\b\g}{}^\s -(-1)^{|e_\a||e_\b|}\rd_\b A^\dia_{\a\g}{}^\s
+\sum_{\r\in J_{\!\dia}} \left(A^\dia_{\b\g}{}^\r A^\dia_{\a\r}{}^\s-(-1)^{|e_\a||e_\b|}A^\dia_{\a\g}{}^\r A^\dia_{\b\r}{}^\s\right)
=0,
\end{align*}
such that the moment generating function $Z(t_{\sS^{\dia}})$ satisfies the following system of formal differential equations 
that, for all $\a,\b \in J_{\!\dia}$:
\begin{align*}
\left(\rd_\a\rd_\b - \sum_{\g\in J_{\!\dia}} A^\dia_{\a\b}{}^\g \rd_\g\right)Z_{\comprandvars^{\dia}}=0,\\
\left(\rd_0-1\right) Z_{\comprandvars^{\dia}}=0.\\
\end{align*}
\end{corollary} 
Remark~\ref{remark: doesnt have to be finite superselection sector} applies word for word again here.


\begin{thebibliography}{AB}


\bibitem{Voi1}
D.V.\ Voiculescu,
{\em Symmetries of some reduced free product $C^*$-algebras}. 
Operator Algebras and Their Connections with Topology and Ergodic Theory (Buşteni, 1983), 556--588, 
Lecture Notes in Math., 1132, Springer, Berlin, 1985. 

\bibitem{Voi2}
D.V.\ Voiculescu,
{\it
Limit laws for random matrices and free products.}
Invent. Math. {\bf 104} (1991), no.\ 1, 201--220. 

\bibitem{Sp}
R.\ Speicher,
{\em 
Multiplicative functions on the lattice of non-crossing partitions and free convolution.} 
Math.\ Ann.\ {\bf 298} (1994), no.\ 4, 611--628. 

\bibitem{SW}
R.\ Speicher and R.\ Woroudi, {\em Boolean convolution}. pp.\ 267--279 in Free Probability Theory (Waterloo, ON, 1995), Fields Inst. Commun., 12, Amer. Math. Soc., Providence, RI, 1997. 

\bibitem{Sta63} J.D.\ Stasheff, 
{\em Homotopy associativity of $H$-spaces I, II}. 
Trans.\ Amer.\ Math.\ Soc.\ {\bf 108} (1963), 275--292; ibid. {\bf 108} (1963), 293--312.

\bibitem{SS}
M.\ Schlessinger and J.D.\ Stasheff,
{\em The Lie algebra structure of tangent cohomology and deformation theory.} 
J.\ of Pure and Appl.\ Algebra 38 (1985), 313--322.



\bibitem{P1}
J.-S.\ Park,
{\em Algebraic Principles of Quantum Field Theory I, II}.
{\tt arXiv:1101.4414},
{\tt arXiv:1102.1533}. 


\bibitem{P}
J.-S.\ Park,
{\em Homotopy theory of quantum fields}. 
In preparation.

\bibitem{Kos}
J.-L.\ Koszul,
{\em Homologie et cohomologie des algèbres de Lie.} 
Bull.\ Soc.\ Math.\ France {\bf 78} (1950). 65--127.

\bibitem{P0}
J.-S.\ Park, {\em Homotopical probability space}, November 2, 2011.,
Einstein Chair Mathematics Seminar, CUNY Grad.\ Center.

\bibitem{PP}
J.-S.\ Park and J.\ Park,
{\em Homotopy theory for period integrals of smooth projective hypersurfaces}. 
{\tt arXiv:1310.6710 [math.AG]}.


\bibitem{DPT1}
G. Drummond-Cole, J.-S.\ Park and J.\ Terilla,
Homotopy probability theory I. 
J.\ Homotopy Relat.\ Struct.\ {\bf 10} (2015),  425--435.
{\tt arXiv:1302.3684 [math.PR]}

\bibitem{DPT2}
G. Drummond-Cole, J.-S.\ Park and J.\ Terilla,
Homotopy probability theory II. 
J.\ Homotopy Relat.\ Struct.\ {\bf 10}  (2015), 623--635.
{\tt arXiv:1302.5325 [math.PR]}

\bibitem{RS}
N.\ Ranade and D.\ Sullivan,
{\it
The Cumulant bijection and differential forms}.
{\tt arXiv:1407.0422v2 [math.AT]}.

\bibitem{VDN}
D.V. Voiculescu, K.J.\ Dykema and A.\ Nica, 
{\em Free Random Variables}.
CRM Monograph Series 1, Amer.\ Math.\ Soc.,  Providence, RI,1992

\bibitem{Voi3}
D.V.\ Voiculescu, {\em Lectures on free probability theory},
pp. 279--349,
Lecture Notes in Mathematics, 1738. Springer-Verlag, Berlin, 2000

\bibitem{Tao}
T.\ Tao, {\em 
Topics in Random Matrix Theory}. 
Graduate Studies in Mathematics, 132. American Mathematical Society, Providence, RI, 2012.

\bibitem{NS}
J.\ Novak and P.\ \'{S}niady, 
 {\em What is .. a free cumulant?}. Amer.\ Math.\ Soc.\ Notices, Feb.\ 2011.

\bibitem{NL}
J.\ Novak and M.\ LaCroix,
{\em Three lectures on free probability.}
{\tt arXiv:1205.2097 [math.CO]}.

\bibitem{Quillen}
D.\ Quillen, {\em Rational homotopy theory}.
Ann.\ of Math.\ (2) {\bf 90} (1969), 205--295.

\bibitem{Su77} 
D.\ Sullivan, {\em Infinitesimal computations in topology}. 
Inst.\ Hautes
\'{E}tudes Sci.\ Publ.\ Math. {\bf 47} (1977), 269--331.

\bibitem{Chen} 
K.T.\ Chen, {\em Iterated path integrals}.
Bull.\ Amer.\ Math.\ Soc.\ {\bf 83} (1977), 831--879.

\bibitem{DGMS}
P.\ Deligne, P. \ Griffiths, J.\ Morgan and D.\ Sullivan, 
{\em Real homotopy theory of K\"ahler manifolds}.
Invent.\ Math.\ {\bf 29} (1975), 245--274

\bibitem{Konst} 
M.\ Kontsevich,
{\em Deformation quantization of Poisson manifolds.} Lett.\ Math.\ Phys.\ {\bf 66} (2003), 157--216.
{\tt arXiv:q-alg/9709040}.

\bibitem{BK}
S.\ Barannikov and M.\ Kontsevich,
{\em Frobenius manifolds and formality of Lie algebras of polyvector fields}.
Internat.\ Math.\ Res.\ Notices no.\ 4 (1998), 201--215.
{\tt arXiv:alg-geom/9710032v2}.

\bibitem{Hain}
R.M.\ Hain, {\em Iterated integrals, minimal models and rational homotopy theory}. Ph.D.\ thesis, Univ.\ of
Illinois, 1980.

\bibitem{KadeA}
T.V.\ Kadeishvili,
{\em On the homology theory of fiber spaces}.
Russian Math.\ Surveys, {\bf 35:3} (1980), 231--238.

\bibitem{Gr}
P.A.\ Griffiths,
{\em On the periods of certain rational integrals. I, II.} 
Ann.\ of Math.\ (2) {\bf 90} (1969), 460--495; ibid. (2) {\bf 90} (1969), 496-541. 

\bibitem{KM}
M.\ Kontsevich, 
{\em Intersection theory on the moduli space of curves and the matrix Airy function}.
Commun.\ Math.\ Phys. {\bf 147} (1992), 1--23.

\bibitem{BV1}
I.A.\ Batalin and G.A.\ Vilkovisky,
{\em Gauge algebra and quantization}.
Phys.\ Lett.\ B {\bf 102} (1981), 27--31.

\bibitem{BV2}
I.A.\ Batalin and G.A.\ Vilkovisky,
{\em Quantization of gauge theories with linearly dependent generators}. 
Phys.\ Rev.\ D {\bf 28} (1983), 2567--2582.

\bibitem{BDA}
K.\ Bering, P.H.\ Damgaard and J.\ Alfaro,
{\em Algebra of higher antibrackets}.
Nucl.\ Phys.\ {\bf B478} (1996) 459--504.
{\tt hep-th/9604027}.

\bibitem{Ak}
F.\ Akman,
{\em On some generalizations of Batalin-Vilkovisky algebras}.
{\tt q-alg/950627}.

\rm\bibitem{Krav}
O.\ Kravchenko,
{\em Deformations of Batalin-Vilkovisky algebra}.
{\tt math.QA/9903191}.

\bibitem{KN}
S.\ Kobayashi and K.\ Nomizu,
{\em Foundations of Differential Geometry},
vol.\ I, John Wiley and Sons, New York, 1969.

\bibitem{MH}
M.\ Hardy, 
{\em Combinatorics of partial derivatives}. 
Electron.\ J.\ Combin.\ 13 (2006), \#R1.
{\tt arXiv:math/0601149 [math.CO]}.

\end{thebibliography}
\end{document}